\documentclass[11pt]{amsart}
\usepackage[utf8]{inputenc}
\usepackage{lmodern}
\usepackage{fullpage}

\usepackage{standalone} 
\usepackage{wrapfig}

\usepackage{bbm}

\usepackage{graphicx}
\usepackage{tikz-cd}
\usepackage{verbatim}

\usepackage{amsthm,amsmath,amsfonts,amssymb}
\usepackage[colorlinks=true, allcolors=blue]{hyperref}
\usetikzlibrary{arrows.meta, decorations.pathreplacing, calc}

\definecolor{region1}{RGB}{35, 150, 255}
\definecolor{region2}{RGB}{0, 160, 150}
\definecolor{region3}{RGB}{255, 165, 79}

\newcommand{\N}{\mathbb{N}}

\newcommand{\R}{\mathbb{R}}

\def\R{{\mathbb R}}
\def\N{{\mathbb N}}

\def\eps{\epsilon}

\date{\today}

\def\cal{\mathcal}
\def\tt{\mathtt}

\newcommand{\var}{{\rm Var}}

\newcommand{\E}{\mathbb{E}}

\newcommand{\h}{{\rm h}}
\newcommand{\lA}[1]{\nocolorref{def  hK}{{\rm h}_K(#1)}}
\newcommand{\rk}[1]{\nocolorref{def rk}{{\rm k}_{#1}}} 
\newcommand{\rkr}{\nocolorref{def  rk}{\rk{\rp}}}

\newcommand{\nocolorref}[2]{\hypersetup{linkcolor=black}\hyperref[#1]{#2}\hypersetup{linkcolor=blue}} 

\newcommand{\maxnorm}[1]{\nocolorref{def  max-norm}{\ensuremath{\|#1\|_{\rm max}}}}
\newcommand{\Unorm}[2]{\nocolorref{def  U-norm}{\ensuremath{\|#1\|}}_{#2}}
\newcommand{\Tnorm}[2]{\nocolorref{eq Tnorm}{\ensuremath{\|#1\|}}_{#2 \otimes \Dm{#2}}}

\newcommand{\CE}[1]{\nocolorref{def  E}{\mathbb{E}_{#1}}}
\newcommand{\EE}[1]{\nocolorref{def  E}{\mathbb{E}^{#1}}}
\newcommand{\D}[1]{\nocolorref{def  E}{\mathbb{D}_{#1}}}
\newcommand{\idtt}[1]{\nocolorref{def  idtt}{\ensuremath{{\bf I}_{#1}}}}

\newcommand{\F}[1]{\nocolorref{def  F}{{\cal F}(#1)}} 
\newcommand{\Fz}[1]{\nocolorref{def  F}{{\mathcal F}_0(#1)}}

\newcommand{\cR}[2]{\nocolorref{def  CR}{{\mathcal R}(#1;#2)}}
\newcommand{\dR}[1]{\nocolorref{def  dR}{{\mathcal R}_{\rp}(#1)}}
\newcommand{\T}[1]{{\nocolorref{def  T}{\mathcal T}_K(#1)}}
\newcommand{\Tp}[1]{{\nocolorref{def  T}{\mathcal T}_{K+1}(#1)}}
\newcommand{\TT}[2]{{\nocolorref{def  Tu}{\mathcal T}_K(#1;#2)}}
\newcommand{\Tr}[1]{{\nocolorref{def  Tu}{\mathcal T}_K\big(#1;[0, \rkr - \rk{#1}]\big)}}
\newcommand{\dT}[1]{{\nocolorref{def  Tu}{\mathcal T}_K\big(#1;[-1, \rkr - \rk{#1}-1]\big)}}

\newcommand{\OO}[2]{{\nocolorref{def  Tu}{ A}(#1;#2)}}

\newcommand{\tm}{{\nocolorref{def tm}{{\tt m}}}}
\newcommand{\Dm}[1]{{\nocolorref{def  Du}{D}_{\tm}(#1)}}

\newcommand{\Deg}[1]{{\nocolorref{def Deg}{\mathcal D}_{1}(#1)}}

\newcommand{\hB}{\nocolorref{def  hB}{\ensuremath{\rm h}^\diamond}}

\newcommand{\PK}[1]{\nocolorref{def  PK}{\ensuremath{{\Pi}_{#1}}}}
\newcommand{\PT}[1]{\nocolorref{def  PT}{\ensuremath{{\overline \Gamma}_{#1}}}}
\newcommand{\PDM}[1]{\nocolorref{def PDM}{\ensuremath{{{\bf \Gamma}}_{#1}}}}
\newcommand{\tlam}{\nocolorref{def  varepsilon}{\tilde \lambda}_{\eps}}
\newcommand{\lame}{\nocolorref{def  varepsilon}{\lambda}_{\eps}}

\renewcommand{\eps}{\nocolorref{def  varepsilon}{\ensuremath{\varepsilon}}}

\newcommand{\CC}{\nocolorref{def  CC}{\ensuremath{{\cal C}}}}
\newcommand{\tC}{\nocolorref{def  CC}{\ensuremath{{\tt C}}}}
\newcommand{\tCB}{\nocolorref{prop PT MAIN}{\ensuremath{{\tt C}^\diamond}}}

\newcommand{\CW}[1]{\nocolorref{def  CW}{\ensuremath{\zeta_{#1}}}}
\newcommand{\dW}[1]{\nocolorref{def  dW}{\ensuremath{\zeta_{#1}}}}

\newcommand{\fc}[1]{\nocolorref{def  anc}{\ensuremath{{\frak c}(#1)}}} 

\newcommand{\tCR}{\nocolorref{def  tCR}{\ensuremath{{\tt C}_{\cal R}}}}
\newcommand{\V}[1]{\nocolorref{def Vk}{\ensuremath{V_{#1}(\rho')}}}
\newcommand{\tCM}{\nocolorref{def  tCM}{\ensuremath{{\tt C}_{M}}}}

\newcommand{\bL}{\mathbf{L}} 
\newcommand{\cW}{\mathcal{W}} 

\newcommand{\RDelta}{\nocolorref{def  RDelta}{\ensuremath{\Delta}}}

\newcommand{\Dt}{\nocolorref{def  Dt}{\ensuremath{\Delta}}}

\newcommand{\pvt}[1]{\nocolorref{def pivot}{\ensuremath{u_{#1}}}} 
\newcommand{\rp}{\rho'} 

\newcommand{\pe}{\preceq}

\newcommand{\step}[1]{ \vspace{0.3cm} \noindent  \underline{#1}: \vspace{0.15cm}}

\newcommand{\anc}[2]{\nocolorref{def anc}{{#1^{(#2)}}}}

\newcommand{\etalchar}[1]{$^{#1}$}

\usepackage{xcolor}
\definecolor{hh}{HTML}{5B68FF}
\definecolor{teal}{HTML}{008080}

\usepackage{tikz}
\usetikzlibrary{mindmap}
\usepackage{subfigure}

\newtheorem{theor}{Theorem}[section]
\newtheorem*{theor*}{Theorem}
\newtheorem*{theorem*}{Theorem}
\newtheorem*{prop*}{Proposition}
\newtheorem*{lemma*}{Lemma}

\newtheorem{lemma}[theor]{Lemma}
\newtheorem{cor}[theor]{Corollary}
\newtheorem{defi}[theor]{Definition}
\newtheorem{prop}[theor]{Proposition}
\newtheorem{fact}[theor]{Fact}

\theoremstyle{definition}
\newtheorem{assumption}[theor]{Assumption}
\newtheorem{rem}[theor]{Remark}

\theoremstyle{plain}

\newtheorem{IH}[theor]{Induction Hypothesis}

\title{Optimal Low degree hardness for Broadcasting on Trees}

\author{
  \begin{minipage}[t]{0.45\textwidth}
    \centering
   Han Huang\\
  University of MIssouri \\
  Columbia, MO \\
  \texttt{hhuang@missouri.edu}
  \end{minipage}
  \hfill
  \begin{minipage}[t]{0.45\textwidth}
    \centering
  Elchanan Mossel \\
  MIT \\
  Cambridge, MA \\
  \texttt{elmos@mit.edu}
  \end{minipage}
}

\begin{document}
\maketitle

\begin{abstract}
    Broadcasting on trees is a fundamental model from statistical physics that plays an important role in information theory, noisy computation and phylogenetic reconstruction within computational biology and linguistics. While this model permits efficient linear-time algorithms for the inference of the root from the leaves, recent work suggests that non-trivial computational complexity may be required for inference.

    The inference of the root state can be performed using the celebrated Belief Propagation (BP) algorithm, which achieves Bayes-optimal performance. Although BP runs in linear time using real arithmetic operations, recent research indicates that it requires non-trivial computational complexity using more refined complexity measures. 

    Moitra, Mossel, and Sandon demonstrated such complexity by constructing a Markov chain for which estimating the root better than random guessing (for typical inputs) is $NC^1$-complete. Kohler and Mossel constructed chains where, for trees with $N$ leaves, achieving better-than-random root recovery requires polynomials of degree $N^{\Omega(1)}$. The papers above raised the question of whether such complexity bounds hold generally below the celebrated Kesten-Stigum bound.

    In a recent work, Huang and Mossel established a general degree lower bound of $\Omega(\log N)$ below the Kesten-Stigum bound. Specifically, they proved that any function expressed as a linear combination of functions of at most $O(log N)$ leaves has vanishing correlation with the root. In this work, we get an exponential improvement of this lower bound by establishing an $N^{\Omega(1)}$ degree lower bound, for any broadcast process in the whole regime below the Kesten-Stigum bound.
\end{abstract}
\maketitle

\begingroup
          \renewcommand\thefootnote{}%
          \footnotetext{The extended abstract of this paper was accepted for presentation at the Conference on Learning Theory (COLT) 2025.}%
          \addtocounter{footnote}{-1}%
\endgroup
\section{Introduction}

\paragraph{Broadcasting on Trees.}
In the \emph{Broadcasting on Trees} (BOT) model, one begins with a $d$-ary tree $T$, assigns a label from a finite state space to the root, and then draws each child's label according to a Markov transition matrix $M$ from its parent. This randomness propagates through each layer down to the leaves, and the central task is to reconstruct the root's label purely from observations at the leaves.

The BOT question arose independently statistical physics~\cite{Higuchi:77,Spitzer:75} as a new question related to phase transitions and in 
 \emph{phylogenetics}~\cite{Farris:73,Neyman:71,Cavender:78} where the goal was to infer ancestral genetic traits from current species. BOT also arises in \emph{community detection} and \emph{network clustering}~\cite{DKMZ:11} and follow up work which impose hierarchical random structures on graphs. More generally, it offers a fundamental example of a Markov Random Field on trees, illustrating how noise or mutations along each edge can affect global correlation between the root and leaves. 

 A landmark result in this area was proven in the language of multi-type branching processes by \cite{KestenStigum:66} which identified a threshold for the qualitative behavior of the limiting distribution of the number of leaves of each type. 
 Two key parameters of the model are tree's arity, $d$, and the magnitude of the second eigenvalue, $\lambda$, of the broadcast chain. The Kesten-Stigum (or KS) threshold $d\lambda^2 =1$ is a key threshold in studying BOT and related problems.
 Formally, if $d\lambda^2>1$, then certain simple (degree‐1) estimators achieve non‐trivial correlation with the root; if $d\lambda^2<1$, they fail. This threshold plays a key role in analyzing BOT and its numerous applications.  We refer the reader to~\cite{Mossel:23} for a survey on the BOT problem. 


\paragraph{Low-Degree Complexity vs.\ Linear-Time Algorithms.}
 The low degree method has emerged as a predictive tool for establishing computational hardness for statistical problems.  
This method postulate that \emph{if any good estimator is a high degree function, then the problem is computationally hard} see e.g.~\cite{BHK+19,hopkins2017efficient,hopkins2018statistical,schramm2020computational}. (see also \cite{kunisky2019notes} for a survey.) 
We note that at least some sense the low-degree framework capture some simple classes of algorithms such as local algorithms~\cite{gamarnik2014limits,chen2019suboptimality} and approximate message passing~\cite{MontanariWein2024}.

Recent works~\cite{KoehlerMossel:22,HanMossel:23} examined BOT from the perspective of \emph{low-degree polynomial} methods for root estimation. There are also related work by~\cite{Mossel:19deep,JKLM:19,MoMoSa:20}. 
From the low-degree method's belief, if polynomials of degree $\omega(\log N)$ fail, one believes it is an indication that the underlying problem “\emph{computationally hard}”, as this regime capture many leading algorithmic approaches such as spectral methods, approximate
message passing and small subgraph counts. In the BOT model, the value $N$ refers to the number of leaves in the tree.  
The work of~\cite{KoehlerMossel:22} showed  polynomials of degree $N^c$ for a small $c > 0$ are not able to correlate with the root label (as $\ell$ tends to $\infty$) but in a very special case of $\lambda = 0$, where one can leverage independence between random variables in the BOT model. Our previous work ~\cite{HanMossel:23} showed that for any polynomial of degree $O(\log N)$, the correlation with the root label vanishes as $\ell \rightarrow \infty$ for general chains across all regimes below the Kesten-Stigum (KS) bound, which marks KS-bound as a threshold for low-degree hardness in BOT but fail to pass the $\omega(\log N)$ mark. 

Yet passing $\omega(\log N)$ seems to clashes with BOT’s known \emph{linear‐time} solution—Belief Propagation (BP)—which perfectly reconstructs the root. 
How can a problem with a fast exact algorithm appear “hard” under the low‐degree lens? 
The above series of papers are based on the following premise: While BP runs in linear time, it does require ``depth" below the KS bound~\cite{MoMoSa:20,KoehlerMossel:22,HanMossel:23}.
Interestingly, this behavior parallels certain phenomena in deep learning, where multiple layers are indispensable for success.


\paragraph{Our Contribution}
From a technical perspective, it is challenging to prove low degree hardness in this setting as intuitively, there is an inherent tension between the existence of linear time algorithm and the desire to establish computational lower bounds. Another technical challenge is the global dependency
between variables. We do not know of a useful way to map this problem to a setting of independent random variables as was done in the study of low-degree hardness, see ~\cite{schramm2020computational} and the references
within. In this paper, we resolve the main open question regarding low‐degree hardness for BOT in full generality. 

\begin{theor*}[Informal]
  For any BOT model with \emph{general}(ergodic) Markov Chains below the KS- threshold, any function expressible as a linear combination of functions of at most $N^c$ leaves (for some $c>0$) has \emph{vanishing} correlation with the root.  
\end{theor*}
This matches the performance of Belief Propagation, which uses \emph{all} $N$ leaves and succeeds in linear time. Consequently, we obtain an \emph{optimal} low-degree hardness result for BOT.


Our work provides exponential improvement over the degree lower bound established in our previous work~\cite{HanMossel:23}. Our result shows that the KS-bound represents a sharp transition from feasibility via degree‐1 estimators (simply counting leaf types) \emph{above} KS to exponential‐degree requirements \emph{below} it.
This exponential improvement also ``properly'' establish computatioinal-hardness for low-degree polynomials below the KS-bound. 
Thus, our result marks BOT as an interesting counterexample to the ``computational-hardness'' belief on low-degree analysis.
 As we mentioned earlier this is an analogy of a real phenomena that is observed but not theoretically understood for deep nets. We explain this in a bit more detail below: 

\paragraph{A Viewpoint in Terms of Neural Network Depth.} The Belief Propagation (BP) algorithm on the BOT (Broadcasting on Tree) model can be presented as a feed-forward model with the computational graph the same as the tree. In other words, this can be interpreted as a neural network of depth $\ell$ (one layer per tree level), with a linear number of parameters in the number of leaves (where one uses universal approximation via ReLU gates to approximate the actual activation function at each node).



Our low‐degree hardness results imply that, below the KS threshold, this \emph{depth} might be essential. If we restrict to bounded‐degree polynomial activations (say of degree bounded by $k$), then any network with fewer than $c\ell$ layers can be expressed as a polynomial of degree $\lesssim k^{c\ell}\approx N^{\log(k)\cdot c/\log(d)}$, which is insufficient to correlate with the root from our result. Thus “shallow” networks fail,   
whereas a “deep” network—like one implementing Belief Propagation—succeeds with linear in $N$ pararmeters. 

This viewpoint bridges classical message‐passing on trees with modern insights into the power of depth in neural networks. Proving depth lower bounds for more general activation functions (such as ReLU or Majority gates) remains notoriously difficult, but the low‐degree paradigm offers a conjectural signature of the depth requirement below the threshold.

\subsection{Additional Background}

 A fundamental result in this area \cite{kesten1966additional}, proven by Kesten and Stigum, is that when $d |\lambda|^2 > 1$  nontrivial reconstruction of the root is possible using a linear estimator in the number of the leaves taking different values, whereas when $d |\lambda|^2 < 1$ such linear estimators have no mutual information with the root.

This threshold $d |\lambda|^2 = 1$ is known as the \emph{Kesten-Stigum threshold}.
A series of works showed that the KS threshold is the information theory threshold for non-trivial root inference for some specific channels, including the binary symmetric
channel~\cite{BlRuZa:95,EvKePeSc:00,Ioffe:96a,Ioffe:96b} and binary channels that are close to symmetric~\cite{BCMR:06}, as well as $3 \times 3$ symmetric channels for large $d$~\cite{Sly:09}.

While the Kesten-Stigum bound is easy to compute, it turns our that in many cases, it is {\em not}  the information-theoretic threshold for root recovery.
This was first established in~\cite{Mossel:01} for symmetric channels with sufficiently many states and later shown for symmetric channels with $q \geq 5$ states in~\cite{Sly:09}. Recent results~\cite{MoSlSo:23} provide more information about the case of
$q=3$ and $q=4$. Many of the finer results in this area prove predictions from statistical physics. The connection between the broadcast problems and phase transitions in statistical physics was made in~\cite{MezardMontanari:06}. More recent predictions include~\cite{Moore:17,AbbeSandon:18,RiSeZd:19}. We also note that already~\cite{Mossel:01}
showed that there are many channels where non-trivial inference of the root is possible, yet $|\lambda| = 0$. Much of the interest in Kesten-Stigum threshold comes from the fundamental role it plays in problems, such as algorithmic recovery in the stochastic block model \cite{DKMZ:11,MoNeSl:15,BoLeMa:15,MoNeSl:18,abbe2017community} and phylogenetic reconstruction \cite{Mossel:04a}.

In~\cite{KoehlerMossel:22} it was shown that $\lambda = 0$ even polynomials of degree $N^c$, where $N = d^\ell$ is the number of leaves of for a $d$-ary tree of depth $\ell$, for a small $c > 0$ are not able to correlate with the root label (as $\ell$ tends to $\infty$) whereas computationally efficient reconstruction is generally possible as long as $d$ is a sufficiently large constant~\cite{Mossel:01}.

The main motivation of~\cite{KoehlerMossel:22} was to prove that low degree polynomials fail below the Kesten Stigum bound:
``It is natural to wonder if the Kesten-Stigum threshold $d|\lambda|^2 = 1$ is sharp for low-degree polynomial reconstruction, analogous to how it is sharp for robust reconstruction."
However the main result of~\cite{KoehlerMossel:22} only established this in the very special case of $\lambda = 0$.
This problem is also stated in the ICM 2022 paper and talk on the broadcast process~\cite{Mossel:23}:
``
The authors of~\cite{KoehlerMossel:22} ask if a similar phenomenon holds through the non-linear regime. For example,
is it true that polynomials of bounded degree have vanishing correlation with $X_0$ in the regime where $d \lambda^2 < 1$?
"
In a recent work by ~\cite{HanMossel:23}, it was shown that for any polynomial of degree $O(\log N)$, where $N = d^\ell$ is the number of leaves of a $d$-ary tree of depth $\ell$, the correlation with the root label vanishes as $\ell \rightarrow \infty$.



We note that predictions in statistical physics related the computational complexity of the root inference in BOT to the computational complexity of inference problems related to the block model, see e.g.~\cite{MezardMontanari:06,decelle2011asymptotic} and follow up work.
While we are not aware of conjectures directly relating the low degree hardness of inference in BOT and the low degree hardness of problems in community detection, we note that some of the foundations results on SoS and low degree hardness established low-degree lower bound for problems associated with community detection starting with~\cite{hopkins2017efficient,hopkins2018statistical}.
\subsection{Definitions and Main Result}
\paragraph{Tree Notations}
Let $T$ be a rooted tree with root vertex $\rho$. We define a natural \emph{partial ordering} $\preceq$ on the vertex set, denoted as $V(T)$, as follows:
\begin{align}
    \label{defi partialOrder}
    v \preceq u
\end{align}
for any two vertices $u, v \in V(T)$ if $u$ lies on the (unique) path from $v$ to the root $\rho$. In this case, $u$ is called an \emph{ancestor} of $v$, and $v$ is called a \emph{descendant} of $u$. In particular, if $(v, u)$ is an edge in $T$, then $v$ is called a \emph{child} of $u$, and $u$ is the \emph{parent} of $v$.

For a vertex $u$, the \emph{$k$-th descendants} of $u$ are the set of vertices $v \in V(T)$ such that $v \preceq u$ and the path from $v$ to $u$ contains exactly $k$ edges. If $u$ has no \emph{children}, we call $u$ a \emph{leaf}. The set of leaves is denoted by $L$.
We also refer the set of $k$-th descendants of the root $\rho$ as the \emph{$k$-th layer} of $T$. The depth of $T$, denoted by $\ell$, is the maximum layer of $T$. In the context of broadcasting on trees, the set of leaves $L$ coincides with the $\ell$-th layer of the tree. We assume that every vertex not in the $\ell$-th layer has at least one child.

Additionally, define the height of a vertex $u$, denoted by $\h(u)$, as:
\[
    \h(u) := \ell - \text{(the layer of } u) = \text{graph distance from } u \text{ to the leaves } L.
\]

In this paper, we will consider trees of the following types.
\begin{defi}
    A rooted tree $T$ with root $\rho$ has degree dominated by $d \ge 1$ with parameter $R \ge 1$ if for every vertex $u$ and positive integer $k$, the number of $k$th descendants of $u$ is at most $Rd^k $.
\end{defi}


 \begin{figure}[h]
 
     \centering
     \includegraphics[width=0.7\textwidth]{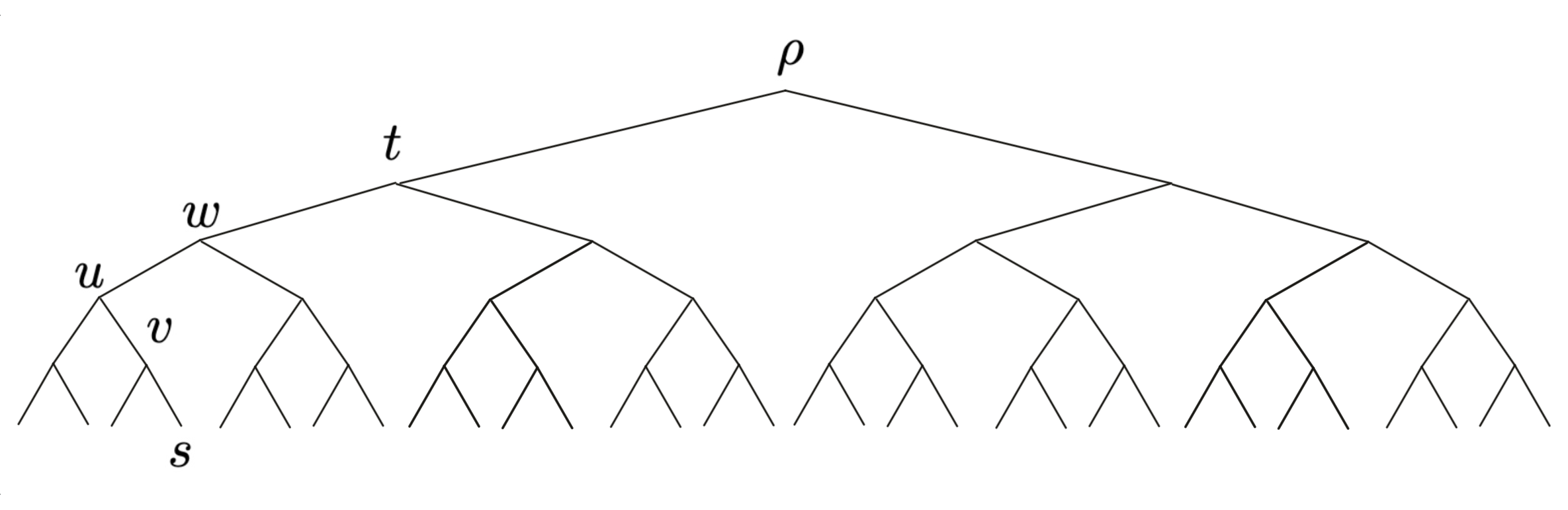} 
     \caption{ An example of a binary rooted tree of depth 5 is shown.
         The vertex $u$ is at the 3rd layer and $\h(u)=2$. Further, the following relationships hold:  $v < u$ and $v$ is a child of $u$,  $s \in L$ is a 2nd descendant of $u$, $w$ is the parent of $u$, and $t$ is the 2nd ancestor of $u$.}
     \label{fig:tree}
 \end{figure}

\paragraph{Broadcasting Process.}
We have a finite state space $[q]$ and an ergodic $q \times q$ transition matrix $M$ with stationary distribution~$\pi$.  The process $(X_u)_{u \in V(T)}$ is defined by drawing $X_\rho \sim \pi$ at the root and then propagating labels down each edge independently according to $M$.  That is, for each $u \in V(T)$ with a child $v$, given $X_u=i$, we draw $X_v=j$ with probability $M_{ij}$. The formal definition (with arbitrary initial distribution) is given below:  
\begin{defi}
    A \emph{broadcasting process} on a rooted tree $T$ is a random process $X = (X_u)_{u \in V(T)}$ with state space $[q]$, transition matrix $M$, and initial distribution $\mu$, defined as follows:
    \[
        \forall x = (x_u)_{u \in V(T)} \in [q]^{V(T)}, \quad
        \mathbb{P}[ X = x ] = \mu(x_\rho) \prod_{(u,v) \in E(T)} M_{x_u x_v},
    \]
    where the product is taken over all edges $(u,v)$ with $v$ being a child of $u$.
\end{defi}
In this rest of the paper, we reserve the notation $X = (X_u)_{u \in V(T)}$ for the broadcasting process on $T$ with initialization $X_\rho \sim \pi$. Note that with this choice of initialization, $X_u \sim \pi$ for all $u \in V(T)$.
\begin{rem} [Markov Property] \label{rem:MarkovProperty}
    The broadcasting process establishes a \emph{ Markov Random Field} on tree $T$: Given any three disjoint subsets $A, B,$ and $C$ of $V(T)$, if every path from a vertex in $A$ to a vertex in $C$ passes through a vertex in $B$, then the random variables $X_A = (x_u)_{u \in A}$ and $X_C = (x_u)_{u \in C}$ are conditionally independent given $X_B = (x_u)_{u \in B}$.
\end{rem}
A natural notion of degree in this setting is:
\begin{defi} [Efron-Stein Degree]
    A function $f$ with variables $x_L = (x_v)_{v\in L}$ is said to have \emph{Efron-Stein} degree at most $k$ if it can be expressed as a finite sum of functions, each depending on no more than  $k$ variables. Formally, this means:
    \begin{align*}
        f(x_L) = \sum_{S\subseteq L, |S| \le k} \phi_S(x_S),
    \end{align*}
    where $\phi_S$ is a function of $x_S$.
\end{defi}
We can now state our main result:
\begin{theor}
    \label{thm: main}
    Consider a broadcasting process $X$ on a rooted tree $T$ with root $\rho$. The tree has $\ell$ layers and its degree is dominated by $d$ with a parameter $R \ge 1$. The transition matrix $M$ is ergodic, and the initial state $X_\rho$ follows the stationary distribution $\pi$.
    Let $\lambda$ be the second largest eigenvalue of $M$ in absolute value.
    If $d\lambda^2<1$, then there exists a constant $c = c(M,d)>0$ such that the following holds:
    For any polynomial $f$ of the leave values $(X_v)_{v\in L}$ with degree bounded by
    $\exp\Big(c\ell / (\log(R)+1) \Big),$
    we have
    \begin{align*}
        \var[ \E[ f(X_{L}) \,|\, X_{\rho}] ]
        \le
        \exp(-c\ell)
        \var[f(X_{L})].
    \end{align*}
\end{theor}

\begin{rem}
    We note that the number of leaves $N := |L| \le Rd^\ell$. Thus, the degree of the polynomials can be as high as a polynomial $N^{c'}$, with the exponent $c'$ depending on $M$ and $d$, but not on the depth of the tree.
\end{rem}
\begin{rem} [Variance Decay implies Vanishing Correlation with Root]
    With the same setting as in Theorem~\ref{thm: main}, for any function $f(x_L)$ of Efron-Stein degree $\le \exp( \frac{c \ell}{\log(R)+1})$, and any function $g(x_\rho)$ of the root value, we can apply conditional expectation and Cauchy-Schwarz inequality to get
    \begin{align*}
        |{\rm Corr}(f(X_L),g(X_\rho)) |
        \le & \exp(-c\ell/2)\,.
    \end{align*}
\end{rem}

\subsection{Proof Overview}
To outline the proof, we begin by introducing some basic notations and definitions.
For any vertex $u \in V(T)$, let $L_u := \{v \in L : v \preceq u\}$ denote the set of leaves that are descendants of $u$.
Further, define
\[
    T_u := \text{the subgraph of } T \text{ induced by the set } \{v \in V(T) : v \preceq u\},
\]
which forms a subtree of $T$ rooted at $u$. Next, for any subset $U \subseteq V(T)$, let $\mathcal{F}(U)$ represent the set of functions of the variables $x_U := (x_v)_{v \in U} \in [q]^U$. For simplicity, we will use the notation:
\[
    \mathcal{F}(u_{\preceq}) := \mathcal{F}( \{ v \in V(T) : v \preceq u \}).
\]

Further, in this paper, we interpret (conditional) expectations as linear maps acting on function spaces:
\begin{defi}
    For each $u \in V(T)$, we define the following linear maps:
    \[
        \CE{u}: \mathcal{F}(u_{\preceq}) \rightarrow \mathcal{F}(u), \quad \EE{u}: \mathcal{F}(u_{\preceq}) \rightarrow \mathbb{R},
        \quad and \quad \D{u} : \mathcal{F}(u_{\preceq}) \rightarrow \mathcal{F}(u)\,.
    \]
    These maps are defined as follows:
    Let $Y = (Y_v)_{v \preceq u}$ be a broadcasting process on the subtree $T_u$ with transition matrix $M$.
    \begin{enumerate}
        \item $\CE{u}$: For each $x_u \in [q]$ and $f \in \F{u_{\preceq}}$,
              \[
                  (\CE{u} f)(x_u) := \mathbb{E}[ f(Y) ] \mbox{ with initialization } Y_u = x_u.
              \]
        \item $\EE{u}$: For $f \in \F{u_{\preceq}}$,
              \[
                  \EE{u} f := \mathbb{E}[ f(Y) ] \mbox{ with initialization } Y_u \sim \pi.
              \]
        \item $\D{u}$: $\D{u}$ is the difference operator, defined as $\D{u} = \CE{u} - \EE{u}$.
    \end{enumerate}
\end{defi}

Note that $(\CE{u}f)(x_u) = \mathbb{E}[ f(X) \mid X_u = x_u]$ for each $f \in \mathcal{F}(u_{\preceq})$, and $\EE{u}f = \mathbb{E}[ f(X) ]$ since $Y_{L_u} \sim X_{L_u}$ due to both $Y_u$ and $X_u$ has the distribution $\pi$.
The significance of this interpretation lies in viewing these expectations as linear maps on function spaces, rather than associating them with a specific distribution of $X$. This distinction is crucial, as we will consider various broadcasting processes on subtrees/subforests throughout the proof.

For any $U \subseteq V(T)$,
we define a natural norm $\maxnorm{\cdot}$ on the spaces $\F{U} = \R^{[q]^U}$ as follows:
\begin{align}
    \label{def  max-norm}
    \maxnorm{\phi} := \max_{\theta \in [q]^U} |\phi(\theta)| \mbox{ for } \phi \in \F{U}\,.
\end{align}
There is a subtle difference between the $\maxnorm{\cdot}$ and the $\ell_\infty$-norm, as the $\ell_\infty$-norm is defined on the support of $X_U$, which might not be the entire space $[q]^U$.  For each $u \in V(T)$, we define the $\ell_2$-norm on $\F{u_{\pe}}$ with respect to $\EE{u}$ as:
\begin{align*}
    \Unorm{f}{u} := \sqrt{\EE{u} f^2} \mbox{ for } f \in \F{u_{\pe}}.
\end{align*}

For discussion of the proof overview, we assume $T$ is a rooted $d$-ary tree of depth $\ell$.

\paragraph{\bf Overall inductive argument.}
For each $K \in \N \cup \{0\}$ and $u \in V(T)$, let
\begin{align*}
    {\cal T}_K(u)  :=
    \left\{
    \mbox{ functions with variables $x_{L_u}$ of degree $\le 2^K$}
    \right\}\,.
\end{align*}

We will choose a suitably small constant \(\varepsilon = \varepsilon(d, \lambda)\). For each \(K \in \mathbb{N} \cup \{0\}\), let \(\h_K\) be the smallest non-negative integer such that the following holds: For every \(u \in V(T)\)
satisfying \(\h(u) \geq \h_K\)
and all \(f, g \in \mathcal{T}_K(u)\):
\begin{align}
    \label{eq introInduction}
    \maxnorm{\D{u}fg}
    =
    \maxnorm{\CE{u} fg - \EE{u} fg}
    \leq
    \exp \big( - \varepsilon (\h(u) - \h_K) \big) \Unorm{f}{u} \Unorm{g}{u}.
\end{align}
The left-hand side of \eqref{eq introInduction} compares the inner products of \(f\) and \(g\) with respect to the law of the broadcasting process on \(T_u\), evaluated under different initializations.
On the right-hand side, the term \(\exp(-\varepsilon \h(u))\) represents an exponential decay associated with the distance from \(u\) to the leaves. The parameter \(\h_K\) serves to offset this decay by accounting for the complexity of the polynomials involved. Consequently, when \(\h(u) - \h_K\) is large, inequality \eqref{eq introInduction} quantifies that
\emph{the inner product on $\T{u}$ introduced by broadcasting process behaves almost identically, regardless of the initial configuration.}

Furthermore, for any polynomial \(f \in \T{\rho}\), applying inequality \eqref{eq introInduction} to \(f - \EE{\rho}f\) and \(g=1\) yields:
$$
    \var[ \E[ f(X_{L}) \,|\, X_{\rho}] ]
    \le
    \maxnorm{ \CE{\rp} f- \EE{\rp}f}^2
    \le
    \exp(- 2\varepsilon(\h(u) - \h_K)) \var[f(X_{L})]\,.
$$

To establish Theorem~\ref{thm: main}, it suffices to demonstrate the existence of a constant \(C = C(M,d) \ge 1\) such that the following two conditions are satisfied:
\begin{align*}
\emph{
    (Base Case) \(\h_0 \le C\), and (Inductive Step) for each \(K \in \N\), \(\h_{K+1} \le \h_K + C\).}
\end{align*}
If these conditions holds, we can select \(K\) to be proportional to \(\ell\) such that \(\h_K \le \ell / 2\), which implies $K \simeq \ell /2 C$. Consequently, the theorem follows.

\paragraph*{\bf Variance decay in the degree 1 case}

It is worth to discuss the proof of the base case, not only because it captures decay below the Kesten-Stigum threshold, but also because it provides insight into how to properly decompose a higher degree polynomial. 
First, every degree $1$ polynomial of the leaves can be expressed in the form
$$
    f = \sum_{u \in L} f_u,
$$
where each $f_u \in \F{u}$. Given our focus on the variance,  we may assume $\E f_u = 0$ for each $u \in L$. (To clarify, $\mathbb{E}$ will always refers to taking expectation with respect to the law of $X$.) Then, our goal is to prove
$\mathbb{E}\big[ (\CE{\rho}f)^2\big]$
is negligible comparing to $\E f^2$.

\begin{align}
    \label{eq intro00}
    \E \big[ (\CE{\rho}f)^2\big]
    \le
    |L| \sum_{u \in L} \mathbb{E}\big[ (\CE{\rho}f_u)^2 \big]
    \lesssim
    d^\ell \sum_{u \in L}  \lambda^{2\ell} \mathbb{E}f_u^2
    =
    (d\lambda^2)^\ell
    \sum_{u \in L} \mathbb{E}f_u^2,\,
\end{align}
where the second inequality is derived from the variance decay property of in a Markov Chain. If we can establish
\begin{align}
    \label{eq intro01}
    \sum_{u \in L}  \E f_u^2 = O\big( \E f^2\big)\,,
\end{align}
then the Base Case is resolved with $\varepsilon$ chosen to satisfy
\begin{align}
\label{eq intro_eps}
\sqrt{d\lambda^2} < \exp( - \varepsilon) < 1.
\end{align}
To see why \eqref{eq intro01} would hold, consider
$u,v \in L$, and $w$ being \emph{the nearest common ancestor} of $u$ and $v$, then
\begin{align*}
    \big|\E \big[ f_u f_v \big]\big|
    =
    \big|\E \big[ \CE{w}f_uf_v]\big|
    =
    \big|\E \big[ \CE{w}f_u \cdot \CE{w}f_v \big]\big|\,,
\end{align*}
where the second equality is due to independence of $X_u$ and $X_v$ given $X_w$. Applying the Cauchy-Schwarz inequality, the above expression is bounded by
\begin{align*}
    (*) \le
    \sqrt{\E (\CE{w}f_u)^2} \sqrt{\E (\CE{w}f_v)^2}
    \lesssim
    \lambda^{2\h(w)} \sqrt{\E f_u^2} \sqrt{\E f_v^2}\,,
\end{align*}
where the last inequality is due to standard decay of a Markov Chain.
Relying on the above inequality, if one carefully sums over all pairs of leaves $u$ and $v$, then the following inequality holds
\begin{align*}
    \Big|\E f^2 - \sum_{u \in L} \E f_u^2\big|
    =
    \frac{1}{2}\Big|\sum_{u \neq v} \E \big[ f_u f_v \big]\Big|
    \le
    \frac{1}{2}\sum_{u \neq v} \Big|\E \big[ f_u f_v \big]\Big|
    \simeq
    \frac{1}{1-d\lambda^2} \sum_{u} \E f_u^2 \,,
\end{align*}
where the term $\frac{1}{1-d\lambda^2}$ arose from the geometric series of summing $(d\lambda^2)^k$ over $k \ge 1$. Clearly, the bound is not strong enough to show \eqref{eq intro01} and indeed some further technical adjustments are needed.
Nevertheless, the computations above still capture two important aspects: (1) the correct decay of correlation between
$f_u$  and  $f_v$  when  $u$  and  $v$  are far apart in graph distance,
and (2) the Kesten-Stigum threshold. These insights will serve as a guideline for establishing the inductive step.

\paragraph*{\bf Decomposition of a $2^{K+1}$ degree polynomial}
When handling functions of degree $\le 2^{K+1}$, we no longer have a natural basis decomposition as shown above for degree 1 functions.
Indeed, the dimension of the space for functions of degree $2^{K+1}$ is at least $ {d^\ell \choose 2^{K+1}}$. At some point the value $2^{K+1}$ will reach exponential in $\ell$, then the dimension become double exponential in $\ell$. \emph{It is hard to get an orthogonal basis of the space, or any rough orthogonal basis which is computationally tractable.} 
Instead, we aim to decompose a function $f$ of degree $2^{K+1}$ into
a form $f = \sum_{u \in V_K} f_u$ such that the decomposition satisfies the properties which are similar to the degree 1 case:
\begin{enumerate}
    \item[$\blacklozenge 1$]
          First, the index set
          $$
              V_K = \{ u \in V(T) \,|\, \h(u) \ge  \h_K + \text{large constant}\}\,
          $$
          is the set of vertices with height greater than $\h_K$ by a large constant.
    \item[$\blacklozenge 2$]
          Second, for each index $u \in V_K$,
          \begin{align}
              \label{eq introDecomposition2}
              \E [ (\CE{\rho} f_u)^2]
              \lesssim
              \lambda^{2(\ell - \h(u))} \exp(-2\eps (\h(u) - \h_K) )  \mathbb{E} f_u^2\,.
          \end{align}
    \item[$\blacklozenge 3$] Third, for $u,v \in V_K$, and $w$ being the nearest common ancestor of $u$ and $v$, then
          \begin{align}
              \label{eq introDecomposition3}
              \big|\E \big[ f_u f_v \big]\big|
              \lesssim
              \lambda^{(\h(w) - \h(u))} \exp(- \eps(\h(u)-\h_K))
              \lambda^{(\h(w) - \h(v))} \exp(- \eps(\h(v)-\h_K))
              \sqrt{\E f_u^2} \sqrt{\E f_v^2}\,.
          \end{align}
\end{enumerate}
Indeed, if such a decomposition exists and $\varepsilon$ is chosen to satisfy \eqref{eq intro_eps},
then, based on the properties outlined above, it is possible to show
$\sum_{u \in V_K} \E f_u^2 = O(\E f^2)$, similar to the degree 1 case.
With these properties in hand, it becomes easier to derive the variance decay of $f$ and extend this to the inductive step with some additional assumption on $\varepsilon$.

\paragraph*{Tensor Products and submultiplicativity of norms}
Let us first show how can we establish some decay properties for simple $\le 2^{K+1}$ degree functions
from the assumption of $\h_K$.
Fix $u \in V_K$ and let $u_1$ and $u_2$ be two children of $u$. Consider a function of the form
$$
    f_u(x_L) = \sum_{i \in I} \phi_{i,1}(x_{L_{u_1}}) \cdot \phi_{i,2}(x_{L_{u_2}})
$$
where $I$ is a finite index set and $\phi_{i,j} \in \T{u_j}$ for $j=1,2$.
In this representation, $f_u$ is a polynomial of degree $\le 2^{K+1}$ and the space of such functions can be identified with the tensor product space $\T{u_1} \otimes \T{u_2} $.

\noindent
\underline{Submultiplicativity of tensor product norms}
Let $g_u$ be another function in $\T{u_1}\otimes \T{u_2}$.
Consider the following term
\begin{align*}
    \CE{u_1}\otimes \CE{u_2} f_ug_u - \EE{u_1}\otimes \EE{u_2}f_ug_u\,,
\end{align*}
which is a function in $\F{u_1} \otimes \F{u_2}$.
By the telescopic sum and the definition of $\D{u_i}$, we have
\begin{align}
    \label{eq introTensorDecompose}
    \CE{u_1}\otimes \CE{u_2} - \EE{u_1}\otimes \EE{u_2}
    =
    \D{u_1} \otimes \D{u_2} +
    \D{u_1} \otimes \EE{u_2} + \EE{u_1}\otimes \D{u_2} \,.
\end{align}

Let us lay out one version of the \emph{submultiplicativity} property of tensor product norms:
Suppose $\bL_i: H_i \times H_i \rightarrow \R^q$ for $i \in [2]$ are two bilinear maps from a Hilbert space $H_i$ to $\R^q$.
Then,
\begin{align*}
    \forall i \in [2],\quad
    \maxnorm{\bL_i (\psi_{i,1},  \psi_{i,2})}& \le  \delta_i \|\psi_{i,1}\| \|\psi_{i,2}\|
    \mbox{ for } \psi_{i,1}, \psi_{i,2} \in H_i\,.                                         \\
                                                  & \Rightarrow  \;
    \maxnorm{\bL_1 \otimes \bL_2 (\psi_1, \psi_2)} \le \delta_1 \delta_2 \|\psi_1\| \|\psi_2\| \mbox{ for } \psi_1, \psi_2 \in H_1 \otimes H_2\,.
\end{align*}

If we apply this to $H_i = \T{u_i}$ with inner product introduced by $\EE{u_i}$
and to each term in \eqref{eq introTensorDecompose}
(with $\delta_i = \exp(-\eps(\h(u_i) - \h_K))$ if $\bL_i = \D{u_i}$, and $\delta_i = 1$ for $\bL_i = \EE{u_i}$), then we have
\begin{align}
    \label{eq introTensorDecay}
    \maxnorm{\CE{u_1}\otimes \CE{u_2} f_ug_u - \EE{u_1}\otimes \EE{u_2}f_ug_u }
    \lesssim &
    \exp(-\eps(\h(u)-\h_K))
    \sqrt{\EE{u_1}\otimes \EE{u_2} f_u^2}
    \sqrt{\EE{u_1}\otimes \EE{u_2} g_u^2} \,.
\end{align}

\noindent
\underline{Probabilistic Interpretation.}
To properly understand the above inequality, 
consider two independent broadcasting processes: \( Y = (Y_v)_{v \pe u_1} \) on the subtree \( T_{u_1} \), with initialization \( Y_{u_1} \sim \pi \),
and \( Z = (Z_v)_{v \pe u_2} \) on the subtree \( T_{u_2} \), with initialization \( Z_{u_2} \sim \pi \).
Then,
$$
    \EE{u_1}\otimes \EE{u_2}f_ug_u =  \E f_u(Y_{L_{u_1}}, Z_{L_{u_2}}) g_u(Y_{L_{u_1}}, Z_{L_{u_2}})\,,
$$
where \( Y \) and \( Z \) are defined as above. On the other hand,
for $(x_{u_1}, x_{u_2})$ such that  \( \mathbb{P}(X_{u_1} = x_{u_1}, X_{u_2} = x_{u_2}) > 0 \), we have
\begin{align*}
    (\CE{u_1}\otimes \CE{u_2} f_ug_u)(x_{u_1}, x_{u_2})
    = &
    \E[ f_u(Y_{L_{u_1}}, Z_{L_{u_2}}) g_u(Y_{L_{u_1}}, Z_{L_{u_2}}) \mid Y_{u_1} = x_{u_1}, Z_{u_2} = x_{u_2}] \\
    = &
    \E[ f_u(X_{L_{u_1}}, X_{L_{u_2}}) g_u(X_{L_{u_1}}, X_{L_{u_2}}) \mid X_{u_1} = x_{u_1}, X_{u_2} = x_{u_2}]\,,
\end{align*}
where the last equality holds since $X_{L_{u_1}}$ and $X_{L_{u_2}}$ are independent given $X_{u_1}=x_{u_1}$ and $X_{u_2}=x_{u_2}$ and they have the same laws as $Y_{L_{u_1}}$ and $Z_{L_{u_2}}$ given $Y_{u_1}=x_{u_1}$ and $Z_{u_2} = x_{u_2}$.
Thus, the inequality \eqref{eq introTensorDecay} suggests that the inner product of $f_u$ and $g_u$ with respect to the law of either \( (Y, Z) \) or \( X \) is roughly the same when \( \h(u) - \h_K \) is large. In particular, this already implies that  $\sqrt{\EE{u_1}\otimes \EE{u_2} f_u^2} \simeq \Unorm{f}{u}$ and $\sqrt{\EE{u_1}\otimes \EE{u_2} g_u^2} \simeq \Unorm{g}{u}$. Using this, along with some algebraic manipulations, we can show that:
\begin{align}
    \label{eq introTensorDecay2}
    \maxnorm{\D{u}f_ug_u} \lesssim \exp(-\eps(\h(u) - \h_K)) \Unorm{f_u}{u} \Unorm{g}{u}.
\end{align}
This result provides a ``toy" version of the inductive step. Further, if we choose $f_u$ to be mean $0$ and $g_u$ simply to be the constant function $1$, then the above inequality leads to \eqref{eq introDecomposition2} in ${\blacklozenge 2}$, where the additional factor of $\lambda^{(\ell - \h(u))}$ arose from the comparison of
$\E(\CE{\rho} f_u)^2$ and $\E(\CE{u}f_u)^2$. In the actual setting, $f_u$ resides in a tensor product space \( \T{A(u)} := \bigotimes_{v \in A(u)} \T{v} \), where \( A(u) \subseteq V(T) \) is an antichain with respect to the $\pe$ partial order of size proportional to $\h(\rho)-\h(u)$. Nevertheless, we rely on the same submultiplicativity property to establish a similar inequality like \eqref{eq introTensorDecay2} for functions in \( \T{A(u)} \).

\paragraph*{\bf A simple case of $f_u$ and $f_v$ satisfying \eqref{eq introDecomposition3} in $\blacklozenge 3$}

Let us extend the previous example slightly.
Fix $u, v \in V_K$ as two incomparable vertices of $T$ (i.e., they do not have an ancestor-descendant relationship). Let $u_1,u_2$ and $v_1,v_2$ be children of $u$ and $v$, respectively.

Consider the space ${\cal R}(u) \subseteq \T{u_1}\otimes \T{u_2}$, which consists of all $ \phi \in \T{u_1}\otimes \T{u_2}$ such that
\begin{align*}
    \EE{u} \phi \psi = 0 \mbox{ for all } \psi \in \T{u}\,.
\end{align*}
Intuitively, $\phi \in {\cal R}(u)$ is a polynomial of degree strictly greater than $2^K$ in $x_{L_u}$ projecting onto the orthogonal complement degree $\le 2^K$ polynomials of $x_{L_u}$.
Similarly, define ${\cal R}(v) \subseteq \T{v_1}\otimes \T{v_2}$ in the same manner.
\vspace{1mm}


 \begin{figure}[h]
     \centering
     \includegraphics[width=0.6\textwidth]{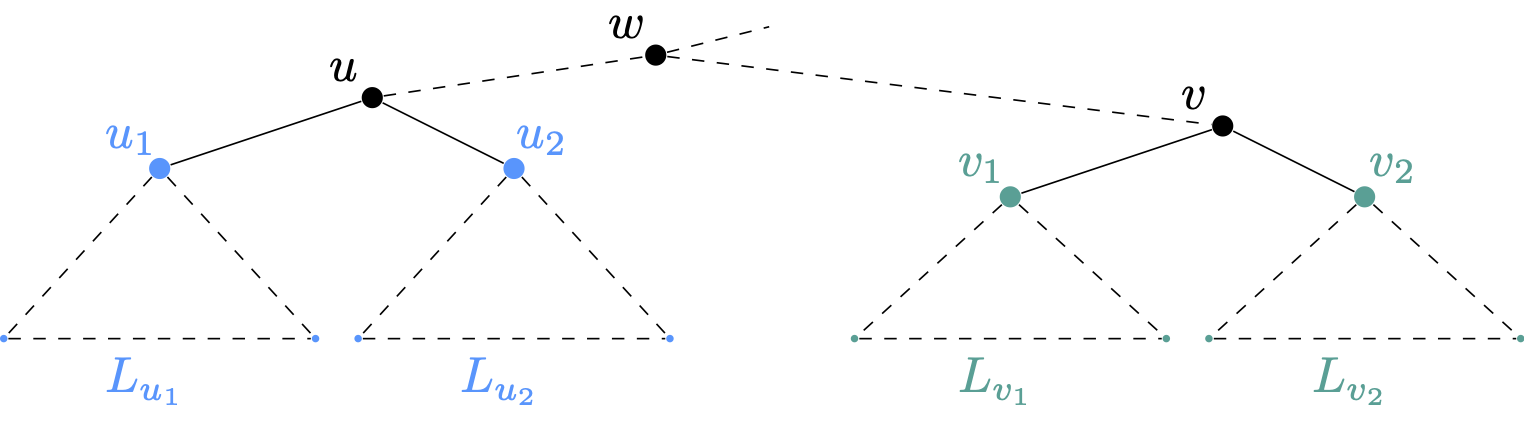}
     \label{fig A}
     \caption{Both $f_u$ and $f_v$ are in variables $x_{L_u}$ and $x_{L_v}$, respectively.
         On the other hand, $f_u$ has higher degrees in $x_{L_u}$ and $f_v$ has higher degrees in $x_{L_v}$.}
 \end{figure}

Now, let  $f_u \in {\cal R}(u) \otimes \T{v}$ and $f_v \in \T{u} \otimes {\cal R}(v)$. In other words, $f_u$ has higher degrees in $x_{L_u}$, but low degrees in $x_{L_v}$, and conversely, $f_v$ has higher degrees in $x_{L_u}$, but lower degrees in $x_{L_v}$. From the definition of ${\cal R}(u)$ and ${\cal R}(v)$, we have
$$
    \EE{u}\otimes \EE{v} f_uf_v = 0 \quad \Rightarrow \quad
    \CE{u}\otimes \CE{v} f_uf_v = \D{u}\otimes \D{v} f_uf_v\,.
$$
From \eqref{eq introTensorDecay2}, we have establish the decay of $\D{u}$ and $\D{v}$ for functions in ${\cal T}(u_1)\otimes {\cal T}(u_2)$ and ${\cal T}(v_1)\otimes {\cal T}(v_2)$, respectively.
We can further apply the submultiplicativity of tensor product norms, leading to the bound
\begin{align}
    \label{eq introTensorDecay3}
    \maxnorm{\CE{u}\otimes \CE{v} f_uf_v}
    \lesssim \exp(-\epsilon(\h(u) - \h_K)) \exp(-\epsilon(\h(v) - \h_K))
    \underbrace{\sqrt{\EE{u}\otimes \EE{v} f_u^2}}_{\simeq \Unorm{f_u}{\rho}}
    \underbrace{\sqrt{\EE{u}\otimes \EE{v} f_v^2}}_{\simeq \Unorm{f_v}{\rho}}\,.
\end{align}
This reveals the origin of the exponential terms in equation \eqref{eq introDecomposition3}.
While this is not the exact form of equation \eqref{eq introDecomposition3}, let us point out that $$ \CE{u} \otimes \CE{v} (f_uf_v) = \mathbb{E}[f_u(X_L) f_v(X_L) \,\vert\, X_u,X_v]
$$ is the conditional expectation given \( X_u \) and \( X_v \). The additional factors \( \lambda^{(h(w) - h(u))} \) and \( \lambda^{(h(w) - h(v))} \) in \eqref{eq introDecomposition3} essentially emerge from taking the conditional expectation $\E[f_u(X_L)f_v(X_L) \mid X_w]$ where
where \( w \) is the nearest common ancestor of \( u \) and \( v \). This captures the decay of the two Markov chains \( (X_u, \dots, X_w) \) and \( (X_v, \dots, X_w) \), in addition to the decay obtained in \eqref{eq introTensorDecay3}.
In the actual setting, the decomposition of $f$ into $f_u$ and $f_v$ is more intricate, where $f_u$ resides in  subspace $\bigotimes_{v \in A(u)} \T{v}$ with

 \begin{figure}[h] 
     \includegraphics[trim=0cm 0.4cm 0cm 0cm, clip, width=0.5\textwidth]{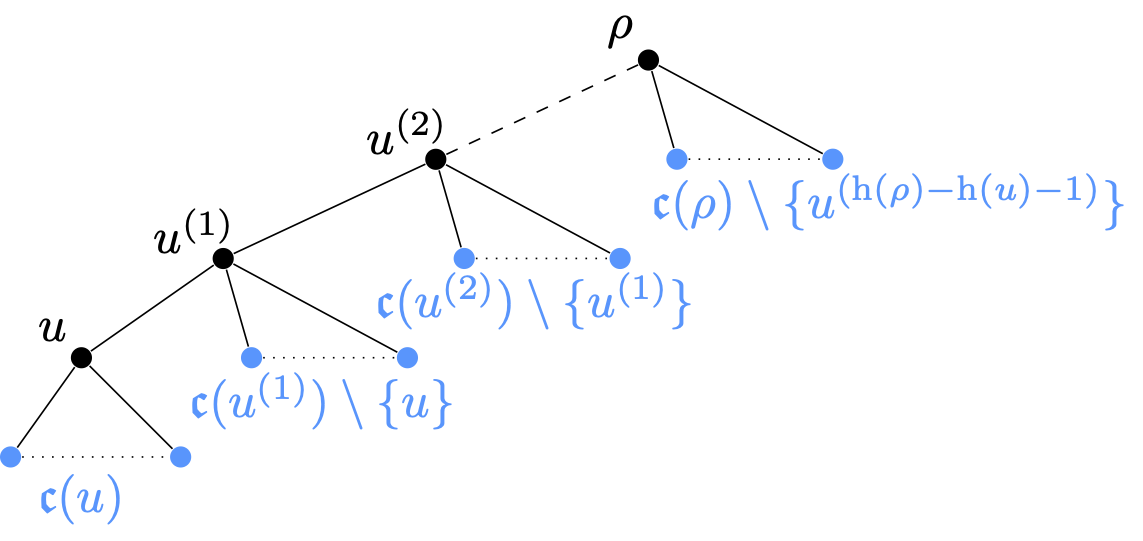}
 \end{figure}
\begin{align*}
    A(u)
    := &
    \Big\{ v \in V(T) \setminus\{u, \, \anc{u}{1},\,  \ldots, \rho\}  \,:\, 
      \in  
    \{u, \, \anc{u}{1},\,  \ldots, \rho\}
     \Big\}                                                     
    =   \fc{u} \cup \bigcup_{k=1}^{\h(\rho)-\h(u)} \Big(\fc{\anc{u}{k}} \setminus \{\anc{u}{k-1}\} \Big)\,,
\end{align*}
\noindent
where, for each $v \in V(T)$, $\fc{v}$ is the set of children $v$. (See the above figure for an illustration.)

\vspace{2mm}
\paragraph*{\bf Comparison to Prior works} 
Our work are based on modification on the inductive and decomposition framework approach from  \cite{HanMossel:23}. The difference arises once we exaime the decomposition of the function $f$ in the inductive step. 

In the work of \cite{HanMossel:23}, the proposed decomposition requires that each component $f_u$ depend solely on the variables $x_{L_u}$. Comparing with the above toy example, the derivation of
\eqref{eq introDecomposition3} follows easily from the fact that $X_{L_u}$ and $X_{L_v}$ are jointly independent given $X_w$, where w is their nearest common ancestor. In short, the their approach relies on \emph{seperation of varaibles to gain correlation decay}.

However, the requirement that each $f_u$ depend solely on $x_{L_u}$ severely limits which functions $f$ can be treated in the inductive step. For instance, even if one establishes decay properties for all polynomials of degree up to $2^K$, applying the inductive argument of \cite{HanMossel:23} only guarantees decay for polynomials of the form
$$
f = \sum_{u \in V_K} f_u, 
\quad 
f_u \in \bigotimes_{v \in \fc{u}} \T{v},
$$
\begin{figure}[h]    
  \includegraphics[width=0.25\textwidth]{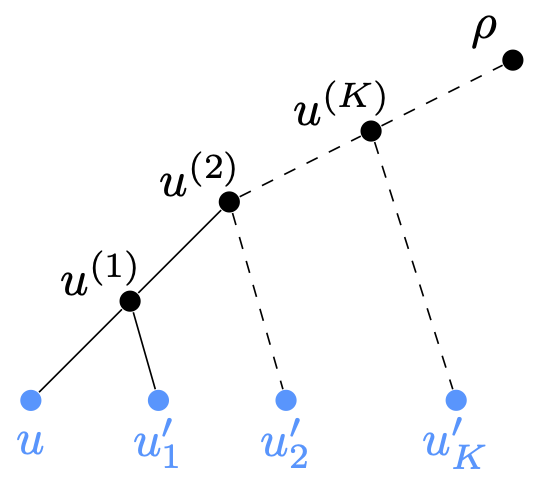}
  \caption{A function of degree $K+1$ with variable $u,u_1,\dots, u_K$ which cannot be handled by prior work \cite{HanMossel:23} when $K$ is beyond $O(\log(N))$.}
\end{figure}

\noindent
which is just a subset of the polynomials of degree $\le 2^{K+1}$. 
Consequently, iterating this procedure $K$ times from degree-1 polynomials covers only polynomials of degree $\le K$. As a concrete example, for any leaf $u \in L$ and leaves $u'_i \in L$ at graph distance $2i$ from $u$, the prior work fails to establish decay for $f = X_u \prod_{i=1}^K X_{u'_i}$ in the first $K$ inductive step, which has degree $K+1$. This essentially leads to the $\Omega(\log(N))$ bound, rather than the $N^{\Omega(1)}$ bound we establish here. 


In contrast, our approach bypasses the constraint of variable separation and instead relies on \emph{separation-degree-density decomposition}. Consequently, we can simply perform induction on $K$, the $\log$ of the degree.


The trade-off is a more intricate analysis, where we need to analyze, for example, $\D{\rho}(f_u f_v)$, with $f_u$ lying in a subspace of $\bigotimes_{w \in A(u)} \T{w}$ and $f_v$ in $\bigotimes_{w \in A(v)} \T{w}$. These two spaces overlap significantly. We address this using an \emph{operator-based analysis}, which is developed based on the basic idea illustrated by the toy example.

\section{Basic Notations, Properties, and Parameter Settings}
For two integers $a < b$, let $[a,b]$ denote the set of integers $\{a, a+1, \ldots, b\}$.
For $u \in V(T)$, let
\begin{align}
    \label{def anc}
    \anc{u}{k} \mbox{ denote the $k$th ancestor of $u$ and } \fc{u} \mbox{ denote the set of children of $u$}.
\end{align}
Further, we abuse the notation and write $\anc{u}{0} = u $.

Let $T$ be a rooted tree and $M$ is an ergodic transition matrix described in Theorem \ref{thm: main}, and
$$
    X = (X_v)_{v \in V(T)} \mbox{ is a broadcasting process on } T \mbox{ with transition matrix } M \mbox{ and initialization } X_\rho \sim \pi\,.
$$

\begin{defi}[Antichains]
    A subset $A \subseteq V(T)$ is called an {\bf antichain} if for any $u,v \in A$, $u \neq v$, we have $u \not \prec v$ and $v \not \prec u$.
    Also, we set
    $$
        A_{\pe} := \{ u \in A\,:\, u \preceq v \mbox{ for some } v \in A'\}\,
        \mbox{ and } A_{\prec} := A_\prec \setminus A \,
    $$
    For an other antichain $A'$,  we write
    $$
        A \pe A'
    $$
    if for every $a \in A$, there exists $a' \in A'$ such that $a \preceq a'$. We will reserves the letter $A$, $A'$, etc. for subsets of $V(T)$ that are antichains.
\end{defi}

\subsection{(Conditional) Expectations as linear operators between function spaces}

\begin{defi}[Space of functions]
    For each set $U \subseteq V(T)$, let
    \begin{align}
        \label{def  F}
        \F{U} := \{ f\,:\, f \mbox{ is a function of } x_{U} \}\,.
    \end{align}
    For simplicity, we will write $\F{u}$ for $\F{\{u\}}$, $\F{u_{\preceq}} = \F{\{v\,:\, v \pe u\}}$.
    Next, we define
    \begin{align*}
        \Fz{u} := \{ f \in \F{u} \,:\, \EE{u}f(X_u) = 0\}\,.
    \end{align*}

    Last, for each non-negative integer $K$, let
    \begin{align}
        \label{def  T}
        \T{u} :=  \left\{ f\,:\, f \mbox{ is a function of variables } x_{L_u} \mbox{ with degree } \le 2^K \right\}\,,
    \end{align}
    and for each antichain $A$,
    \begin{align*}
        \T{A} := \bigotimes_{u \in A} \T{u}\,.
    \end{align*}

\end{defi}
We note that the degree of a function $f \in \T{A}$ can be up to $2^{K|A|}$.

\begin{defi}[Identification of tensor product spaces]
    \label{def  IdentificationXi}
    Suppose ${\cal W}_1,\dots {\cal W}_k$ are subspaces of $\F{V(T)}$, we write
    $$
        \Xi : {\cal W}_1 \otimes {\cal W}_2 \otimes \cdots \otimes {\cal W}_k \rightarrow \F{V(T)}
    $$
    be the multilinear map defined by
    \begin{align}
        \label{eq N_IdentificationOfTensor}
        w_1 \otimes w_2 \otimes \cdots \otimes w_k \mapsto w_1(x)w_2(x)\cdots w_k(x)\,.
    \end{align}
    If there exists $U_1, U_2, \dots, U_k$ which are \emph{disjoint subsets} of $V(T)$ such that ${\cal W}_i \subseteq \F{U_i}$, then the map $\Xi$ is \emph{injective} and, in this paper, if not mentioned, we always identify
    $$
        {\cal W}_1 \otimes {\cal W}_2 \otimes \cdots \otimes {\cal W}_k \equiv \Xi\Big({\cal W}_1 \otimes {\cal W}_2 \otimes \cdots \otimes {\cal W}_k \Big)\,.
    $$
\end{defi}
\begin{rem}
    There is only one occasion where the above identification does not holds, which shows up in Section \ref{sec: CWTK}, and we will explicitly mention it and provide a detailed explanation. For the rest of the paper, we will always identify the tensor product spaces as in Definition \ref{def  IdentificationXi} and it should be clear from the context.
\end{rem}

Given we have identify $\T{A}$ as a subspace of $\F{V(T)}$, here is a basic fact about inclusion relation of two $\T{A}$ spaces: 
\begin{lemma}
    \label{lem TAinclusion}
    Suppose two antichains $A \preceq A'$ satisfies that
    \begin{align*}
         \{L_v \}_{v \in A} \mbox{ is a finer or the same partition as } 
            \{L_v \}_{v \in A'} \mbox{ for the set } \sqcup_{v \in A'}L_v\,.
    \end{align*} 
    Then, 
    \begin{align*}
        \T{A} \subseteq \T{A'}\,.
    \end{align*}
\end{lemma}

\begin{proof}
\step{Basic Case: $A'$ is a singleton}

Suppose $A' = \{a\}$. It is suffice to show that a monomial $\phi \in \T{A'} = \T{a}$ is contained in $\T{A}$. 
Consider any monomial $\phi \in \T{A'} = \T{a}$, 
which is a function of a subset $S \subseteq L_a$ satisfying $|S| \le 2^K$. 
Because $A \preceq A'$, the sets $\{\,S_v\}_{v\in A}$ 
defined by 
$$
    S_v = S \cap L_v
$$
form a partition of $S$. Thus, $\phi$ can equivalently be seen as a function of the variables 
$\bigl\{x_{S_v} \mid v \in A\bigr\}$.

Next, we write \(\phi\) as a sum of indicator functions over all realizations of \(x_S\). Specifically,
\[
  \phi\bigl((x_{S_v})_{v \in A}\bigr)
  \;=\;
  \sum_{(x'_{S_v})_{v \in A}}
    \phi\bigl((x'_{S_v})_{v \in A}\bigr)
    \;\prod_{v \in A}\;
    \mathbbm{1}_{\{x'_{S_v}\}}\bigl(x_{S_v}\bigr)
  \;\in\;
  \bigotimes_{v \in A}\;\T{v}
  \;=\;
  \T{A},
\]
where the sum is taken over all possible realizations \(\bigl(x'_{S_v}\bigr)_{v \in A}\), and \(\mathbbm{1}_{\{x'_{S_v}\}}\) is the indicator function for the event \(x_{S_v} = x'_{S_v}\). Since each factor belongs to \(\T{v}\), their product lies in \(\otimes_{v\in A}\,\T{v} = \T{A}\), completing the proof of the lemma.

\step{Generalization}
We now generalize the argument to show that any function of the form
\[
  \bigotimes_{u \in A'} \phi_u
  \;\in\;
  \bigotimes_{u \in A'} \T{u}
\]
is contained in \(\T{A}\).

For each \(u \in A'\), our assumption on \(A\) and \(A'\) implies that \(L_u\) is partitioned into the sets \(\{\,L_v : v \preceq u,\,v \in A\}\). By the base case established earlier, every \(\phi_u\) belongs to 
\(\T{\{\,v \in A' : v \preceq u\}}\). 
Hence,
\[
  \bigotimes_{u \in A'} \phi_u
  \;\in\;
  \bigotimes_{u \in A'} \T{\{v \in A' : v \preceq u\}}
  \;=\;
  \T{A}.
\]
Given simple tensors formed a basis of $\T{A'}$, we conclude that \(\T{A'} \subseteq \T{A}\).
\end{proof}

In this paper, we will focus on $\T{A}$ for the following types of antichain $A$:
\begin{defi}
    \label{def  TuOu}
    \label{def  Tu}
    For $u \in V(T)$, let
    \begin{align*}
        \OO{u}{k}
        :=
        \begin{cases}
            {\mathfrak c}(u)                                     & \mbox{ if } k = -1,  \\
            {\mathfrak c}(\anc{u}{k+1}) \setminus \{\anc{u}{k}\} & \mbox{ if } k \ge 0.
        \end{cases}
    \end{align*}
    (See Figure \ref{fig:Ouk} for a visual illustration.)
    It is worth to remark every vertex in $\OO{u}{k}$ has height $\h(u)+k$. Next, let
    $$
        \OO{u}{[s,t]} := \bigcup_{k \in [s,t]} \OO{u}{k}
        \mbox{ and }
        \OO{u}{} := \OO{u}{[-1,\infty)}.
    $$
    Further, we define the corresponding space of functions:
    \begin{align*}
        \TT{u}{k}
        :=
        \T{\OO{u}{k}}
        \mbox{ and }
        \TT{u}{[s,t]}
        :=
        \T{\OO{u}{[s,t]}}.
    \end{align*}

    It is also worth to remark two identities that will be used in the paper:
    For $0 \le a \le b$, we have
    \begin{align*}
        \OO{u}{[a,b]} = \OO{\anc{u}{a}}{[0,b]} \Rightarrow \TT{u}{[a,b]} = \TT{\anc{u}{a}}{[0,b]}\,,
    \end{align*}
    and
    \begin{align*}
        \{u\} \cup \OO{u}{[0,b]} = \OO{\anc{u}{1}}{[-1,b-1]} \Rightarrow
        \T{u} \otimes \TT{u}{[0,b]} = \TT{\anc{u}{1}}{[-1,b-1]}\,.
    \end{align*}

\end{defi}

\begin{figure}[h]
    \centering
    \includegraphics[width=0.8\textwidth]{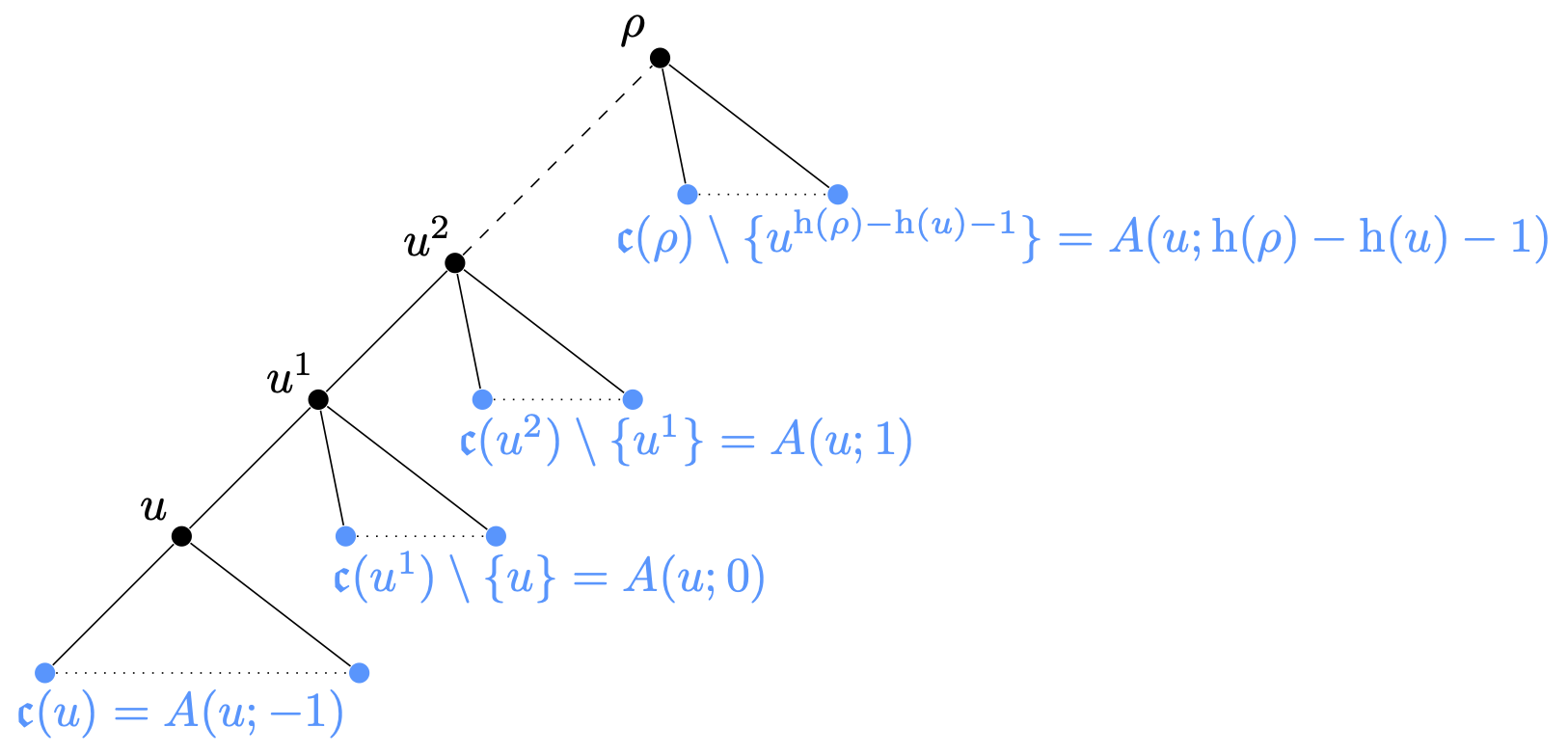}
    \label{fig:Ouk}
    \caption{An illustration of the sets $A(u;k)$.}
\end{figure}

\begin{defi}[Linear Operators]
    For each $u \in V(T)$, we define the linear maps
    \begin{align}
        \label{def  E}
        \CE{u} : \F{u_{\preceq}}  \rightarrow \F{u}\,\,\,\,  \EE{u} : \F{u_{\preceq}}  \rightarrow \mathbb{R}
        \mbox{ and }
        \D{u} := \CE{u} - \EE{u}\,,
    \end{align}
    as follows:
    \begin{enumerate}
        \item[$\CE{u}$:] For each $x_u \in [q]$, let $Y = (Y_v)_{v \le u}$ be the broadcasting process on the substree $T_u$ initialization $Y_u = x_u$. For each $f \in \F{u_{\preceq}}$, set $(\E_u f)(x_u) = \E f(Y)$.
        \item[$\EE{u}$:] Let $Y$ be the broadcasting process on the subtree $T_u$ with $Y_u \sim \pi$, where $\pi$ is the stationary distribution of $M$.  For each $f \in \F{u_{\preceq}}$, $(\E_u f)$ is defined as the expectation of $f(Y)$.
        \item[$\D{u}$:] The difference of $\CE{u}$ and $\EE{u}$, where we simply interpret $\EE{u}f$ as a constant function in $\F{u}$.
    \end{enumerate}
    For an antichain $A \subseteq V(T)$, we denote their tensorization as
    $$
        \CE{A} := \bigotimes_{u \in A} \CE{u},\, \EE{A} := \bigotimes_{u \in A} \EE{u}\,,
        \mbox{ and } \D{A} := \CE{A} - \EE{A}\,.
    $$
    Further, we define
    \begin{align}
        \label{def  idtt}
        \idtt{A} : \F{A_{\pe}} \rightarrow \F{A_{\pe}}\,,
    \end{align}
    to be the identity map.
\end{defi}

\begin{lemma}
    [Basic Properties of $\CE{A}, \EE{A}, $ and $\D{A}$]
    \label{lem basicEECED}
    Consider two antichains $A \pe A'$. For any function $\phi \in \F{A_{\pe}}$,
    \begin{align}
        \label{eq EEU'ECU}
        \CE{A'}\phi = \CE{A'} \circ \CE{A}\phi \,,\,
        \EE{A'}\phi = \EE{A'} \circ \CE{A}\phi \,,\, \text{ and }
        \D{A'}\phi = \D{A'} \circ \CE{A}\phi\,.
    \end{align}

    Further, for $v \prec u$, we have
    \begin{align*}
        \EE{u} f = \EE{v} f \,\, \mbox{ for } f \in \F{v_{\pe}}\,.
    \end{align*}
\end{lemma}
\begin{proof}
    For each $u \in A$, let $T_u$ be the induced subgraph of $T$ with vertex set $\{v\in V(T)\,:\, v \le u\}$ which is a subtree with root $u$. Let $\{Y_{u_{\preceq}}\}$ be independent random broadcasting process where $Y_{u_{\preceq}}$ is the broadcasting process on $T_u$ with $Y_u \sim \pi$. In other words, $Y_{u_{\preceq}} \sim X_{u_{\preceq}}$. For simplicity, let $Y = \{Y_{u_{\preceq}}\}_{u \in A'}$ and
    $$
        Y_{ A_{\not \prec}} = \big\{Y_v\,:\, \not\exists u \in A  \text{ s.t. } v \prec u \big\} \,.
    $$
    Given that $\{v \,:\, \not \exists u \in A \text{ s.t. } v \prec u \} \supseteq A' $,
    for each $\phi \in \F{A_{\pe}}$,
    \begin{align*}
        \E\left[ \phi(Y) \right]
        = &
        \E\Big[
            \E\big[ \phi(Y) \,|\, Y_{A_{\not \prec}}
                \big]\Big]
        \text{ and }
        \E\left[ \phi(Y) \,\vert\, Y_{A'} \right]
        =
        \E\Big[
            \E\big[ \phi(Y) \,|\, Y_{A_{\not \prec}} \big ]
            \Big\vert\, Y_{A'} \Big]\,.
    \end{align*}
    Then, following by the correspondence of $\EE{A'}, \EE{A'}$ and $\CE{A}$ with the above expectation and conditional expectation, we have the desired result
    \begin{align*}
        \EE{A'}\phi = \EE{A'} \circ \CE{A}\phi
        \quad  \mbox{ and } \quad
        \CE{A'}\phi = \CE{A'} \circ \CE{A}\phi \,.
    \end{align*}
    and $\D{A'} \circ \CE{A}\phi = \D{A'}\phi$ follows immediately from the above two identities.

    Now, it remains to show that $\EE{u}f = \EE{v}f$ for $f \in \F{v_{\pe}}$.
    Let $Y_{u_{\pe}}$ be the broadcasting process on $T_u$ with $Y_u \sim \pi$
    and $Z_{v_{\pe}}$ be the broadcasting process on $T_v$ with $Z_v \sim \pi$.
    Observe that $Y_v, Y_{\anc{v}{1}},\cdots Y_u$ is a Markov Chain with $Y_u \sim \pi$. This implies that
    $Y_v \sim \pi$ as well. In other words, $Z_v \sim Y_v$.  Next, we invoke the first part of the proof to conclude that
    $$
        \EE{u}f = \EE{u} (\CE{v}f) = \mathbb{E}[(\CE{v}f)(Y_v)] = \mathbb{E}[(\CE{v}f)(Z_v)]
        = \EE{v}f\,.
    $$

\end{proof}

\begin{defi}[Norms of Spaces]
    \label{def norm}
    For $u \in V(T)$, we define
    \begin{align*}
        \forall f \in \F{u_{\preceq}}\,,  \|f\|_u := \sqrt{\EE{u} f^2}\,,
    \end{align*}
    and for an antichain $A \subseteq V(T)$, we define
    \begin{align}
        \label{def  U-norm}
        \forall f \in \F{A_{\pe}}\,,  \Unorm{f}{A} := \sqrt{\EE{A} f^2}\,.
    \end{align}
    Next, for a set $U \subseteq V(T)$ (not necessarily an antichain), we define
    \begin{align*}
        \maxnorm{f} :=  \max_{x_U \in [q]^U} |f(x_U)|\,.
    \end{align*}
\end{defi}

We remark that, given $A$ is an antichain, $\Unorm{\cdot}{A}$ is the tensorization of $\|\cdot\|_u$ for $u \in A$.

\subsection{Induction Step Setup}
\begin{defi}[Utilizing the gap between $d\lambda^2$ and $1$: Parameters $\eps, \lame, \tlam, \kappa$ and $\CC$]
    \label{def  varepsilon}
    \label{def  kappaBarLambda}
    First, we define $\varepsilon > 0$ to be a small constant such that
    \begin{align}
        \label{defi: epsilon}
        \exp( - 1.2 \eps) =  \sqrt{\max\{ d\lambda^2, \lambda\}}.
    \end{align}
    Then, we define $ \lame > \tlam > \lambda$ to be two values such that
    \begin{align*}
        \exp(-\eps) > &
        \sqrt{\max\{ d\lame^2, \lame\}} = \exp(-1.1\eps)                   \\
        >             & \sqrt{ \max\{ d\tlam^2, \tlam\}} = \exp(-1.15\eps) \\
        >             & \sqrt{\max\{ d\lambda^2, \lambda\}}\,.
    \end{align*}
    The actual values of $1.2$, $1.15$, and $1.1$ are not important, as long as there is a gap between them. The technical reason for introducing $\lame$ and $\tlam$ is later we will encounter geometric series of the following kind:
    \begin{align*}
        \sum_{k=0}^\ell \big( \sqrt{d \lame^2} \big)^k \exp(-\eps (\ell - k))
        \mbox{ or }
        \sum_{k=0}^\ell \lambda^k \exp(-\eps (\ell - k))\,,
    \end{align*}
    the gap allows us to bound the above series by a huge (but independent of $\ell$) constant times $\exp(-\eps \ell)$.

    Further, let $\kappa = \kappa(\eps,d)>0$ be a sufficient small constant such that
    $$
        (1+ \kappa)^6 \tlam \le \lame\,.
    $$
  \end{defi}

  \begin{defi} [$(M,d)$--dependent values]
    \label{def CC}
    We define 
    \begin{align*}
        \CC= {\cal C}(M, d)
        := \{ \tC(M, d) \,:\,  0< \tC(M, d) <+\infty \mbox{ if } d\lambda^2 <1 \},
    \end{align*}
    the collection of $(M, d)$-dependent positive values which are bounded away from $0$ and $\infty$ as long as  $d\lambda^2<1$. For example, $\eps, \lame, \tlam,$ and $\kappa$ are elements in $\CC$.
    As suggested by the notation, we will reserve $\tC, \tC'$ for the elements in $\CC$.
\end{defi}
\begin{rem}
  \label{rem:constant-convention}
  Throughout this paper, whenever we refer to a Lemma or Proposition X whose statement is of the form 
  \emph{“There exists a constant \(\tC \in \CC\) such that \(\dots\)”}, 
  we denote this specific constant by \(\tC_{\rm X}\). In other words, \(\tC_{\rm X}\) always 
  refers to the constant \(\tC\) asserted in the statement of Lemma or Proposition ~X.
  \end{rem}

Let us group the properties we want from the Markov Chain $M$ in the following definition:
\begin{defi}
    \label{def  tCM}
    Let $\tCM \in \CC$ be a large enough constant so that the following holds:
    \begin{enumerate}
        \item  For $f \in \F{u}$,
              \begin{align*}
                  \frac{1}{\tCM} \maxnorm{f} \le
                  \EE{u}|f|
                  \le
                  \Unorm{f}{u} \le \tCM \maxnorm{f}\,.
              \end{align*}
        \item (Markov Chain Decay) For $f \in \Fz{u}$ and $\anc{u}{k}$ exists,
              \begin{align}
                  \label{eq 1variableDecay}
                  \maxnorm{\CE{\anc{u}{k}}f} \le  \tCM\tlam^k \maxnorm{f}\,.
              \end{align}
    \end{enumerate}
\end{defi}

For the sake of completeness, we restate the definition of $\h_K$ in \eqref{eq introInduction} from the introduction, along with a notion of relative height $\lA{u}$ for $u \in V(T)$:
\begin{defi}
    \label{def  hK}
    For each non-negative integer $K$, let $\h_K$ be the smallest
    non-negative integer such that the following property holds: For every $u \in V(T)$ satisfying $\h(u) \ge \h_K$, we have
    \begin{align*}
        \forall f,g \in {\cal T}_K(u), \, \,
        \maxnorm{\D{u}(fg)} \le \exp(-\eps (\h(u) - \h_K)) \|f\|_u \|g\|_u\,.
    \end{align*}
    Further, for simplicity, we will define the relative height of $u$ as
    \begin{align}
        \label{def  hKu}
        \lA{u} := \h(u) - \h_K\,.
    \end{align}
\end{defi}

\begin{defi} [Decay Parameter $\tCR$]
    Let
    \begin{align}
        \label{def  tCR}
        \tCR \in \CC
    \end{align}
    be a large constant (but only depending on $M$ and $d$), with its value to be determined later.
    As we will see soon, the parameter $\tCR$ plays a crucial role: We will always consider $u \in V(T)$ satisfying
    $$
        \lA{u} \ge \tCR (\log(R)+1)\,,
    $$
    which not only ensures some decay for $\D{u}fg$ with $f,g \in \T{u}$, but also provides room for the decay of
    $\D{A}fg $ for some antichain $A \in V(T)$, where the above condition holds for every with $u \in A$, and $f,g \in \T{A}$ or some of its variations. Moreover, the larger the value of $\tCR$, the more room we have for the decay of $\D{A}fg$. On the other hand, the cost we pay for a larger $\tCR$ will be reflected when bounding $\h_{K+1}-\h_K$ in the later stage. However, it can be done in a way so that $\tCR$ only depends on $M$ and $d$.
\end{defi}

Finally, the case $K=0$, or degree 1 polynomials, was already resolved in the work of \cite{HanMossel:23}, let us cite it as a Proposition.
\begin{prop}
    \label{prop K=0}
    With the same assumption as in Theorem \ref{thm: main},
    there exists $\tC_0 \in \CC$ such that
    \begin{align*}
        \h_0 \le \tC_0 (\log(R)+1)\,.
    \end{align*}
\end{prop}

Now, let us state the inductive step as a theorem
\begin{theor}
    \label{theor main}
    With the same assumption as in Theorem \ref{thm: main},  there exists a constant $\tCR \in \CC$ such that for each $K \ge 0$,
    \begin{align*}
        \h_{K+1} - \h_K \le 2\tCR (\log(R)+1)\,.
    \end{align*}
\end{theor}
\begin{rem}
    First, the constant $2$ appeared only for technical reasons, as $\tCR$ also has a different role discussed previously. Second, while this theorem holds for arbitrary $K$, once
    $\h_K$ exceeds $\ell$, the depth of the tree, the statement of the theorem becomes meaningless.
\end{rem}

Finally, the Theorem \ref{thm: main} follows by choosing
\begin{align*}
    K =  \frac{1}{2\tCR (\log(R)+1)} \cdot \frac{\ell}{2} \,.
\end{align*}
The above theorem implies that $\h_K \le \ell/2$.
It implies that for every polynomial $\phi$ of degree at most
$$
    \exp\left( \frac{\ell}{4 \tCR (\log(R)+1)}\right)\,.
$$
Applying the property corresponds to $\h_K$ with $f = \phi - \EE{\rho}\phi$ and $g=1$, we have
\begin{align*}
    ({\rm Var}[ \E[f(X_L)\,|\,X_\rho] ])^{1/2}
    \le
    \max_{x_\rho} \left|\mathbb{E}[ f(X_L) \,\vert\, X_\rho ]\right|
    \le
    \exp(- \eps (\ell - \ell/2)) {\rm Var}[f]^{1/2}  .
\end{align*}

\paragraph*{Organization of the paper}
In Section \ref{sec T}, we introduce some basic properties of submultiplicativity of tensor operator norms specifically for the spaces of our interest. The paper is then divided into two parts: In {\bf Part I}, we define the ${\cal R}(u)$ types of spaces and derive their corresponding properties. In {\bf Part II}, we focus on the structural properties of the decomposition $f = \sum_u f_u$, which leads to the proof of Theorem \ref{theor main}.

\section{Submultiplicativity of tensor operator norms}
\label{sec T}
In this section, we begin by stating two fundamental submultiplicativity properties of tensor operator norms. They form the building block for the arguments of the proof of Theorem \ref{theor main}.

Before stating the first lemma, let us briefly recall how tensor products allow us to combine multiple bilinear maps into a single multilinear map.
Let \( k \) be a positive integer, and for each \( i \in [k] \) let \( H_i^+ \) and \( H_i^- \) be finite-dimensional vector spaces.
Suppose we have bilinear maps
\[
    L_i : H_i^+ \times H_i^- \to W_i \quad \text{for each } i \in [k],
\]
where each \( W_i \) is a finite-dimensional vector space.

Consider the following map from
$H_1^+ \times H_2^+ \times \cdots \times H_k^+ \times H_1^- \times H_2^- \times \cdots \times H_k^-$
to \( W_1 \otimes W_2 \otimes \cdots \otimes W_k \):
\[
    (f_1,\dots, f_k , g_1, \dots, g_k) \mapsto
    \bigotimes
    L_i(f_i, g_i)
\]

This map is multilinear, and by the universal property of tensor products, it extends uniquely to a linear map
\[
    \bigotimes_{i \in [k]} H_i^+ \otimes
    \bigotimes_{i \in [k]} H_i^- \;\to\;
    \bigotimes_{i \in [k]} W_i.
\]

We may also view \(\bigotimes_{i \in [k]} H_i^+\) and \(\bigotimes_{i \in [k]} H_i^-\) themselves as vector spaces. Applying the universal property in the opposite direction, we obtain a bilinear map
\[
    \bigotimes_{i \in [k]} L_i :
    \left(\bigotimes_{i \in [k]} H_i^+\right) \times \left(\bigotimes_{i \in [k]} H_i^-\right)
    \;\to\; \bigotimes_{i \in [k]} W_i,
\]
which on pure tensors is given by
\[
    \bigotimes_{i \in [k]} L_i( f_1 \otimes \cdots \otimes f_k, \; g_1 \otimes \cdots \otimes g_k )
    = \bigotimes_{i \in [k]} L_i(f_i, g_i).
\]

With this background in mind, we now present the following lemma.

\begin{lemma}
    \label{lem: mainTensorProduct}
    Let $k$ be a positive integer, and for each $i \in [k]$ let
    $H_i^+$ and $H_i^-$ be two finite-dimensional vector spaces equipped with symmetric, semi-positive definite bilinear forms $E_i^+$ and $E_i^-$, respectively.

    Let $U_1, \dots, U_k$ be finite sets, and consider bilinear maps
    $L_i: H_i^+ \times H_i^- \to \R^{U_i}$.  Suppose there exists $\delta_i >0$ such that for all $f \in H_i^+$ and $g \in H_i^-$,
    \begin{align*}
        \|L_i(f,g)\|_{\rm max} \le \delta_i \sqrt{E_i^+(f,f)} \sqrt{E_i^-(g,g)} \,.
    \end{align*}

    Then, for all $f \in \bigotimes_{i \in [k]} H_i^+$ and $g \in \bigotimes_{i \in [k]} H_i^-$,

    \begin{align*}
        \|\bigotimes_{i \in [k]} L_i (f,g)\|_{\rm max} \le \prod_{i \in k}\delta_i \sqrt{\bigotimes_{i \in [k]} E_i^+(f,f)}
        \sqrt{\bigotimes_{i \in [k]} E_i^-(g,g) }\,.
    \end{align*}
\end{lemma}

To build intuition, we now consider a simplified setting and connect it to the familiar notion of submultiplicativity of operator norms.

In a simpler scenario where $H_i^+ = H_i^- = \mathbb{R}^n$ with the standard inner product and each $U_i$ a signleton, each $L_i$ reduces to a linear operator represented by an $n \times n$ matrix. In this case, the condition that on $L_i$ is simply translates to saying that $L_i$ has operator norm at most $\delta_i$. The lemma then asserts that the operator norm of the tensor product of the $L_i$s is at most the product of the $\delta_i$s, reflecting the standard submultiplicativity of matrix norms.

If $U_i$ is not a singleton but remains finite, we consider each $a \in U_i$. For each such $a$, the bilinear form $f,g \mapsto L_i(f,g)(a)$ corresponds to a matrix $L_{i,a}$. The lemma's assumption ensure that all these matrices $L_{i,a}$ have operator nomrs uniformly bounded by $\delta_i$.
Correspondly, the lemma's assertion is that for each $(a_1,a_2,\dots,a_k) \in U_1\times U_2 \cdots \times U_k$, the operator norm of the tensor product of $L_{i,a_i}$ is bounded by $\prod_{i \in k}\delta_i$. The rest of the generalization is also straightforward, but we leave it to the appendix.

Given the proof is an exercise in linear algebra, we defer it to the appendix.Further, let us state the second lemma about submultiplicativity of tensor product norms:
\begin{lemma}
    \label{lem tensorNorm}
    Consider $k$ finite-dimensional vector spaces $H_1, H_2, \dots, H_k$, each equipped with a symmetric, semi-positive definite bilinear form
    $E_i: H_i \times H_i \to \R$ for $i = 1,2,\dots, k$.
    Suppose there exists linear operators $L_i: H_i \rightarrow  H_i$ for $i \in [k]$ such that
    \begin{align*}
        E_i(L_i(f), L_i(f)) \le \delta_i E_i(f,f) \mbox{ for every } f \in H_i\,,
    \end{align*}
    for some $\delta_i \ge 0$.
    Then it follows that
    \begin{align*}
        \bigotimes_{i=1}^k E_i \left(
        \bigotimes_{i=1}^k L_i f, \bigotimes_{i=1}^k L_i f \right)
        \le \prod_{i=1}^k \delta_i
        \bigotimes_{i=1}^k E_i(f,f) \mbox{ for every } f \in \bigotimes_{i=1}^kH_i \,.
    \end{align*}
\end{lemma}
We will also defer the proof of this lemma to the appendix.

In the second part of this section, we will apply the lemmas discussed above to the spaces of interest—specifically, the spaces of functions defined on the vertices of the tree. The following result follows as a straightforward consequence of Lemma \ref{lem: mainTensorProduct}.
\begin{lemma}
    \label{lem TU}
    Consider an antichain $A \subseteq V(T)$ where every $u \in A$ has height satisfying $\h(u) \ge \h_K$. For any $f,g \in \T{A}$, the expression $\CE{A}fg$ can be decomposed as follows:

    \begin{align*}
        (\CE{A} fg) (x_A)
        =
        \sum_{A' \subseteq A} a_{A'}(x_{A'})\,,
    \end{align*}
    where each $a_{A'} \in \F{A'}$ satisfies
    \begin{align*}
        \maxnorm{a_{A'}}
        \le \prod_{u \in A'}\exp( - \eps \h_K(u)) \Unorm{f}{A}\Unorm{g}{A}\,.
    \end{align*}
    Furthermore, we have $a_{\emptyset} = \EE{A}fg$ and
    \begin{align*}
        \maxnorm{\D{A}fg}
        \le
        \left( \sum_{\emptyset \neq A' \subseteq A} \prod_{u \in A'}\exp( - \eps \h_K(u)) \right) \Unorm{f}{A}\Unorm{g}{A}\,.
    \end{align*}
\end{lemma}
\begin{proof}
    We start with the decomposition
    \begin{align*}
        \CE{A}
        =
        \bigotimes_{u \in A} \CE{u}
        =
        \bigotimes_{u \in A} \left( \EE{u} + \D{u} \right)
        =
        \sum_{A' \subseteq A} \underbrace{\bigotimes_{u \in A'} \D{u} \bigotimes_{u \in A \setminus A'} \EE{u}}_{:=\bL_{A'}}\,.
    \end{align*}

    For each subset $A' \subseteq A$, define
    \begin{align*}
        a_{A'}
        :=
        \bigotimes_{u \in A'} \bL_{A'}fg \,.
    \end{align*}
    By construction, $a_{A'}$ is an element of
    \[
        \bigotimes_{u \in A'} \F{u} \otimes \bigotimes_{u \in A \setminus A'} \R  = \F{A'}\,.
    \]

    Consider $u \in A$.
    By definition of $\lA{u}$,
    \begin{align*}
        \maxnorm{\D{u}f_ug_u}
        \le
        \exp(-\varepsilon \lA{u}) \Unorm{f_u}{u} \Unorm{g_u}{u}\,,
        \quad
        f_u,g_u \in \T{u}\,,
    \end{align*}
    and by Cauchy-Schwarz inequality, we also have
    $$
        \maxnorm{ \EE{u}f_ug_u}
        =
        |\EE{u} f_u g_u|
        \le
        \Unorm{f_u}{u} \Unorm{g_u}{u}\,,
        \quad
        f_u,g_u \in \T{u}\,,
    $$
    Applying Lemma~\ref{lem: mainTensorProduct}, we can “tensorize” these inequalities. For each subset $A' \subseteq A$, when we combine the factors corresponding to each $u \in A'$, the resulting inequality for $\bL_{A'}(f,g)$ reads
    \[
        \maxnorm{\bL_{A'}fg} \le \left(\prod_{u \in A'} \exp(-\eps \lA{u})\right)\Unorm{f}{A}\Unorm{g}{A}.
    \]
    Thus, we obtain the desired bound for each $a_{A'}$.

    Finally, the second statement of the lemma follows directly from the first one, using the fact that $a_{\emptyset} = \EE{A}fg$ and recalling $\D{A} = \CE{A} - \EE{A}$.

\end{proof}

Now, we extend the above lemma dedicated to the types of antichain we are interested in, namely any
$$
    A \subseteq \OO{u}{} \mbox{ for } u \in V(T)\,.
$$
(See Definition \ref{def  TuOu} for the definition of $\OO{u}{}$.)

\begin{lemma}
    \label{lem basicDecayAu}
    There exists $\tC \in \CC$ such that the following statement holds for sufficiently large $\tCR$:
    Let $u \in V(T)$ satisfy $\lA{u} \geq \tCR (\log(R)+1)$,
    and let $A \subseteq \OO{u}{}$. Then, for any functions $f,g \in \T{A}$ we have
    \begin{align*}
        \maxnorm{\D{A}fg}
        \le
        \tC R \exp(-\eps \lA{u}) \Unorm{f}{A} \Unorm{g}{A}\,.
    \end{align*}
    and
    \begin{align*}
        \maxnorm{\CE{A}fg}
        \le
        \left( 1 + \tC R \exp(-\eps \lA{u}) \right) \Unorm{f}{A} \Unorm{g}{A}\,.
    \end{align*}
\end{lemma}

\begin{proof}
    First, as an immediate consequence of Lemma \ref{lem TU}, we have
    \begin{align*}
        \maxnorm{\D{A}fg}
        \le
        \left( \sum_{\emptyset \neq A' \subseteq A} \prod_{u \in A'}\exp( - \eps \h_K(u)) \right) \Unorm{f}{u}\Unorm{g}{u}\,.
    \end{align*}

    Next, we estimate this sum.

    For each integer $t \ge -1$, the set $A \cap \OO{u}{t}$ is contained in ${\mathfrak c}(u)$, and thus can have at most $Rd$ vertices at height $\h(u') = \h(u) + t$. Since $A \subset \OO{u}{}$, we derive
    \begin{align*}
        \left( \sum_{\emptyset \neq A' \subseteq A} \prod_{u \in A'}\exp( - \eps \h_K(u)) \right)
        =   &
        \prod_{u' \in A} \left( 1 + \exp(-\eps \lA{u'}) \right)  - 1                  \\
        \le &
        \prod_{t=-1}^\infty \Big( 1 + \exp \big(-\eps (\lA{u}+t)\big) \Big)^{Rd}  - 1 \\
        \le &
        \exp\left(
        \sum_{t=-1}^\infty \exp(-\eps (\lA{u}+t))Rd
        \right)
        -1\,,
    \end{align*}
    where the last inequality follows from the fact that $(1+s) \le \exp(s)$ for $s \in \mathbb{R}$.
    Next, we consider the assumption in the lemma that
    $$
        \lA{u} \ge \tCR (\log(R)+1) \mbox{ and } \tCR \mbox{ is sufficiently large.}
    $$
    If $\tC$ is a constant multiple of $ \frac{\exp(\eps)}{1-\exp(-\eps)}d$, then
    \[
        (*) \le \tC R \exp(-\eps \lA{u})\,.
    \]

    To prove the second inequality, note that by the Cauchy–Schwarz inequality,
    \[
        |\EE{A}fg| \le \Unorm{f}{A}\Unorm{g}{A}.
    \]
    Hence, with
    \[
        \maxnorm{\CE{A}fg}
        = \maxnorm{\EE{A}fg + \D{A}fg}
        \le \maxnorm{\EE{A}fg} + \maxnorm{\D{A}fg}\,,
    \]
    the second statement follows.

\end{proof}

In particular, having established controlled decay (Lemma \ref{lem basicDecayAu}), we can now use these estimates to compare norms defined on different antichains. The next lemma shows that if the perturbation operator \(\D{A}\) is sufficiently small, then norms defined via different antichains remain close to each other.

\begin{lemma}
    [Norm comparison]
    \label{lem normComparison}
    For any given antichain $A \subseteq V(T)$, suppose ${{\cal W}} \subseteq \F{A_{\pe}}$ is a subspace of functions such that
    $$
        \maxnorm{ \D{A}fg }
        \le
        c \Unorm{f}{A} \Unorm{g}{A} \mbox{ for } f,g \in {{\cal W}},
    $$
    for some $c \in  (0,1/2)$.
    Then, for any antichain $A' \subseteq V(T)$ such that $A \pe A'$,
    $$
        \frac{1}{1+c} \|f\|_A \le \|f\|_{A'} \le (1+c) \|f\|_A \mbox{ for } f \in {{\cal W}}.
    $$
\end{lemma}
\begin{proof}
    First, we invoke Lemma \ref{lem basicEECED} to get
    $$
        \EE{A'} \circ \CE{A} \phi = \EE{A'}\phi \mbox{ for } \phi \in \F{A_{\pe}}\,.
    $$

    Then, for any $f \in \F{A_{\pe}}$, we have
    $$
        \big| \|f\|_{A'}^2 - \|f\|_A^2 \big|
        =
        |\EE{A'} \CE{A}f^2 - \EE{A}f^2|
        \le
        \EE{A'} |\CE{A}f^2 - \EE{A}f^2|
        =
        \EE{A'} |\D{A}f^2|
        \le
        \maxnorm{\D{A}f^2}
        \le
        c \Unorm{f}{A}^2,
    $$
    which in turn implies
    \begin{align*}
        (1-c) \Unorm{f}{A}^2 \le \Unorm{f}{A'}^2 \le (1+c) \Unorm{f}{A}^2.
    \end{align*}
    Now, taking square root on each term above, and together with the fact that
    $$
        \sqrt{1-c} \le \frac{1}{1+c} \mbox{ and } \sqrt{1+c} \le 1+ c
        \mbox{ for } c \in (0,1/2)\,,
    $$
    the lemma follows.
\end{proof}

Now we state a specific version of norm comparisons for the spaces of interest.

\begin{lemma}
    \label{lemma: coreNorm}
    The following statement holds for sufficiently large $\tCR$.
    Let $u \in V(T)$ satisfy $\lA{u} \ge \tCR (\log(R)+1)$,
    and let $k \in \mathbb{N}$ be an integer for which $\anc{u}{k}$ is well-defined.
    Then, for any antichain $A'$ such that
    \begin{align*}
        \{u\} \cup  \OO{u}{[0,k-1]} \pe  A',
    \end{align*}
    the following inequality holds: for each $f \in \F{u_{\pe}} \otimes \TT{u}{[0,k-1]}$,
    \begin{align*}
        \frac{1}{\sqrt{1+\kappa}}
        \Unorm{f}{\{u\} \cup \OO{u}{[0,k-1]}}
        \le
        \Unorm{f}{A'}
        \le
        \sqrt{1+\kappa}
        \Unorm{f}{\{u\} \cup \OO{u}{[0,k-1]}} .
    \end{align*}
\end{lemma}
\begin{rem}
    Although we do not use the following observation,
    any set $A'$ satisfying the condition in the lemma is necessarily of the form
    \begin{align*}
        \{\anc{u}{s}\} \cup \OO{u}{[s,k-1]}
        =
        \{\anc{u}{s}\} \cup \OO{\anc{u}{s}}{[0,k-s-1]}
        \mbox{ for some } s \in [0,k]\,.
    \end{align*}
\end{rem}
\begin{rem}
    \label{rem: coreNorm}
    If $A''$ is another antichain satisfies the description of $A'$ above, then invoking the Lemma twice (applied to $A'$ and $A''$) we have
    \begin{align*}
        \frac{1}{1+\kappa}
        \Unorm{f}{A''}
        \le
        \Unorm{f}{A'}
        \le
        (1+\kappa)
        \Unorm{f}{A''}.
    \end{align*}
\end{rem}
\begin{proof}
    For simplicity, let $A := \TT{u}{[0,k-1]}$.
    We invoke Lemma \ref{lem basicDecayAu} to show that for
    $\phi \in \TT{u}{[0,k-1]}$,
    \begin{align*}
         &                                       &
         & \maxnorm{\CE{A}\phi^2 - \EE{A}\phi^2}
        \le
        \underbrace{\tC R
            \exp( - \eps \lA{u})}_{:=\delta} \EE{A}\phi^2
        \\
         & \Leftrightarrow                       &
         & (1- \delta)  \EE{A}\phi^2
        \le
        (\CE{A}\phi^2)(x_A)
        \le
        (1+ \delta) \EE{A}\phi^2 \mbox{ for every }x_A \in [q]^A\,.
    \end{align*}

    \step{Random processes corresponding to $\EE{A}$ and $\EE{A'}$}

    Let \(Y = (Y_{v_{\pe}})_{v \in A}\) be a family of independent random processes, where each \(Y_{v_{\pe}}\) is a broadcasting process on the subtree \(T_v\) initialized by \(Y_v \sim \pi\). Likewise, let \(Y' = (Y'_{v_{\pe}})_{v \in A'}\) be the analogous random processes for \(A'\), defined independently from \(Y\).

    Now, fix \(f \in \F{u_{\pe}} \otimes \TT{u}{[0,k-1]}\) and choose any \(x_{u_{\pe}} \in [q]^{u_{\pe}}\). Define
    \[
        \phi(x_{A_{\pe}})
        := f\bigl(x_{u_{\pe}},\,x_{A_{\pe}}\bigr)
        \;\in\;
        \TT{u}{[0,k-1]}.
    \]
    In particular,
    \[
        \EE{A} \phi^2
        \;=\;
        \E\bigl[f^2(x_{u_{\pe}},Y)\bigr]
        \quad\text{and}\quad
        \CE{A}\phi^2(x_A)
        \;=\;
        \E\!\bigl[f^2(x_{u_{\pe}},Y)\,\big\vert\,Y_A = x_A\bigr],
    \]
    where \(Y_A = (Y_v)_{v \in A}\). Hence, the inequality from the previous step implies that for each \(x_{u_{\pe}}\) and \(x_A\),
    \[
        (1-\delta)\,\E\bigl[f^2(x_{u_{\pe}},Y)\bigr]
        \;\;\le\;\;
        \E\!\Bigl[f^2(x_{u_{\pe}},Y)\,\Big\vert\,Y_A = x_A\Bigr]
        \;\;\le\;\;
        (1+\delta)\,\E\bigl[f^2(x_{u_{\pe}},Y)\bigr].
    \]

    Next, we take the expectation of each term above with respect to \(Y'\):
    \[
        (1-\delta)\,\E\bigl[f^2\bigl(Y'_{u_{\pe}},Y\bigr)\bigr]
        \;\;\le\;\;
        \E_{Y'}\!\Bigl[
        \E\!\bigl[f^2(Y'_{u_{\pe}},Y)\,\big\vert\,Y_{A'} = Y'_{A'}\bigr]
        \Bigr]
        \;\;\le\;\;
        (1+\delta)\,\E\bigl[f^2\bigl(Y'_{u_{\pe}},Y\bigr)\bigr].
    \]

    Observe that
    \[
        \E\bigl[f^2\bigl(Y'_{u_{\pe}},Y\bigr)\bigr]
        \;=\;
        \EE{u} \otimes \EE{A}\,f^2
        \;=\;
        \EE{\{u\}\cup A}\,f^2
        ,
    \]
    because \(Y'_{u_{\pe}}\) can itself be viewed as a broadcasting process on \(T_u\) with initialization \(Y_u \sim \pi\). Moreover, when \(Y_{A_{\pe}}\) is conditioned on \(Y_A = Y_{A'}\), it has the same distribution as \(Y'_{A_{\pe}}\) conditioned on \(Y_{A'}\). Consequently,
    \[
        \E_{Y'}\!\Bigl[
        \E\!\bigl[f^2(Y'_{u_{\pe}},Y)\,\big\vert\,Y_{A'} = Y'_{A'}\bigr]
        \Bigr]
        \;=\;
        \E \big[
        f^2(Y'_{u_{\pe}},Y')
        \big]
        =
        \EE{A'}\,f^2.
    \]
    %
    %
    %
    %
    %
    %
    Therefore,
    we get
    \begin{align*}
        (1-\delta)
        \EE{\{u\}\cup A} f^2
        \le
        \EE{A'} f^2
        \le
        (1+\delta)
        \EE{\{u\}\cup A} f^2\,.
    \end{align*}
    Finally, recall that $\kappa \in \CC$ is prefixed, and $\lA{u} \ge \tCR (\log(R)+1)$, it is sufficient to set $\tCR$ to be large enough so that $\delta \le \kappa$.
    Now, the rest of the proof follows from Lemma \ref{lem normComparison}.

\end{proof}

\section{\bf Part I: \texorpdfstring{$\mathcal{R}$}{R}-space and its corresponding properties}
In this section, we define the $\mathcal{R}$-spaces and outline the properties we aim to establish in {\bf Part I}. We begin by introducing certain projection-like operators.

\begin{defi}
    For each $u \in V(T)$, let
    \begin{align}
        \label{def  PK}
        \PK{u} = \Pi_{u,K} : \F{u_{\pe}} \to \T{u},\,
    \end{align}
    be a linear operator satisfying the following two conditions:
    \begin{itemize}
        \item For any $f \in \F{u_{\pe}}$,
              \begin{align}
                  \label{eq PiProperty1}
                  \EE{u} fg = \EE{u} \left[(\PK{u} f)g\right] \mbox{ for } g \in \T{u}\,.
              \end{align}
        \item For any $f \in \F{u_{\pe}}$,
              \begin{align}
                  \label{eq PiProperty2}
                  \mbox{ if }\quad \EE{u} fg = 0 \mbox{ for every } g \in \T{u}\,, \quad \text{then } \quad \PK{u} f  = 0\,.
              \end{align}
    \end{itemize}
\end{defi}
\begin{rem}
    If the law of the broadcasting process \( Y_{u_{\preceq}} \) on \( T_u \), initialized with \( Y_u \sim \pi \), assigns non-zero probtbility on each \( y_{u_{\preceq}} \in [q]^{u_{\preceq}} \), then \( \PK{u} \) is simply the projection operator from \( \mathcal{F}(u_{\preceq}) \) to \( \T{u} \) under the inner product induced by the law of \( Y_{u_{\preceq}} \).
    However, if the transition matrix \( M \) has zero entries, then the support of the law of \( Y_{u_{\preceq}} \) will be strictly smaller than \( [q]^{u_{\preceq}} \).
    Because we frequently switch probability laws, we should not define \( \PK{u} \) in an almost-sure sense; instead, we specifically require the two properties stated above. The existence of such an operator is a straightforward result of linear algebra. We include the proof in the appendix (see Lemma \ref{lem LA_Pi2}).
\end{rem}

\begin{defi}
    \label{def  CR}
    For each $u \in V(T)$, consider a linear subspace ${{\cal W}_u} \subseteq \F{u_{\pe}}$.  We define a sequence of subspaces
    $ \cR{{{\cal W}_u}}{k} \subseteq \F{\anc{u}{k}_{\pe}}$ for $k \ge 0$, provided $\anc{u}{k}$ exists:  \\
    Recall that for every $v \in V(T)$, $\idtt{v}$ is the identity operator from $\F{v_{\pe}}$ to itself.
    Let
    \begin{align*}
        \cR{{{\cal W}_u}}{0}
        =
        (\idtt{u} - \PK{u}) {{\cal W}_u}
        \subseteq
        \F{u_{\pe}}\,,
    \end{align*}
    and recursively define
    \begin{align*}
        \cR{{{\cal W}_u}}{k}
        =
        (\idtt{\anc{u}{k}}-\PK{\anc{u}{k}})
        \underbrace{
            \left(\cR{{{\cal W}_u}}{k-1} \otimes \TT{u}{k-1} \right)
        }_{
        \subseteq
        \F{\anc{u}{k-1}} \otimes \TT{u}{k-1}
        \subseteq \F{\anc{u}{k}_{\pe}}
        }
        \subseteq
        {\cal F}(\anc{u}{k}_{\pe})\,.
    \end{align*}
\end{defi}
\begin{figure}[h]
    \centering
    \includegraphics[width=0.8\textwidth]{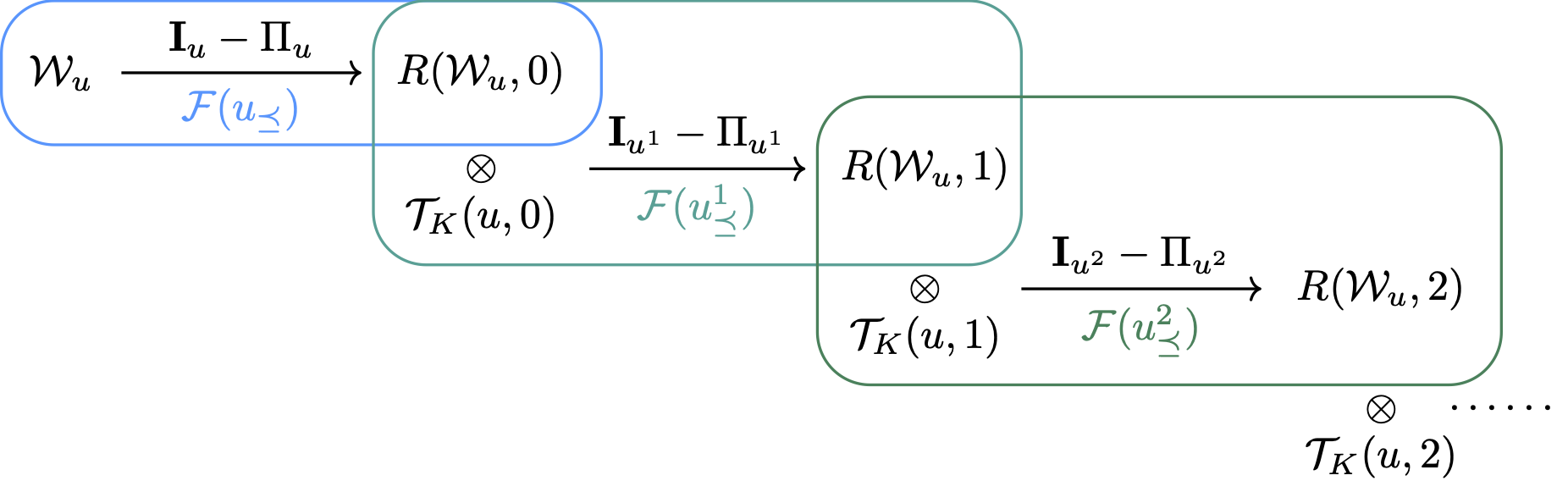}
    \caption{An illustration of the definition of ${\cal R}({{\cal W}_u},k)$ for $k=0,1,2$.}
    \label{fig:R}
\end{figure}

\begin{lemma}[Orthogonality]
    \label{lem PK_Orthogonal}
    For each $f \in \cR{{\cal W}_u}{k}$ and $g \in \T{\anc{u}{k}}$, we have
    \begin{align*}
        \EE{\anc{u}{k}}[f \cdot g] = 0\,.
    \end{align*}
    Equivalently, as a function of $x_{\anc{u}{k}}$, it has mean $0$ with respect to the law corresonding to $\EE{\anc{u}{k}}$:
    \begin{align}
        \label{eq R_CRortho}
        \CE{\anc{u}{k}}[f \cdot g]  \in \Fz{\anc{u}{k}}\,.
    \end{align}
\end{lemma}
\begin{proof}
    Let
    $$
        \tilde f \in \cR{{\cal W}_u}{k-1} \otimes \TT{u}{[0,k-1]}
    $$
    be a function so that $f = ({\bf I}_{\anc{u}{k}} -\PK{\anc{u}{k}}) \tilde f$. Then, based on the projection property \eqref{eq PiProperty1} for $\PK{\anc{u}{k}}$, we have
    \begin{align}
        \label{eq R_remark_00}
        \EE{\anc{u}{k}} [f \cdot g]
        =
        \EE{\anc{u}{k}} \left[
            ((\idtt{\anc{u}{k}} - \PK{\anc{u}{k}}) \tilde f)
            \cdot
            g
            \right]
        \stackrel{\eqref{eq PiProperty1}}{=}
        \EE{\anc{u}{k}} \left[
            (\PK{\anc{u}{k}} (\idtt{\anc{u}{k}} - \PK{\anc{u}{k}}) \tilde f)
            \cdot g
            \right]
    \end{align}
    With
    $$
        \PK{\anc{u}{k}} (\idtt{\anc{u}{k}} - \PK{\anc{u}{k}}) \tilde f
        =
        (\PK{\anc{u}{k}}\tilde f)
        - \PK{\anc{u}{k}}( \PK{\anc{u}{k}} \tilde f)\,,
    $$
    we have
    \begin{align*}
        \eqref{eq R_remark_00}
        =
        \EE{\anc{u}{k}} \left[(\PK{\anc{u}{k}}\tilde f) \cdot g \right]
        -
        \EE{\anc{u}{k}} \left[
            (\PK{\anc{u}{k}}( \PK{\anc{u}{k}} \tilde f))
            \cdot
            g
            \right]
        \stackrel{\eqref{eq PiProperty1}}{=}
        0\,.
    \end{align*}
\end{proof}

\begin{lemma}\label{lem R Decompose}
    Let $u \in V(T)$, and suppose ${{\cal W}_u} \subseteq \F{u_{\pe}}$ is a linear subspace.  
    For every integer $k \ge 0$ and every function
    \[
      \phi \;\in\; \mathcal{W}_u \;\otimes\; \TT{u}{[0,k-1]},
    \]
    there exists some $\psi \,\in\, \cR{{\cal W}_u}{k}$ such that
    \[
      \psi \;-\;\phi
      \;\in\;
      \T{u}\,\otimes\, \TT{u}{[0,k-1]}.
    \]
    Further, we also have 
    \begin{align*}
        \cR{{\cal W}_u}{k} 
    \subseteq 
        ({\cal W}_u + \T{u}) \otimes \TT{u}{[0,k-1]}.
    \end{align*}
  \end{lemma}

\begin{proof}
    We prove both statements by induction on \(k\).
    
    \step{Base Case (\texorpdfstring{$k=0$}{k=0})}
    If \(k=0\), set
    \[
      \psi \;=\; \bigl(\idtt{u} - \PK{u}\bigr)\,\phi \;\in\; \cR{{\cal W}_u}{0}.
    \]
    Then
    \[
      \psi \;-\;\phi
      \;=\;
      -\,\PK{u}\,\phi
      \;\in\;
      \T{u},
    \]
    as desired. Moreover, any \(\psi \in \cR{{\cal W}_u}{0}\) is obtained by applying the operator \(\idtt{u} - \PK{u}\) to some function in \({\cal W}_u\). Hence, the above argument shows
    \[
      \cR{{\cal W}_u}{0} \;\subseteq\; {\cal W}_u + \T{u}.
    \]
    
    \step{Inductive Step}
    Assume the statements hold for \(k-1 \ge 0\). We now prove them for \(k\).
    
    First, consider the simple case
    \[
      \phi \;=\; \phi_1 \,\otimes\, \phi_2,
    \]
    where \(\phi_1 \in \mathcal{W}_u \otimes \TT{u}{[0,k-2]}\) and \(\phi_2 \in \TT{u}{k-1}\). By the inductive hypothesis, there exists
    \[
      \psi_1 \;\in\; \cR{{\cal W}_u}{\,k-1}
    \]
    such that
    \[
      \psi_1 \;-\;\phi_1
      \;\in\;
      \T{u}\,\otimes\,\TT{u}{[0,k-2]}.
    \]
    Then
    \[
      \psi_1 \,\otimes\, \phi_2
      \;\;\in\;\;
      \cR{{\cal W}_u}{\,k-1}\,\otimes\, \TT{u}{\,k-1}
    \]
    also satisfies the required condition, because
    \[
      \bigl(\psi_1 \otimes \phi_2\bigr) \;-\; \bigl(\phi_1 \otimes \phi_2\bigr)
      \;=\;
      \bigl(\psi_1 - \phi_1\bigr)\otimes \phi_2
      \;\in\;
      \T{u}\,\otimes\,\TT{u}{[0,k-1]}.
    \]
    
    For the general case, write \(\phi\) as a finite sum
    \[
      \phi
      \;=\;
      \sum_i 
        \bigl(\phi_{1,i} \,\otimes\, \phi_{2,i}\bigr),
    \]
    where each summand is of the simple form \(\phi_{1,i} \otimes \phi_{2,i}\). Applying the above argument to each summand completes the inductive step.
    
    As for the second statement, we have
    \begin{align*}
          \cR{{\cal W}_u}{k}
      &=  (\idtt{\anc{u}{k}} - \PK{\anc{u}{k}})
          \Bigl(\cR{{\cal W}_u}{k-1} \otimes \TT{u}{k-1}\Bigr) \\
      &=  (\idtt{\anc{u}{k}} - \PK{\anc{u}{k}})
          \Bigl(({\cal W}_u + \T{u}) \otimes \TT{u}{[0,k-1]}\Bigr) \\
      &=  ({\cal W}_u + \T{u}) \otimes \TT{u}{[0,k-1]}
      \;-\;
      \PK{\anc{u}{k}}
      \Bigl(({\cal W}_u + \T{u}) \otimes \TT{u}{[0,k-1]}\Bigr).
    \end{align*}
    Since the image of \(\PK{\anc{u}{k}}\) is contained in
    \[
      \T{\anc{u}{k}}
      \;\subseteq\;
      \T{u} \otimes \TT{u}{[0,k-1]},
    \]
    where the inclusion follows from Lemma~\ref{lem TAinclusion}. We conclude that 
    \[
      \cR{{\cal W}_u}{k}
      \;\subseteq\;
      ({\cal W}_u + \T{u}) \otimes \TT{u}{[0,k-1]}.
    \]
    \end{proof}

\begin{rem}
    \label{rem: RInclusion}
    In this paper, we only focus on two specific choices for ${\cal W}_u$: ${\cal W}_u = \TT{u}{-1}$ and ${\cal W}_u = {\cal T}_{K+1}(u)$.

    Note that in both caess we have $\T{u} \subseteq \TT{u}{-1}$
    and also $\T{u} \subseteq {\cal T}_{K+1}(u)$.
    Consequently, applying the inclusion property from the previous lemma, we obtain
    \begin{align*}
        \cR{\TT{u}{-1}}{k} \subseteq \TT{u}{-1} \otimes \TT{u}{[0,k-1]} = \TT{u}{[-1,k-1]}\,,
    \end{align*}
    and similarly,
    \begin{align*}
        \cR{ {\cal T}_{K+1}(u)}{k} \subseteq {\cal T}_{K+1}(u) \otimes \TT{u}{[0,k-1]} \,.
    \end{align*}
    These two cases cover all the instances of ${\cal W}_u$ appearing in our later arguments.
\end{rem}

Next, we introdcue a decay related parameter $\CW{{\cal W}_u}$ for each subspace ${\cal W}_u$.
\begin{defi}
    \label{def CW}
    For each $u \in V(T)$ and a linear subspace ${{\cal W}_u} \subseteq \F{u_{\pe}}$, we define $\CW{{\cal W}_u}$ to be the smallest constant such that for any $f \in \cR{{{\cal W}_u}}{0}$ and $g \in \T{u}$,
    \begin{align*}
        \maxnorm{\CE{u} fg}
        \le
        \CW{{\cal W}_u} \Unorm{f}{u} \Unorm{g}{u}\,.
    \end{align*}
\end{defi}

In {\bf Part I}, we aim to establish two properties of $\cR{{{\cal W}_u}}{k}$.


\begin{prop}
    \label{prop: coreR}
    There exists a constant $\tC \in \CC$ such that the following statement holds for sufficiently large $\tCR$:
    Let $u \in V(T)$ satisfy $\lA{u} \ge \tCR (\log(R)+1)$ and assume $\anc{u}{k}$ is well-defined.
    Consider a linear subspace ${{\cal W}_u} \subseteq \F{u_{\pe}}$.
    For any $f \in \cR{{{\cal W}_u}}{k}$, the following properties holds:
    \begin{enumerate}
        \item For any $g \in \T{\anc{u}{k}}$,
              $$
                  \maxnorm{ \CE{\anc{u}{k}} fg}
                  \le
                  \tC \CW{{\cal W}_u} \lame^k
                  \Unorm{f}{\anc{u}{k}} \Unorm{g}{\anc{u}{k}} \,.
              $$
        \item For any $t \in [0,k-1]$, there exists
              $$
                  f_t \in \cR{{{\cal W}_u}}{t} \otimes \TT{u}{[t,k-1]}
              $$
              such that $f - f_t \in \T{\anc{u}{t}} \otimes \TT{u}{[t,k-1]}$ and
              \begin{align*}
                  \Unorm{f - f_t}{\anc{u}{k}}
                  \le
                  \CW{{\cal W}_u} \cdot
                  \tC R \exp(-\eps \lA{u})
                  \lame^t \exp(- \eps t)
                  \Unorm{f}{\anc{u}{k}}\,.
              \end{align*}
    \end{enumerate}
\end{prop}

This proposition demonstrates that for any $f \in \cR{{{\cal W}_u}}{k}$, its correlation with any polynomials $g \in \T{\anc{u}{k}}$ is small with respect to the broadcasting process $Y_{\anc{u}{k}_{\pe}}$ on $T_{\anc{u}{k}}$, regardless of the initial distribution of $Y_{\anc{u}{k}}$. 

To state the result, we need to introduce a paper-wide parameter $\hB$:
\begin{align}
    \label{def  hB}
    \hB = \left\lceil \tCR (\log(R) +1) + \frac{\eps}{10 d} \tCR (\log(R) +1) \right\rceil\,.
\end{align}

As mentioned in remark \ref{rem: RInclusion}, we will consider only two types of ${\cal W}_u$ in this paper: ${\cal W}_u = \TT{u}{-1}$ and ${\cal W}_u = {\cal T}_{K+1}(u)$. Specifically, we consider the first case when $\lA{u} > \hB$ and the second case when $\lA{u} = \hB$. Therefore, we need to bound their corresponding $\CW{{\cal W}_u}$ values.

\begin{prop}
    \label{prop PT MAIN}
    For any $ 1 < \tCB \in \CC$, the following holds when $\tCR$ is sufficiently large:
    For any $u \in V(T)$,  we have
    \begin{align*}
        \CW{\F{u_{\pe}}} \le \tCM\,,
    \end{align*}
    where $\tCM \in \CC$ is the constant defined related to the Markov Chain in Definition \ref{def  tCM}.
    If $ \lA{u} \ge \tCR (\log(R)+1)$, then
    \begin{align*}
        \CW{\TT{u}{-1}} \le \tC R \exp(-\eps \lA{u})\,.
    \end{align*}
    where $\tC \in \CC$ is a constant independent of $\tCB$.
    If $ \lA{u} \ge \hB$, then
    \begin{align*}
        \CW{{\cal T}_{K+1}(u)} \le \frac{1}{\tCB R}\,,
    \end{align*}
\end{prop}

The proof of Proposition \ref{prop: coreR} and the bound on $\CW{{\cal T}_{K+1}(u)}$ stated in Proposition \ref{prop PT MAIN} constitute the main technical results of {\bf Part I}, which will be presented in the next two sections. For the remainder of this section, we will prove the first two bounds in Proposition \ref{prop PT MAIN}, as they are straightforward and will be needed for the proof of Proposition \ref{prop: coreR}.

\begin{proof}[Proof of first two bounds in Proposition \ref{prop PT MAIN}]

    For any given $f\in \cR{{\cal W}_u}{0} \subseteq \F{u_{\pe}}$ and $g \in \T{u}$, we have
    $$
        \maxnorm{\CE{u}fg}
        \stackrel{\eqref{eq 1variableDecay}}{\le}
        \tC_M  \EE{u} |\CE{u}[fg]|
        \le
        \tC_M  \EE{u} |fg|
        \le
        \tC_M \Unorm{f}{u} \Unorm{g}{u}\,,
    $$
    where the last inequality follows from Cauchy-Schwarz inequality.

    Now we consider the case when  ${{\cal W}_u} = \TT{u}{-1}$.
    Note that $\T{u} \subseteq \TT{u}{-1}$.
    For $\tilde f \in \TT{u}{-1}$, we have $\PK{u} \tilde f \in \T{u} \subseteq \TT{u}{-1}$, and thus, $(\idtt{u} - \PK{u}) \tilde f \in \TT{u}{-1}$. In other words, we have ${\cal R}(\TT{u}{-1},0) \subseteq \TT{u}{-1}$.

    For any $f \in \cR{{\cal W}_u}{0}$ and $g \in \T{u}$,
    since both functions lie in $\TT{u}{-1}$,
    we can apply Lemma \ref{lem basicDecayAu} to obtain
    \begin{align*}
        \maxnorm{\D{u}fg}
        \le &
        \tC_{\ref{lem basicDecayAu}} R \exp(-\eps \lA{u})
        \Unorm{f}{\OO{u}{-1}} \Unorm{g}{\OO{u}{-1}} \\
        \le &
        (1+\kappa) \tC_{\ref{lem basicDecayAu}} R \exp(-\eps \lA{u})
        \Unorm{f}{u} \Unorm{g}{u}\,,
    \end{align*}
    where the last inequality follows from Lemma \ref{lemma: coreNorm}.
    Moreover, since $\EE{u}fg = 0$ for $f \in \cR{{\cal W}_u}{0}$ and $g \in \T{u}$, it follows that $\D{u}fg = \CE{u}fg$. This completes the proof of the second statement.
\end{proof}

\section{Decay Properties of \texorpdfstring{$\mathcal{R}$}{R} function spaces and low degree polynomials}

The goal of this section is to establish Proposition \ref{prop: coreR}. Let us restate it here for convenience.

\begin{prop*}
    There exists a constant $\tC \in \CC$ such that the following statement holds for sufficiently large $\tCR$:
    Let $u \in V(T)$ satisfy $\lA{u} \ge \tCR (\log(R)+1)$ and $\anc{u}{k}$ is well-defined.
    Consider a subspace ${{\cal W}_u} \subseteq \F{u_{\pe}}$.
    For any $f \in \cR{{{\cal W}_u}}{k}$, the following holds:
    \begin{enumerate}
        \item For any $g \in \T{\anc{u}{k}}$,
              $$
                  \maxnorm{ \CE{\anc{u}{k}} fg}
                  \le
                  \tC \CW{{\cal W}_u} \lame^k
                  \Unorm{f}{\anc{u}{k}} \Unorm{g}{\anc{u}{k}} \,.
              $$
        \item For any $t \in [0,k-1]$, there exists
              $$
                  f_t \in \cR{{{\cal W}_u}}{t} \otimes \TT{u}{[t,k-1]}
              $$
              such that $f - f_t \in \T{\anc{u}{t}} \otimes \TT{u}{[t,k-1]}$ and
              \begin{align*}
                  \Unorm{f - f_t}{\anc{u}{k}}
                  \le
                  \CW{{\cal W}_u} \cdot
                  \tC R \exp(-\eps \lA{u})
                  \lame^t \exp(- \eps t)
                  \Unorm{f}{\anc{u}{k}}\,.
              \end{align*}
    \end{enumerate}
\end{prop*}

The proof of Proposition \ref{prop: coreR} is technical and relies on induction.
We start by establishing the induction hypothesis.
\subsection{Induction Hypothesis and its Implication}
\begin{IH}
    \label{IH: coreR}
    For each $k \ge 0$, we can express
    $$
        \CE{\anc{u}{k}} \left[ f \cdot g \right]
        =
        \sum_{t \in [0,k]} \CE{\anc{u}{k}}\bL_{k,t}(f,g) \mbox{ for } f \in \cR{{\cal W}_u}{k}, g \in \T{\anc{u}{k}},
    $$
    where
    $$
        \bL_{k,t} : \cR{{\cal W}_u}{k} \times \T{\anc{u}{k}} \to \Fz{\anc{u}{t}}
    $$
    is a bilinear map such that
    \begin{align*}
        \maxnorm{\bL_{k,t}(f,g)}
        \le &
        \begin{cases}
            \CW{{\cal W}_u} (1+\kappa)^{3k}
            \Unorm{f}{\anc{u}{k}} \Unorm{g}{\anc{u}{k}} & t = 0\,,       \\
            \tlam^{t} \exp(- \eps t)  \cdot \tC_{\ref{IH: coreR}}
            R\exp(-\eps \lA{u})
            \cdot \CW{{\cal W}_u} (1+\kappa)^{3k}
            \Unorm{f}{\anc{u}{k}} \Unorm{g}{\anc{u}{k}} & t \in [1,k]\,,
        \end{cases}
    \end{align*}
    where
    \begin{align*}
        \tC_{\ref{IH: coreR}} =  (1 + \tCM) \frac{2\tCM \tC_{\ref{lem basicDecayAu}}\exp(\eps)}{ \tlam} \in \CC\,.
    \end{align*}
\end{IH}

Let us now consider a straightforward implication of the induction hypothesis.
\begin{lemma}
    \label{lem RIHImplication}
    There exists $\tC \in \CC$ so that when $\tCR$ is large enough,
    the following holds:
    The induction hypothesis holds for $k$, then
    \begin{align*}
        \maxnorm{\CE{\anc{u}{k}}fg} \le
        2 \tCM \CW{{\cal W}_u} \tlam^k (1+\kappa)^{3k} \Unorm{f}{\anc{u}{k}} \Unorm{g}{\anc{u}{k}}\,
    \end{align*}
    for $f \in \cR{{\cal W}_u}{k}$ and $g \in \T{\anc{u}{k}}$.
\end{lemma}
In other words, the first statement of Proposition \ref{prop: coreR} is a direct consequence of the induction hypothesis.

\begin{proof}
    Given that $\bL_{k,t}fg \in \Fz{\anc{u}{t}}$, we use the Markov Chain decay \eqref{eq 1variableDecay} to obtain
    \begin{align*}
        \maxnorm{\CE{\anc{u}{k}}\bL_{k,t}fg}
        \le
        \tCM \tlam^{k-t} \maxnorm{ \bL_{k,t}fg}\,.
    \end{align*}
    From the induction hypothesis, we have
    \begin{align*}
        \maxnorm{\CE{\anc{u}{k}}\bL_{k,t} fg}
        \le
        \begin{cases}
            \tCM \tlam^{k}
            \CW{{\cal W}_u} (1+\kappa)^{3k}
            \Unorm{f}{\anc{u}{k}}\Unorm{g}{\anc{u}{k}}                                & \mbox{ if } t = 0\,,       \\
            \tCM \tlam^{k} \exp(-\eps t) \cdot \tC_{\ref{IH: coreR}}\exp(-\eps \lA{u})
            \CW{{\cal W}_u}(1+\kappa)^{3k} \Unorm{f}{\anc{u}{k}}\Unorm{g}{\anc{u}{k}} & \mbox{ if } t \in [1,k]\,.
        \end{cases}
    \end{align*}
    Therefore, by summing over $t$, we obtain
    \begin{align*}
        \maxnorm{\CE{\anc{u}{k}}fg}
        \le &
        \tCM \CW{{\cal W}_u} \tlam^k (1+\kappa)^{3k} \bigg( 1 +
        \tC_{\ref{IH: coreR}}\exp \Big(-\eps \tCR (\log(R)+1) \Big)
        \sum_{t = 0 }^\infty  \exp(-\eps t) \bigg) \Unorm{f}{\anc{u}{k}} \Unorm{g}{\anc{u}{k}}
        \\
        \le & 2 \tCM \CW{{\cal W}_u} \tlam^k (1+\kappa)^{3k} \Unorm{f}{\anc{u}{k}} \Unorm{g}{\anc{u}{k}}\,,
    \end{align*}
    if $\tCR$ is large enough,
    since $\sum_{t} \exp(-\varepsilon t) \le \frac{1}{1-\exp(-\varepsilon)} \in \CC$.
\end{proof}
\begin{rem}
    We note that in the proof, we are able to capture the decay of a Markov Chain of length $k-t$ for each term $\bL_{k,t}fg$ in the sum. This is one of the reasons why decompose $\CE{\anc{u}{k}}fg$ into a sum of $\bL_{k,t}fg$ in the induction hypothesis.
\end{rem}
As a corollary of Lemma \ref{lem RIHImplication}, we have the following result.
\begin{cor}
    \label{cor: coreR-1}
    When $\tCR$ is large enough, the following holds:
    Assume that $u \in V(T)$ satisfies
    $$\lA{u} \ge \tCR(\log(R)+1)\,,$$
    and the Induction Hypothesis \ref{IH: coreR} holds for $k$. \\
    Then, for $f \in \cR{{\cal W}_u}{k} \otimes \TT{u}{k}$
    and $g \in  \T{\anc{u}{k}}\otimes \TT{u}{k} \supseteq \T{\anc{u}{k+1}}$, we have
    $$
        \maxnorm{ \CE{\anc{u}{k}} \otimes \D{\OO{u}{k}} fg}
        \le
        \RDelta_k
        \Unorm{f}{\anc{u}{k+1}} \Unorm{g}{\anc{u}{k+1}}\,,
    $$
    where
    \begin{align}
        \label{def  RDelta}
        \RDelta_k := &
        \frac{2 \tCM \tC_{\ref{lem basicDecayAu}}\exp(\eps)}{\tlam}
        R \exp( -\eps \lA{u}) \cdot \underbrace{\CW{{\cal W}_u}}_{\le \tCM}
        \underbrace{\tlam^{k+1} (1+\kappa)^{3k+2}}_{\le 1} \exp(- \eps k)\,,
    \end{align}
    where $\tC_{\ref{lem basicDecayAu}} \in \CC$ is the constant introduced in Lemma \ref{lem basicDecayAu}.
\end{cor}
We will keep the definition of $\RDelta_k$ for the rest of this section.
\begin{rem}
    \label{rem RDelta}
    Observe that the definition of $\RDelta_k$ satisfies the recursive relation:
    \begin{align*}
        \RDelta_{s+1} = \tlam (1+\kappa)^3 \exp(-\eps) \RDelta_s .
    \end{align*}
    Due to the parameter configuration that
    \begin{align*}
        \tlam (1+\kappa)^2 \exp(-\eps) < \lame \exp(-\eps) < 1\,,
    \end{align*}
    from Definition \ref{def varepsilon}, we know that $k \mapsto \RDelta_{k}$ is monotone decreasing.
    Further,
    \[
        \Delta_0
        \le
        \frac{2 \tCM \tC_{\ref{lem basicDecayAu}}\exp(\eps)}{\tlam}
        R \exp( -\eps \tCR(\log(R)+1)) \cdot \tCM
    \]
    can be as close to $0$ as we want by choosing $\tCR$ large enough.
\end{rem}
\begin{proof}
    \step{Tensorization}
    Given that $\OO{u}{k} = \OO{\anc{u}{k}}{0}$, we first invoke from Lemma \ref{lem basicDecayAu} that, for $\phi_1,\phi_2 \in \TT{u}{k}$,
    \begin{align*}
        \maxnorm{\D{\OO{u}{k}}\phi_1\phi_2}
        =   &
        \maxnorm{\D{\OO{\anc{u}{k}}{0}}\phi_1\phi_2}                                                       \\
        \le &
        \tC_{\ref{lem basicDecayAu}} R \exp(-\eps \lA{\anc{u}{k}}) \Unorm{\phi_1}{A_k} \Unorm{\phi_2}{A_k} \\
        \le &
        \exp(-\eps k)
        \tC_{\ref{lem basicDecayAu}} R \exp(-\eps \lA{u}) \Unorm{\phi_1}{A_k} \Unorm{\phi_2}{A_k}
    \end{align*}
    where $\tC_{\ref{lem basicDecayAu}}$ is the constant introduced in Lemma \ref{lem basicDecayAu}.
    Then, we invoke Lemma \ref{lem: mainTensorProduct} to get
    \begin{align*}
        \maxnorm{\CE{\anc{u}{k}} \otimes \D{\OO{u}{k}} fg}
        \le &
        2 \tCM  \CW{{\cal W}_u} \tlam^k (1+\kappa)^{3k} \cdot                   \\
            & \phantom{AAA}
        \cdot
        \exp(-\eps k)
        \tC_{\ref{lem basicDecayAu}} R \exp(-\eps \lA{u})
        \Unorm{f}{\{\anc{u}{k}\}\cup A_k} \Unorm{g}{\{\anc{u}{k}\}\cup A_k}     \\
        =   &
        \frac{2 \tCM \tC_{\ref{lem basicDecayAu}}}{\tlam}
        R \exp(-\eps \lA{u})
        \CW{{\cal W}_u} \cdot                                                   \\
            & \phantom{AAA AAA} \cdot\tlam^{k+1} (1+\kappa)^{3k} \exp(- \eps k)
        \Unorm{f}{\{\anc{u}{k}\}\cup A_k} \Unorm{g}{\{\anc{u}{k}\}\cup A_k}\,.
    \end{align*}

    \step{Convert the norms}
    Here we apply Lemma \ref{lemma: coreNorm} with $\{\anc{u}{k}\} \cup A_k \preceq \anc{u}{k+1}$
    to get
    \begin{align*}
        \Unorm{f}{\{\anc{u}{k}\}\cup A_k} \le (1+\kappa) \Unorm{f}{\anc{u}{k+1}}\,\, \text{ and } \,\,
        \Unorm{g}{\{\anc{u}{k}\}\cup A_k} \le (1+\kappa) \Unorm{g}{\anc{u}{k+1}}\,.
    \end{align*}
    Substituting this into the above inequality, the lemma follows.
\end{proof}

\subsection{Operator Norm of the projection-like operators $\Pi_v$}

\begin{lemma}
    \label{lem RProjectionNorm}
    When $\tCR$ is large enough, the following holds:
    For any $k \ge 0$, we have
    \begin{align*}
        \Unorm{\PK{\anc{u}{k+1}} \phi}{\anc{u}{k+1}} \le
        \RDelta_k \Unorm{\phi}{\anc{u}{k+1}}\,,
    \end{align*}
    for $\phi \in \cR{{{\cal W}_u}}{k}\otimes \TT{u}{k} \subseteq \F{\anc{u}{k+1}_{\pe}}$.
\end{lemma}
\begin{proof}
    \step{Representing $\Unorm{\PK{\anc{u}{k+1}} \phi}{\anc{u}{k+1}}$}
    By definition of $\PK{\anc{u}{k+1}}$,
    \begin{align*}
        \Unorm{\PK{\anc{u}{k+1}} \phi}{\anc{u}{k+1}}^2
        =
        \EE{\anc{u}{k+1}}\left[ \left(\PK{\anc{u}{k+1}} \phi\right)^2 \right]
        \stackrel{\eqref{eq PiProperty1}}{=} &
        \EE{\anc{u}{k+1}} \left[\phi (\PK{\anc{u}{k+1}} \phi) \right] \\
                                             & =
        \EE{\anc{u}{k+1}} \left[ \CE{\anc{u}{k}}\otimes \CE{\OO{u}{k}} \left[
                \phi(\PK{\anc{u}{k+1}} \phi)\right]\right]\,,
    \end{align*}
    where the last equality follows from Lemma \ref{lem basicEECED}.

    Consider the simple decomposition
    \begin{align*}
        \CE{\anc{u}{k+1}}\otimes \CE{\OO{u}{k}}
        =
        \CE{\anc{u}{k+1}}\otimes \EE{\OO{u}{k}} + \CE{\anc{u}{k+1}}\otimes \D{\OO{u}{k}}\,.
    \end{align*}

    For the first summand, recall from \eqref{eq R_CRortho} that
    \begin{align*}
        \cR{{{\cal W}_u}}{k} \times \T{\anc{u}{k}}  \xrightarrow{\CE{\anc{u}{k}}}
        \Fz{\anc{u}{k}}\,,
    \end{align*}
    we have
    \begin{align*}
        (\cR{{{\cal W}_u}}{k}\otimes \TT{u}{k}) \times
        (\T{\anc{u}{k}} \otimes \TT{u}{k}) \xrightarrow{\CE{\anc{u}{k}}\otimes \EE{\OO{u}{k}}}
        \Fz{\anc{u}{k}}\otimes \R = \Fz{\anc{u}{k}}
        \xrightarrow{\EE{\anc{u}{k+1}}}
        \{0\}\,.
    \end{align*}
    Hence,
    \begin{align*}
        \Unorm{\PK{\anc{u}{k+1}} \phi}{\anc{u}{k+1}}^2
        =
        \EE{\anc{u}{k+1}} \left[ \CE{\anc{u}{k}}\otimes \CE{\OO{u}{k}}
            \phi(\PK{\anc{u}{k}} \phi)\right]
        = &
        \EE{\anc{u}{k+1}} \left[ \CE{\anc{u}{k}}\otimes \D{\OO{u}{k}}
            \phi(\PK{\anc{u}{k}} \phi)\right]\,.
    \end{align*}

    \step{Invoke Corollary \ref{cor: coreR-1}}
    Since $\PK{\anc{u}{k+1}} \phi \in \T{\anc{u}{k}}\otimes\TT{u}{k}$, we can invoke Corollary \ref{cor: coreR-1} bound the term in the above equality
    \begin{align*}
        \EE{\anc{u}{k+1}} \left[ \CE{\anc{u}{k}}\otimes \D{\OO{u}{k}}
            \phi(\PK{\anc{u}{k}} \phi)\right]
        \le &
        \maxnorm{ \CE{u}\otimes \D{A(u;k)}\phi \PK{\anc{u}{k+1}} \phi}
        \le
        \RDelta_k
        \Unorm{\phi}{\anc{u}{k+1}} \Unorm{\PK{\anc{u}{k+1}} \phi}{\anc{u}{k+1}}\,,
    \end{align*}
    which in turn implies
    \begin{align*}
        \Unorm{\PK{\anc{u}{k+1}} \phi}{\anc{u}{k+1}}
        \le &
        \RDelta_k \Unorm{\phi}{\anc{u}{k+1}}\,.
    \end{align*}
\end{proof}

\subsection{Construction of $\bL_{k+1,t}$}

Since we will work directly with functions in $\cR{{{\cal W}_u}}{k+1}$,
it is convenient to consider the pseudo inverse:
\begin{defi}
    \label{def  RPseudoInverse}
    [Pseudo-Inverse of $\idtt{\anc{u}{k+1}} - \PK{\anc{u}{k+1}}$]
    Recall that
    $$
        \cR{{{\cal W}_u}}{k+1} = (\idtt{\anc{u}{k+1}} - \PK{\anc{u}{k+1}}) \left(\cR{{{\cal W}_u}}{k} \otimes \TT{u}{k}\right)\,.
    $$
    Let
    $$
        {\bf Q}: \cR{{{\cal W}_u}}{k+1} \rightarrow \left(\cR{{{\cal W}_u}}{k} \otimes \TT{u}{k}\right)
    $$
    be the Moore-Penrose pseudo-inverse of $(\idtt{\anc{u}{k+1}} - \PK{\anc{u}{k+1}})$.
    The fact that $\cR{{{\cal W}_u}}{k+1}$ is the image of the map implies
    \begin{align*}
        (\idtt{\anc{u}{k+1}} - \PK{\anc{u}{k+1}})
        {\bf Q} f = f \mbox{ for } f \in \cR{{{\cal W}_u}}{k+1}\,.
    \end{align*}
\end{defi}

Let $f \in \cR{{{\cal W}_u}}{k+1}$ and $g \in \T{\anc{u}{k+1}}$, consider the following decomposition of $\CE{\anc{u}{k+1}}fg$:
\begin{align*}
    \nonumber
    \CE{\anc{u}{k+1}}fg
    = &
    \CE{\anc{u}{k+1}} ({\bf Q} f) g +
    \CE{\anc{u}{k+1}} (f - {\bf Q} f) g \,.
\end{align*}
Notice that
\begin{align*}
    f -{\bf Q}f = (\idtt{\anc{u}{k+1}} - \PK{\anc{u}{t+1}})
    {\bf Q} f - {\bf Q} f = - \PK{\anc{u}{t+1}} {\bf Q} f
    \in \T{\anc{u}{k+1}}\,,
\end{align*}
and thus
\begin{align}
    \nonumber
    \CE{\anc{u}{k+1}}fg
    = &
    \CE{\anc{u}{k+1}} \left[({\bf Q}f) g \right]
    - \CE{\anc{u}{k+1}} \left[ (\PK{\anc{u}{t+1}} {\bf Q}f) g  \right]                  \\
    \nonumber
    = &
    \CE{\anc{u}{k+1}} \left[\CE{\anc{u}{k}} \otimes \EE{\OO{u}{k}} ({\bf Q}f) g \right] \\
      & + \underbrace{
        \CE{\anc{u}{k+1}} \left[\CE{\anc{u}{k}} \otimes \D{\OO{u}{k}} ({\bf Q}f) g \right]
        - \CE{\anc{u}{k+1}} \left[ (\PK{\anc{u}{k+1}} {\bf Q}f) g  \right]
    }_{:=\bL_{k+1,k+1}(f,g)}\,.
    \label{eq RCEfgDecomposition}
\end{align}
Observe that $\bL_{k+1,k+1}$ is a bilinear map from $\cR{{{\cal W}_u}}{k+1} \times \T{\anc{u}{k+1}}$ to $\F{\anc{u}{k+1}}$.
\begin{lemma}
    \label{lem RbLk+1k+1}
    The following holds when $\tCR$ is large enough:
    Suppose the Induction Hypothesis \ref{IH: coreR} holds for $k$. Then, $\bL_{k+1,k+1}$ satisfies the condition described in Induction Hypothesis \ref{IH: coreR} for $k+1$.
\end{lemma}

\begin{proof}
    \step{The image of \(\bL_{k+1,k+1}\) lies in \(\Fz{\anc{u}{k+1}}\).}
    We begin by writing
    \begin{align}
        \label{eq RbLk+1k+100}
        \EE{\anc{u}{k+1}} \bL_{k+1,k+1}(f,g)
        \;=\;
        \EE{\anc{u}{k+1}}\bigl[\CE{\anc{u}{k}} \otimes \D{\OO{u}{k}}\bigl({\bf Q}f\bigr)\,g\bigr]
        \;-\;
        \EE{\anc{u}{k+1}}\bigl[(\PK{\anc{u}{k+1}} {\bf Q}f)\,g\bigr].
    \end{align}
    Since \(g \in \T{\anc{u}{k+1}}\), we have
    \begin{align*}
        \EE{\anc{u}{k+1}}\bigl[(\PK{\anc{u}{k+1}} {\bf Q}f)\,g\bigr]
         & \stackrel{\eqref{eq PiProperty1}}{=}
        \EE{\anc{u}{k+1}}\bigl[({\bf Q}f)\,g\bigr]
        \;=\;
        \EE{\anc{u}{k+1}}\bigl[\CE{\anc{u}{k}} \otimes \CE{\OO{u}{k}}\bigl({\bf Q}f\bigr)\,g\bigr].
    \end{align*}
    Substituting back into \eqref{eq RbLk+1k+100}, we get
    \begin{align*}
        \EE{\anc{u}{k+1}} \bL_{k+1,k+1}(f,g)
        \;=\;
        -\, \EE{\anc{u}{k+1}}\Bigl[\CE{\anc{u}{k}} \otimes \EE{\OO{u}{k}}\bigl({\bf Q}f\bigr)\,g\Bigr].
    \end{align*}
    Recall that \({\bf Q}f \in \cR{{{\cal W}_u}}{k}\otimes \TT{u}{k}\) and \(g \in \T{\anc{u}{k+1}} \subset \T{\anc{u}{k}} \otimes \TT{u}{k}\).
    By the orthogonality property \eqref{eq R_CRortho}, namely
    \[
        \CE{\anc{u}{k}}: \cR{{{\cal W}_u}}{k}\times \T{\anc{u}{k}} \;\to\; \Fz{\anc{u}{k}},
    \]
    we obtain
    \[
        (\cR{{{\cal W}_u}}{k}\otimes \TT{u}{k}) \times (\T{\anc{u}{k}} \otimes \TT{u}{k})
        \;\xrightarrow{\;\CE{\anc{u}{k}} \otimes \EE{\OO{u}{k}}\;}
        \Fz{\anc{u}{k}} \otimes \R \;=\; \Fz{\anc{u}{k}}
        \;\xrightarrow{\;\EE{\anc{u}{k+1}}\;}
        \{0\}.
    \]
    Hence,
    \[
        \EE{\anc{u}{k+1}} \bL_{k+1,k+1}(f,g)
        \;=\;
        -\, \EE{\anc{u}{k+1}}\Bigl[\CE{\anc{u}{k}} \otimes \EE{\OO{u}{k}}\bigl({\bf Q}f\bigr)\,g\Bigr]
        \;=\; 0.
    \]

    \step{Bounding \(\maxnorm{\bL_{k+1,k+1}(f,g)}\) by \(\Unorm{{\bf Q}f}{\anc{u}{k+1}}\Unorm{g}{\anc{u}{k+1}}\).}
    By Corollary~\ref{cor: coreR-1},
    \[
        \maxnorm{\CE{\anc{u}{k}} \otimes \D{\OO{u}{k}} \bigl({\bf Q}f\bigr)\,g}
        \;\le\;
        \RDelta_k \,\Unorm{{\bf Q}f}{\anc{u}{k+1}}\Unorm{g}{\anc{u}{k+1}}.
    \]
    Next,
    \begin{align*}
        \maxnorm{\CE{\anc{u}{k+1}}\bigl[(\PK{\anc{u}{k+1}} {\bf Q}f)\,g\bigr]}
         & \le
        \tCM \,\EE{\anc{u}{k+1}}\Bigl|\CE{\anc{u}{k+1}}\bigl[(\PK{\anc{u}{k+1}} {\bf Q}f)\,g\bigr]\Bigr| \\
         & \le
        \tCM \,\EE{\anc{u}{k+1}}\bigl|(\PK{\anc{u}{k+1}} {\bf Q}f)\,g\bigr|
        \;\le\;
        \tCM \,\Unorm{\PK{\anc{u}{k+1}} {\bf Q}f}{\anc{u}{k+1}} \,\Unorm{g}{\anc{u}{k+1}}.
    \end{align*}
    By Lemma~\ref{lem RProjectionNorm},
    \[
        \Unorm{\PK{\anc{u}{k+1}} {\bf Q}f}{\anc{u}{k+1}}
        \;\le\;
        \RDelta_k \,\Unorm{{\bf Q}f}{\anc{u}{k+1}}.
    \]
    Combining these, we conclude
    \[
        \maxnorm{\bL_{k+1,k+1}(f,g)}
        \;\le\;
        (1 + \tCM)\,\RDelta_k \,\Unorm{{\bf Q}f}{\anc{u}{k+1}} \,\Unorm{g}{\anc{u}{k+1}}.
    \]

    \step{Converting the norms.}
    By Lemma~\ref{lem RProjectionNorm},
    \begin{align*}
        \Unorm{f}{\anc{u}{k+1}}
         & \ge
        \Unorm{{\bf Q}f}{\anc{u}{k+1}} \;-\; \Unorm{f - {\bf Q}f}{\anc{u}{k+1}}              \\
        \nonumber
         & =
        \Unorm{{\bf Q}f}{\anc{u}{k+1}} \;-\; \Unorm{\PK{\anc{u}{t+1}}{\bf Q}f}{\anc{u}{k+1}} \\
        \nonumber
         & \ge
        (1 - \RDelta_k)\,\Unorm{{\bf Q}f}{\anc{u}{k+1}}\,.
    \end{align*}
    From Remark~\ref{rem RDelta},
    \[
        \RDelta_k
        \;\le\;
        \RDelta_0
        \;\le\;
        \frac{2\,\tCM \,\tC_{\ref{lem basicDecayAu}} \,\exp(\eps)}{\tlam}
        \,R \,\exp\bigl(-\eps\,\tCR(\log(R)+1)\bigr)\,\tCM,
    \]
    which can be made arbitrarily small by choosing \(\tCR\) large.  We thus assume
    \[
        (1-\RDelta_k) \;\ge\; \frac{1}{1+\kappa},
    \]
    so that
    \begin{align}
        \label{eq RbLk+1k+101}
        \Unorm{f}{\anc{u}{k+1}}
        \;\ge\; \frac{1}{1+\kappa}\,\Unorm{{\bf Q}f}{\anc{u}{k+1}}.
    \end{align}
    Substituting back into our previous inequality and unwrapping the definitions of \(\RDelta_k\) and \(\tC_{\ref{IH: coreR}}\), we obtain
    \begin{align*}
        \maxnorm{\bL_{k+1,k+1}(f,g)}
        \le &
        (1 + \tCM)\,\RDelta_k \,(1+\kappa)\,\Unorm{f}{\anc{u}{k+1}} \,\Unorm{g}{\anc{u}{k+1}} \\
        =   &
        \underbrace{(1+\tCM) \,\frac{2\,\tCM \,\tC_{\ref{lem basicDecayAu}} \,\exp(\eps)}{\tlam}}_{=\;\tC_{\ref{IH: coreR}}}
        \, R \,\exp\bigl(-\eps\,\lA{u}\bigr)\,\CW{{\cal W}_u}\, \cdot                         \\
            & \cdot
        \tlam^{k+1}\,(1+\kappa)^{3k+3}\,\exp\bigl(- \eps (k+1)\bigr)\,
        \Unorm{f}{\anc{u}{k+1}} \,\Unorm{g}{\anc{u}{k+1}}.
    \end{align*}
    Thus, the constant \(\tC_{\ref{IH: coreR}}\) is chosen to match the one above, completing the proof.
\end{proof}

After introducing the bilinear map \(\bL_{k+1,k+1}\), we now proceed to define the bilinear map \(\bL_{k+1,t}\) for \(t \in [0,k]\). Recall the first summand from \eqref{eq RCEfgDecomposition} and decompose it using the operators
${\bL}_{k,t}$ for $t \in [0,k]$:
\begin{align*}
    \CE{\anc{u}{k+1}} \Bigl[\CE{\anc{u}{k}} \otimes \EE{\OO{u}{k}} \bigl({\bf Q}f\bigr)\,g \Bigr]
     & =
    \CE{\anc{u}{k+1}}\Biggl[\sum_{t=0}^{k}
        \bigl(\CE{\anc{u}{k}}\,{\bL}_{k,t}\bigr)
        \,\otimes\,
        \EE{\OO{u}{k}}\bigl({\bf Q}f\bigr)\,g
        \Biggr]\,.
\end{align*}

Note that this decomposition is well-defined, since
${\bf Q}f \in \cR{{{\cal W}_u}}{k+1} \otimes \TT{u}{k}$ and
$g \in \T{\anc{u}{k+1}} \subseteq \T{\anc{u}{k}} \otimes \TT{u}{k}$.

For each summand, we use the fact that the image of $\EE{\OO{u}{k}}$ lies in $\R$:
\begin{align*}
    \bigl(\CE{\anc{u}{k}}\,{\bL}_{k,t}\bigr) \otimes \EE{\OO{u}{k}} \bigl({\bf Q}f\bigr)\,g
     & =
    \Bigl(\CE{\anc{u}{k}} \otimes \idtt{\R}\Bigr)\,\circ\,
    \Bigl({\bL}_{k,t} \otimes \EE{\OO{u}{k}}\Bigr)\bigl[\bigl({\bf Q}f\bigr)\cdot g\bigr].
\end{align*}
Since
\({\bL}_{k,t} \otimes \EE{\OO{u}{k}}\bigl({\bf Q}f\bigr)\,g \)
lies in \(\Fz{\anc{u}{t}}\otimes \R = \Fz{\anc{u}{t}}\), we obtain
\begin{align*}
    \bigl(\CE{\anc{u}{k}}\,{\bL}_{k,t}\bigr) \otimes \EE{\OO{u}{k}} \bigl({\bf Q}f\bigr)\,g
     & =
    \CE{\anc{u}{k}} \Bigl({\bL}_{k,t} \otimes \EE{\OO{u}{k}}\bigl({\bf Q}f\bigr)\,g\Bigr).
\end{align*}

Applying the identity \eqref{eq EEU'ECU}, we then have
\begin{align*}
    \CE{\anc{u}{k+1}} \Bigl[\CE{\anc{u}{k}} \otimes \EE{\OO{u}{k}} \bigl({\bf Q}f\bigr)\,g \Bigr]
     & =
    \CE{\anc{u}{k+1}} \circ \CE{\anc{u}{k}}
    \sum_{t=0}^{k} {\bL}_{k,t} \,\otimes\, \EE{\OO{u}{k}}\bigl({\bf Q}f\bigr)\,g \\
     & =
    \CE{\anc{u}{k+1}}
    \sum_{t=0}^{k} {\bL}_{k,t} \,\otimes\, \EE{\OO{u}{k}}\bigl({\bf Q}f\bigr)\,g.
\end{align*}

Because
\[
    \Fz{\anc{u}{t}}\otimes \R
    \;=\;
    \Fz{\anc{u}{t}}
    \;\xrightarrow{\EE{\anc{u}{k+1}}}\;
    \{0\},
\]
we can define the bilinear map
\[
    {\bf L}_{k+1,t}(f,g)
    \;:=\;
    {\bL}_{k,t} \otimes \EE{\OO{u}{k}} \bigl({\bf Q}f\bigr)\,g,
\]
which takes \(\cR{{{\cal W}_u}}{k+1} \times \T{\anc{u}{k+1}}\) to \(\Fz{\anc{u}{t}}\). Substituting these definitions of
\({\bf L}_{k+1,t}\) into \eqref{eq RCEfgDecomposition} yields
\begin{align}
    \label{eq RLDecomposition}
    \CE{\anc{u}{k+1}} \bigl[f\,g\bigr]
    \;=\;
    \sum_{t=0}^{k+1} {\bf L}_{k+1,t}(f,g).
\end{align}
Thus, each \({\bf L}_{k+1,t}\) captures a piece of the decomposition corresponding to level \(t\), and the sum recovers the full correlation term
\(\CE{\anc{u}{k+1}}[f\,g].\)

\begin{lemma}\label{lem RbLk+1t}
    Suppose the Induction Hypothesis~\ref{IH: coreR} holds for $k$. Then for each $t \in [0,k]$,
    the operator $\bL_{k+1,t}$ satisfies the corresponding condition of Induction Hypothesis~\ref{IH: coreR}
    for $k+1$.
\end{lemma}

\begin{proof}
    The result follows directly from the induction hypothesis for $k$, combined with a tensorization argument.

    First, recall that
    \[
        {\bf Q}f\;\in\;\cR{{{\cal W}_u}}{k}\,\otimes\,\TT{u}{k}
        \quad\text{and}\quad
        g \;\in\;\T{\anc{u}{k+1}} \;\subseteq\;\T{\anc{u}{k}}\,\otimes\,\TT{u}{k}.
    \]
    By the induction hypothesis for $k$, we know that for all
    \(\phi_1 \in \cR{{{\cal W}_u}}{k}\) and \(\phi_2 \in \T{\anc{u}{k}}\),
    \[
        \maxnorm{\bL_{k,t}(\phi_1,\phi_2)}
        \;\le\;
        \begin{cases}
            \CW{{\cal W}_u}\,(1+\kappa)^{3k}\,
            \Unorm{\phi_1}{\anc{u}{k}}\;\Unorm{\phi_2}{\anc{u}{k}},
             & t=0,         \\[6pt]
            \tlam^{t}\,\exp\bigl(-\eps t\bigr)\,\tC_{\ref{IH: coreR}}\,
            R\,\exp\bigl(-\eps\,\lA{u}\bigr)\,\CW{{\cal W}_u}\,(1+\kappa)^{3k}\,
            \Unorm{\phi_1}{\anc{u}{k}}\;\Unorm{\phi_2}{\anc{u}{k}},
             & t \in [1,k].
        \end{cases}
    \]
    Additionally, for any \(\phi_1,\phi_2 \in \TT{u}{k}\), we have the trivial bound (via Cauchy-Schwarz inequality):
    \[
        \maxnorm{\EE{\OO{u}{k}}\!\bigl[\phi_1\,\phi_2\bigr]}
        \;\le\;
        \Unorm{\phi_1}{\OO{u}{k}}\;\Unorm{\phi_2}{\OO{u}{k}}.
    \]

    Next, we apply
    Lemma~\ref{lem: mainTensorProduct},
    Lemma~\ref{lemma: coreNorm},
    and
    \eqref{eq RbLk+1k+101} from Lemma~\ref{lem RbLk+1k+1}
    to handle the expression
    \(\maxnorm{ \bL_{k,t} \otimes \EE{\OO{u}{k}}\bigl(({\bf Q}f)\,g\bigr)}\).
    For $t \in [1,k]$, we obtain:
    \[
        \begin{aligned}
               & \phantom{\le\;}\maxnorm{\bL_{k+1,t} (f,\,g)}                    \\
            \; & \le\;
            \maxnorm{\bL_{k,t} \otimes \EE{\OO{u}{k}} \bigl(({\bf Q}f)\,g\bigr)} \\
            \; & \le\;
            \tlam^{t}\,\exp\bigl(-\eps t\bigr)\,\tC_{\ref{IH: coreR}}\,
            R\,\exp\bigl(-\eps\,\lA{u}\bigr)\,\CW{{\cal W}_u}\,(1+\kappa)^{3k}
            \;\Unorm{{\bf Q}f}{\anc{u}{k}\cup\OO{u}{k}}\;\Unorm{g}{\anc{u}{k}\cup\OO{u}{k}}
            \\[6pt]
               & \le\;
            \tlam^{t}\,\exp\bigl(-\eps t\bigr)\,\tC_{\ref{IH: coreR}}\,
            R\,\exp\bigl(-\eps\,\lA{u}\bigr)\,\CW{{\cal W}_u}\,(1+\kappa)^{3k+2}
            \;\Unorm{{\bf Q}f}{\anc{u}{k+1}}\;\Unorm{g}{\anc{u}{k+1}}
            \\[6pt]
               & \le\;
            \tlam^{t}\,\exp\bigl(-\eps t\bigr)\,\tC_{\ref{IH: coreR}}\,
            R\,\exp\bigl(-\eps\,\lA{u}\bigr)\,\CW{{\cal W}_u}\,(1+\kappa)^{3k+3}
            \;\Unorm{f}{\anc{u}{k+1}}\;\Unorm{g}{\anc{u}{k+1}}.
        \end{aligned}
    \]
    A similar argument for $t=0$ shows that
    \[
        \maxnorm{\bL_{k,t} \otimes \EE{\OO{u}{k}}\bigl(({\bf Q}f)\,g\bigr)}
        \;\le\;
        \CW{{\cal W}_u}\,(1+\kappa)^{3k+3}
        \;\Unorm{f}{\anc{u}{k+1}}\;\Unorm{g}{\anc{u}{k+1}}.
    \]

    Thus, in every case (i.e., for all $t \in [0,k]$), the operator
    $\bL_{k+1,t}$ satisfies the required correlation bounds with respect to
    \(\cR{{{\cal W}_u}}{k+1}\) and \(\T{\anc{u}{k+1}}\). This completes the proof.
\end{proof}

\subsection{Conclusion: Proof of Proposition~\ref{prop: coreR}}
\begin{proof}[Proof of Proposition \ref{prop: coreR}]
    \step{Bounds for $\maxnorm{\CE{\anc{u}{k+1}}fg}$ (First Statement).}
    First, observe that if the Induction Hypothesis~\ref{IH: coreR} holds at level $k$,
    then Lemma~\ref{lem RIHImplication} implies the first statement of the proposition.
    Since the base case $k=0$ of the hypothesis is true by definition, and the induction step is verified
    in Lemmas~\ref{lem RbLk+1k+1} and~\ref{lem RbLk+1t}, it follows that the induction hypothesis holds for all $k$.
    Consequently, the first statement of Proposition~\ref{prop: coreR} is established.

    Now we consider the second statement of the proposition.

    \step{Pseudo-inverse and the functions $f_t$}
    Recall the pseudo-inverse introduced in Definition~\ref{def  RPseudoInverse}.
    For each $t \in [0,k-1]$, let
    \[
        {\bf Q}_t \;:\; \cR{{{\cal W}_u}}{t+1} \;\longrightarrow\;
        \cR{{{\cal W}_u}}{t} \otimes \TT{u}{t}
    \]
    be the pseudo-inverse of the operator $(\idtt{\anc{u}{t+1}} - \PK{\anc{u}{t+1}})$.
    For simplicity, define
    \[
        {\bf I}_{t+1} \;=\; \idtt{\OO{u}{[t+1,k-1]}},
    \]
    so that the tensor product \({\bf Q}_t \otimes {\bf I}_{t+1}\) induces a linear map
    \[
        \cR{{{\cal W}_u}}{t+1} \,\otimes\, \TT{u}{[t+1,k-1]}
        \;\xrightarrow{{\bf Q}_t \otimes {\bf I}_{t+1}}\;
        \cR{{{\cal W}_u}}{t} \,\otimes\, \TT{u}{[t,k-1]}.
    \]
    We now define $f_t$ recursively by
    \[
        f_t \;:=\; \bigl({\bf Q}_t \otimes {\bf I}_{t+1}\bigr)\,f_{t+1},
        \quad\text{with the base case } f_k = f.
    \]

    \step{Showing $f_{t+1} - f_t \in \T{\anc{u}{t+1}} \otimes \TT{u}{[t+1,k-1]}$}
    From the definition of $f_t$, we compute
    \[
        f_{t+1} - f_t
        \;=\;
        (\idtt{\anc{u}{t+1}} - \PK{\anc{u}{t+1}})\,\otimes\,{\bf I}_{t+1}\,\circ\,
        \bigl({\bf Q}_t \otimes {\bf I}_{t+1}\bigr)f_{t+1}
        \;-\;
        f_t
        \;=\;
        -\,\PK{\anc{u}{t+1}}\otimes {\bf I}_{t+1}\,f_t.
    \]
    Clearly, this lies in
    \(\T{\anc{u}{t+1}} \otimes \TT{u}{[t+1,k-1]}\).
    Summing over $s$ from $t$ to $k-1$, we find that
    \[
        f_k - f_t
        \;=\;
        \sum_{s=t}^{k-1}\!\bigl(f_{s+1} - f_s\bigr)
        \;\;\in\;\;
        \T{\anc{u}{t+1}} \otimes \TT{u}{[t+1,k-1]},
    \]
    since each difference $(f_{s+1} - f_s)$ also lies in that space.

    \step{Bounding $\Unorm{f_{t+1} - f_t}{\anc{u}{k}}$ by $\Unorm{f_k}{\anc{u}{k}}$}
    Using Lemma~\ref{lem RProjectionNorm} and the submultiplicativity of the tensor-product norm (Lemma~\ref{lem tensorNorm}), we obtain:
    \begin{align*}
        \Unorm{f_{t+1} - f_t}{\anc{u}{t+1} \cup \OO{u}{[t+1,k-1]}}
        \;=\;   &
        \Unorm{-\,\PK{\anc{u}{t+1}} \otimes {\bf I}_{t+1}\,f_t}
        {\anc{u}{t+1} \cup \OO{u}{[t+1,k-1]}} \\
        \;\le\; &
        \RDelta_t\,\Unorm{f_t}{\anc{u}{t+1}\cup \OO{u}{[t+1,k-1]}}.
    \end{align*}

    Since $\anc{u}{t+1} \cup \OO{u}{[t+1,k-1]} \preceq \anc{u}{k}$, we now apply Lemma~\ref{lemma: coreNorm} to convert the norms on both sides:
    \[
        \Unorm{f_{t+1} - f_t}{\anc{u}{k}}
        \;\le\;
        \RDelta_t\,(1+\kappa)^2 \,\Unorm{f_t}{\anc{u}{k}},
    \]
    which in turn implies (by the triangle inequality) that
    \[
        \Unorm{f_t}{\anc{u}{k}}
        \;\le\;
        \bigl(1 + \RDelta_t\,(1+\kappa)^2\bigr)\,\Unorm{f_{t+1}}{\anc{u}{k}}.
    \]
    Iterating these inequalities from $t$ to $k-1$ then gives:
    \begin{align*}
        \Unorm{f_{t+1} - f_t}{\anc{u}{k}}
        \;\le\; &
        \RDelta_t\,(1 + \kappa)^{2}
        \prod_{s=t}^{k-1} \Bigl(1+\RDelta_s\,(1+\kappa)^{2}\Bigr)\,
        \Unorm{f_k}{\anc{u}{k}} \\
        \;\le\; &
        \RDelta_t\,(1 + \kappa)^2
        \exp\!\Bigl(\!\sum_{s=t}^{k-1}\RDelta_s\,(1+\kappa)^{2}\Bigr)\,
        \Unorm{f_k}{\anc{u}{k}}.
    \end{align*}

    \step{Estimating $\sum_{s=t} \RDelta_s$}
    To estimate the above term, we need to estimate $\sum_{s=t} \RDelta_s$.
    The definition of $\RDelta_s$ has a recurrence relation
    \begin{align*}
        \RDelta_{s+1} = \tlam (1+\kappa)^3 \exp(-\eps) \RDelta_s .
    \end{align*}
    With $\tlam (1+\kappa)^2 \exp(-\eps) < \lame \exp(-\eps) <1$, the sum is bounded by a geometric series:
    $$
        \sum_{s=t} \RDelta_s \le \tC \RDelta_t\,
    $$
    for some $\tC \in \CC$.

    Recall that $\RDelta_t \le \RDelta_0$ and the term $\RDelta_0$ can be made arbitrarily small by choosing $\tCR \in \CC$ large enough (see Remark \ref{rem RDelta}). We may assume that
    $$
        \exp\left(\sum_{s=t}^{k-1} \RDelta_s (1+ \kappa)^{2}\right)
        \le
        1+\kappa\,,
    $$
    and we conclude that
    $$
        \Unorm{f_{t+1} - f_t}{\anc{u}{k}}
        \le
        \RDelta_t (1+ \kappa)^{3} \Unorm{f_{k}}{\anc{u}{k}}\,.
    $$

    \step{Bounding $\Unorm{f_{k} - f_t}{\anc{u}{k}}$ by $\Unorm{f}{\anc{u}{k}}$}
    As an immediate consequence of the above inequality, we have
    \begin{align*}
        \Unorm{f_k - f_t}{\anc{u}{k}}
        \le
        \left(\sum_{s=t} \RDelta_s (1+\kappa)^{3}\right) \Unorm{f}{\anc{u}{k}}
        \le
        \tC \RDelta_t (1+\kappa)^3 \Unorm{f}{\anc{u}{k}}\,.
    \end{align*}
    Finally, recall the definition of $\RDelta_t$:
    $$
        \RDelta_t =
        \frac{2 \tCM \tC_{\ref{lem basicDecayAu}}\exp(\eps)}{\tlam}
        R \exp( -\eps \lA{u}) \cdot \CW{{\cal W}_u}
        \tlam^{t+1} (1+\kappa)^{3t+2} \exp(- \eps t)\,.
    $$
    Together with
    $$
        \tlam (1+\kappa)^6 \le \lame
    $$
    from the Definition \ref{def varepsilon},
    we obtain the second statement of the proposition:
    \begin{align*}
        \Unorm{f - f_t}{\anc{u}{k}}
        \le
        \CW{{\cal W}_u} \cdot
        \tC R \exp(-\eps \lA{u})
        \lame^t \exp(- \eps t)
        \Unorm{f}{\anc{u}{k}}\,,
    \end{align*}
    where the term $\tC$ has been increased, while remaining in $\CC$, since $  \tCM$,  $\tC_{\ref{lem basicDecayAu}}$, $\exp(\eps)$, $\tlam$, and $\kappa$ are all in $\CC$.

\end{proof}

\bigskip
\section{\texorpdfstring{\(\mathcal{R}({\mathcal T}_{K+1}(u);0)\)}{R} for $u$ with small \texorpdfstring{${\rm h}_K(u)$}{hK(u)}}
\label{sec: CWTK}
\subsection{Overview}
Recall that $\Tp{u}$ denotes the space of functions in the variables $x_{L_u}$, with degree bounded by $2^{K+1}$, where $L_u$ is the set of leaves that are descendants of $u$. The space
$$
    \cR{\Tp{u}}{0} = (\idtt{u} - \PK{u}) \Tp{u}
$$
is the image of \( \Tp{u} \) under the ``projection" to the orthogonal complement of \( \T{u} \).

In this section, we aim to prove the third statement of Proposition \ref{prop PT MAIN}. For the reader's convenience, we restate it here.

\begin{prop*}
    For any $ 1 < \tCB \in \CC$, the following holds when $\tCR$ is sufficiently large:
    For any $u \in V(T)$ satisfying
    \begin{align*}
        \lA{u} \ge \hB : = \left\lceil \tCR (\log(R) +1) + \frac{\eps}{10 d} \tCR (\log(R) +1) \right\rceil.
    \end{align*}
    we have
    \begin{align*}
        \CW{{\cal T}_{K+1}(u)} \le \frac{1}{\tCB R}\,.
    \end{align*}
\end{prop*}

\noindent
For later use, we also define
\begin{align}
    \label{def tm}
    \tm =  \left\lfloor \frac{\eps}{10 d} \tCR (\log(R) +1) \right\rfloor\,.
\end{align}
\begin{rem}\label{rem hBtmTechnical}
    \textbf{Why do we need \(\hB\) and \(\tm\)?}
    Both parameters are mainly technical. Previously, we only required
    \(\lA{u} \ge \tCR(\log(R)+1)\) to derive all the decay estimates used so far.
    However, because \(\Tp{u}\) cannot be decomposed into \(\T{A}\) for some \(A \subseteq \OO{u}{}\),
    we now need an additional “margin” of \(\tm\) layers to leverage some new decay properties
    \emph{and} to keep using the earlier estimates.
    Consequently, \(\hB\) and \(\tm\) ensure that \(u\) is sufficiently tall to satisfy both the original condition
    and the extra decay requirements introduced in this section.
\end{rem}

Recall from the definition of $\CW{\Tp{u}}$, the proposition states that
$f \in \cR{\Tp{u}}{0}$ and $g \in \T{u}$, we have
\begin{align*}
    \maxnorm{ \D{u} [f \cdot g]}
    \le
    \frac{1}{\tCB R} \Unorm{f}{u} \Unorm{g}{u}\,.
\end{align*}

To proceed, let us introduce some necessary notation.
For each $u \in V(T)$ with $\lA{u} \ge \hB$, define
\begin{align}
    \label{def Du}
    \Dm{u} := \{ v' \le u \,:\, \h(u)-\h(v') = \tm \}\,,
\end{align}
the set of $\tm$--th descendants of $u$. By the assumption $\lA{u} \ge \hB$ in the Proposition, it follows that
$$
    \lA{v} \ge \tCR (\log(R) +1)
    \mbox{ for all } v \in \Dm{u}\,,
$$
enabling us to apply the decay properties established in earlier sections to functions in $\T{\Dm{u}}$.

However, we cannot directly handle the functions in $\Tp{u}$ relying on any previous results.
To bridge this gap, we introduce a ``projection-like'' operator $\PDM{u}$ that maps functions in $\Tp{u}$ into a more tractable space.

\begin{prop}
    \label{prop PT PDM}
    For each $u \in V(T)$ satisfying $\lA{u} \ge \hB$, there exists a linear operator
    \begin{align*}
        \PDM{u}: \Tp{u} \to  \Deg{\Dm{u}} \otimes \T{\Dm{u}} ,\,
    \end{align*}
    where
    \begin{align}
        \label{def Deg}
        \Deg{\Dm{u}} := \left\{
        \mbox{ functions of variables $x_{\Dm{u}}$ of Efron-Stein degree $\le 1$}
        \right\}\,,
    \end{align}
    such that the following holds:
    \begin{itemize}
        \item  For any $f \in \mathcal{T}_{K+1}(u)$ and $g \in \T{u}$, we have
              \begin{align}
                  \label{eq PDMSame}
                  \CE{u}\big[((\idtt{u} - \PK{u})f)g\big] = \CE{u}\big[((\idtt{u} - \PK{u}) \PDM{u}f)g\big]\,.
              \end{align}
        \item For any $f \in \mathcal{T}_{K+1}(u)$, we have
              \begin{align}
                  \label{eq PDMnorm}
                  \Unorm{\PDM{u}f}{\Dm{u}} \le  \Unorm{f}{\Dm{u}}\,.
              \end{align}
    \end{itemize}
\end{prop}
Intuitively,  $\PDM{u}$ can be viewed as a projection operator that send $\mathcal{T}_{K+1}(u)$ into $\Deg{\Dm{u}} \otimes \T{\Dm{u}}$.
From this perspective, the
proposition suggests
$$
    \CW{{\cal T}_{K+1}(u)} \simeq \CW{\Deg{\Dm{u}} \otimes \T{\Dm{u}}}\,,
$$
since working with $\Deg{\Dm{u}} \otimes \T{\Dm{u}}$ allow us to
exploit the decay properties for degree-1 polynomials
as well as functions in $\T{\Dm{u}}$.
This allows us to establish the following proposition.
\begin{prop}
    \label{prop PDM DEG1T}
    With the same assumptions as in Proposition \ref{prop PT MAIN}, the following holds when $\tCR$ is sufficiently large:
    For $f \in \Deg{\Dm{u}} \otimes \T{\Dm{u}}$ and $g \in \T{u}$,
    we have
    \begin{align}
        \label{eq PDM DEG1T}
        \maxnorm{\D{u} [((\idtt{u} - \PK{u})f)g]}
        \le
        \frac{1}{\tCB R} \Unorm{f}{u} \Unorm{g}{u}\,.
    \end{align}
\end{prop}
If the two propositions above hold, then we can conclude the main result:

\begin{proof}[Proof of the third statement in Proposition~\ref{prop PT MAIN}]
    Let \(f\) and \(g\) be functions in \(\cR{\Tp{u}}{0}\) and \(\T{u}\), respectively.
    We first claim that \(f = (\idtt{u} - \PK{u})f\), or equivalently \(\PK{u}f = 0\).

    Indeed, for any \(\phi \in \T{u}\), property~\eqref{eq PiProperty1} of \(\PK{u}\) gives
    \[
        \EE{u}\bigl[(\PK{u}f)\,\phi \bigr]
        \;=\;
        \EE{u}\bigl[f\,\phi\bigr].
    \]
    Then, by Lemma~\ref{lem PK_Orthogonal},
    \[
        \EE{u}\bigl[f\,\phi\bigr]
        \;=\; 0.
    \]
    Given that it holds for every $\phi \in \T{u}$,
    applying property~\eqref{eq PiProperty2} of \(\PK{u}\) therefore implies
    \[
        \PK{u}f
        \;=\; 0.
    \]

    Now,
    \begin{align*}
        \maxnorm{\D{u}\,[f\,g]}
        \;=\;
        \maxnorm{\D{u}\,\bigl[\bigl(\idtt{u} - \PK{u}\bigr)f \cdot g\bigr]}
        \;\stackrel{\eqref{eq PDMSame}}{=}\; &
        \maxnorm{\D{u}\,\bigl[\bigl(\idtt{u} - \PK{u}\bigr)\,\PDM{u}f \cdot g\bigr]}    \\
                                             & \;\stackrel{\eqref{eq PDM DEG1T}}{\le}\;
        \frac{1}{\tCB\,R}\,
        \Unorm{\PDM{u}f}{u}\,\Unorm{g}{u}
        \;\stackrel{\eqref{eq PDMnorm}}{\le}\;
        \frac{1}{\tCB\,R}\,
        \Unorm{f}{u}\,\Unorm{g}{u}.
    \end{align*}
\end{proof}

In the next two subsections, we will prove Proposition \ref{prop PT PDM} and Proposition \ref{prop PDM DEG1T}, respectively.

\subsection{The Operator \texorpdfstring{${\bf \Gamma}_u$}{P} and Proof of Proposition \ref{prop PT PDM}}
We begin by constructing some operators that will lead to the definition of $\PDM{u}$.
\begin{defi}\label{def PT}
    For each \(u \in V(T)\), we define a linear operator
    \[
        \PT{u} : \F{u_{\pe}} \;\longrightarrow\; \F{u} \,\otimes\, \T{u}
    \]
    as follows.

    First, recall that for \(f, g \in \F{u_{\pe}}\), the expression \(\CE{u}[f\,g]\) is a function in \(\F{u}\), i.e., a function of \(x_u\).
    For each \(\theta \in [q]\), consider the map
    \[
        (f,g) \;\mapsto\; (\CE{u}[f\,g])(\theta),
    \]
    where we evaluate \(\CE{u}[f\,g]\) at \(x_u = \theta\).  This induces a semi-positive definite bilinear form on \(\F{u_{\pe}}\).

    By Lemma~\ref{lem LA_Pi}, there exists a projection operator
    \[
        \Gamma_{u,\theta}: \F{u_{\pe}} \;\longrightarrow\; \T{u}
    \]
    such that
    \begin{align}\label{eq PT GammauTheta00}
        (\CE{u}[f\,g])(\theta)
        \;=\;
        \bigl(\CE{u}\bigl[\Gamma_{u,\theta}f \cdot g\bigr]\bigr)(\theta)
        \quad
        \text{for all } f \in \F{u_{\pe}} \text{ and } g \in \T{u}.
    \end{align}

    We then define
    \[
        \PT{u}f
        \;:=\;
        \sum_{\theta \in [q]}
        \mathbf{1}_\theta(x_u)\,\Gamma_{u,x_u}f
        \;\in\;
        \F{u}\,\otimes\,\T{u},
    \]
    where \(\mathbf{1}_\theta(x_u)\) is the indicator function for \(x_u = \theta\).
\end{defi}

\begin{lemma}\label{lem:PTprops}
    For each \(u \in V(T)\), the operator \(\PT{u}\) has the following properties:
    \begin{itemize}
        \item \textbf{(Norm Reduction)}
              For all \(f \in \F{u_{\pe}}\),
              \begin{align}\label{eq PT Ortho}
                  \Unorm{\PT{u}f}{u} \;\le\; \Unorm{f}{u}.
              \end{align}
        \item \textbf{(Identity)}
              For all \(g \in \T{u}\),
              \begin{align}\label{eq PTTu}
                  \PT{u}\,g \;=\; g.
              \end{align}
        \item \textbf{(Strong Orthogonality)}
              For any \(f \in \F{u_{\pe}}\) and \(g \in \T{u}\),
              \begin{align}\label{eq PT StrongOrtho}
                  \CE{u}\bigl[(f - \PT{u}f)\,g\bigr] \;=\; 0 \;\in\; \F{u},
              \end{align}
              meaning that \(\CE{u}[(f - \PT{u}f)\,g]\) is identically zero in \(\F{u}\) (i.e., it vanishes for every \(x_u \in [q]\)).
              We refer to this property as \emph{strong orthogonality} because it is strictly stronger than the usual condition
              \(\EE{u}\bigl[(f - \PT{u}f)\,g\bigr] = 0\), requiring pointwise vanishing rather than just a zero mean.
    \end{itemize}
\end{lemma}

\begin{proof}
    \step{Norm Reduction \eqref{eq PT Ortho}}
    By definition,
    \[
        \PT{u}f \;=\; \sum_{\theta \in [q]} \mathbf{1}_\theta(x_u)\,\Gamma_{u,\theta}\,f.
    \]
    Then,
    \[
        \CE{u}\bigl[f\,(\PT{u}f)\bigr](\theta)
        \;=\;
        \CE{u}\bigl[f\,\Gamma_{u,\theta}f\bigr](\theta)
        \;\stackrel{\eqref{eq PT GammauTheta00}}{=}\;
        \CE{u}\bigl[(\Gamma_{u,\theta}f)^2\bigr](\theta).
    \]
    Consequently,
    \[
        \EE{u}\bigl[f\,(\PT{u}f)\bigr]
        \;=\;
        \Unorm{\PT{u}f}{u}^2.
    \]
    By the Cauchy–Schwarz inequality,
    \[
        \Unorm{\PT{u}f}{u}^2
        \;=\;
        \EE{u}\bigl[f\,(\PT{u}f)\bigr]
        \;\le\;
        \Unorm{f}{u}\,\Unorm{\PT{u}f}{u}.
    \]
    If \(\Unorm{\PT{u}f}{u}\), then \eqref{eq PT Ortho} follows immediately.
    Otherwise, dividing both sides by \(\Unorm{\PT{u}f}{u}\) yields \eqref{eq PT Ortho}.

    \step{Identity \eqref{eq PTTu}}
    Since \(\Gamma_{u,\theta}\) is a projection, we have \(\Gamma_{u,\theta}g = g\) for \(g \in \T{u}\). Hence,
    \[
        \PT{u}g
        \;=\;
        \sum_{\theta \in [q]} \mathbf{1}_\theta(x_u)\,\Gamma_{u,\theta}g
        \;=\;
        g.
    \]

    \step{Strong Orthogonality \eqref{eq PT StrongOrtho}}
    For each \(\theta \in [q]\),
    \[
        \CE{u}\bigl[(f - \PT{u}f)\,g\bigr](\theta)
        \;=\;
        \CE{u}[f\,g](\theta)
        \;-\;
        \CE{u}[(\PT{u}f)\,g](\theta)
        \;=\;
        \CE{u}[f\,g](\theta)
        \;-\;
        \CE{u}\bigl[\Gamma_{u,\theta}f\,g\bigr](\theta).
    \]
    By \eqref{eq PT GammauTheta00} from the assumption of $\Gamma_{u,\theta}$,
    \[
        \CE{u}\bigl[\Gamma_{u,\theta}f\,g\bigr](\theta)
        \;=\;
        \CE{u}[f\,g](\theta),
    \]
    so the difference is zero. Since this holds for every \(\theta \in [q]\), it follows that
    \(\CE{u}[(f - \PT{u}f)\,g]\) is the zero function in \(\F{u}\).
\end{proof}

Now we are ready to define the operator \(\PDM{u}\).

\begin{defi}\label{def PDM}
    For each \(u \in V(T)\) with \(\lA{u} \ge \hB\), we define
    \begin{align*}
        \PDM{u} \;:=\; \bigotimes_{v \in \Dm{u}}\;\PT{v}.
    \end{align*}
\end{defi}

Our first goal is to verify that \(\PDM{u}\) maps \(\mathcal{T}_{K+1}(u)\) into
\(\Deg{\Dm{u}} \otimes \T{\Dm{u}}\), as claimed in Proposition~\ref{prop PT PDM}.

\begin{lemma}\label{lem PT PDMrange}
    For each \(u \in V(T)\) with \(\lA{u} \ge \hB\),
    \begin{align}\label{eq PDM}
        \PDM{u}\bigl(\Tp{u}\bigr)
        \;\subseteq\;
        \Deg{\Dm{u}} \;\otimes\; \T{\Dm{u}}.
    \end{align}
\end{lemma}

\begin{proof}
    We claim that
    \begin{align}\label{eq TPDecompose}
        \Tp{u}
        \;\subseteq\;
        \sum_{v \in \Dm{u}} \F{v_{\pe}} \,\otimes\, \T{\Dm{u}\setminus \{v\}}.
    \end{align}
    Proving this for monomials suffices.

    Suppose \(\phi \in \Tp{u}\) is a monomial in variables \(x_S\), where
    \(S \subseteq L_u\) and \(|S|\le 2^{K+1}\).
    By definition, \(\phi\) may have degree higher than $2^K$ in the coordinates
    \(x_{L_v}\) for at most one node \(v\in \Dm{u}\).
    - If there is such a \(v \in \Dm{u}\) with \(\deg_{x_{L_v}}(\phi) > 2^K\), then
    \(\phi \in \F{v_{\pe}} \otimes \T{\Dm{u}\setminus\{v\}}\).
    - Otherwise, \(\phi\) already belongs to \(\T{\Dm{u}}\), which is contained in every summand on the right-hand side of~\eqref{eq TPDecompose}.

    Since any \(f \in \Tp{u}\) is a linear combination of such monomials, the decomposition~\eqref{eq TPDecompose} holds.

    Next, for each \(v \in \Dm{u}\),
    \begin{align*}
        \PDM{u}\Bigl(\F{v_{\pe}} \otimes \T{\Dm{u}\setminus \{v\}}\Bigr)
         & \;=\;
        \Bigl(\!\!\bigotimes_{w\in \Dm{u}} \PT{w}\Bigr)
        \Bigl(\F{v_{\pe}} \otimes \T{\Dm{u}\setminus \{v\}}\Bigr) \\
         & \;\subseteq\;
        \bigl(\F{v}\otimes \T{v}\bigr) \;\otimes\; \T\!\bigl(\Dm{u}\setminus\{v\}\bigr)
        \quad\text{(by \eqref{eq PTTu})}                          \\[6pt]
         & \;=\;
        \F{v}\,\otimes\,\T{\Dm{u}}.
    \end{align*}
    Summing over \(v\in \Dm{u}\) gives
    \begin{align*}
        \sum_{v\in \Dm{u}} \F{v}\,\otimes\,\T{\Dm{u}}
        \;=\;
        \Deg{\Dm{u}} \;\otimes\; \T{\Dm{u}},
    \end{align*}
    completing the proof.
\end{proof}

\begin{proof}[Proof of Proposition \ref{prop PT PDM}]
    By definition of \(\PDM{u}\) and Lemma~\ref{lem PT PDMrange}, we may view 
    \(\PDM{u}\) as a linear operator from \(\Tp{u}\) into \(\Deg{\Dm{u}} \otimes \T{\Dm{u}}\).
    We now proceed to prove \eqref{eq PDMSame}.

    \step{Strong orthogonality for \(\idtt{u} - \PDM{u}\).}
    Let \(f \in \mathcal{T}_{K+1}(u)\). Label the vertices in \(\Dm{u}\) as 
    \(v_1, v_2, \dots, v_{|\Dm{u}|}\) in any fixed order. We can write 
    \(\idtt{u} f - \PDM{u}f\) in a telescoping sum:
    \begin{align*}
      \idtt{u} f \;-\; \PDM{u}f
      &= 
      \biggl(\!\!\bigotimes_{i \in [|\Dm{u}|]}\! \idtt{v_i}\biggr)f
      \;-\;
      \biggl(\!\!\bigotimes_{i \in [|\Dm{u}|]}\! \PT{v_i}\biggr)f \\[6pt]
      &=
      \sum_{i \in [|\Dm{u}|]}
        \biggl(\!\!\bigotimes_{j \in [i-1]} \PT{v_j}\biggr)
        \;\otimes\;
        \bigl(\idtt{v_i} - \PT{v_i}\bigr)
        \;\otimes\;
        \biggl(\!\!\bigotimes_{j \in [i+1,\,|\Dm{u}|]} \idtt{v_j}\biggr)f.
    \end{align*}
  
    Now fix any \(g \in \T{u} \subseteq \T{\Dm{u}}\). For each summand above, 
    strong orthogonality \eqref{eq PT StrongOrtho} on the factor $\idtt{v_i} - \PT{v_i}$
    implies
    \[
      \CE{\Dm{u}}\Bigl[
        \Bigl(\!\!\bigotimes_{j \in [i-1]} \PT{v_j}\Bigr)\,\otimes\,
        (\idtt{v_i} - \PT{v_i}) \,\otimes\,
        \Bigl(\!\!\bigotimes_{j \in [i+1,\,|\Dm{u}|]} \idtt{v_j}\Bigr)\,f
        \;\cdot\; g
      \Bigr]
      \;=\; 0.
    \]
    Consequently,
    \begin{align}\label{eq PT PDM 00}
      \CE{u} \Bigl[\bigl(\idtt{u} f - \PDM{u}f\bigr) g \Bigr]
      \;\stackrel{\eqref{eq EEU'ECU}}{=}\;
      \CE{u}\circ \CE{\Dm{u}} \Bigl[\bigl(\idtt{u} f - \PDM{u}f\bigr) g \Bigr]
      \;=\;
      0.
    \end{align}

    \step{Establishing \eqref{eq PDMSame}.}
    Recall property \eqref{eq PiProperty2} of \(\PK{u}\): for \(\phi_1 \in \F{u_{\pe}}\),
    \[
      \CE{u}\bigl[\phi_1\,\phi_2\bigr] = 0 \text{ for all } \phi_2 \in \T{u}^\perp
      \quad\implies\quad
      \PK{u}\,\phi_1 = 0.
    \]
    From \eqref{eq PT PDM 00}, we see that for every \(g \in \T{u}\),
    \begin{align}\label{eq PT PDM 01}
      \CE{u}\Bigl[(\idtt{u}f - \PDM{u}f)\,g\Bigr] 
      \;=\; 0
      \quad\stackrel{\eqref{eq PiProperty2}}{\Longrightarrow}\quad
      \PK{u}\,\bigl(\idtt{u} f - \PDM{u}f\bigr) 
      \;=\;
      0.
    \end{align}
    Using \eqref{eq PT PDM 00} again, for \(g \in \T{u}\) we get
    \[
      \CE{u}\Bigl[\bigl((\idtt{u} - \PK{u})f\bigr)\,g\Bigr]
      \;=\;
      \CE{u}\Bigl[\bigl((\idtt{u} - \PK{u})\circ(\idtt{u}-\PDM{u})\,f\bigr)\,g\Bigr]
      \;+\;
      \CE{u}\Bigl[\bigl((\idtt{u} - \PK{u})\,\PDM{u}f\bigr)\,g\Bigr].
    \]
    By \eqref{eq PT PDM 01}, the first term vanishes:
    \[
      \CE{u}\Bigl[\bigl((\idtt{u} - \PK{u})\circ(\idtt{u}-\PDM{u})\,f\bigr)\,g\Bigr]
      \;=\; 
      \CE{u}\Bigl[\bigl(\idtt{u}-\PDM{u}\bigr)f\,g\Bigr]
      \;=\; 0
      \quad\text{(from \eqref{eq PT PDM 00})}.
    \]
    Hence,
    \[
      \CE{u}\Bigl[\bigl((\idtt{u} - \PK{u})f\bigr)\,g\Bigr]
      \;=\;
      \CE{u}\Bigl[\bigl((\idtt{u} - \PK{u})\,\PDM{u}f\bigr)\,g\Bigr],
    \]
    which completes the proof of \eqref{eq PDMSame}.

    \step{Norm reduction via submultiplicativity.}
    It remains to show \eqref{eq PDMnorm}. First, observe that for each \(v\in\Dm{u}\), 
    \(\PT{v}: \F{v_{\pe}} \to \F{v_{\pe}}\) satisfies
    \[
      \EE{v}\bigl[(\PT{v}\,\phi)^2\bigr]
      \;\le\;
      \EE{v}\bigl[\phi^2\bigr]
      \quad\text{for all }\phi\in\F{v_{\pe}},
    \]
    where the inequality follows from \eqref{eq PT Ortho} (the norm reduction property).
  
    By taking \(V_v = \F{v_{\pe}}\), \(E_v = \EE{v}\), \(L_v = \PT{v}\), and \(\delta_v=1\), we can apply Lemma~\ref{lem tensorNorm} to obtain
    \[
      \Unorm{\PDM{u}f}{\Dm{u}}
      \;\le\;
      \Unorm{f}{\Dm{u}},
    \]
    proving \eqref{eq PDMnorm}.
  \end{proof}

\subsection{Properties of ${\cal D}_1(D_{\tt m}(u)) \otimes {\cal T}_K(D_{\tt m}(u))$}

\noindent
As a prerequisite for proving Proposition~\ref{prop PDM DEG1T}, we need to establish some properties of the space \(\T{\Dm{u}}\). 
These properties are essentially \emph{variations} of the results from Section~\ref{sec T}, 
but specialized to the collection of vertices \(\Dm{u}\):

\begin{lemma}\label{lem PT TDmu}
  The following results hold for sufficiently large \(\tCR\). 
  Suppose \(u \in V(T)\) satisfies \(\lA{u} \ge \hB\). 
  Then, for any two functions \(f, g \in \T{\Dm{u}}\):
  \begin{itemize}
    \item First, it is a decay property:  
    \begin{align}
        \label{eq PT TDmu}
        \maxnorm{\D{\Dm{u}} \bigl[f\,g\bigr]}
        \;\le\;
        2 \,\exp\!\Bigl(-\,\tfrac{1}{2}\,\eps\,\bigl(\lA{u} - \tm\bigr)\Bigr)\,
        \Unorm{f}{\Dm{u}}\;\Unorm{g}{\Dm{u}}.
    \end{align}
    \item Second, we also have the norm comparion: 
      \begin{align}
        \label{eq PT TDmu2}
        \frac{1}{1+\kappa}\,\Unorm{f}{u}
        \;\;\le\;\;
        \Unorm{f}{\Dm{u}}
        \;\;\le\;\;
        (1+\kappa)\,\Unorm{f}{u}.
      \end{align}
  \end{itemize}
\end{lemma}

\noindent
While the proof of Lemma~\ref{lem PT TDmu} introduces no fundamentally new ideas,
it does illuminate the reason for setting
\[
    \hB = \left\lceil \tCR (\log(R) +1) + \frac{\eps}{10 d} \tCR (\log(R) +1) \right\rceil\,.
\]
For brevity, we postpone this proof until the end of the section. 
We now turn to the main argument for Proposition~\ref{prop PDM DEG1T}.

\smallskip

Consider \(f \in \Deg{\Dm{u}} \,\otimes\, \T{\Dm{u}}\). 
By the tensor product structure, we define a “tensor-product norm”:
\begin{align}
  \label{eq Tnorm}
  \Tnorm{f}{u} 
  \;:=\; 
  \sqrt{\EE{u}\,\otimes\,\EE{\Dm{u}}\bigl[f^2\bigr]}.
\end{align}
\begin{figure}[h]
    \centering
    \includegraphics[width = 0.9\textwidth]{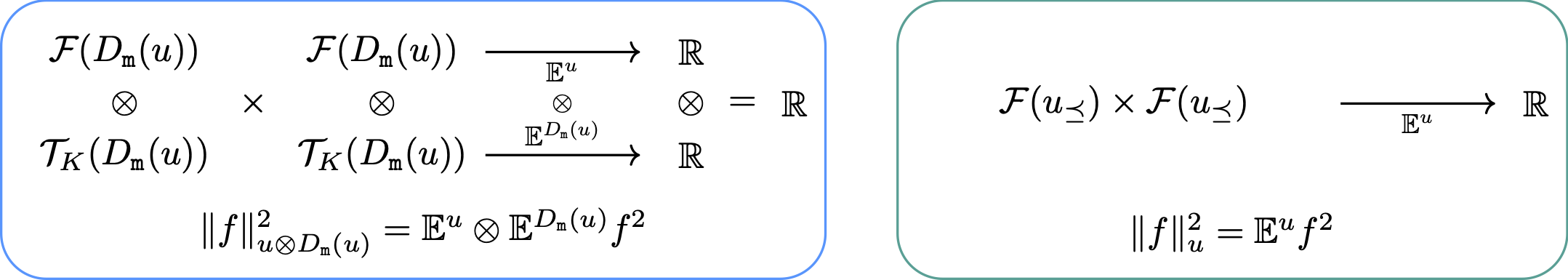}
    \caption{Two norms of $f$}
    \label{fig: Dmu2}
\end{figure}

\noindent
\textbf{Interpretation via two broadcasting processes.}
Consider two \emph{independent} broadcasting processes:
\begin{itemize}
  \item \(Y = (Y_w : w \preceq u)\), a broadcasting process over the tree \(T_u\), initialized with \(Y_u \sim \pi\).
  \item For each \(v \in \Dm{u}\), let \(Z_{v_{\pe}}\) be a broadcasting process over the subtree \(T_v\), initialized with \(Z_v \sim \pi\).
\end{itemize}
We assume all these processes \(\{\,Y\}\) and \(\{\,Z_{v_\pe}: v \in \Dm{u}\}\) are jointly independent.

\smallskip

Recall:
\begin{itemize}
  \item Functions in \(\F{\Dm{u}}\) are functions of the variables \(y_{\Dm{u}} \in [q]^{\Dm{u}}\).
  \item Functions in \(\T{\Dm{u}}\) are functions of  the variables \(z_{L_u} \in [q]^{L_u}\). (The usage of different letters $y$ and $z$ are intentional, but that are just different labels.)
\end{itemize}
Hence, the space 
\(\Deg{\Dm{u}} \,\otimes\, \T{\Dm{u}}\)
can be seen as finite sums of terms \(\phi_1\bigl(y_{\Dm{u}}\bigr)\,\phi_2\bigl(z_{L_u}\bigr)\),
where \(\phi_1 \in \Deg{\Dm{u}}\) and \(\phi_2 \in \F{\Dm{u}}\).
Under the above independence assumptions, the norm \(\Tnorm{f}{u}\) from~\eqref{eq Tnorm} 
can be expressed as 
\[
  \Tnorm{f}{u}
  \;=\;
  \sqrt{\mathbb{E}\Bigl[f\bigl(Y_{\Dm{u}_{\pe}},\,Z_{L_u}\bigr)^2\Bigr]},
\]
where the expectation is over both \(Y\) and \(\{Z_{v_\pe}\}_{v\in \Dm{u}}\).

\smallskip

\noindent
\textbf{Comparison with the \(\Unorm{\cdot}{\Dm{u}}\) norm.}
If we identify 
\(\Deg{\Dm{u}} \otimes \T{\Dm{u}}\)
as a subspace of \(\F{u_{\pe}}\) 
and use the norm \(\Unorm{\cdot}{u}\), we get
\[
  \Unorm{f}{\Dm{u}}
  \;=\;
  \sqrt{\E\Bigl[f\bigl(Y_{\Dm{u}_{\pe}},\,Y_{L_u}\bigr)^2\Bigr]},
\]
where now \emph{a single} broadcasting process \(Y\) (over \(T_u\)) 
generates both \(Y_{\Dm{u}_{\pe}}\) and \(Y_{L_u}\). 

\smallskip

We shall use these two different random-process constructions, 
\(\Tnorm{\cdot}{u}\) and \(\Unorm{\cdot}{\Dm{u}}\),
throughout this subsection to facilitate the proof of Proposition~\ref{prop PDM DEG1T}.
Intuitively, the independence in the \(\Tnorm{\cdot}{u}\) setting 
allows for certain factorized estimates, 
while \(\Unorm{\cdot}{\Dm{u}}\) uses a global broadcasting process on \(T_u\). 
Each perspective plays a role in capturing the decay and norm-comparison properties we need. 
Our next goal is to show that working with \(\Tnorm{f}{u}\) 
instead of \(\Unorm{f}{\Dm{u}}\) only costs 
a small multiplicative factor.  

\begin{lemma}\label{lem PT NormConversion}
  The following holds for sufficiently large \(\tCR\):
  suppose \(u \in V(T)\) satisfies \(\lA{u} \ge \hB\). 
  Then, for every \(f \in \F{\Dm{u}} \otimes \T{\Dm{u}}\),
  \begin{align}\label{eq PT NormConversion}
    (1-\kappa)\,\Tnorm{f}{u}
    \;\;\le\;\;
    \Unorm{f}{u}
    \;\;\le\;\;
    (1+\kappa)\,\Tnorm{f}{u}.
  \end{align}
\end{lemma}

\begin{proof}
  \step{The bilinear map \(\idtt{\Dm{u}} \otimes \D{\Dm{u}}\).}
  For any \(f,g \in \F{\Dm{u}} \otimes \T{\Dm{u}}\), 
  consider the map
  \[
    \bigl(\idtt{\Dm{u}} \,\otimes\, \D{\Dm{u}}\bigr)\,[f\,g],
  \]
  which acts as a bilinear form from \(\bigl(\F{\Dm{u}} \otimes \T{\Dm{u}}\bigr)^2\)
  to \(\F{\Dm{u}} \otimes \F{\Dm{u}}\). Figure~\ref{fig: Dmu3} illustrates this map.
  \begin{figure}[h]
    \centering
    \includegraphics[width = 0.5\textwidth]{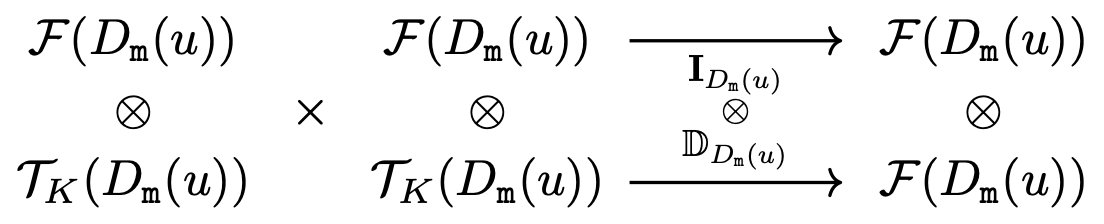}
    \caption{The bilinear map}
    \label{fig: Dmu3}
\end{figure}

  \step{\(\F{\Dm{u}} \otimes \F{\Dm{u}}\) as functions of 
         \(y_{\Dm{u}_{\pe}}, z_{\Dm{u}_{\pe}} \in [q]^{\Dm{u}_{\pe}}\).}
  Note that \(\F{\Dm{u}} \otimes \F{\Dm{u}}\) 
  \emph{cannot} be identified with
  \[
    \Xi\bigl(\F{\Dm{u}} \otimes \F{\Dm{u}}\bigr) 
    \;=\; 
    \F{\Dm{u}},
  \]
  because both “tensor factors” involve the same underlying variables. 
  Instead, we may view elements in 
  \(\F{\Dm{u}} \otimes \F{\Dm{u}}\)
  as functions of two distinct copies of those variables. 
  Concretely, let \(y_{\Dm{u}_{\pe}}\) and \(z_{\Dm{u}_{\pe}}\) 
  be two disjoint sets of variables in \([q]^{\Dm{u}_{\pe}}\). 
  Then each element of 
  \(\F{\Dm{u}} \otimes \F{\Dm{u}}\)
  can be seen as a finite sum of terms 
  \(\phi_1\bigl(y_{\Dm{u}}\bigr)\,\phi_2\bigl(z_{\Dm{u}}\bigr)\),
  where \(\phi_1, \phi_2 \in \F{\Dm{u}}\).

  \step{Applying \eqref{eq PT TDmu} after fixing one variable.}
  Fix \(y_{\Dm{u}}\in[q]^{\Dm{u}}\). 
  Then \(z_{L_u}\mapsto f\bigl(y_{\Dm{u}}, z_{L_u}\bigr)\)
  is a function in \(\T{\Dm{u}}\). From \eqref{eq PT TDmu}, 
  \begin{align*}
    &\max_{z_{\Dm{u}},\,z'_{\Dm{u}}\in[q]^{\Dm{u}}}\!
    \Bigl|\,
      \mathbb{E}_Z\bigl[f^2\!\bigl(y_{\Dm{u}},Z_{L_u}\bigr) \,\big|\,
                       Z_{\Dm{u}}=z_{\Dm{u}}\bigr]
      \;-\;
      \mathbb{E}_Z\bigl[f^2\!\bigl(y_{\Dm{u}},Z_{L_u}\bigr) \,\big|\,
                       Z_{\Dm{u}}=z'_{\Dm{u}}\bigr]
    \Bigr| \\
    \;\le\; &
    2\,\max_{z_{\Dm{u}}\in[q]^{\Dm{u}}}\!
    \Bigl|\,
      \mathbb{E}_Z\bigl[f^2\!\bigl(y_{\Dm{u}},Z_{L_u}\bigr) \,\big|\,
                       Z_{\Dm{u}}=z_{\Dm{u}}\bigr]
      - 
      \mathbb{E}_Z\bigl[f^2\!\bigl(y_{\Dm{u}},Z_{L_u}\bigr)\bigr]
    \Bigr| \\
    \;\le\; &
    4\,\exp\!\Bigl(-\tfrac12\,\eps\,(\lA{u}-\tm)\Bigr)\,
    \mathbb{E}_Z\!\bigl[f^2\bigl(y_{\Dm{u}},Z_{L_u}\bigr)\bigr].
  \end{align*}

  Moreover, for any given \(y_{\Dm{u}}\in[q]^{\Dm{u}}\), 
  the conditional distribution of \(Y_{L_u}\) given \(Y_{\Dm{u}}=y_{\Dm{u}}\) 
  coincides with that of \(Z_{L_u}\) given \(Z_{\Dm{u}}=y_{\Dm{u}}\). 
  Hence,
  \begin{align*}
  & \Bigl|\,
      \mathbb{E}\bigl[f^2\!\bigl(Y_{\Dm{u}},\,Z_{L_u}\bigr)\,\big|\,
                    Y_{\Dm{u}},\,Z_{\Dm{u}}\bigr]
      \;-\;
      \mathbb{E}\bigl[f^2\!\bigl(Y_{\Dm{u}},\,Y_{L_u}\bigr)\,\big|\,
                    Y_{\Dm{u}}\bigr]
    \Bigr| \\
    \;\le\;&
    4\,\exp\!\Bigl(-\tfrac12\,\eps\,(\lA{u}-\tm)\Bigr)\,
    \mathbb{E}\bigl[f^2\!\bigl(Y_{\Dm{u}},\,Z_{L_u}\bigr)\,\big|\,
                   Y_{\Dm{u}}\bigr].
  \end{align*}

  \step{Norm comparison.}
  Using the above inequality, we compare the squared norms 
  \(\Unorm{f}{u}^2\) and \(\Tnorm{f}{u}^2\):
  \begin{align*}
    & \Bigl|\,
      \mathbb{E}\bigl[f^2\!\bigl(Y_{\Dm{u}},Z_{L_u}\bigr)\bigr]
      \;-\;
      \mathbb{E}\bigl[f^2\!\bigl(Y_{\Dm{u}},Y_{L_u}\bigr)\bigr]
    \Bigr| \\
    \;=\; &
    \Bigl|\,
      \mathbb{E}_Y\Bigl(
        \mathbb{E}\bigl[f^2\!\bigl(Y_{\Dm{u}},\,Z_{L_u}\bigr)
                       - f^2\!\bigl(Y_{\Dm{u}},\,Y_{L_u}\bigr)
                       \,\big|\,
                       Y_{\Dm{u}},\,Z_{\Dm{u}}\bigr]
      \Bigr)\Bigr| \\
    \;\le\; &
    \mathbb{E}_Y\!\Bigl[
      \Bigl|\,
        \mathbb{E}\bigl[
          f^2(Y_{\Dm{u}},Z_{L_u}) 
          - 
          f^2(Y_{\Dm{u}},Y_{L_u})
          \,\big|\,
          Y_{\Dm{u}},\,Z_{\Dm{u}}
        \bigr]
      \Bigr| 
    \Bigr] \\
    \;\le\; &
    4\,\exp\!\Bigl(-\tfrac12\,\eps\,(\lA{u}-\tm)\Bigr)\,
    \mathbb{E}\bigl[f^2\!\bigl(Y_{\Dm{u}},Z_{L_u}\bigr)\bigr]\,.
  \end{align*}
  Expressing this in terms of the norms we have 
  \[
        \Bigl|
            \Tnorm{f}{u}^2 - \Unorm{f}{u}^2 
        \Bigr| 
    \le 
    4\,\exp\!\Bigl(-\tfrac12\,\eps\,(\lA{u}-\tm)\Bigr)\,
    \Tnorm{f}{u}^2\,.
  \]

  Since \(\lA{u}-\tm \ge \tCR\,(\log(R)+1)\) under \(\lA{u}\ge \hB\),
  and \(\kappa>0\) was chosen a priori, we can make \(\tCR\) large enough so that
  \eqref{eq PT NormConversion}, 
  \[
        (1-\kappa) \Tnorm{f}{u}
    \le
        \Unorm{f}{u}
    \le
        (1+\kappa) \Tnorm{f}{u}\,,
  \]
  follows from the above bound.
\end{proof}

\begin{proof}[Proof of Proposition \ref{prop PDM DEG1T}]
    \step{Representing \(\CE{u}\bigl[((\idtt{u} - \PK{u})f)\,g\bigr]\).}
    \begin{align*}
      \CE{u}\bigl[((\idtt{u} - \PK{u})f)\,g\bigr]
      \;=\;
      \CE{u}[f\,g]
      \;-\;
      \CE{u}\bigl[(\PK{u}f)\,g\bigr]
      \;=\;
      \D{u}[f\,g]
      \;+\;
      \EE{u}[f\,g]
      \;-\;
      \CE{u}\bigl[(\PK{u}f)\,g\bigr].
    \end{align*}
    
    Since \(g \in \T{u}\), the projection property of \(\PK{u}\) 
    (see \eqref{eq PiProperty1}) yields
    \begin{align}
    \nonumber
      (*)
      \;=\;
      \D{u}[f\,g]
      \;+\;
      \EE{u}\bigl[(\PK{u}f)\,g\bigr]
      \;-\;
      \CE{u}\bigl[(\PK{u}f)\,g\bigr]
      \;=\;&
      \D{u}[f\,g] \;-\; \D{u}\bigl[(\PK{u}f)\,g\bigr]\\
    \label{eq PDM DEG1T decompose00}
      &\; \stackrel{\eqref{eq EEU'ECU}}{=}\;
      \D{u} \,\circ\, \CE{\Dm{u}}[f\,g]
      \;-\; \D{u}\bigl[(\PK{u}f)\,g\bigr]\,.
    \end{align}

    \step{Decomposing \(\CE{\Dm{u}}[f\,g]\)}
    Here, we plan to leverage the decay in the “\(\T{\Dm{u}}\)-part” of \(f\).  
    To that end, notice \(g\) can also be viewed as an element of 
    \(\Deg{\Dm{u}} \otimes \T{\Dm{u}}\), since
    \[
      \T{u} \;\subseteq\; \T{\Dm{u}}
      \;\subseteq\; \Deg{\Dm{u}} \,\otimes\, \T{\Dm{u}}.
    \]
    We claim there is a commutative diagram ensuring
    \[
      \CE{\Dm{u}}[f\,g]
      \;=\;
      \Xi\Bigl(\idtt{\Dm{u}} \otimes \CE{\Dm{u}}[f\,g]\Bigr),
    \]
    namely,
    \[
      \begin{tikzcd}[column sep=huge]
        \bigl(\F{\Dm{u}} \otimes \T{\Dm{u}}\bigr) 
          \times 
          \bigl(\F{\Dm{u}} \otimes \T{\Dm{u}}\bigr) 
          \ar[r, 
            "\,\idtt{\Dm{u}}\,\otimes\,\D{\Dm{u}}", 
            shift left] 
          \ar[dr, "\D{\Dm{u}}"']
        &
        \F{\Dm{u}} \otimes \F{\Dm{u}}
        \ar[d, "\,\Xi", shift right]
        \\[-4pt]
        & \F{\Dm{u}}
      \end{tikzcd}
    \]
      
    \noindent
    \textbf{Verifying the claim.}  
    First observe:
    \begin{align}
      \bigl(\idtt{\Dm{u}} \otimes \CE{\Dm{u}}[f\,g]\bigr)(y_{\Dm{u}},\,z_{\Dm{u}})
      \;=\;
      \mathbb{E}_Z\!\Bigl[
        f\bigl(y_{\Dm{u}},\,Z_{L_u}\bigr)\,
        g\bigl(y_{\Dm{u}},\,Z_{L_u}\bigr)
        \,\Big|\,
        Z_{\Dm{u}} = z_{\Dm{u}}
      \Bigr],
      \label{eq PT decomposeCDfg00}
    \end{align}
    Since \(\Xi\) is the identification map that merges the two sets of variables, we have  
    \begin{align*}
      \Xi\Bigl(\idtt{\Dm{u}} \otimes \CE{\Dm{u}}[f\,g]\Bigr)\!(y_{\Dm{u}})
      &= 
        \bigl(\idtt{\Dm{u}} \otimes \CE{\Dm{u}}[f\,g]\bigr)\bigl(y_{\Dm{u}},\,y_{\Dm{u}}\bigr)
        \\[-1pt]
      &\stackrel{\eqref{eq PT decomposeCDfg00}}{=}
        \mathbb{E}_Z\!\Bigl[
          f\bigl(y_{\Dm{u}},\,Z_{L_u}\bigr)\,
          g\bigl(y_{\Dm{u}},\,Z_{L_u}\bigr)
          \,\Big|\,
          Z_{\Dm{u}}=y_{\Dm{u}}
        \Bigr] 
        \\[-1pt]
      &=
        \mathbb{E}_Y\!\Bigl[
          f\bigl(y_{\Dm{u}},\,Y_{L_u}\bigr)\,
          g\bigl(y_{\Dm{u}},\,Y_{L_u}\bigr)
          \,\Big|\,
          Y_{\Dm{u}}=y_{\Dm{u}}
        \Bigr] 
        \\[-1pt]
      &=
        \CE{\Dm{u}}[f\,g](y_{\Dm{u}}),
    \end{align*}
    where in the second-to-last line we used that the conditional distributions of 
    \((Y_{L_u}\mid Y_{\Dm{u}}=y_{\Dm{u}})\) 
    and 
    \((Z_{L_u}\mid Z_{\Dm{u}}=y_{\Dm{u}})\)
    are identical. The claim thus follows. 
    
    \smallskip
    
    Next, consider the split
    \[
      \idtt{u}\,\otimes\,\CE{\Dm{u}}[f\,g]
      \;=\;
      \underbrace{\idtt{u}\,\otimes\,\D{\Dm{u}}[f\,g]}_{\in \,\F{\Dm{u}}\otimes \F{\Dm{u}}}
      \;+\;
      \underbrace{\idtt{u}\,\otimes\,\EE{\Dm{u}}[f\,g]}_{\in \,\F{\Dm{u}}\otimes\R}.
    \]
    Now \eqref{eq PDM DEG1T decompose00} can be expressed as 
    \begin{align}
      \label{eq PT decomposeCDfg01}
      \CE{u}\bigl[((\idtt{u} - \PK{u})f)\,g\bigr]
      \;=\;
      \D{u}\Bigl[
        \Xi\bigl(\idtt{\Dm{u}} \otimes \D{\Dm{u}}[f\,g]\bigr)
      \Bigr]
      \;+\;
      \D{u}\Bigl[
        \Xi\bigl(\idtt{\Dm{u}} \otimes \EE{\Dm{u}}[f\,g]\bigr)
      \Bigr]
      \;-\;
      \D{u}\bigl[(\PK{u}f)\,g\bigr].
    \end{align}
    We will bound each term on the right-hand side (in the \(\maxnorm{\cdot}\) sense) individually.

    \step{Bounding second term of \eqref{eq PT decomposeCDfg01}:
    \(\displaystyle \maxnorm{\D{u}\!\Bigl[\Xi\bigl(\idtt{\Dm{u}} \otimes \EE{\Dm{u}}[f\,g]\bigr)\Bigr]}\)}
  To estimate 
  \(\D{u}\!\Bigl[\Xi\bigl(\idtt{\Dm{u}} \otimes \EE{\Dm{u}}[f\,g]\bigr)\Bigr],\)
  it is helpful to restate it in terms of the random processes \(Y\) and \(Z\). Note that
  \[
    \Xi\bigl(\idtt{\Dm{u}} \otimes \EE{\Dm{u}}[f\,g]\bigr)\!(y_{\Dm{u}})
    \;=\;
    \bigl(\idtt{\Dm{u}} \otimes \EE{\Dm{u}}[f\,g]\bigr)\!(y_{\Dm{u}})
    \;=\;
    \mathbb{E}_Z\Bigl[f\bigl(y_{\Dm{u}},\,Z_{L_u}\bigr)\,g\bigl(Z_{L_u}\bigr)\Bigr].
  \]
  Hence,
  \begin{align*}
    &\Bigl| \D{u}\Bigl[\Xi\bigl(\idtt{\Dm{u}} \otimes \EE{\Dm{u}}[f\,g]\bigr)\Bigr]\!(y_u) \Bigr|
    \\
    =\;&
    \Bigl|
    \mathbb{E}_{Y,Z}\Bigl[
      f\bigl(Y_{\Dm{u}},\,Z_{L_u}\bigr)\,g\bigl(Z_{L_u}\bigr)
      \,\big|\,
      Y_u = y_u
    \Bigr]
    \;-\;
    \mathbb{E}_{Y,Z}\Bigl[f\bigl(Y_{\Dm{u}},\,Z_{L_u}\bigr)\,g\bigl(Z_{L_u}\bigr)\Bigr]
    \Bigr|
    \\[2pt]
    =\;&
    \Bigl|
    \mathbb{E}_{Y,Z}\Bigl[
      \Bigl(f\bigl(Y_{\Dm{u}},\,Z_{L_u}\bigr)
            - 
            \mathbb{E}_Y\,f\bigl(Y_{\Dm{u}},\,Z_{L_u}\bigr)\Bigr)\,
      g\bigl(Z_{L_u}\bigr)
      \,\Big|\,
      Y_u=y_u
    \Bigr]
    \Bigr|
    \\[4pt]
    \le\;&
    \mathbb{E}_Z\,\Bigl[
      \mathbb{E}_Y\Bigl[
        \Bigl|\,
          f\bigl(Y_{\Dm{u}},\,Z_{L_u}\bigr)
          \;-\;
          \mathbb{E}_Y\,f\bigl(Y_{\Dm{u}},\,Z_{L_u}\bigr)
        \Bigr|
        \,\Big|\,
        Y_u=y_u
      \Bigr]
      \;\cdot\;
      \bigl|\,g\bigl(Z_{L_u}\bigr)\bigr|
    \Bigr].
  \end{align*}
  
  Fix \(z_{L_u} \in [q]^{L_u}\). Then the map 
  \(\;y_{\Dm{u}}\;\mapsto\;f\bigl(y_{\Dm{u}},\,z_{L_u}\bigr)\)
  is a degree-\(1\) function in \(\F{\Dm{u}}\). Define
  \[
    T' \;=\; \bigl\{\,
       v \in V(T)\,\colon\, v \preceq u,\;\h(v)\,\ge\,\h(u)-\tm
    \bigr\},
  \]
  which is the subtree of \(T\) with root \(u\), leaves \(\Dm{u}\), and depth~\(\tm\).
  Applying the base case of our induction on \(T'\) (or an analogous argument for this smaller subtree), we obtain
  \[
    \maxnorm{
      \mathbb{E}_{Y}\Bigl[
        f\bigl(Y_{\Dm{u}},\,z_{L_u}\bigr)
        \;-\;
        \mathbb{E}_Y\,f\bigl(Y_{\Dm{u}},\,z_{L_u}\bigr)
        \,\Big|\,
        Y_u
      \Bigr]
    }
    \;\le\;
    \exp\!\bigl(-\,\eps\,(\tm - \h_0)\bigr)\,
    \sqrt{
      \mathbb{E}_Y\Bigl[f^2\bigl(Y_{\Dm{u}},\,z_{L_u}\bigr)\Bigr]
    },
  \]
  provided \(\tm \ge \h_0\), where \(\h_0\) is a base-level decay parameter. By Proposition~\ref{prop K=0}, we have
  \(\,\h_0 \le \tC_0\,(\log R + 1)\), 
  and thus choosing \(\tCR\) sufficiently large ensures \(\tm\ge\h_0\).
  
  Consequently, for each \(y_u \in [q]\),
  \begin{align*}
    \Bigl|\,
      \D{u}\Bigl[\Xi\bigl(\idtt{\Dm{u}} \otimes \EE{\Dm{u}}[f\,g]\bigr)\Bigr]\!(y_u)
    \Bigr|
    &\;\le\;
    \mathbb{E}_Z\,\Bigl[
      \exp\!\bigl(-\,\eps\,(\tm - \h_0)\bigr)\,
      \sqrt{\mathbb{E}_Y\Bigl[f^2\bigl(Y_{\Dm{u}},\,Z_{L_u}\bigr)\Bigr]}
      \;\cdot\;
      \bigl|\,g\bigl(Z_{L_u}\bigr)\bigr|
    \Bigr]
    \\
    &\;\le\;
    \exp\!\bigl(-\,\eps\,(\tm - \h_0)\bigr)\,
    \sqrt{
      \mathbb{E}_{Y,Z}\Bigl[f^2\bigl(Y_{\Dm{u}},\,Z_{L_u}\bigr)\Bigr]
    }
    \;\cdot\;
    \sqrt{\mathbb{E}_Z\Bigl[g^2\bigl(Z_{L_u}\bigr)\Bigr]}
    \\[-4pt]
    &\;=\;
    \exp\!\bigl(-\,\eps\,(\tm - \h_0)\bigr)\,\Tnorm{f}{u}\,\Unorm{g}{\Dm{u}},
  \end{align*}
  where in the last two steps we apply Cauchy–Schwarz and recall the definition
  \(\Tnorm{f}{u}^2 = \EE{Y,Z}[f^2(Y_{\Dm{u}},\,Z_{L_u})]\).  
  Hence,
  \begin{align}
  \nonumber 
    \maxnorm{
      \D{u}\Bigl[\Xi\bigl(\idtt{\Dm{u}} \otimes \EE{\Dm{u}}[f\,g]\bigr)\Bigr]
    }
    \;\le\; \;\;\;\,&
    \exp\!\bigl(-\,\eps\,(\tm - \h_0)\bigr)\,\Tnorm{f}{u}\,\Unorm{g}{\Dm{u}} \\
    \;\stackrel{\eqref{eq PT NormConversion},\,\eqref{eq PT TDmu2}}{\le}\; &
    \exp\!\bigl(-\,\eps\,(\tm - \h_0)\bigr)\,(1+\kappa)\,\Unorm{f}{u}\,\Unorm{g}{u}.
    \label{eq PT decomposeCDfg03}
  \end{align}

  \step{Bounding third term of \eqref{eq PT decomposeCDfg01} : 
    \(\displaystyle \maxnorm{\D{u}\bigl[(\PK{u}f)\,g\bigr]}\).}
  For \(\D{u}[(\PK{u}f)\,g]\), note that both \(\PK{u}f\) and \(g\) lie in \(\T{u}\). Therefore,
  \[
    \maxnorm{\D{u}\bigl[(\PK{u}f)\,g\bigr]}
    \;\le\;
    \exp\!\bigl(-\,\eps\,\lA{u}\bigr)\,\Unorm{\PK{u}f}{u}\,\Unorm{g}{u}.
  \]
  Using the “orthogonal projection” property of \(\PK{u}\), one obtains
  \[
    \Unorm{\PK{u}f}{u}^2
    \;\stackrel{\eqref{eq PiProperty1}}{=}\;
    \EE{u}\bigl[f\,(\PK{u}f)\bigr]
    \;\le\;
    \Unorm{f}{u}\,\Unorm{\PK{u}f}{u}
    \quad\Longrightarrow\quad
    \Unorm{\PK{u}f}{u}
    \;\le\;
    \Unorm{f}{u}.
  \]
  Hence,
  \begin{align}
    \maxnorm{\D{u}\bigl[(\PK{u}f)\,g\bigr]}
    \;\le\;
    \exp\!\bigl(-\,\eps\,\lA{u}\bigr)\,\Unorm{f}{u}\,\Unorm{g}{u},
    \label{eq PT decomposeCDfg04}
  \end{align}
  completing the bound for this term.

  \step{Bounding the first term in \eqref{eq PT decomposeCDfg01} :
  \(\displaystyle \maxnorm{\D{u}\Bigl[\Xi\bigl(\idtt{\Dm{u}} \otimes \D{\Dm{u}}[f\,g]\bigr)\Bigr]}\)}
First, note that
\[
  \idtt{\Dm{u}} \otimes \D{\Dm{u}}[f\,g]
  \;=\;
  \E_Z\!\Bigl[f\bigl(y_{\Dm{u}},\,Z_{L_u}\bigr)\,g\bigl(Z_{L_u}\bigr)
  \,\Big|\,
  Z_{\Dm{u}}=z_{\Dm{u}}\Bigr]
  \;-\;
  \E_Z\!\Bigl[f\bigl(y_{\Dm{u}},\,Z_{L_u}\bigr)\,g\bigl(Z_{L_u}\bigr)\Bigr].
\]
When \(y_{\Dm{u}}\) is fixed, the map
\[
  z_{L_u}
  \;\mapsto\;
  f\bigl(y_{\Dm{u}},\,z_{L_u}\bigr)
\]
is a function in \(\T{\Dm{u}}\), so Lemma~\ref{lem PT TDmu} applies. Therefore, for every \(\,y_{\Dm{u}}, z_{\Dm{u}}\in[q]^{\Dm{u}}\),
\[
  \Bigl|\bigl(\idtt{\Dm{u}}\otimes\D{\Dm{u}}[f\,g]\bigr)(y_{\Dm{u}},\,z_{\Dm{u}})\Bigr|
  \;\le\;
  2\,\exp\!\Bigl(-\tfrac12\,\eps\,\bigl(\lA{u}-\tm\bigr)\Bigr)\,
  \sqrt{\E_Z\!\Bigl[f^2\bigl(y_{\Dm{u}},\,Z_{L_u}\bigr)\Bigr]}\,
  \sqrt{\E_Z\!\Bigl[g^2\bigl(Z_{L_u}\bigr)\Bigr]}.
\]
In particular, 
\begin{align*}
  \Bigl|\Xi\bigl(\idtt{\Dm{u}}\otimes\D{\Dm{u}}[f\,g]\bigr)(y_{\Dm{u}})\Bigr|
  \;=\; &
  \Bigl|\bigl(\idtt{\Dm{u}}\otimes\D{\Dm{u}}[f\,g]\bigr)\bigl(y_{\Dm{u}},\,y_{\Dm{u}}\bigr)\Bigr| \\
  \;\le\; &
  2\,\exp\!\Bigl(-\tfrac12\,\eps\,\bigl(\lA{u}-\tm\bigr)\Bigr)\,
  \sqrt{\E_Z\!\Bigl[f^2\bigl(y_{\Dm{u}},\,Z_{L_u}\bigr)\Bigr]}\,
  \sqrt{\E_Z\!\Bigl[g^2\bigl(Z_{L_u}\bigr)\Bigr]}.
\end{align*}

\noindent
Consequently,
\begin{align}
\label{eq PT decomposeCDfg02}
  &\Bigl|\D{u}\Bigl[\Xi\bigl(\idtt{\Dm{u}}\otimes \D{\Dm{u}}[f\,g]\bigr)\Bigr](y_u)\Bigr|
  \\[4pt]
  \nonumber
  \;\le\; &
  \mathbb{E}_Y\!\Bigl[
    2\,\exp\!\Bigl(-\tfrac12\,\eps\,\bigl(\lA{u}-\tm\bigr)\Bigr)\,
    \sqrt{\E_Z\!\Bigl[f^2\bigl(Y_{\Dm{u}},\,Z_{L_u}\bigr)\Bigr]}\,
    \sqrt{\E_Z\!\Bigl[g^2\bigl(Z_{L_u}\bigr)\Bigr]}
  \Bigr] \\
  \nonumber
  &\;+\;
  \mathbb{E}_Y\!\Bigl[
    2\,\exp\!\Bigl(-\tfrac12\,\eps\,\bigl(\lA{u}-\tm\bigr)\Bigr)\,
    \sqrt{\E_Z\!\Bigl[f^2\bigl(Y_{\Dm{u}},\,Z_{L_u}\bigr)\Bigr]}\,
    \sqrt{\E_Z\!\Bigl[g^2\bigl(Z_{L_u}\bigr)\Bigr]}
    \,\Big|\,
    Y_u = y_u
  \Bigr]
  \\[4pt]
  \nonumber
  \;\le\;&
  2\,\exp\!\Bigl(-\tfrac12\,\eps\,(\lA{u}-\tm)\Bigr)\,
  \sqrt{\E_{Y,Z}\!\Bigl[f^2\bigl(Y_{\Dm{u}},\,Z_{L_u}\bigr)\Bigr]}\,
  \sqrt{\E_Z\!\Bigl[g^2\bigl(Z_{L_u}\bigr)\Bigr]}\\
  \nonumber 
  & \;+\;
  2\,\exp\!\Bigl(-\tfrac12\,\eps\,(\lA{u}-\tm)\Bigr)\,
  \sqrt{\E_{Y,Z}\!\Bigl[
    f^2\bigl(Y_{\Dm{u}},\,Z_{L_u}\bigr)
    \,\Big|\,
    Y_u=y_u
  \Bigr]}\,
  \sqrt{\E_Z\!\Bigl[g^2\bigl(Z_{L_u}\bigr)\Bigr]},
  \notag
\end{align}
where we use Jensen’s inequality in the last step.

\smallskip

Define
\[
  \phi(y_u)
  \;:=\;
  \mathbb{E}_{Y,Z}\!\Bigl[
    f^2\bigl(Y_{\Dm{u}},\,Z_{L_u}\bigr)
    \,\Big|\,
    Y_u=y_u
  \Bigr]
  \quad\Longrightarrow\quad
  \E_{Y,Z}\!\Bigl[f^2\bigl(Y_{\Dm{u}},\,Z_{L_u}\bigr)\Bigr]
  \;=\;
  \mathbb{E}_{Y_u}\bigl[\phi(Y_u)\bigr].
\]
Since \(\phi\) is nonnegative and \(\pi(\theta)>0\) for every \(\theta\in[q]\),
\[
  \phi(y_u)
  \;=\;
  \frac{1}{\pi(y_u)}\,\pi\bigl(y_u\bigr)\,\phi\bigl(y_u\bigr)
  \;\le\;
  \frac{1}{\pi(y_u)}\,
  \mathbb{E}_{Y_u}\bigl[\phi(Y_u)\bigr]
  \;\;\;\mbox{ for all }y_u\in[q].
\]
Because \(0<\min_{\theta}\pi(\theta)\in\CC\), there exists \(\tC\in\CC\) such that
\[
  \E_{Y,Z}\!\Bigl[
    f^2\bigl(Y_{\Dm{u}},\,Z_{L_u}\bigr)
    \,\Big|\,
    Y_u=y_u
  \Bigr]
  \;\le\;
  \tC\,
  \E_{Y,Z}\!\Bigl[
    f^2\bigl(Y_{\Dm{u}},\,Z_{L_u}\bigr)
  \Bigr].
\]
Hence, we can combine the two terms in \eqref{eq PT decomposeCDfg02} to obtain
\[
  \eqref{eq PT decomposeCDfg02}
  \;\le\;
  \tC\,\exp\!\Bigl(-\tfrac12\,\eps\,(\lA{u}-\tm)\Bigr)\,
  \Tnorm{f}{u}\,\Unorm{g}{\Dm{u}},
\]
where we may absorb the sum into the constant \(\tC\) if necessary.

\smallskip

In other words,
\begin{align}
    \nonumber
  \maxnorm{
    \D{u}\Bigl[\Xi\bigl(\idtt{\Dm{u}} \otimes \D{\Dm{u}}[f\,g]\bigr)
  \Bigr]}
  \;\le\;&
  \tC\,\exp\!\Bigl(-\tfrac12\,\eps\,(\lA{u}-\tm)\Bigr)\,
  \Tnorm{f}{u}\,\Unorm{g}{\Dm{u}} \\
  \;\stackrel{\eqref{eq PT NormConversion},\eqref{eq PT TDmu2}}{\le}\;&
  \tC\,\exp\!\Bigl(-\tfrac12\,\eps\,(\lA{u}-\tm)\Bigr)\,\bigl(1+\kappa\bigr)^2\,
  \Unorm{f}{u}\,\Unorm{g}{u}.
 \label{eq PT decomposeCDfg05}
\end{align}

\step{Combining the bounds.}
Putting together \eqref{eq PT decomposeCDfg03}, \eqref{eq PT decomposeCDfg04}, and 
\eqref{eq PT decomposeCDfg05} yields
\[
  \maxnorm{
    \CE{u}\Bigl[
      \bigl((\idtt{u}-\PK{u})f\bigr)\,g
    \Bigr]
  }
  \;\le\;
  \Bigl[
    \tC\,\exp\!\Bigl(-\tfrac12\,\eps\,(\lA{u}-\tm)\Bigr)
    \;+\;
    \exp\bigl(-\eps\,(\tm-\h_0)\bigr)
  \Bigr]\,
  (1+\kappa)^2\,\Unorm{f}{u}\,\Unorm{g}{u}.
\]
Recall we set 
\[
  \tm 
  = 
  \left\lfloor 
    \frac{1}{10\,d}\,\tCR\,(\log R + 1) 
  \right\rfloor
  \quad\text{and}\quad
  \h_0 \;\le\; \tC_0\,(\log R+1).
\]
Moreover, \(\lA{u}\ge \hB\) implies 
\(\lA{u}-\tm \;\ge\;\tCR\,(\log R+1).\)
By choosing \(\tCR\) large enough, we ensure
\[
  \bigl[
    \tC\,\exp\!\bigl(-\tfrac12\,\eps\,(\lA{u}-\tm)\bigr)
    \;+\;
    \exp\!\bigl(-\eps\,(\tm-\h_0)\bigr)
  \bigr]\,
  (1+\kappa)^2
  \;\le\;
  \frac{1}{\tCB\,R},\,
\]
and the Proposition follows. 
\end{proof}

It remains to prove Lemma~\ref{lem PT TDmu}.

\begin{proof}[Proof of Lemma~\ref{lem PT TDmu}]
  \step{Bounding 
    \(\displaystyle \maxnorm{\D{\Dm{u}}\bigl[f\,g\bigr]}\).}
  We invoke Lemma~\ref{lem TU}, which gives
  \[
    \maxnorm{\D{\Dm{u}}\bigl[f\,g\bigr]}
    \;\le\;
    \Bigl(
      \sum_{\varnothing \neq A' \subseteq \Dm{u}}
      \prod_{v \in A'} \exp\!\bigl(-\,\eps\,\lA{v}\bigr)
    \Bigr)\,
    \Unorm{f}{\Dm{u}}\;\Unorm{g}{\Dm{u}}.
  \]
  Since every \(v\in \Dm{u}\) satisfies \(\lA{v} = \lA{u}-\tm\),
  \[
    \prod_{v \in A'} \exp\!\bigl(-\,\eps\,\lA{v}\bigr)
    \;\le\;
    \exp\bigl(-\,|A'|\,\eps\,(\lA{u}-\tm)\bigr).
  \]
  Hence,
  \[
    \sum_{\varnothing \neq A'\subseteq \Dm{u}}
      \exp\bigl(-\,|A'|\,\eps\,(\lA{u}-\tm)\bigr)
    \;\le\;
    \sum_{s=1}^{|\Dm{u}|}
      \binom{|\Dm{u}|}{s}\,\exp\!\Bigl(-\,s\,\eps\,(\lA{u}-\tm)\Bigr).
  \]
  From the tree assumption, we have \(|\Dm{u}| \le \,R\,d^\tm\). Thus, 
  \[
    \sum_{\varnothing \neq A' \subseteq \Dm{u}}
    \exp\!\bigl(-\,|A'|\,\eps\,(\lA{u}-\tm)\bigr)
    \;\le\;
    \sum_{s=1}^{R\,d^\tm}
      \bigl(R\,d^\tm\bigr)^s\,
      \exp\bigl(-\,s\,\eps\,(\lA{u}-\tm)\bigr).
  \]
  We rewrite
  \[
    \bigl(R\,d^\tm\bigr)^s
    \;=\;
    \exp\!\Bigl(s\,\log(R) + s\,\tm\,\log(d)\Bigr),
  \]
  so
  \[
    \sum_{s=1}^{R\,d^\tm}
      \exp\bigl(-\,s\,\eps\,(\lA{u}-\tm)\bigr)
      \,\exp\!\bigl(s\,\log(R) + s\,\tm\,\log(d)\bigr)
    \;\le\;
    \sum_{s=1}^{\infty}
      \exp\!\Bigl(
        -\,s\,\bigl[-\,\log(R) - \tm\,\log(d)
                   + \eps\,(\lA{u}-\tm)\bigr]
      \Bigr).
  \]
  By our choice 
  \[
    \tm 
    \;=\; 
    \left\lfloor 
      \frac{\eps}{10\,d}\,\tCR\,\bigl(\log(R)+1\bigr)
    \right\rfloor
  \]
  and the assumption 
  \[
    \lA{u} \ge  \hB \Rightarrow \lA{u} - \tm \ge \tCR(\log(R)+1)\,,
  \]
  we have
  \[
    -\,\log(R)\;-\;\tm\,\log(d)\;+\;
    \eps\,\bigl(\lA{u}-\tm\bigr)
    \;\ge\;
    \tfrac12\,\eps\,\bigl(\lA{u}-\tm\bigr)
    \;\ge\;
    2,
  \]
  whenever \(\tCR\) is chosen sufficiently large. In that case, the above infinite series decays at least as fast as a geometric series with ratio \(\tfrac12\) and initial term \(\exp\!\bigl(-\tfrac12\,\eps\,(\lA{u}-\tm)\bigr)\). Hence, it is bounded by
  \[
    2\,\exp\!\Bigl(-\tfrac12\,\eps\,(\lA{u}-\tm)\Bigr).
  \]
  Combining all factors, we obtain
  \[
    \maxnorm{\D{\Dm{u}}\bigl[f\,g\bigr]}
    \;\le\;
    2\,\exp\!\Bigl(-\tfrac12\,\eps\,(\lA{u}-\tm)\Bigr)\,
    \Unorm{f}{\Dm{u}}\;\Unorm{g}{\Dm{u}},
  \]
  which establishes \eqref{eq PT TDmu}.

  \step{Proof of the norm-comparison statement \eqref{eq PT TDmu2}.}
  We compare \(\Unorm{f}{u}^2\) and \(\Unorm{f}{\Dm{u}}^2\). Observe:
  \[
    \Unorm{f}{u}^2
    \;-\;
    \Unorm{f}{\Dm{u}}^2
    \;=\;
    \EE{u}\bigl[f^2\bigr]
    \;-\;
    \EE{\Dm{u}}\bigl[f^2\bigr]
    \;\stackrel{\eqref{eq EEU'ECU}}{=}\;
    \EE{u}\Bigl[\CE{\Dm{u}}[f^2]\Bigr]
    \;-\;
    \underbrace{\EE{\Dm{u}}\bigl[f^2\bigr]}_{\text{constant}},
  \]
  so
  \[
    \Unorm{f}{u}^2
    \;-\;
    \Unorm{f}{\Dm{u}}^2
    \;=\;
    \EE{u}\Bigl[
      \CE{\Dm{u}}[f^2] \;-\; \EE{\Dm{u}}[f^2]
    \Bigr].
  \]
  Applying \eqref{eq PT TDmu} (i.e., the previous result with \(g=f\)) yields
  \[
    \Bigl|\CE{\Dm{u}}\bigl[f^2\bigr](y_{\Dm{u}}) 
           - 
           \EE{\Dm{u}}\bigl[f^2\bigr]\Bigr|
    \;\le\;
    2\,\exp\!\Bigl(-\tfrac12\,\eps\,(\lA{u}-\tm)\Bigr)\,
    \Unorm{f}{\Dm{u}}^2
    \quad
    \forall\,y_{\Dm{u}}\in[q]^{\Dm{u}}.
  \]
  Hence
  \[
    \Bigl|\,
      \Unorm{f}{u}^2
      \;-\;
      \Unorm{f}{\Dm{u}}^2
    \Bigr|
    \;\le\;
    2\,\exp\!\Bigl(-\tfrac12\,\eps\,(\lA{u}-\tm)\Bigr)\,
    \Unorm{f}{\Dm{u}}^2,
  \]
  which implies
  \[
    \sqrt{
      1
      \;-\;
      2\,\exp\!\Bigl(-\tfrac12\,\eps\,(\lA{u}-\tm)\Bigr)
    }\,
    \Unorm{f}{\Dm{u}}
    \;\le\;
    \Unorm{f}{u}
    \;\le\;
    \sqrt{
      1
      \;+\;
      2\,\exp\!\Bigl(-\tfrac12\,\eps\,(\lA{u}-\tm)\Bigr)
    }\,
    \Unorm{f}{\Dm{u}}.
  \]
  Since \(\lA{u}-\tm \ge \tCR\,(\log(R)+1)\) and \(\kappa\) is fixed in advance, we can choose \(\tCR\) large enough such that \eqref{eq PT TDmu2} holds. This completes the proof of Lemma~\ref{lem PT TDmu}.
\end{proof}

\section{\bf Part II: Global Structure and the Proof of the Inductive Step}

\subsection{Overview of Part II}
In this \textbf{Part II}, our main objective is to decompose any polynomial of degree 
\(\le 2^{K+1}\) into a sum of polynomials lying in the \(\mathcal{R}\)-spaces introduced previously, 
and then use this decomposition to derive the main theorem.

Recall from \eqref{def hB} that we define
\[
  \hB
  \;=\;
  \left\lceil 
    \tCR\,\bigl(\log R +1\bigr)
    \;+\;
    \frac{\eps}{10\,d}\,\tCR\,\bigl(\log R +1\bigr)
  \right\rceil.
\]
Throughout this part, we focus on polynomials
\(\,f \in \mathcal{T}_{K+1}(\rp)\,\)
where \(\rp \in V(T)\) satisfies
\[
  \h(\rp) \;\ge\; \h_K + \hB
  \quad\Longleftrightarrow\quad
  \lA{\rp} \;\ge\; \hB.
\]
In other words, we restrict to those vertices \(\rp\) that are “tall enough” above level \(\h_K\) by at least~\(\hB\).

Given such \(\rp\), every polynomial \(f \in \mathcal{T}_{K+1}(\rp)\) will be written as
\[
  f
  \;=\;
  \sum_{u}\,f_u,
  \quad
  \text{where the sum is indexed by }
  \bigl\{
    u \preceq \rp 
    : 
    \h(u)\,\ge\,\h_K + \hB
  \bigr\}.
\]
Each piece \(f_u\) lies in one of the \(\mathcal{R}(\mathcal{W}_u)\)-type spaces introduced in \textbf{Part~I}, 
apart from one small exception that we discuss later.

To facilitate our analysis, we introduce a new \emph{relative height} \(\rk{u}\). 
For each \(u \in V(T)\),
\begin{align}\label{def rk}
  \rk{u} 
  \;=\;
  \rk{u,K}
  \;:=\;
  \lA{u} \;-\; \hB
  \;=\;
  \h(u) 
  \;-\; 
  \h_K
  \;-\; 
  \hB.
\end{align}
This notation is especially convenient when we get into more technical details, 
and placing \(u\) in the subscript (i.e.\ \(\rk{u}\)) helps keep lengthy formulas more manageable. See Figure~\ref{fig relativeHeight} for a visual illustration of \(\rk{u}\) and a chosen \(\rho'\).

\begin{figure}[h]
    \centering
    \includegraphics[width=\textwidth]{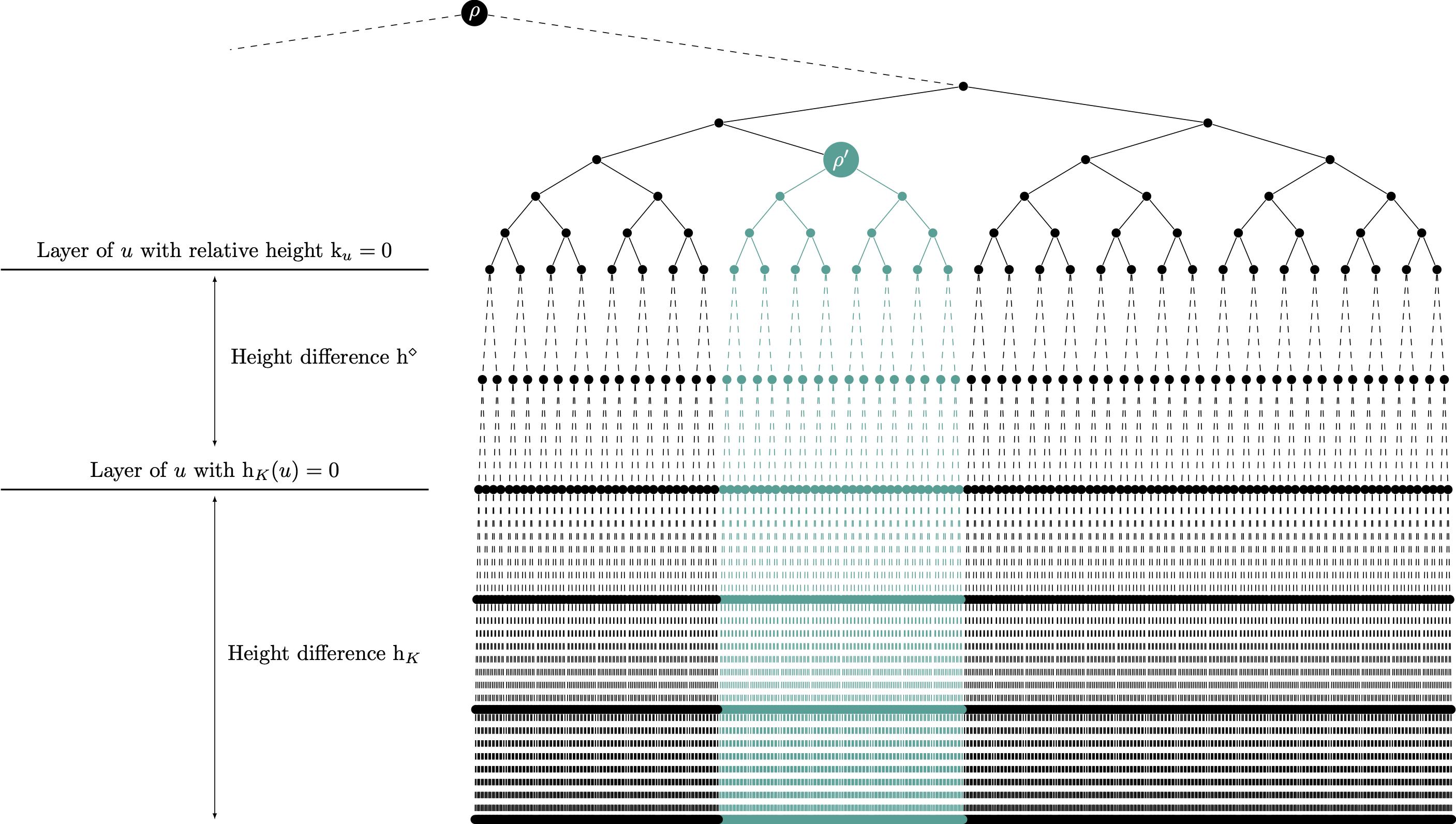}
    \caption{A visual illustration of ${\rm k}_u$ and  $\rho'$ must satisfies ${\rm k}_{\rho'}\ge 0$.}
    \label{fig relativeHeight}
\end{figure}

\begin{figure}[h]
    \centering
    \includegraphics[trim=0cm 4cm 0cm 4cm, clip,width=\textwidth]{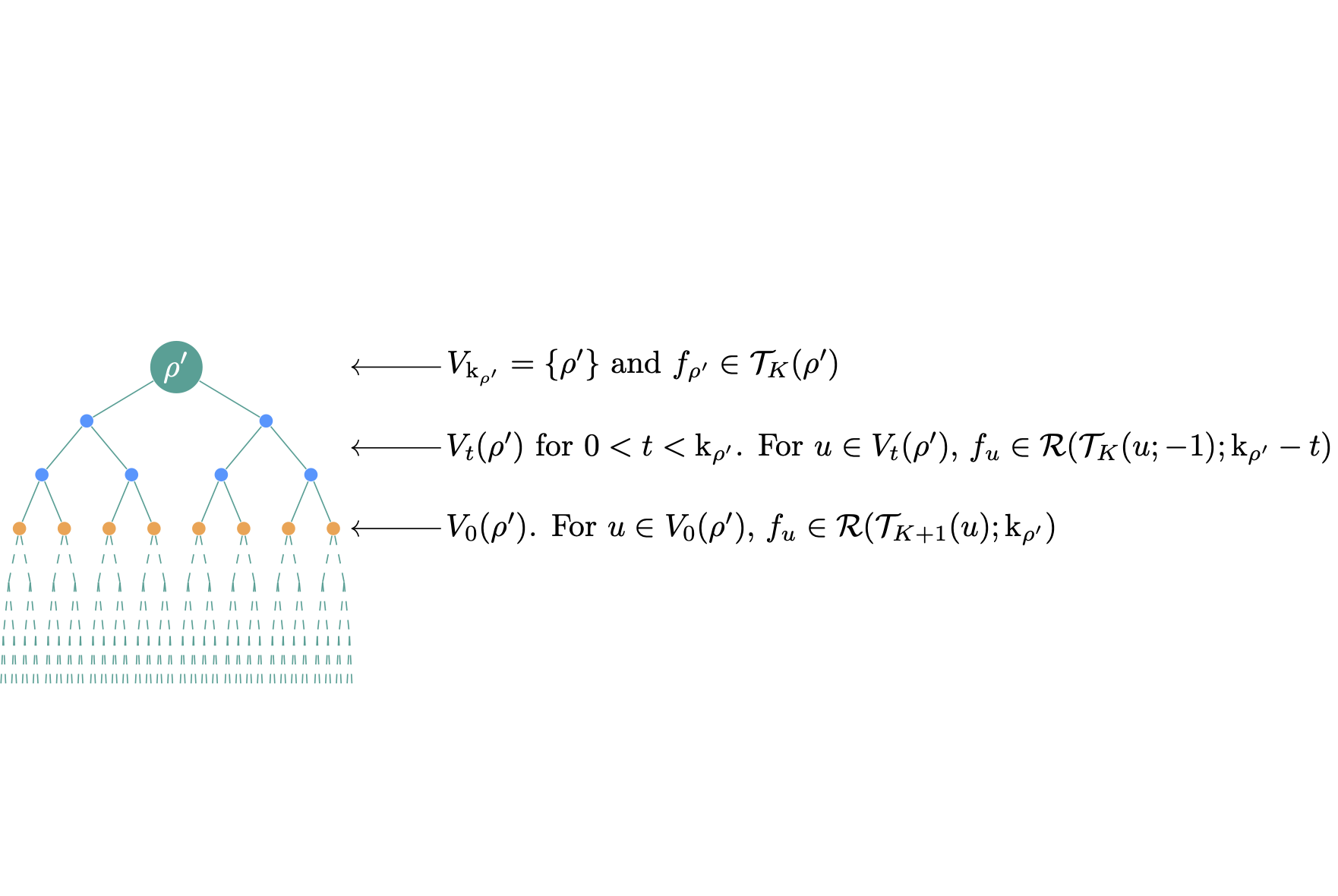}
    \caption{When $f$ is chosen in ${\mathcal T}_{K+1}(\rho')$, we decompose $f$ into sum of $f_u$'s, where the index vertex $u$ ranges over all vertices $\preceq \rho'$ with ${\rm k}_u \ge 0$\,.}
    \label{fig frho'Decompose}
\end{figure}

\smallskip

Moreover, note that for any \(u\preceq w\),
\[
  \rk{w} \;-\;\rk{u}
  \;=\;
  \h(w)\;-\;\h(u)
  \;\;\Longrightarrow\;\;
  w \;=\;
  \anc{u}{\,\rk{w} - \rk{u}}.
\]

\noindent
\textbf{Fixing a Base Vertex $\rho'$ and associated terms.}
Let us now fix a vertex \(\rp\in V(T)\) for which 
\(\rkr \;=\;\rk{\rp}\;\ge 0\).
This choice of \(\rp\) (together with its relative height \(\rkr\)) 
will form the basis for our subsequent constructions and proofs in Part II. Let us now define several key objects tied to our fixed vertex \(\rp\) with relative height \(\rkr\ge0\).

\noindent
\textbf{Layer Sets \(\V{t}\).}
For each integer \(t\in[0,\rkr]\), we define
\begin{align}\label{def Vk}
  \V{t} 
  \;=\;
  \bigl\{\,
    u \preceq \rp : \rk{u} = t
  \bigr\}
  \;\subseteq\;
  V(T).
\end{align}
We also set
\[
  \V{[a,b]}
  \;=\;
  \bigcup_{t=a}^{b}\,\V{t},
  \qquad
  \V{}
  \;=\;\V{[0,\rkr]}.
\]
Hence, \(\V{t}\) consists of all nodes \(u\) whose relative height \(\rk{u}\) equals \(t\), and \(\V{}\) is the full “vertical chain” from height~0 up to~\(\rkr\).

\noindent
\textbf{Spaces \(\dR{u}\).}
Recall (Definition~\ref{def  CR}) the definition of the spaces \(\cR{{\cal W}_u}{k}\).  
For each \(u\in\V{t}\) with \(t\in[0,\rkr-1]\) and each integer \(k\in[0,\rkr-\rk{u}]\), define
\begin{align}\label{def dR}
  \dR{u;k}
  \;:=\;
  \begin{cases}
    \cR{\mathcal{T}_{K+1}(u)}{k},  & \text{if }u\in \V{0},\\
    \cR{\TT{u}{-1}}{k},           & \text{if }u\in \V{t}\,\text{with }t>0.
  \end{cases}
\end{align}
We then set 
\[
  \dR{u} \;=\; \dR{u;\,\rkr-\rk{u}-1}.
\]

As a sanity check, one always has 
\(\dR{u;k} \subseteq \F{\anc{u}{k}_{\pe}}\) and \(\dR{u}\subseteq \F{\rp_{\pe}}\).  

Observe that if \(t=\rkr-\rk{u}\), then \(\V{t}=\{\rp\}\).  
In this case, by a slight abuse of notation we define
\begin{align*}
  \dR{\rp}
  \;:=\;
  \TT{\rp}{-1}.
\end{align*}
This choice, rather than \(\cR{\TT{\rp}{-1}}{0}\), ensures that \(\dR{\rp}=\TT{\rp}{-1}\) includes \emph{all} polynomials of degree \(\le2^K\).  
Hence any “low-degree remainder” of \(f\) can be absorbed into \(f_{\rp}\).

\begin{prop}[Decomposition into \(\dR{u}\)-Spaces]
  \label{prop PT 2K+1}
  Fix \(\rp\) such that \(\rkr\ge0\). 
  Then for every polynomial \(f\in\mathcal{T}_{K+1}(\rp)\), one can write
  \[
    f
    \;=\;
    \sum_{u\in \V{}}\,f_u,
    \qquad\text{with}\quad
    f_u\;\in\;\dR{u}.
  \]
\end{prop}

Before proving our main proposition, we gather the results from Part~I into one consolidated statement for easy reference, so we can cite it later without jumping between multiple individual lemmas. In order to do that, we introduce two parameters that encapsulate the constants from Part~I and control decay in our decomposition:

\begin{defi}[Global and Local Decay Parameters \(\Dt\) and \(\dW{u}\)]
    \label{def:Dt-and-dWu}
    Let $\hB$ be defined by \eqref{def hB}, ensuring $\hB \ge \tCR\bigl(\log(R)+1\bigr)$. 
    We introduce two parameters that work in tandem to control decay properties in our decomposition:
  
    \begin{enumerate}
      \item \textbf{Global smallness factor \(\Dt\).} 
        Define
        \begin{align}
        \label{def Dt}
          \Dt 
          \;:=\; 
          \max\!\Bigl\{\,
            \tC_{\ref{lem basicDecayAu}},\,
            \tC_{\ref{prop: coreR}},\,
            \tC_{\ref{prop PT MAIN}}
          \Bigr\}
          \times
          R\,
          \exp\bigl(-\eps\,\hB\bigr),
        \end{align}
        where the constants \(\tC_{\ref{lem basicDecayAu}}, \tC_{\ref{prop: coreR}}, 
        \tC_{\ref{prop PT MAIN}}\) are from 
        Lemma~\ref{lem basicDecayAu}, 
        Proposition~\ref{prop: coreR}, 
        and Proposition~\ref{prop PT MAIN}, respectively. 
        Because $\hB \ge \tCR(\log(R)+1)$, we can make \(\Dt\) arbitrarily small by taking \(\tCR\) large.
  
      \item \textbf{Local budget \(\dW{u}\).}
  For each node \(u\) with relative height \(\rk{u} = t \in [0,\rkr]\), we define
  \begin{align}
    \label{def dW}
    \dW{u} 
    \;=\; 
    \dW{t} 
    \;:=\;
    \begin{cases}
      \dfrac{1}{\tCB\,R}, 
         & \text{if } u \in \V{0} \;\bigl(\text{i.e.\ }t=0\bigr), \\[6pt]
      \Dt\,\exp\!\bigl(-\eps\,t\bigr), 
         & \text{if } u \in \V{t}\,\text{with }t>0.
    \end{cases}
\end{align}
  Here, \(\tCB \ge 1\) is the parameter that appears in 
  Proposition~\ref{prop PT MAIN}.

  \smallskip
  We emphasize that the main purpose of \(\dW{u}\) is to unify or “absorb” the two bounds that appear in Proposition~\ref{prop PT MAIN}.  
  Concretely, for each such node \(u\), the statement of Proposition~\ref{prop PT MAIN} 
  simplifies to saying 
  \begin{align}
    \label{eq: role of dWu}
    \dW{u} 
    \;\ge\;
    \begin{cases}
      \CW{\Tp{u}}, 
         & \text{if } u \in \V{0} \,(\text{i.e.\ } t=0), \\[5pt]
      \CW{\TT{u}{-1}},
         & \text{if } u \in \V{t}\,\text{with }t>0.
    \end{cases}
  \end{align}
  Thus, \(\dW{u}\) provides a single notation to handle both bounds,
  depending on whether \(u\) is in the bottom layer (\(t=0\)) or higher up (\(t>0\)).
    \end{enumerate}
  
    Together, $\Dt$ provides a global smallness factor (made arbitrarily small by choosing large $\tCR$), 
    while $\dW{u}$ (or $\dW{t}$) encodes the layer-by-layer decay at node $u$. 
    Both will be used to control norm bounds and correlation estimates 
    when we decompose polynomials in the next sections.
  \end{defi}

\noindent
Recall that a frequently appearing factor is
\[
  \exp\bigl(-\eps\,\lA{u}\bigr)
  \;=\;
  \exp\bigl(-\eps\,\hB\bigr)
  \,\cdot\,
  \exp\bigl(-\eps\,\rk{u}\bigr).
\]
Hence, the term \(\exp\bigl(-\eps\,\hB\bigr)\) will be absorbed into our newly defined parameters \(\Dt\) ( and consequently for \(\dW{u}\) as well).

We now restate Lemma~\ref{lem basicDecayAu} in terms of \(\rk{u}\) and \(\Dt\):

\begin{lemma}\label{lem basicDecayAu2}
  The following holds for sufficiently large \(\tCR\):
  let \(u \in V(T)\) satisfy \(\rk{u} \ge 0\), and let \(A \subseteq \OO{u}{}\).  
  Then for any functions \(f,g \in \T{A}\), we have
  \[
    \maxnorm{\D{A}\bigl[f\,g\bigr]}
    \;\le\;
    \Dt\,\exp\!\bigl(-\eps\,\rk{u}\bigr)\,\Unorm{f}{A}\,\Unorm{g}{A},
  \]
  and
  \[
    \maxnorm{\CE{A}\bigl[f\,g\bigr]}
    \;\le\;
    \Bigl(1 + \Dt\,\exp\!\bigl(-\eps\,\rk{u}\bigr)\Bigr)\,
    \Unorm{f}{A}\,\Unorm{g}{A}.
  \]
\end{lemma}

\noindent
In this formulation, the role of \(\dW{u}\) is to replace the original 
\(\CW{{\cal W}_u}\) and incorporate additional decay factors. 
We now gather the essential conditions from \textbf{Part~I} into a single unified assumption for the vertex \(\rp\). 

\begin{assumption}\label{Assume rp}
  Suppose \(\rp\in V(T)\) comes with parameters 
  \begin{align}
        \label{def tCB2}
        1 \le \tCB \in \CC \quad \text{and} \quad \tCR  \in \CC\,.
  \end{align}
  We say “\(\rp\) satisfies this assumption” if:

  \begin{enumerate}
    \item[\textit{(i)}] \(\rkr \ge 0.\)

    \smallskip
    
    \item[\textit{(ii)}] For each \(u\in \V{t}\) with $0 \le t < \rkr$ and every $w$ satisfying \(\,u\preceq w\preceq \rp\). We require:
      \begin{enumerate}
        \item For each \(f_u\in \dR{u}\), there exists a decomposition
          \[
            f_u \;=\; f_{u,\mathcal{R}} \;+\; f_{u,\mathcal{T}},
          \]
          such that
          \[
            f_{u,\mathcal{R}}
            \;\in\;
            \underbrace{\dR{u;\,\rk{w}-\rk{u}} }_{\subseteq \F{w_{\pe}}}
            \,\otimes\, 
            \TT{w}{[0,\,\rkr-\rk{w}-1]} 
            \;\subseteq\;\F{\rp_{\pe}},
          \]
          \[
            f_{u,\mathcal{T}}
            \;\in\;
            \T{w} \,\otimes\, \TT{w}{[0,\,\rkr-\rk{w}-1]} 
            \;\subseteq\;\F{\rp_{\pe}},
          \]
          and
          \[
            \Unorm{f_{u,\mathcal{T}}}{\rp}
            \;\le\;
            \dW{t}\,\Dt\,\exp\!\bigl(-\eps\,\rk{w}\bigr)\,
            \lame^{\rk{w}-\rk{u}}
            \,\Unorm{f}{\rp}.
          \]

        \item For every \(\phi_1\in \dR{u;\,\rk{w}-\rk{u}} \subseteq \F{w_{\pe}}\) 
          and \(\phi_2\in \T{w}\),
          \[
            \maxnorm{\CE{w}\bigl[\phi_1\,\phi_2\bigr]}
            \;\le\;
            \tC_{\ref{prop: coreR}}\,
            \dW{t}\,\lame^{\rk{w}-\rk{u}}
            \,\Unorm{\phi_1}{w}\,\Unorm{\phi_2}{w}.
          \]
      \end{enumerate}

    \smallskip

    \item[\textit{(iii)}] For \(u \in \V{\rkr}\) (i.e., \(u=\rp\)), and for any \(f,g\in \dR{\rp}\),
      \[
        \maxnorm{\D{\rp}[f\,g]}
        \;\le\;
        \dW{\rkr}\,\Unorm{f}{\rp}\,\Unorm{g}{\rp}.
      \]
  \end{enumerate}
\end{assumption}

By combining Proposition~\ref{prop: coreR} and Proposition~\ref{prop PT MAIN}, we obtain the following corollary:

\begin{cor}
    \label{cor PartII group}
  For any fixed \(1 \le \tCB \in \CC\), if \(\tCR\) is chosen large enough, then every vertex \(\rp\) with \(\rkr\ge0\) satisfies Assumption~\ref{Assume rp}.
\end{cor}

    \begin{proof}[Proof of the Corollary]
        We claim that once $\tCR$ is sufficiently large, every vertex \(\rp \in V(T)\) with \(\rkr \ge 0\) satisfies Assumption~\ref{Assume rp}. We verify each item in that assumption:
        The first \textbf{(i)} is assumed.  

        \noindent
        \textbf{(ii)}
        Consider any $u \in \V{t}$ with $0 \le t < \rkr$, and let $u \preceq w \preceq \rp$. 
        We need to check parts (a) and (b):
      
        \begin{enumerate}
          \item \emph{Existence of the decomposition} $f_u = f_{u,\mathcal{R}} + f_{u,\mathcal{T}}$.  
          First, we have $w = \anc{u}{\rk{w} - \rk{u}}$. We apply item~(2) in Proposition~\ref{prop: coreR} with the parameter $t$ in the Proposition setting to  $\rk{w}-\rk{u}$ to get a decomposition 
          \[
            f_u \;=\; f_{u,\mathcal{R}} \;+\; f_{u,\mathcal{T}}
          \] 
          with 
          \begin{align*}
                f_{u,\mathcal{R}} \in &  
                \dR{u;\,\rk{w}-\rk{u}} 
                \,\otimes\, 
                \TT{u}{[\rk{w}-\rk{u},\,\rkr - \rk{u}-1]}  \\
                =&
                \dR{u;\,\rk{w}-\rk{u}} 
                 \,\otimes\, 
                \TT{w}{[0,\,\rkr - \rk{w}-1]} \,.
          \end{align*}
          where the last equality follows from the fact that 
          $$\TT{u}{[\rk{w}-\rk{u},\,\rkr - \rk{u}-1]} = \TT{w}{[0,\,\rkr - \rk{w}-1]}\,.$$
          And the same holds for 
          \begin{align*}
            f_{u,\mathcal{T}} \in &  
            \T{w} 
            \,\otimes\, 
            \TT{w}{[0,\,\rkr - \rk{w}-1]} \,.
          \end{align*}
          Further, from item~(2) in Proposition~\ref{prop: coreR}, we have
          \begin{align*}
            \Unorm{f_{u,\mathcal{T}}}{\rp}
            \;\le\;
             \CW{\Tp{u}}
            \cdot
                  \tC_{\ref{prop: coreR}} R \exp(-\eps \lA{u})
                  \lame^{\rk{w}-\rk{u}} \exp(- \eps (\rk{w} - \rk{u})) 
                  \Unorm{f}{\rp}\,.
            \end{align*}
            With $\exp(-\eps \lA{u}) = \exp(-\eps \rk{u}) \exp(-\eps \hB)$ and \eqref{eq: role of dWu}, the above term can be bounded by
            \begin{align*}
                (*) 
            \le & 
                \dW{u}
                \cdot
                \Dt 
                  \lame^{\rk{w} - \rk{u}} \exp(- \eps \rk{w} ) 
                  \Unorm{f}{\rp}\,,
            \end{align*}

          \item \emph{Correlation bound for $\phi_1 \in \dR{u;\,\rk{w}-\rk{u}}$ and $\phi_2 \in \T{w}$.}
          From item~(1) of Proposition~\ref{prop: coreR} and \eqref{eq: role of dWu}, we have a uniform correlation bound of the form
          \[
            \maxnorm{\CE{w}\bigl[\phi_1\,\phi_2\bigr]}
            \;\le\;
            \tC_{\ref{prop: coreR}}
            \dW{t}
            \,\lame^{\rk{w} - \rk{u}} \,\Unorm{\phi_1}{w}\,\Unorm{\phi_2}{w},
          \]
          Thus, condition~(ii)(b) is also satisfied if $\tCR$ is large enough.
        \end{enumerate}

        \noindent
        \textbf{(iii)}
        Lastly, for $u \in \V{\rkr}$ we have $u = \rp$. Then the statement claims 
        \[
          \maxnorm{\D{\rp}\bigl[f\,g\bigr]} 
          \;\le\; 
          \dW{\rkr}\,\Unorm{f}{\rp}\,\Unorm{g}{\rp}
          \quad
          \text{for any } f,g\in \dR{\rp}.
        \]
        This simply follows from Lemma \ref{lem basicDecayAu2}, definition of $\dW{\rkr}$ and the norm comparison Lemma \ref{lemma: coreNorm}. 
               
     \end{proof}

\paragraph{Structure of \textbf{Part~II}}
The remainder of this section focuses on proving Proposition~\ref{prop PT 2K+1}, 
where we decompose a polynomial \(f\in \TT{\rp}{K+1}\) into polynomials in the spaces \(\dR{u}\).  
In Section~\ref{sec: RuRv}, we derive bounds on 
\(\maxnorm{\D{\rp}[f_u\,g_v]}\) 
and 
\(\bigl|\EE{\rp}[f_u\,g_v]\bigr|\) 
for \(f_u\in\dR{u}\) and \(g_v\in\dR{v}\).  
In Section~\ref{sec: layerwise}, we extend these to 
\(\maxnorm{\D{\rp}[f_t\,g_r]}\) 
and 
\(\bigl|\EE{\rp}[f_t\,g_r]\bigr|\), 
where 
\(f_t=\sum_{u\in\V{t}}f_u\) and \(g_r=\sum_{v\in\V{r}}g_v\), 
for \(t,r\in[\,0,\rkr\,]\).  
Finally, in Section~\ref{sec: proof main}, we combine these results to establish Theorem~\ref{theor main}.

\subsection{Decomposition of $\le 2^{K+1}$-degree polynomial}
Before we proceed to prove Proposition \ref{prop PT 2K+1},
we define a pivot vertex $\pvt{S}$ for each $S \subseteq L_{\rp}$ with $|S| \le 2^{K+1}$, and show that any function $\phi_S$ with variables $x_S$ is contained in $\TT{\pvt{S}}{-1}\otimes \Tr{\pvt{S}}$.
\begin{defi}
    \label{defi pivot}
    Consider a fixed $\rp \in V(T)$.
    For each $S \subseteq L_{\rp}$ with $|S| \le 2^{K+1}$, we define a pivot vertex
    \begin{align}
        \label{def pivot}
        \pvt{S} = \pvt{S, \rp} \in V(T_{\rp})
    \end{align}
    associated with $S$, constructed through the following process:
    \begin{enumerate}
        \item Initialize a pointer $\bf p$ at $\rp$.
        \item If theres exists a child vertex $v$ of the vertex to which $\bf p$ is currently pointing
              such that $|S \cap L_{v'}| > 2^K$, then update $\bf p$ to point to $v$ (note there is at most one such vertex $v$ given $|S| < 2^{K+1}$) and repeat this step . If no such vertex exists, then set $\pvt{S}$ to be the vertex where $\bf p$ is currently pointing to and terminate the process.
    \end{enumerate}
\end{defi}
\begin{rem}
    Note that if $S \subseteq L_{\rp}$ with $|S| \le 2^{K}$, then $\pvt{S} = \rp$.
\end{rem}
\begin{lemma}\label{lem Decompose pivot to TT}
    Suppose $S \subseteq L_{\rp}$ has $|S| \le 2^{K+1}$, and let $\phi_S$ be a function with variables $x_S$. 
    Then:
    \begin{enumerate}
      \item $\phi_S \in \dT{\pvt{S}}.$
      \item For any $v$ with $\pvt{S} \preceq v \preceq \rp$, we have 
      \[
        \phi_S \;\in\; \Tp{v}\;\otimes\;\TT{v}{[0,\;\rkr - \rk{v}-1]}.
      \]
      \item For any $v$ satisfying $v \preceq \pvt{S}$, we have $\phi_S \in \dT{v}.$
    \end{enumerate}
  \end{lemma}

\begin{proof}

    \step{Proof of first statement}
    Now we fix $S$ as described in the lemma, and let $\phi_S$ be a function in the variables $x_S$. For simplicity, let $A$ denote the set
    \[
      \OO{\pvt{S}}{[-1, \rkr - \rk{\pvt{S}} - 1]}.
    \]
    From the procedure used to find the pivot vertex $\pvt{S}$ for a set $S$, it follows directly that each $w \in A$ was once a child vertex of the vertex to which $\mathbf{p}$ was pointing at some stage, but was never itself visited by the pointer $\mathbf{p}$. Consequently, for every $w \in A$, we have
    \[
      S_w \;:=\; |S \cap L_w| \;\le\; 2^K.
    \]
    Observe that $\bigcup_{w \in A} L_w$ is a partition of $L_{\rp}$.
    Hence, we can view $\phi_S$ as a function of the variables $\bigl(x_{S_w}\bigr)_{w \in A}$. From this observation, we conclude that $\phi_S$ is a function of
    \[
      \T{A} \;=\; \TT{\pvt{S}}{[-1, \rkr - \rk{u_S}]}\,.
    \]
    
    \step{Proof of second statement} 
    Next, consider any $\pvt{S} \pe v \pe \rp$. Since 
    \[
      \OO{v}{[0, \rkr - \rk{v}-1]} \;\subseteq\; A,
    \]
    it follows that 
    \[
      |S_w| \;\le\; 2^K \quad\text{for all } w \in \OO{v}{[0, \rkr - \rk{v}-1]}.
    \]
    Let $A' := \OO{v}{[0, \rkr - \rk{v} - 1]}$. Because
    \[
      L_v \;\sqcup\; \bigsqcup_{w \in A'} L_w \;=\; L_{\rp}
    \]
    is a partition of $L_{\rp}$, we can express $\phi_S$ as a function of the variables
    \[
      \bigl(x_{S_w}\bigr)_{w \in \{v\} \cup A'}\,,
    \]
    where 
    \[
      S_v \;:=\; |S \cap L_v|\quad\text{and}\quad |S_v|\;\le\;|S|\;\le\;2^{K+1}.
    \]

    Finally, we write $\phi_S$ as a sum of indicator functions over all realizations of $x_S$. Specifically,
    \begin{align*}
      \phi\bigl((x_{S_w})_{w \in \{v\}\cup A'}\bigr)
      \;=\;&
      \sum_{(x'_{S_w})_{w \in \{v\} \cup A'}}
        \underbrace{\phi\bigl((x'_{S_w})_{w \in \{v\} \cup A'}\bigr)}_{\text{constant}}
        \;\underbrace{\mathbbm{1}_{x'_{S_v}}\bigl(x_{S_v}\bigr)}_{\in \Tp{v}}
        \;\prod_{w \in A'}\;
        \underbrace{\mathbbm{1}_{\{x'_{S_w}\}}\bigl(x_{S_w}\bigr)}_{\in \T{w}}
      \\[6pt]
      \in\;&
      \Tp{v} \;\otimes\; \TT{v}{[0, \rkr - \rk{v}-1]}\,.
    \end{align*}

    \step{Proof of third statement}
    This follows immediately from the first statement of the lemma together with Lemma \ref{lem TAinclusion}. 
\end{proof}

\begin{proof} [Proof of Proposition \ref{prop PT 2K+1}]
        \phantom{.}
        
        \step{Decomposition of $f$, bottom layer}
        We construct the decomposition layer by layer, starting with the bottom layer $f_u$ for $u \in \V{0}$.
        
        First, since $f$ is a sum of monomials $\phi_S$ with $S \subseteq L_{\rp}$ and $|S| \le 2^{K+1}$, each $\phi_S$ depends on the variables $x_S$. Moreover, each $\phi_S$ belongs to $\dT{\pvt{S}}$ by the first statement of Lemma~\ref{lem Decompose pivot to TT}. 
        Hence, we can write
        \begin{align*}
          f 
          \;=\;
            \,\sum_{\substack{S\,:\,\pvt{S} \pe \rp \\|S|\le 2^{K+1}}}\!\phi_S 
          \;=\;
          \sum_{u \in \V{0}}
            \,\sum_{\substack{S\,:\,\pvt{S}\,\preceq u\\|S|\le 2^{K+1}}}\!\phi_S
          \;+\;
          \sum_{\substack{S\,:\,\pvt{S}\,\in\,\V{[1,\rkr]}}}\!\phi_S.
        \end{align*}
        
        Now, fix $u \in \V{0}$. For each $S$ with $\pvt{S} \preceq u$, we have $\phi_S \in \Tp{u}\,\otimes\,\TT{u}{[0,\rkr-\rk{u}-1]}$ by the second statement of Lemma~\ref{lem Decompose pivot to TT} with $v = u$.  
        Hence,
        \[
          \sum_{\substack{S\,:\,\pvt{S}\,\preceq u\\|S|\le 2^{K+1}}}\!\phi_S
          \;\;\in\;\;
          \mathcal{T}_{K+1}(u)\,\otimes\,\TT{u}{[0,\rkr-\rk{u}-1]},
        \]
        and by Lemma~\ref{lem R Decompose} together with the definition of $\dR{u}$ in \eqref{def dR}, there exists 
        \[
          f_u \;\in\;\dR{u}
        \]
        such that
        \begin{align*}
          f_u 
          \;-\;
          \sum_{\substack{S\,:\,\pvt{S}\,\preceq u\\|S|\le 2^{K+1}}}\!\phi_S 
          &\;\;\in\;\;
          \T{u}\,\otimes\,\TT{u}{[0,\rkr-\rk{u}-1]}
          \;=\;
          \dT{\anc{u}{1}}
        \end{align*}
        From this and third statement of Lemma~\ref{lem Decompose pivot to TT}, we obtain a decomposition
        \[
          f \;-\; \sum_{u \in \V{0}} f_u
          \;=\;
          \sum_{v \in \V{1}}\! \phi_v,
        \]
        where each $\phi_v \in \dT{v}$.

        \step{Decomposition of $f$, intermediate layers}
We claim inductively that for each \(k \in [0,\rkr-1]\), we can write
\begin{align}\label{eq PT 2K+1 00}
  f \;-\; \sum_{u \in \V{[0,k]}} f_u
  \;=\;
  \sum_{v \in \V{k+1}}\! \phi_v,
\end{align}
where \(f_u \in \dR{u}\) for each \(u \in \V{[0,k]}\) and \(\phi_v \in \dT{v}\) for each \(v \in \V{k+1}\).

The base case \(k=0\) was established above.

Assume the statement is true for some \(k < \rkr-1\). Consider any \(v \in \V{k+1}\). Note that
\[
  \phi_v \;\in\;\dT{v}
  \;=\;
  \TT{v}{-1}\,\otimes\,\TT{v}{[0,\rkr-\rk{v}-1]}.
\]
By Lemma~\ref{lem R Decompose} and the definition of \(\dR{v}\) in \eqref{def dR}, there is an element \(f_v \in \dR{v}\) such that
\[
  f_v \;-\;\phi_v
  \;\in\;
  \T{v}\,\otimes\,\TT{v}{[0,\rkr-\rk{v}-1]} 
  \;=\;
  \dT{\anc{v}{1}}.
\]
Hence, by the induction hypothesis for \(k\), we have
\[
  f \;-\;\sum_{u \in \V{[0,k+1]}}\! f_u
  \;=\;
  \sum_{u \in \V{k}} \bigl(\phi_u - f_u\bigr)
  \;=\;
  \sum_{v \in \V{k+1}}
    \Bigl(\,\sum_{u \in \frak{c}(v)}(\phi_u - f_u)\Bigr).
\]
Since each sum \(\sum_{u \in \frak{c}(v)}(\phi_u - f_u)\) is contained in \(\dT{v}\), this establishes the claim for \(k+1\). 

Therefore, \eqref{eq PT 2K+1 00} holds for all \(k \in [0,\rkr-1]\).

\step{Decomposition of $f$, top layer}
From the previous step, we have 
\[
  f
  \;-\;
  \sum_{u \in \V{[0,\rkr-1]}} f_u
  \;=\;
  \phi_{\rp},
\]
where \(\phi_{\rp} \in \dT{\rp} = \TT{\rp}{-1}\). Hence, setting \(f_{\rp} = \phi_{\rp}\) completes the proof of the proposition.
    \end{proof}
\bigskip

\section{Pairwise Analysis}
\label{sec: RuRv}
The goal of this section is dedicated to proving the following Proposition.
\begin{prop}
    \label{prop RuRv}
    There exists a constant $\tC \in \CC$ such that
    the following statement holds for sufficiently large $\tCR$:\\
    Fix $\rp$ satisfying Assumption \ref{Assume rp}.
    For a pair of vertices  $u,v \in \V{}$, consider
    $f_u \in \dR{u}$ and $g_v \in \dR{v}$.
    \begin{itemize}
        \item

              If $u \neq v$, let $w$ be the nearest common ancestor of $u$ and $v$. Then,
              \begin{align}
                  \label{eq RuRv D}
                  \maxnorm{\D{\rp} f_u g_v}
                  \le &
                  \tC \cdot \dW{u}\lame^{\rk{w}-\rk{u}} \cdot
                  \dW{v} \lame^{\rk{w} - \rk{v}} \cdot
                  \exp \big(-\eps(\rkr- \rk{w}) \big)
                  \Unorm{f_u}{\rp} \Unorm{g_v}{\rp}\,, \quad \mbox{ and } \\
                  \label{eq RuRv E}
                  \left| \EE{\rp} f_u g_v \right|
                  \le &
                  \tC \cdot \dW{u}\lame^{\rk{w}-\rk{u}} \cdot
                  \dW{v} \lame^{\rk{w} - \rk{v}}
                  \Unorm{f_u}{\rp} \Unorm{g_v}{\rp}\,,
              \end{align}
        \item
              If $u = v$, we then
              \begin{align*}
                  \maxnorm{\D{\rp} f_u g_v}
                  \le &
                  \begin{cases}
                      \tC \left( \tlam^{\rkr} + \Dt \exp( -\eps \rkr)
                      \right) \Unorm{f_u}{\rp} \Unorm{g_v}{\rp}                  & \mbox{ if } u \in \V{0},                        \\
                      \tC \Dt \exp(-\eps \rkr) \Unorm{f_u}{\rp} \Unorm{g_v}{\rp} & \mbox{ if } u \in \V{k} \mbox{ for } k \ge 1\,. \\
                  \end{cases}  \mbox{ and } \\
                  \left| \EE{\rp} f_u g_v \right|
                  \le &
                  \Unorm{f_u}{\rp} \Unorm{g_v}{\rp}\,.
              \end{align*}
    \end{itemize}
\end{prop}
The last statement of the Proposition is a direct consequence of the Cauchy-Schwarz inequality; it is included here for completeness.

\subsection{Decay generalization}
To prove Proposition \ref{prop RuRv}, we require a generalization of the decay properties. In this subsection, we establish two lemmas that capture these extended decay properties.

\begin{lemma}
    \label{lem PT Correlation}
    There exists a constant $\tC \in \CC$ such that the following statement holds for sufficiently large $\tCR$:\\
    Let $w \in \V{}$ and let $w_1$ and $w_2$ be two children of $w$.
    For $i \in [2]$, suppose ${\cal W}_i$ is a subspace of $\F{w_{i \pe}}$ associated with a value $\delta_i >0$ such that
    \begin{align*}
        \maxnorm{\CE{w_i}\phi_1\phi_2} \le \delta_i \Unorm{\phi_1}{w_i} \Unorm{\phi_2}{w_i} \mbox{ for } \phi_1 \in {\cal W}_i \mbox{ and } \phi_2 \in \T{w_i}\,.
    \end{align*}

    Let $k = \rkr-\rk{w}$, and set 
    $$
        B = \OO{w}{[-1,k-1]} \setminus \{w_1,w_2\}\,.
    $$
    Then for
    \begin{align*}
        f \in & \phantom{A,} {\cal W}_1 \phantom{A,} \otimes \T{w_2} \otimes \T{B}  \subseteq \F{\rp_{\pe}} \quad \mbox{ and } \\
        g \in & \T{w_1} \otimes \phantom{A,} {\cal W}_2 \phantom{A,} \otimes \T{B} \subseteq \F{\rp_{\pe}} \,,
    \end{align*}
    we have
    \begin{align*}
        \maxnorm{\CE{\rp}fg} \le & 2 \delta_1 \delta_2 \Unorm{f}{\rp} \Unorm{g}{\rp}
        \phantom{AAA AAA AAA AAA AAA AAA} \mbox{and}
        \\
        \maxnorm{\D{\rp}fg} \le  &
        \tC \delta_1 \delta_2 \exp( - \eps (\rkr - \rk{w})) \Unorm{f}{\rp} \Unorm{g}{\rp}\,.
    \end{align*}
\end{lemma}
See Figure \ref{fig: PT Correlation} for an illustration of the vertices and sets described in the lemma. 

\begin{figure}[h]
    \centering
    \label{fig: PT Correlation}
    \includegraphics[width=0.6\textwidth]{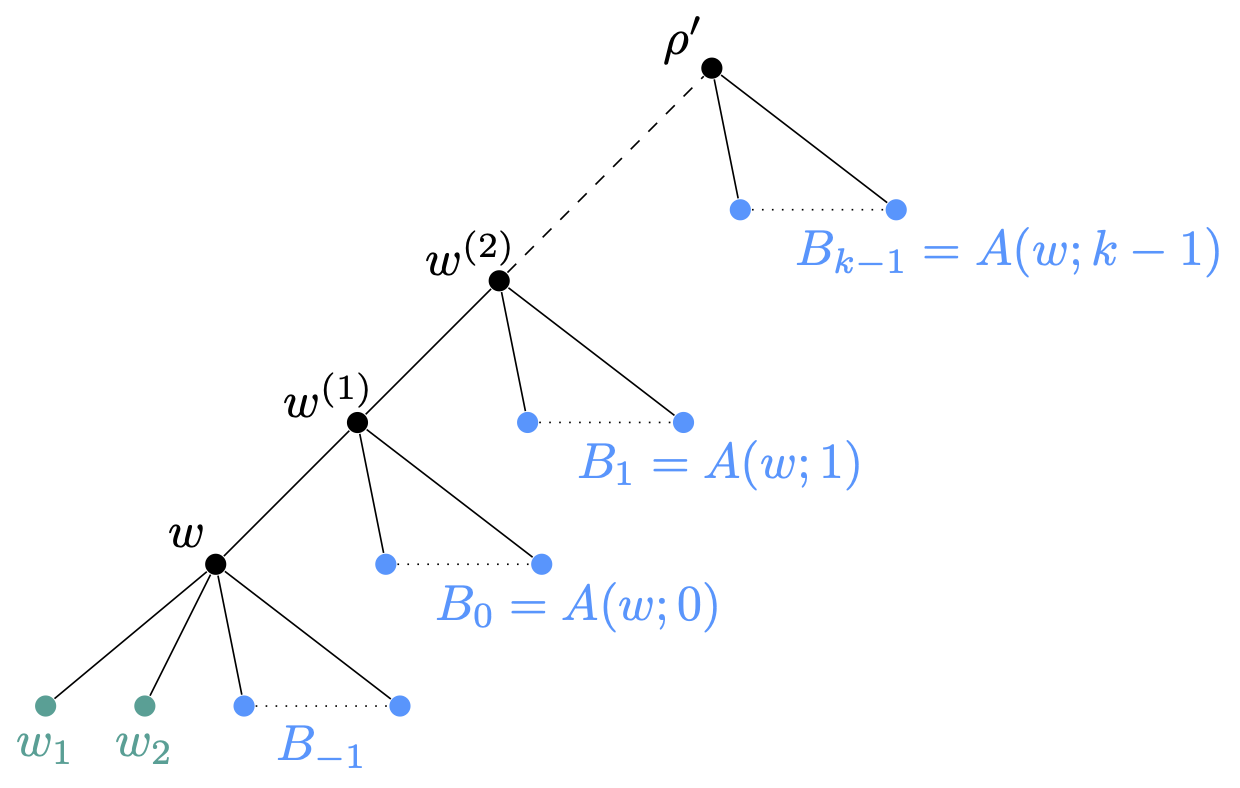}
    \caption{An illustrations of the vertices and sets described in the lemma.}
\end{figure}

\begin{lemma}
    \label{lem fugv self}
    There exists a constant $\tC \in \CC$ such that the following statement holds for sufficiently large $\tCR$:\\
    Let $w \in \V{}$ and suppose $\cW$ is a subspace of $\F{w_{\pe}}$ such that
    \begin{align*}
        \maxnorm{\D{w}fg} \le \delta \Unorm{f}{w} \Unorm{g}{w} \mbox{ for } f,g \in \cW
    \end{align*}
    for some $\delta >0$. Let $k = \rkr-\rk{w}$, and let
    $$
        B = \OO{w}{[0,k-1]}\,.
    $$
    For
    \begin{align*}
        f,g \in & {\cal W} \otimes \T{B}\,,
    \end{align*}
    we have
    \begin{align*}
        \maxnorm{\D{\rp}fg} \le &
        \tC \left( \delta \tlam^k +  \Dt \exp( -\eps \rkr)
        \right)
        \Unorm{f}{\rp} \Unorm{g}{\rp}\,.
    \end{align*}
\end{lemma}

For the proof of above two lemmas, we need the following simple arithmetic expansion identity.

\begin{lemma}
\label{lem:TensorArithmeticIdentity}
Let $t_1 \le t_2$ be integers. For each integer $t \in [t_1,t_2]$, write
\[
  c_t \;=\; a_t \;+\; b_t,
\]
where $a_t$ and $b_t$ lie in some vector space (or module) $H_t$. Define 
\[
  b_t \;=\; \mathbf{1}
  \quad\text{and}\quad
  c_t \;=\; \mathbf{1}
  \quad
  \text{for } t < t_1,
\]
where $\mathbf{1}$ is the identity element for the tensor product in your ambient algebraic setting. Then
\begin{equation}\label{eq:TensorArithmeticIdentity}
  \bigotimes_{t=t_1}^{t_2} c_t
  \;=\;
  \sum_{t = t_1 - 1}^{t_2}
    \Bigl(\,\bigotimes_{s = t_1}^{\,t-1} c_s\Bigr)
    \;\otimes\;
    b_t
    \;\otimes\;
    \Bigl(\,\bigotimes_{s = t+1}^{\,t_2} a_s\Bigr),
\end{equation}
where, by convention, an empty tensor product is set to $\mathbf{1}$.
\end{lemma}

\begin{proof}
\textbf{Base Case ($t_2 = t_1$).}  
In this case,
\[
  \bigotimes_{t = t_1}^{t_1} c_t
  \;=\;
  c_{t_1}
  \;=\;
  a_{t_1} + b_{t_1}.
\]
On the right-hand side of \eqref{eq:TensorArithmeticIdentity}, the sum ranges over $t = t_1 - 1$ to $t_1$.  
\begin{itemize}
\item 
For $t = t_1 - 1$, both $\displaystyle \bigotimes_{s = t_1}^{t-1} c_s$ and $\displaystyle \bigotimes_{s = t+1}^{t_1} a_s$ are empty tensor products, hence each equals $\mathbf{1}$. Also, $b_{t_1 - 1} = \mathbf{1}$ by definition. Thus the summand is
\[
  (\mathbf{1}) \;\otimes\; \mathbf{1} \;\otimes\; (\mathbf{1})
  \;=\;
  \mathbf{1}.
\]
\item
For $t = t_1$, the factors outside $b_{t_1}$ are again empty tensors, so the summand is
\[
  \mathbf{1} \;\otimes\; b_{t_1} \;\otimes\; \mathbf{1}
  \;=\;
  b_{t_1}.
\]
\end{itemize}
Depending on the precise convention for ``$c_t = \mathbf{1}$ for $t < t_1$,'' one obtains $a_{t_1} + b_{t_1}$ overall. Hence the identity holds for $t_2 = t_1$.

\noindent
\textbf{Inductive Step.}  
Assume \eqref{eq:TensorArithmeticIdentity} holds for all pairs $(t_1, r)$ with $r - t_1 < N$. Let $t_2 = r+1$. We show it holds for $\displaystyle \bigotimes_{t = t_1}^{r+1} c_t$.  
By definition,
\[
  \bigotimes_{t = t_1}^{r+1} c_t
  \;=\;
  \biggl(\bigotimes_{t = t_1}^{r} c_t\biggr)
  \;\otimes\;
  \bigl(a_{r+1} \;+\; b_{r+1}\bigr).
\]
Applying the inductive hypothesis to $\displaystyle \bigotimes_{t = t_1}^{r} c_t$ yields
\[
  \bigotimes_{t = t_1}^{r} c_t
  \;=\;
  \sum_{\,t = t_1 - 1}^{r}
    \Bigl(\!\!\!\bigotimes_{s = t_1}^{\,t-1} c_s\Bigr)
    \;\otimes\;
    b_t
    \;\otimes\;
    \Bigl(\!\!\bigotimes_{s = t+1}^{\,r} a_s\Bigr).
\]
Hence
\[
  \biggl(\bigotimes_{t = t_1}^{r} c_t\biggr)
  \;\otimes\;
  \bigl(a_{r+1} + b_{r+1}\bigr)
  \;=\;
  \sum_{\,t = t_1 - 1}^{\,r}
    \Bigl(\!\!\bigotimes_{s = t_1}^{\,t-1} c_s\Bigr)
    \;\otimes\;
    b_t
    \;\otimes\;
    \Bigl(\!\!\bigotimes_{s = t+1}^{\,r} a_s\Bigr)
    \;\otimes\;
    \bigl(a_{r+1} + b_{r+1}\bigr).
\]
We now distribute $a_{r+1} + b_{r+1}$ via the bilinearity of $\otimes$:
\[
  X \;\otimes\; (\,a_{r+1} + b_{r+1}\,)
  \;=\;
  (\,X \;\otimes\; a_{r+1}\,)
  \;+\;
  (\,X \;\otimes\; b_{r+1}\,).
\]
So each summand splits into two:
\begin{itemize}
\item One summand picks up $a_{r+1}$, yielding
\[
  \Bigl(\!\!\bigotimes_{s = t_1}^{\,t-1} c_s\Bigr)
  \;\otimes\;
  b_t
  \;\otimes\;
  \Bigl(\!\!\bigotimes_{s = t+1}^{\,r+1} a_s\Bigr),
\]
where we have extended the range of $a_s$ to include $s = r+1$.
\item The other summand adds a new term for $t = r+1$ itself, namely
\[
  \bigl(\!\!\bigotimes_{t = t_1}^{r} c_t\bigr)
  \;\otimes\;
  b_{r+1}.
\]
\end{itemize}
Reindexing, we combine these to form
\[
  \sum_{\,t = t_1 - 1}^{\,r+1}
    \Bigl(\!\!\bigotimes_{s = t_1}^{\,t-1} c_s\Bigr)
    \;\otimes\;
    b_t
    \;\otimes\;
    \Bigl(\!\!\bigotimes_{s = t+1}^{\,r+1} a_s\Bigr),
\]
exactly matching the right-hand side of \eqref{eq:TensorArithmeticIdentity} for $t_2 = r+1$.  
Thus, by induction, the identity holds for all $t_2 \ge t_1$.
\end{proof}

\begin{proof}[Proof of Lemma \ref{lem PT Correlation}]

    \step{Notations}
    For simplicity, let 
    \[
      B_{-1} 
      \;=\; 
      \OO{w}{-1} \,\setminus\, \{w_1,w_2\}
      \quad\text{and}\quad
      B_{t} 
      \;=\; 
      \OO{w}{t}
      \quad\text{for } t \ge 0.
    \]
    Then define
    \[
      B_{[a,b]} 
      \;=\;
      \bigcup_{t \in [a,b]} B_t.
    \]
    For \(-1 \le a \le b \le k-1\), we have
    \[
      \T{B_{[a,b]}} 
      \;\subseteq\;
      \TT{w}{[a,b]}.
    \]
    With
    \(
      \{w_1,w_2\} \,\cup\, B_{[-1,k-1]}
      \;\pe\;
      \{\rp\}
    \),
    we use the identity from Lemma~\ref{lem basicEECED} to write
    \[
      \CE{\rp} \bigl(fg\bigr)
      \;=\;
      \CE{\rp} \Bigl[\CE{\{w_1,w_2\} \cup B_{[-1,k-1]}} \bigl(fg\bigr)\Bigr]
      \quad\text{and}\quad
      \D{\rp} \bigl(fg\bigr)
      \;=\;
      \D{\rp} \Bigl[\CE{\{w_1,w_2\} \cup B_{[-1,k-1]}} \bigl(fg\bigr)\Bigr].
    \]
    
    \step{Deriving the bound for 
    \(\maxnorm{\CE{\{w_1,w_2\} \cup B_{[-1,k-1]}} fg}\)}
    For \(\phi_1,\phi_2 \in \T{B_{[-1,k-1]}}\), Lemma~\ref{lem basicDecayAu2} gives
    \[
      \maxnorm{\CE{B_{[-1,k-1]}} \phi_1 \,\phi_2}
      \;\le\;
      \Bigl(1 + \Delta \,\exp\bigl(-\eps\,\rk{w}\bigr)\Bigr)\,
      \Unorm{\phi_1}{B_{[-1,k-1]}}\,
      \Unorm{\phi_2}{B_{[-1,k-1]}}\,,
    \]
    provided \(\tCR\) is sufficiently large. 
    Combining this with the assumptions on 
    \({\cal W}_1\) and \({\cal W}_2\), we can tensorize the norm via Lemma~\ref{lem: mainTensorProduct} to obtain
    \[
      \maxnorm{\CE{\{w_1,w_2\} \cup B_{[-1,k-1]}} \bigl(fg\bigr)}
      \;\le\;
      \delta_1 \,\delta_2\,
      \Bigl(1 + \Dt \,\exp\bigl(-\eps\,\rk{w}\bigr)\Bigr)\,
      \Unorm{f}{\{w_1,w_2\}\cup B_{[-1,k-1]}}
      \,\Unorm{g}{\{w_1,w_2\}\cup B_{[-1,k-1]}}.
    \]
    Using Lemma~\ref{lemma: coreNorm}, which implies 
    \[
      \Unorm{f}{\{w_1,w_2\}\cup B_{[-1,k-1]}}
      \,\le\,
      (1+\kappa)\,\Unorm{f}{\rp}
      \quad\text{and}\quad
      \Unorm{g}{\{w_1,w_2\}\cup B_{[-1,k-1]}}
      \,\le\,
      (1+\kappa)\,\Unorm{g}{\rp},
    \]
    we conclude
    \[
      \maxnorm{\CE{\{w_1,w_2\} \cup B_{[-1,k-1]}} \bigl(fg\bigr)}
      \;\le\;
      (1+\kappa)\,\delta_1\,\delta_2\,
      \Bigl(1 + \Dt \,\exp\bigl(-\eps\,\rk{w}\bigr)\Bigr)\,
      \Unorm{f}{\rp}\,\Unorm{g}{\rp}.
    \]
    Finally, noting that
    \[
      \maxnorm{\CE{\rp} \Bigl[\CE{\{w_1,w_2\} \cup B_{[-1,k-1]}} \bigl(fg\bigr)\Bigr]}
      \;\le\;
      \maxnorm{\CE{\{w_1,w_2\} \cup B_{[-1,k-1]}} \bigl(fg\bigr)},
    \]
    the first statement of the lemma follows by choosing \(\tCR\) sufficiently large.

    \step{Decomposition of $\CE{\{w_1,w_2\} \cup B_{[-1,k-1]} }$}
    To estimate 
    \[
      \maxnorm{\D{\rp} fg}
      \;=\;
      \maxnorm{\D{\rp} \bigl(\CE{\{w_1,w_2\} \cup B_{[-1,k-1]}}\,fg\bigr)},
    \]
    we first decompose the operator 
    \(\CE{\{w_1,w_2\} \cup B_{[-1,k-1]}}\)
    in the same way as ~\eqref{eq:TensorArithmeticIdentity} in Lemma~\ref{lem:TensorArithmeticIdentity}:
    \begin{align}
    \label{eq arithemetic tensor decomposition}
      \CE{\{w_1,w_2\} \cup B_{[-1,k-1]}} 
      \;=\;&
      \CE{w_1}\,\otimes\,\CE{w_2}
      \,\otimes\,
      \bigotimes_{t=-1}^{k-1} \CE{B_t}
      \\
      =\;&
      \CE{w_1}\,\otimes\,\CE{w_2}
      \,\otimes\,
      \bigotimes_{t=-1}^{k-1}
        \bigl(\EE{B_t} \;+\; \D{B_t}\bigr)
      \nonumber
      \\
      =\;&
      \sum_{t = -2}^{k-1}
        \CE{w_1}\,\otimes\,\CE{w_2}
        \,\otimes\,
        \bigotimes_{s \in [-1,t-1]} \CE{B_s}
        \,\otimes\,
        \D{B_t}
        \,\otimes\,
        \bigotimes_{s \in [t+1,k-1]} \EE{B_s}
      \nonumber
      \\
      =\;&
      \sum_{t = -2}^{k-1}
        \underbrace{
          \CE{w_1} \,\otimes\, \CE{w_2}
          \,\otimes\, \CE{B_{[-1,t-1]}}
          \,\otimes\, \D{B_t}
          \,\otimes\, \EE{B_{[t+1,k-1]}}
        }_{:= \bL_t}.
      \nonumber
    \end{align}
    In particular, the first two terms in this sum are:
    \[
      \bL_{-2}
      \;=\;
      \CE{w_1} \,\otimes\, \CE{w_2}
      \,\otimes\, \EE{B_{[-1,k-1]}}
      \quad\text{and}\quad
      \bL_{-1}
      \;=\;
      \CE{w_1} \,\otimes\, \CE{w_2}
      \,\otimes\, \D{B_{-1}}
      \,\otimes\, \EE{B_{[0,k-1]}}.
    \]
    Hence,
    \[
      \D{\rp} \bigl(fg\bigr)
      \;=\;
      \sum_{t=-2}^{k-1}
        \D{\rp}\bigl[\bL_t \,(fg)\bigr].
    \]
    
    \step{Estimate of $\maxnorm{\bL_t fg}$ for $t \in [-1,k-1]$}
    We now establish a bound for each summand \(\bL_t fg\). Observe:
    
    \begin{enumerate}
    \item
      For \(\phi_1, \phi_2 \in \T{B_{[-1,t-1]}}\), 
      Lemma~\ref{lem basicDecayAu2} implies
      \[
        \maxnorm{\CE{B_{[-1,t-1]}} \,\phi_1 \,\phi_2}
        \;\le\;
        \Bigl(1 + \Dt \exp\bigl(-\eps\,\rk{w}\bigr)\Bigr)\,
        \Unorm{\phi_1}{B_{[-1,t-1]}}
        \,\Unorm{\phi_2}{B_{[-1,t-1]}}
        \;\le\;
        2\,
        \Unorm{\phi_1}{B_{[-1,t-1]}}
        \,\Unorm{\phi_2}{B_{[-1,t-1]}},
      \]
      when \(\tCR\) is large enough.
    
    \item
      For \(\phi_1,\phi_2 \in \T{B_t} = \TT{w}{t} = \TT{w^t}{0}\), again by Lemma~\ref{lem basicDecayAu2},
      \[
        \maxnorm{\D{B_t}\,\phi_1\,\phi_2}
        \;\le\;
        \Dt\,
        \exp\bigl(-\,\eps\,\rk{w} \;-\;\eps\,t\bigr)\,
        \Unorm{\phi_1}{B_t}\,\Unorm{\phi_2}{B_t}.
      \]
    
    \item
      For \(\phi_1,\phi_2 \in \T{B_{[t+1,k-1]}}\),
      \[
        \maxnorm{\EE{B_{[t+1,k-1]}} \,\phi_1 \,\phi_2}
        \;=\;
        \bigl|\EE{B_{[t+1,k-1]}} \,\phi_1 \,\phi_2\bigr|
        \;\le\;
        \Unorm{\phi_1}{B_{[t+1,k-1]}}\,\Unorm{\phi_2}{B_{[t+1,k-1]}}.
      \]
    \end{enumerate}
    
    Combining these bounds with the assumptions on \({\cal W}_1\) and \({\cal W}_2\), and applying norm ternsorization Lemma~\ref{lem: mainTensorProduct}, we obtain
    \begin{align}
    \label{eq coreDecay0-0}
      \maxnorm{\bL_t \,f\,g}
      &\;\le\;
      \delta_1\,\delta_2
      \;\cdot\;
      2
      \;\cdot\;
      \Dt \,\exp\bigl(-\eps\,\rk{w} \;-\;\eps\,t\bigr)
      \,\Unorm{f}{\OO{w}{[-1,k-1]}}
      \,\Unorm{g}{\OO{w}{[-1,k-1]}}
      \\[6pt]
      &\;\le\;
      (1+\kappa)
      \,\delta_1\,\delta_2\,2\,\Dt\,
      \exp\bigl(-\eps\,\rk{w} \;-\;\eps\,t\bigr)
      \,\Unorm{f}{\rp}\,\Unorm{g}{\rp},
      \nonumber
    \end{align}
    where we have used Lemma~\ref{lemma: coreNorm} in the final inequality. To clarify, in the case $t=-1$, we don't need the factor $2$ in the bound.

    \step{Compare \(\maxnorm{\D{\rp}\bL_tfg}\) with \(\maxnorm{\bL_tfg}\) for \(t \in [-1,k-1]\)}
By the definition of \(\bL_t\), we have 
\[
  \bL_t\,fg 
  \;\in\;
  \F{\OO{w}{[-1,t]}}
  \;\subseteq\;
  \F{w^{t+1}_{\pe}}.
\]
Hence, applying both parts of Lemma~\ref{lem basicEECED}, we get
\begin{align}
\label{eq coreDecay0-05}
  \D{\rp} \bigl(\bL_t\,fg\bigr)
  &= 
  \CE{\rp} \bigl(\bL_t\,fg\bigr)
  \;-\;
  \EE{\rp} \bigl(\bL_t\,fg\bigr)
  \\[5pt]
  &=
  \CE{\rp}\,\CE{w^{t+1}}\bigl(\bL_t\,fg\bigr)
  \;-\;
  \EE{w^{t+1}}\bigl(\bL_t\,fg\bigr)
  \;=\;
  \CE{\rp}
  \underbrace{\D{w^{t+1}}\bigl(\bL_t\,fg\bigr)}_{\in \Fz{w^{t+1}}}.
  \nonumber
\end{align}
Therefore,
\begin{align}
\label{eq coreDecay0-00}
  \maxnorm{\D{\rp}\bigl(\bL_t\,fg\bigr)}
  &\;\stackrel{\eqref{eq 1variableDecay}}{\le}\;
  \tCM\,\tlam^{\,k-1-t}\,\maxnorm{\D{w^{t+1}}\bigl(\bL_t\,fg\bigr)}
  \\[6pt]
  &=
  \tCM\,\tlam^{\,k-1-t}\,\maxnorm{\CE{w^{t+1}}\bigl(\bL_t\,fg\bigr)
    \;-\;\EE{w^{t+1}}\bigl(\bL_t\,fg\bigr)}
  \nonumber
  \\[6pt]
  &\le
  2\,\tCM\,\tlam^{\,k-1-t}\,\maxnorm{\CE{w^{t+1}}\bigl(\bL_t\,fg\bigr)}
  \nonumber
  \\[6pt]
  &\le
  2\,\tCM\,\tlam^{\,k-1-t}\,\maxnorm{\bL_t\,fg}
  \nonumber
  \\[6pt]
  &\;\stackrel{\eqref{eq coreDecay0-0}}{\le}\;
  \tC\,
  \tlam^{k-1-t}\,\delta_1\,\delta_2\,
  \Dt\,\exp\!\bigl(-\,\eps\,\rk{w} \;-\;\eps\,t\bigr)\,
  \Unorm{f}{\rp}\,\Unorm{g}{\rp},
  \nonumber
\end{align}
for some \(\tC \in \CC\).

\step{Estimating \(\maxnorm{\D{\rp}\bL_{-2}fg}\)}
Recall
\[
  \bL_{-2}
  \;=\;
  \CE{w_1}\,\otimes\,\CE{w_2}
  \,\otimes\,\EE{B_{[-1,k-1]}}.
\]
The difference here is that \(\bL_{-2}\) does not contain any \(\D{}\)-type operator; otherwise, the argument is similar. We sketch the proof:

By an argument analogous to \eqref{eq coreDecay0-0}, one obtains
\begin{align}
    \label{eq coreDecay0-0x}
  \maxnorm{\bL_{-2}\,f\,g}
  \;\le\;
  \delta_1\,\delta_2\,(1+\kappa)\,\Unorm{f}{\rp}\,\Unorm{g}{\rp}.
\end{align}
Since \(\bL_{-2}\,f\,g\) depends only on the variables \(x_{w_1}\) and \(x_{w_2}\), we treat it as a function of $\F{w_{\pe}}$.    
Then, by Lemma \ref{lem basicEECED}, we have
\begin{align*}
    \D{\rp}\,\bigl(\bL_{-2}\,f\,g\bigr)
= &
    \CE{\rp} \,\bigl(\bL_{-2}\,f\,g\bigr)
-
    \EE{\rp} \,\bigl(\bL_{-2}\,f\,g\bigr) \\
= &
    \CE{\rp} \CE{w} \,\bigl(\bL_{-2}\,f\,g\bigr)
-
    \EE{w} \,\bigl(\bL_{-2}\,f\,g\bigr) \\
= &
  \CE{\rp} \CE{w} \,\bigl(\bL_{-2}\,f\,g\bigr)
-
   \CE{\rp} \EE{w} \,\bigl(\bL_{-2}\,f\,g\bigr) \\
= &
    \CE{\rp} \D{w}  \,\bigl(\bL_{-2}\,f\,g\bigr) \,.
\end{align*}

We write
\begin{align}
\label{eq coreDecay0-01}
  \maxnorm{\D{\rp}\,\bigl(\bL_{-2}\,f\,g\bigr)}
  &= 
  \maxnorm{\CE{\rp}\,\underbrace{\D{w}\bigl(\bL_{-2}\,f\,g\bigr)}_{\in\,\Fz{w}}}
  \;\le\;
  \tCM\,\tlam^{k}\,\maxnorm{\D{w}\bigl(\bL_{-2}\,f\,g\bigr)}.
\end{align}
Furthermore,
\[
  \maxnorm{\D{w}\bigl(\bL_{-2}\,f\,g\bigr)}
  \;\le\;
  \maxnorm{\CE{w}\bigl(\bL_{-2}\,f\,g\bigr)}
  \;+\;
  \maxnorm{\EE{w}\bigl(\bL_{-2}\,f\,g\bigr)}
  \;\stackrel{\eqref{eq coreDecay0-0x}}{=}\;
  2\,\delta_1\,\delta_2\,(1+\kappa)\,\Unorm{f}{\rp}\,\Unorm{g}{\rp}.
\]
Consequently,
\[
  \maxnorm{\D{\rp}\bigl(\bL_{-2}\,f\,g\bigr)}
  \;\le\;
  2\,\delta_1\,\delta_2\,(1+\kappa)\,\tCM\,\tlam^{k}\,
  \Unorm{f}{\rp}\,\Unorm{g}{\rp}.
\]

\step{Summation over \(t\)}
Finally, summing over \(t\) yields
\begin{align*}
  \maxnorm{\D{\rp}\bigl(f\,g\bigr)}
  \;\le\; &
  \sum_{t = -2}^{k-1}
    \maxnorm{\D{\rp}\,\bigl(\bL_t\,f\,g\bigr)} \\
  \;\le\; &
  \tC\,\delta_1\,\delta_2
  \Bigl(
    \tlam^k
    \;+\;
    \Dt\,\exp\!\bigl(-\,\eps\,\rk{w}\bigr)
    \sum_{t=-1}^{k-1}
      \tlam^{k-1-t}\,\exp\!\bigl(-\,\eps\,t\bigr)
  \Bigr)
  \,\Unorm{f}{\rp}\,\Unorm{g}{\rp},
\end{align*}
for some \(\tC \in \CC\). Since \(\tlam\,\exp(\eps) < 1\),
\[
  \sum_{t=-1}^{k-1}
    \tlam^{k-1-t}\,\exp\!\bigl(-\,\eps\,t\bigr)
  \;\le\;
  \exp\bigl(-\,\eps\,(k-1)\bigr)\,
  \frac{1}{1-\tlam\,\exp(\eps)}
  \;\le\;
  \tC'\,\exp\!\bigl(-\,\eps\,k\bigr)
  \quad
  \text{for some }\tC' \in \CC.
\]
Therefore,
\[
  \tlam^k
  \;+\;
  \Dt\,\exp\!\bigl(-\,\eps\,\rk{w}\bigr)
  \sum_{t=-1}^{k-1}
    \tlam^{k-1-t}\,\exp\!\bigl(-\,\eps\,t\bigr)
  \;\le\;
  \tlam^k
  \;+\;
  \tC'\,\Dt\,\exp\!\bigl(-\,\eps\,\rk{w}\bigr)\,\exp\!\bigl(-\,\eps\,k\bigr)
  \;\le\;
  \tC''\,\exp\!\bigl(-\,\eps\,k\bigr),
\]
for some \(\tC'' \in \CC\), provided \(\tCR\) is sufficiently large. Substituting back into the previous bound for \(\maxnorm{\D{\rp}f\,g}\) completes the proof.
\end{proof}

\begin{proof}[Proof of Lemma \ref{lem fugv self}]

    \step{Overview}

The overall proof structure is similar to that of Lemma~\ref{lem PT Correlation}, but is actually simpler. Here, we outline the main steps and highlight the key differences. In this case, define
\[
  B_t \;:=\; \OO{w}{t}
  \quad
  \text{for } t \ge 0,
\]
and set
\[
  B \;=\; \bigotimes_{t=0}^{k-1} B_t.
\]
First, we express
\begin{align*}
  \D{\rp}fg \;=\; \D{\rp}\,\bigl(\CE{w} \otimes \CE{B}\bigr)\,fg.
\end{align*}
We then decompose \(\CE{w} \otimes \CE{B}\) into a sum of terms, much like in \eqref{eq arithemetic tensor decomposition}, obtaining
\begin{align*}
  \CE{w} \otimes \CE{B}
  \;=\;&
  \sum_{t=-1}^{k-1} \CE{w} \otimes 
  \CE{B_{[0,t-1]}} \otimes \D{B_t} \otimes \EE{B_{[t+1,k-1]}}
\end{align*}
For each \(t \in [0,k-1]\), we note that
\[
  \CE{w} \otimes \CE{B_{[0,t-1]}} \otimes \D{B_t} \otimes \EE{B_{[t+1,k-1]}}\,fg
  \;\in\;
  \F{w}\,\otimes\,\F{B_{[0,t-1]}}\otimes \F{B_t}.
\]
With the path
\[
  \{w\} \cup B_{[0,t-1]} \cup B_t
  \;\pe\;
  w^t \cup B_t
  \;\pe\;
  \rp,
\]
we apply Lemma~\ref{lem basicEECED} to get
\begin{align*}
  & \D{\rp}\,
    \Bigl[\CE{w} \otimes \CE{B_{[0,t-1]}} \otimes \D{B_t} \otimes \EE{B_{[t+1,k-1]}}\,fg\Bigr]
  \\[6pt]
  =\;&
    \D{\rp}\,\circ\,\bigl(\CE{w^t} \otimes \CE{B_t}\bigr)\circ
    \Bigl[\CE{w} \otimes \CE{B_{[0,t-1]}} \otimes \D{B_t} \otimes \EE{B_{[t+1,k-1]}}\,fg\Bigr]
  \\
  =\;&
    \D{\rp}\,
    \Bigl[
      \underbrace{\CE{w^t}\,\otimes\,\D{B_t}\,\otimes\,\EE{B_{[t+1,k-1]}}}_{:=\,\bL_t}
      fg
    \Bigr].
\end{align*}
Putting everything together, we obtain
\begin{align*}
  \D{\rp}fg
  \;=\;
  \D{\rp}\,\bigl[\CE{w} \otimes \EE{B}\,fg\bigr]
  \;+\;
  \sum_{t=0}^{k-1}
    \D{\rp}\,\bigl[\bL_t\,fg\bigr].
\end{align*}

\step{Estimate \(\maxnorm{\D{\rp}\,\bL_t\,fg}\) for \(t \in [0,k-1]\)}

Similarly to \eqref{eq coreDecay0-05}, we have
\[
  \D{\rp}\,\bigl[\bL_t\,fg\bigr]
  \;=\;
  \CE{\rp}\,\Bigl[
    \underbrace{\D{w^{t+1}}\,\bL_t\,fg}_{\in\,\Fz{w^{t+1}}}
  \Bigr],
\]
and hence
\begin{align}
  \label{eq fugv self 00}
  \maxnorm{\D{\rp}\,\bL_t\,fg}
  \;\le\;
  2\,\tCM\,\tlam^{\,k-t-1}\,\maxnorm{\D{w^{t+1}}\,\bL_t\,fg}
  \;\le\;
  2\,\tCM\,\tlam^{\,k-t-1}\,\maxnorm{\bL_t\,fg}.
\end{align}

Next, we apply Lemma~\ref{lem: mainTensorProduct} to 
\[
  \bL_t
  \;=\;
  \CE{w^t}\,\otimes\,\D{B_t}\,\otimes\,\EE{B_{[t+1,k-1]}},
\]
treating \(f,g\) as elements of \(\F{w^t_{\pe}} \otimes \F{B_t} \otimes \F{B_{[t+1,k-1]}}\). 
We examine each component of \(\bL_t\):

\begin{itemize}
  \item 
  For \(\phi_1,\phi_2 \in \F{w^t_{\pe}}\), we apply Lemma~\ref{lem basicEECED} and the Cauchy--Schwarz inequality to obtain
  \[
    \maxnorm{\CE{w^t}\,\phi_1\,\phi_2}
    \;\le\;
    \tCM\,\EE{w^t}\bigl[\bigl|\CE{w^t}\,\phi_1\,\phi_2\bigr|\bigr]
    \;\le\;
    \tCM\,\EE{w^t}\Bigl[
      \sqrt{\CE{w^t}\,\phi_1^2}\,\sqrt{\CE{w^t}\,\phi_2^2}
    \Bigr]
    \;\le\;
    \tCM\,\Unorm{\phi_1}{w^t}\,\Unorm{\phi_2}{w^t}.
  \]

  \item 
  For \(\phi_1,\phi_2 \in \T{B_t} \subseteq \TT{w}{t}=\TT{w^t}{0}\), we apply Lemma~\ref{lem basicDecayAu2} to obtain
  \[
    \maxnorm{\D{B_t}\,\phi_1\,\phi_2}
    \;\le\;
    \Dt\,\exp(-\eps\,\rk{w} \;-\;\eps\,t)\,
    \Unorm{\phi_1}{B_t}\,\Unorm{\phi_2}{B_t}.
  \]

  \item 
  For \(\phi_1,\phi_2 \in \T{B_{[t+1,k-1]}}\), applying Cauchy--Schwarz directly yields
  \[
    \maxnorm{\EE{B_{[t+1,k-1]}}\,\phi_1\,\phi_2}
    \;=\;
    \bigl|\EE{B_{[t+1,k-1]}}\,\phi_1\,\phi_2\bigr|
    \;\le\;
    \Unorm{\phi_1}{B_{[t+1,k-1]}}\,
    \Unorm{\phi_2}{B_{[t+1,k-1]}}.
  \]
\end{itemize}

By Lemma~\ref{lem: mainTensorProduct}, we conclude
\begin{align*}
  \maxnorm{\bL_t\,fg}
  \;\le\; &
  \tCM
  \;\cdot\;
  \Dt\,\exp\bigl(-\eps\,\rk{w} \;-\;\eps\,t\bigr)
  \;\Unorm{f}{\{w\}\cup B}\,\Unorm{g}{\{w\}\cup B} \\
  \;\le\; &
  (1+\kappa)\,\tCM
  \;\cdot\;
  \Dt\,\exp\bigl(-\eps\,\rk{w} \;-\;\eps\,t\bigr)
  \;\Unorm{f}{\rp}\,\Unorm{g}{\rp}.
\end{align*}
Combining this with \eqref{eq fugv self 00}, we obtain
\[
  \maxnorm{\D{\rp}\,\bL_t\,fg}
  \;\le\;
  \tC\,\tlam^{\,k-t-1}\,
  \Dt\,\exp\bigl(-\eps\,\rk{w} \;-\;\eps\,t\bigr)\,
  \Unorm{f}{\rp}\,\Unorm{g}{\rp}
  \quad
  \text{for some } \tC \in \CC.
\]

\step{Estimate \(\maxnorm{\D{\rp}(\CE{w}\otimes \EE{B}\,fg)}\)}

Because \(\CE{w}\otimes \EE{B}\,fg \,\in\, \F{w}\), we have
\begin{align*}
  \D{\rp}\,\bigl(\CE{w}\otimes \EE{B}\,fg\bigr)
  \;&=\;
  \CE{\rp}\,\bigl(\CE{w}\otimes \EE{B}\,fg\bigr)
  \;-\;
  \EE{\rp}\,\bigl(\CE{w}\otimes \EE{B}\,fg\bigr)
  \\[4pt]
  &=\;
  \CE{\rp}\,\bigl(\CE{w}\otimes \EE{B}\,fg\bigr)
  \;-\;
  \bigl(\EE{w}\otimes \EE{B}\,fg\bigr)
  \\[4pt]
  &=\;
  \CE{\rp}\,\bigl(\D{w} \otimes \EE{B}\,fg\bigr).
\end{align*}

Since \(\D{w} \otimes \EE{B}\,fg \,\in\, \Fz{w} \otimes \R \,=\, \Fz{w}\), 
Lemma~\ref{lem basicEECED} implies 
\begin{align}
  \label{eq fugv self 01}
  \maxnorm{\CE{\rp}\,\bigl(\D{w} \otimes \EE{B}\,fg\bigr)}
  \;\le\;
  \tCM\,\tlam^k\;\maxnorm{\D{w}\otimes \EE{B}\,fg}.
\end{align}
Next, by our assumption on \(\cal W\) and Lemma~\ref{lem: mainTensorProduct}, we obtain
\[
  \maxnorm{\D{w}\otimes \EE{B}\,fg}
  \;\le\;
  \delta\;\Unorm{f}{\{w\}\cup B}\,\Unorm{g}{\{w\}\cup B}
  \;\le\;
  (1+\kappa)\,\delta\;\Unorm{f}{\rp}\,\Unorm{g}{\rp},
\]
where we also used Lemma~\ref{lemma: coreNorm}. Combining this with \eqref{eq fugv self 01} yields
\[
  \maxnorm{\D{\rp}\,\bigl(\CE{w}\otimes \EE{B}\,fg\bigr)}
  \;\le\;
  \tC\,\tlam^k\,\delta\;\Unorm{f}{\rp}\,\Unorm{g}{\rp}.
\]

\step{Summation over \(t\)}

Summing over \(t\), we conclude
\begin{align*}
  \maxnorm{\D{\rp}\,fg}
  &\;\le\;
  \tC\;\Bigl(
    \delta\,\tlam^k
    \;+\;
    \Dt\,\exp\bigl(-\eps\,\rk{w}\bigr)\,
    \sum_{t=0}^{k-1} \tlam^{k-t-1}\,\exp\bigl(-\eps\,t\bigr)
  \Bigr)
  \\
  &\;\le\;
  \tC\;\Bigl(
    \delta\,\tlam^k
    \;+\;
    \Dt\,\exp\bigl(-\eps\,\rk{w} \;-\;\eps\,k\bigr)
  \Bigr)
  \;\Unorm{f}{\rp}\,\Unorm{g}{\rp},
\end{align*}
where we may increase \(\tC\) if necessary to absorb the constant arising from the sum.

    \step{Estimate $\maxnorm{\D{\rp}(\CE{w}\otimes \EE{B}fg)}$}
    Given that $\CE{w}\otimes \EE{B}fg \in \F{w}$, we have
    \begin{align*}
        \D{\rp} (\CE{w}\otimes \EE{B}fg)
        \le &
        \CE{\rp} (\CE{w}\otimes \EE{B}fg)
        -
        \EE{\rp} (\CE{w}\otimes \EE{B}fg) \\
        =   &
        \CE{\rp}  (\CE{w}\otimes \EE{B}fg)
        -
        (\EE{w}\otimes \EE{B} fg)         \\
        =   &
        \CE{\rp} (\D{w} \otimes \EE{B}fg)\,.
    \end{align*}

    Given that $\D{w} \otimes \EE{B}fg \in \Fz{w} \otimes \R = \Fz{w}$, by Lemma \ref{lem basicEECED}, we have
    \begin{align}
        \label{eq fugv self 02}
        \maxnorm{\CE{\rp} (\D{w}\otimes \EE{B}fg)}
        \le &
        \tCM \tlam^k \maxnorm{\D{w}\otimes \EE{B}fg}\,.
    \end{align}

    Now, we rely on our assumption on ${\cal W}$ together with Lemma \ref{lem: mainTensorProduct} to get
    \begin{align*}
        \maxnorm{\D{w}\otimes \EE{B}fg}
        \le
        \delta \Unorm{f}{\{w\}\cup B} \Unorm{g}{\{w\}\cup B}
        \le
        (1+\kappa) \delta \Unorm{f}{\rp} \Unorm{g}{\rp}\,,
    \end{align*}
    by invoking Lemma \ref{lemma: coreNorm}. Together with \eqref{eq fugv self 02}, we conclude that
    \begin{align*}
        \maxnorm{\D{\rp} (\CE{w}\otimes \EE{B}fg)}
        \le
        \tC \tlam^k \delta \Unorm{f}{\rp} \Unorm{g}{\rp}\,.
    \end{align*}

    \step{Summation over $t$}
    Finally, we sum over $t$ to get
    \begin{align*}
        \maxnorm{\D{\rp}fg}
        \le &
        \tC \left(
        \delta \tlam^k +
        \Dt \exp( -\eps \rk{w})
        \sum_{t=0}^{k-1} \tlam^{k-t-1} \exp(-\eps t)
        \right) \\
        \le &
        \tC \left( \delta \tlam^k + \Dt \exp( -\eps \rk{w} -\eps k) \right)  \Unorm{f}{\rp} \Unorm{g}{\rp}\,,
    \end{align*}
    where we increase $\tC$ if necessary to absorb the constant in the sum.
\end{proof}

\subsection{Proof of Proposition \ref{prop RuRv}}
\begin{proof}
    \step{Setup}
We first consider the case \(u \neq v\) and assume \(u \prec w\) and \(v \prec w\). 
Let \(w_1\) and \(w_2\) be the two children of \(w\) such that \(u \pe w_1\) and \(v \pe w_2\). 
Set
\[
  B \;:=\; 
  \bigl(\fc{w}\setminus\{w_1,w_2\}\bigr)\,\cup\,\OO{w}{\bigl[0,\rkr-\rk{w}-1\bigr]}.
\]
Since neither \(u\) nor \(v\) is \(\rp\), we may invoke the second statement of Assumption~\ref{Assume rp},  
together with the observation that 
$\{w_2\} \cup B = \OO{w_1}{\bigl[0,\rkr-\rk{w_1}-1\bigr]}$ and $\{w_1\} \cup B = \OO{w_2}{\bigl[0,\rkr-\rk{w_2}-1\bigr]}$, to obtain the decompositions
\[
  f_u \;=\; f_{u,{\cal R}} + f_{u,{\cal T}}
  \quad\text{and}\quad
  g_v \;=\; g_{v,{\cal R}} + g_{v,{\cal T}},
\]
with
\begin{align*}
  f_{u,{\cal R}}
  \;&\in\;
  \underbrace{
  \dR{u\,;\,\rk{w_1}-\rk{u}}}_{\subseteq\,\F{(w_1)_{\pe}}}\;
  \;\otimes\;\T{w_2}\;\otimes\;\T{B},
  \\[2pt]
  f_{u,{\cal T}}
  \;&\in\;
  \T{w_1}\;\otimes\;\T{w_2}\;\otimes\;\T{B},
  \\[2pt]
  g_{v,{\cal R}}
  \;&\in\;
  \T{w_1}\;\otimes\;\dR{v\,;\,\rk{w_2}-\rk{v}}\;\otimes\;\T{B},
  \\[2pt]
  g_{v,{\cal T}}
  \;&\in\;
  \T{w_1}\;\otimes\;\T{w_2}\;\otimes\;\T{B}.
\end{align*}
Furthermore, from the assumption and $\rk{w_1} = \rk{w_2} = \rk{w}-1$, we have 
\begin{align}
    \label{eq RuRv 00}
  \Unorm{f_{u,{\cal T}}}{\rp}
  \;\le\; &
  \dW{u}
  \;\Dt\;\exp\bigl(-\eps\,(\rk{w}-1)\bigr)\;\lame^{\rk{w}-\rk{u}-1}
  \;\Unorm{f_{u,{\cal R}}}{\rp}, \\
  \nonumber
  \Unorm{g_{v,{\cal T}}}{\rp}
  \;\le\; &
  \dW{v}
  \;\Dt\;\exp\bigl(-\eps\,(\rk{w}-1)\bigr)\;\lame^{\rk{w}-\rk{v}-1}
  \;\Unorm{g_{v,{\cal R}}}{\rp}.
\end{align}
By the definitions of \(\dW{u}\) and \(\dW{v}\), we know
\[
  \max\{\dW{u},\,\dW{v}\}
  \;\le\;
  \max\Bigl\{\tfrac{1}{\tCB\,R},\,\Delta\Bigr\},
\]
which is strictly less than \(1\) whenever \(\tCR\) is large enough, noting also that \(\tCB \ge 1\) by \eqref{def tCB2}. 
Thus, if \(\tCR\) is sufficiently large, 
\[
    \Unorm{f_u}{\rp}  \ge  \Unorm{f_{u,{\cal R}}}{\rp} - \Unorm{f_{u,{\cal T}}}{\rp} \stackrel{\eqref{eq RuRv 00}}{\ge}
    \frac{1}{2} \Unorm{f_{u,{\cal R}}}{\rp} \,.
    \]
Also, we have an analogous inequality for \(g_v\) as well. Here we conclude that
\begin{align}
    \label{eq RuRv 01}
\Unorm{f_{u,{\cal R}}}{\rp}\;\le\; 2\,\Unorm{f_u}{\rp}
\quad \text{and} \quad
  \Unorm{g_{v,{\cal R}}}{\rp}
  \;\le\; 2\,\Unorm{g_v}{\rp}.
\end{align}

\step{Invoke Lemma \ref{lem PT Correlation}}
We set \({\cal W}_1 = \dR{u\,;\,\rk{w_1}-\rk{u}}\) and \({\cal W}_2 = \dR{v\,;\,\rk{w_2}-\rk{v}}\). 
From (ii)(b) of Assumption~\ref{Assume rp}, set 
\[
  \delta_1 
  \;=\;
  \tC_{\ref{prop: coreR}}\;\dW{u}\;\lame^{\rk{w_1}-\rk{u}},
  \qquad
  \delta_2 
  \;=\;
  \tC_{\ref{prop: coreR}}\;\dW{v}\;\lame^{\rk{w_2}-\rk{v}}.
\]
Applying Lemma~\ref{lem PT Correlation}, we obtain
\begin{align*}
  \maxnorm{\D{\rp}\,f_{u,{\cal R}}\,g_{v,{\cal R}}}
  &\;\le\;
  \tC
  \;\dW{u}\;\lame^{\rk{w_1}-\rk{u}}
  \;\dW{v}\;\lame^{\rk{w_2}-\rk{v}}
  \;\exp\bigl(-\eps\,(\rkr-\rk{w})\bigr)
  \;\Unorm{f_{u,{\cal R}}}{\rp}
  \;\Unorm{g_{v,{\cal R}}}{\rp}
  \\[4pt]
  &\;\le\;
  \tC
  \;\dW{u}\;\lame^{\rk{w_1}-\rk{u}}
  \;\dW{v}\;\lame^{\rk{w_2}-\rk{v}}
  \;\exp\bigl(-\eps\,(\rkr-\rk{w})\bigr)
  \;\Unorm{f_{u}}{\rp}
  \;\Unorm{g_{v}}{\rp}
  \\[4pt]
  &\;\le\;
  \tC
  \;\dW{u}\;\lame^{\rk{w}-\rk{u}}
  \;\dW{v}\;\lame^{\rk{w}-\rk{v}}
  \;\exp\bigl(-\eps\,(\rkr-\rk{w})\bigr)
  \;\Unorm{f_{u}}{\rp}
  \;\Unorm{g_{v}}{\rp},
\end{align*}
where \(\tC \in \CC\), and we allow \(\tC\) to absorb constant factors from line to line. 

To bound \(\maxnorm{\D{\rp}\,f_{u,{\cal R}}\,g_{v,{\cal T}}}\), we again apply Lemma~\ref{lem PT Correlation}, this time taking 
\({\cal W}_1 = \dR{u\,;\,\rk{w_1}-\rk{u}}\) 
and 
\({\cal W}'_2 = \T{w_2}\). 
By Lemma~\ref{lem basicDecayAu2}, we may 
assume \(\tCR\) is sufficiently large so that \(\Dt \le 1\)
which allows us to choose \(\delta_2' = 2\).
Applying Lemma~\ref{lem PT Correlation} gives
\begin{align*}
  \maxnorm{\D{\rp}\,f_{u,{\cal R}}\,g_{v,{\cal T}}}
  &\;\le\;
  \tC
  \;\dW{u}\;\lame^{\rk{w}-\rk{u}}
  \;\exp\bigl(-\eps\,(\rkr-\rk{w})\bigr)
  \;\Unorm{f_{u,{\cal R}}}{\rp}
  \;\Unorm{g_{v,{\cal T}}}{\rp}
  \\[4pt]
  & \stackrel{\eqref{eq RuRv 00}}{\;\le\;}
  \tC
  \;\dW{u}\;\lame^{\rk{w}-\rk{u}}
  \;\dW{v}\;\lame^{\rk{w}-\rk{v}}
  \;\exp\bigl(-\eps\,(\rkr-\rk{w})\bigr)
  \;\Unorm{f_{u,{\cal R}}}{\rp}
  \;\Unorm{g_{v,{\cal R}}}{\rp}
  \\[4pt]
  &\stackrel{\eqref{eq RuRv 00},\eqref{eq RuRv 01}}{\;\le\;}
  \tC
  \;\dW{u}\;\lame^{\rk{w}-\rk{u}}
  \;\dW{v}\;\lame^{\rk{w}-\rk{v}}
  \;\exp\bigl(-\eps\,(\rkr-\rk{w})\bigr)
  \;\Unorm{f_{u}}{\rp}
  \;\Unorm{g_{v}}{\rp},
\end{align*}
where we again allow \(\tC\) to increase from line to line to absorb various constants. 
By symmetry, the same bound applies to \(\maxnorm{\D{\rp}\,f_{u,{\cal T}}\,g_{v,{\cal R}}}\).

Finally, for \(\maxnorm{\D{\rp}\,f_{u,{\cal T}}\,g_{v,{\cal T}}}\), we set \({\cal W}'_1 = \T{w_1}\), \({\cal W}'_2 = \T{w_2}\), and \(\delta_1' = \delta_2' = 2\).  
Using Lemma~\ref{lem PT Correlation} once more, along with the aforementioned norm bounds, we obtain the same tail bound. 
Summing all four terms, we conclude
\[
  \maxnorm{\D{\rp}\,f_u\,g_v}
  \;\le\;
  \tC
  \;\dW{u}\;\lame^{\rk{w}-\rk{u}}
  \;\dW{v}\;\lame^{\rk{w}-\rk{v}}
  \;\exp\bigl(-\eps\,(\rkr-\rk{w})\bigr)
  \;\Unorm{f_u}{\rp}
  \;\Unorm{g_v}{\rp}.
\]

    \step{Estimate of \(\bigl|\EE{\rp} f_u\,g_v\bigr|\)}

We can also repeat the above argument by invoking the first statement of Lemma~\ref{lem PT Correlation}. In particular, we obtain
\begin{align*}
  &\max \Bigl\{
      \maxnorm{\CE{\rp} f_{u,{\cal R}}\,g_{v,{\cal R}}},
      \maxnorm{\CE{\rp} f_{u,{\cal R}}\,g_{v,{\cal T}}},
      \maxnorm{\CE{\rp} f_{u,{\cal T}}\,g_{v,{\cal R}}},
      \maxnorm{\CE{\rp} f_{u,{\cal T}}\,g_{v,{\cal T}}}
  \Bigr\} \\
  \;\le\; &
  \tC\,\dW{u}\,\lame^{\rk{w}-\rk{u}}
  \;\dW{v}\;\lame^{\rk{w}-\rk{v}}
  \;\Unorm{f_u}{\rp}\,\Unorm{g_v}{\rp},
\end{align*}
which in turn implies
\[
  \bigl|\EE{\rp}\,f_u\,g_v\bigr|
  \;\le\;
  \maxnorm{\CE{\rp}\,f_u\,g_v}
  \;\le\;
  \tC\,\dW{u}\,\lame^{\rk{w}-\rk{u}}
  \;\dW{v}\;\lame^{\rk{w}-\rk{v}}
  \;\Unorm{f_u}{\rp}\,\Unorm{g_v}{\rp}.
\]

We have now completed the proof of the proposition for the case where \(u \neq v\) and none of \(u\) or \(v\) is \(w\).

\step{Case 2: \(u \neq v\) and \(w \in \{u,v\}\)}

Without loss of generality, assume \(v = w\). We still let \(w_1\) be the child of \(w\) such that \(u \pe w_1\), and let \(w_2\) be any other child of \(w\).

In this situation,
\[
  u \prec v \;\Longrightarrow\; v \in \V{k}
  \quad\text{for some }k\ge1.
\]
For \(v \in \V{k}\) with \(k \in [1,\rkr-1]\), we have \(\dR{v} = \cR{\TT{v}{-1}}{\rkr-\rk{v}}\), which implies
\[
  \dR{v}
  \;\subseteq\;
  \TT{v}{[-1,\rkr-\rk{v}]}
  \;=\;
  \T{w_1}\,\otimes\,\T{w_2}\,\otimes\,\T{B}
\]
from Lemma \ref{lem R Decompose}.
When \(v \in \V{\rkr}\) (i.e.\ \(v = \rp\)), we have
\[
  \dR{v}
  \;=\;
  \TT{\rp}{-1}
  \;=\;
  \T{w_1}\,\otimes\,\T{w_2}\,\otimes\,\T{B}.
\]
We decompose \(f_u = f_{u,{\cal R}} + f_{u,{\cal T}}\) as before while keeping \(g_v\) unchanged. Using Lemma~\ref{lem PT Correlation} with \({\cal W}_1 = \dR{u;0}\), \({\cal W}_1' = \T{w_1}\), and \({\cal W}'_2 = \T{w_2}\), we can derive the same bounds as in the previous case. This covers both statements of the proposition in the scenario \(u \neq v\) and \(w \in \{u,v\}\).

\step{Case 3: \(u = v = w\) and \(\rk{w}>0\)}

When \(u=v\), we use Lemma~\ref{lem fugv self} instead of Lemma~\ref{lem PT Correlation}. Let
\[
  B \;=\; \OO{w}{[-1,\,\rkr - \rk{w}-1]}
  \quad\text{and}\quad
  k \;=\;\rkr-\rk{w}.
\]
If \(\rk{w}>0\), then from Lemma \ref{lem R Decompose}, we have
\[
  f_u,\;g_v
  \;\in\;
  \dR{w}
  \;\subseteq\;
  \TT{w}{-1}
  \;\otimes\;
  \T{B}.
\]
Thus, we set \(\cW = \TT{w}{-1}\). Before applying our lemmas, we must estimate
\(\maxnorm{\D{w}\,\phi_1\,\phi_2}\) for \(\phi_1,\phi_2 \in \TT{w}{-1}\).

\smallskip

\noindent
\textbf{Claim.} For \(\phi_1,\phi_2 \in \TT{w}{-1}\),
\[
  \maxnorm{\D{w}\,\phi_1\,\phi_2}
  \;\le\;
  2\,\Dt\,\exp\bigl(-\eps\,\rk{w}\bigr)
  \,\Unorm{\phi_1}{w}
  \,\Unorm{\phi_2}{w}.
\]

\noindent
\emph{Sketch of the argument.} For \(\phi_1,\phi_2 \in \TT{w}{-1}\), we note
\[
  \D{w}\,\phi_1\,\phi_2
  \;=\;
  \CE{w}\,\phi_1\,\phi_2
  \;-\;
  \EE{w}\,\phi_1\,\phi_2,
\]
so
\[
  \maxnorm{\D{w}\,\phi_1\,\phi_2}
  \;\le\;
  \max_{x_w}
    \Bigl[\CE{w}\,\phi_1\,\phi_2\Bigr](x_w)
  \;-\;
  \min_{x_w'}
    \Bigl[\CE{w}\,\phi_1\,\phi_2\Bigr](x_w').
\]
We can view 
\(\CE{w}\,\phi_1\,\phi_2 = \CE{w}\,\CE{\OO{w}{-1}}\,\phi_1\,\phi_2\) as the conditional expectation of  a function, $\CE{\OO{w}{-1}}\,\phi_1\,\phi_2$,  in variables \(x_{\OO{w}{-1}}\), we see
\begin{align*}
  (\ast)
  \;\le\; &
  \max_{x_{\OO{w}{-1}}}
    \Bigl[\CE{w}\,\phi_1\,\phi_2\Bigr]
  \;-\;
  \min_{x_{\OO{w}{-1}}'}
    \Bigl[\CE{w}\,\phi_1\,\phi_2\Bigr] \\
\;=\; &
\max_{x_{\OO{w}{-1}}}
    \Bigl[\CE{w}\,\phi_1\,\phi_2\Bigr] - \EE{w}\,\phi_1\,\phi_2
  \;-\; \left(
  \min_{x_{\OO{w}{-1}}'}
    \Bigl[\CE{w}\,\phi_1\,\phi_2\Bigr]
- \EE{w}\,\phi_1\,\phi_2
\right)
  \;\le\;
  2\,\maxnorm{\D{\OO{w}{-1}}\,\phi_1\,\phi_2}.
\end{align*}
Hence, applying Lemma~\ref{lem basicDecayAu2} and then convert the norm using Lemma~\ref{lemma: coreNorm} yields 
\[
  (\ast)
  \;\le\;
  2\,\Dt\,\exp\bigl(-\eps\,\rk{w}\bigr)
  \,\Unorm{\phi_1}{w}
  \,\Unorm{\phi_2}{w}.
\]
This completes the proof of the claim.

\smallskip

Thus, we may now apply Lemma~\ref{lem fugv self} with 
\(\delta = 2\,\Dt\,\exp\bigl(-\eps\,\rk{w}\bigr)\), obtaining
\begin{align*}
  \maxnorm{\D{\rp}\,f_u\,g_v}
  &\;\le\;
  \tC\,\Dt
  \Bigl[\exp\bigl(-\eps\,\rk{w}\bigr)\,\tlam^{\rkr-\rk{w}}\,+
        \,\exp\bigl(-\eps\,\rkr\bigr)\Bigr]
  \,\Unorm{f_u}{\rp}\,\Unorm{g_v}{\rp}
  \\[4pt]
  &\;\le\;
  \tC\,\Dt\,\exp\bigl(-\eps\,\rkr\bigr)
  \,\Unorm{f_u}{\rp}\,\Unorm{g_v}{\rp},
\end{align*}
where the final inequality follows from \(\tlam < \exp(-\eps)\) and 
the constant \(\tC\) may increase from line to line to absorb various constants.

\step{Case 4: \(u = v = w\) and \(\rk{w} = 0\)}

Finally, consider \(w \in \V{0}\). 
Again, set 
\[
  B \;=\; \OO{w}{[-1,\,\rkr - \rk{w}-1]}
  \quad\text{and}\quad
  k \;=\;\rkr-\rk{w}.
\]

Here, we simply treat
\[
  f_u,\;g_v
  \;\in\;
  \dR{w}
  \;\subseteq\;
  \F{w_{\pe}}
  \;\otimes\;
  \T{B}.
\]
In this case, we can rely on Cauchy-Schwarz inequality to obtain
\begin{align*}
  \maxnorm{\D{w}\,\phi_1\,\phi_2}
  \;\le\;
  2\,\maxnorm{\CE{w}\,\phi_1\,\phi_2}
  \;\le\; &
  2\,\tCM
  \,\EE{w}\bigl|\CE{w}\,\phi_1\,\phi_2\bigr|\\
  & \;\le\;
  2\,\tCM
  \,\EE{w}\Bigl[\sqrt{\CE{w}\,\phi_1^2}\,\sqrt{\CE{w}\,\phi_2^2}\Bigr]
  \;\le\;
  2\,\tCM
  \,\Unorm{\phi_1}{w}
  \,\Unorm{\phi_2}{w},
\end{align*}
for all \(\phi_1,\phi_2 \in \F{w_{\pe}}\). Therefore, we may set \(\delta = 2\,\tCM\) and apply Lemma~\ref{lem fugv self}, yielding
\[
  \maxnorm{\D{\rp}\,f_u\,g_v}
  \;\le\;
  \tC\,
  \bigl[\tlam^{\rkr} + \Dt\,\exp\bigl(-\eps\,\rkr\bigr)\bigr]
  \;\Unorm{f_u}{\rp}\,\Unorm{g_v}{\rp}.
\]
\end{proof}

\bigskip

\section{Layerwise Analysis}
\label{sec: layerwise}

In this section, we establish the following proposition.
\begin{prop}\label{prop layerwise main}
    There exists a constant $\tC \in \CC$ such that if
    \[
        \tCB \;\ge\; \sqrt{\frac{2}{\kappa}\,\tC},
    \]
    then for sufficiently large $\tCR$, the following statement holds:
    Consider $\rp$ satisfying Assumption~\ref{Assume rp}. 
    For $t,r \in [0, \rkr]$, let
    \[
        f_u \in \dR{u} \quad \bigl(\forall u \in \V{t}\bigr),
        \quad\text{and}\quad
        g_v \in \dR{v} \quad \bigl(\forall v \in \V{r}\bigr).
    \]
    Define
    \[
        f_t = \sum_{u \in \V{t}} f_u,
        \quad\text{and}\quad
        g_r = \sum_{v \in \V{r}} g_v.
    \]
    Then,
    \begin{align*}
        \maxnorm{\D{\rp}\bigl(f_t \,g_r\bigr)}
        &\le
        \frac{\tC}{(\tCB)^2}
        \exp\bigl(-\eps \,\rkr\bigr)\,
        \exp\bigl(-1.1\,\eps\,|t-r| \;-\;\eps\,\min\{t,r\}\bigr)\,
        \Unorm{f_t}{\rp}\,\Unorm{g_r}{\rp} \\
        &\quad
        +\,{\bf 1}(r=t)\,\tC
        \Bigl(
            \tlam^{\rkr} \;+\; \Dt\,\exp\bigl(-\eps \,\rkr\bigr)
        \Bigr)\,\Unorm{f_t}{\rp}\,\Unorm{g_r}{\rp}\,.
    \end{align*}
    Furthermore, for $t \neq r$, we have
    \begin{align*}
        \bigl|\EE{\rp}\bigl(f_t\,g_r\bigr)\bigr|
        &\le
        \tC\,\frac{\Dt}{\tCB}\,
        \exp\bigl(-\eps\,(t+r)\bigr)\,
        \exp\bigl(-1.1\,\eps\,|t-r|\bigr)\,
        \Unorm{f_t}{\rp}\,\Unorm{g_r}{\rp}\,.
    \end{align*}
\end{prop}

The following lemma is a key technical result that will be used in the proof of Proposition~\ref{prop layerwise main}. In some sense, the statement of that proposition is a weaker version of this lemma, included for the sake of readability.

\begin{lemma}\label{lem: layerwise}
    There exists a constant $\tC \in \CC$ such that the following holds.

    Fix $\rp$ satisfying Assumption~\ref{Assume rp}. For $t,r \in [0, \rkr]$, consider
    a collection of functions
    \[
        f_u \in \dR{u} \quad \forall\,u \in \V{t}, 
        \quad\mbox{and}\quad 
        g_v \in \dR{v} \quad \forall\,v \in \V{r}.
    \]
    Then,
    \begin{align*}
        \sum_{\substack{u \in \V{t},\,v \in \V{r}\\u \neq v}}
        \maxnorm{\D{\rp}\bigl(f_u\,g_v\bigr)}
        &\;\le\;
        \tC\,\dW{t}\,\dW{r}\,R\,
        \bigl(d\,\lame^2\bigr)^{\tfrac{|t-r|}{2}}\,
        \exp\Bigl(-\eps\bigl(\rkr - \max\{t,r\}\bigr)\Bigr)
        \\
        &\qquad{}\cdot
        \sqrt{\sum_{u \in \V{t}} \Unorm{f_u}{\rp}^2 }
        \,\sqrt{\sum_{v \in \V{r}} \Unorm{g_v}{\rp}^2 },
    \end{align*}
    and
    \begin{align*}
        \sum_{\substack{u \in \V{t},\,v \in \V{r}\\u \neq v}}
        \bigl|\EE{\rp}\bigl(f_u\,g_v\bigr)\bigr|
        &\;\le\;
        \tC\,\dW{t}\,\dW{r}\,R\,
        \bigl(d\,\lame^2\bigr)^{\tfrac{|t-r|}{2}}
        \,\sqrt{\sum_{u \in \V{t}} \Unorm{f_u}{\rp}^2 }
        \,\sqrt{\sum_{v \in \V{r}} \Unorm{g_v}{\rp}^2 }.
    \end{align*}
\end{lemma}

\begin{rem}
    Note that the condition \(u \neq v\) is redundant when \(t \neq r\) because
    \(\V{t}\) and \(\V{r}\) are naturally disjoint in that case.
\end{rem}

Let us state a quick corollary of Lemma~\ref{lem: layerwise} that will be used in the proof of Proposition~\ref{prop layerwise main}.
\begin{cor}\label{cor layerwise norm}
    Suppose \(\tCB \ge \sqrt{\frac{2}{\kappa}\,\tC_{\ref{lem: layerwise}}}\). 
    Then the following holds for sufficiently large \(\tCR\): 
    Consdier \(\rp\) satisfying Assumption~\ref{Assume rp}. 
    For \(t \in [0, \rkr]\) and a collection of functions
    \[
        f_u \in \dR{u} \quad \forall\,u \in \V{t},
    \]
    we have
    \begin{align*}
        \frac{1}{(1+\kappa)} \sum_{u \in \V{t}} \EE{\rp}\!\bigl(f_u^2\bigr)
        \;\;\le\;\;
        \EE{\rp}\!\Bigl(\sum_{u \in \V{t}} f_u\Bigr)^{2}
        \;\;\le\;\;
        (1+\kappa)\,\sum_{u \in \V{t}} \EE{\rp}\!\bigl(f_u^2\bigr)\,.
    \end{align*}
\end{cor}

\begin{proof}
    First, if \(t = \rkr\), there is only a single term since \(\V{\rkr} = \{\rp\}\), so there is nothing to prove. 
    
    Now, assume \(t \in [0, \rkr - 1]\). Consider the difference
    \begin{align*}
        \Bigl|\,
            \EE{\rp}\!\Bigl(\sum_{u \in \V{t}} f_u\Bigr)^{2}
            \;-\;
            \sum_{u \in \V{t}} \EE{\rp}\!\bigl(f_u^2\bigr)
        \Bigr|
        &= 
        \Bigl|\,
            \sum_{\substack{u,v \in \V{t}\\u \neq v}}
            \EE{\rp}\!\bigl(f_u\,f_v\bigr)
        \Bigr|
        \;\le\;
        \sum_{\substack{u,v \in \V{t}\\u \neq v}}
        \Bigl|\EE{\rp}\!\bigl(f_u\,f_v\bigr)\Bigr|.
    \end{align*}
    By applying Lemma~\ref{lem: layerwise} with \(r = t\) and \(f_u = g_u\), we obtain
    \begin{align*}
        \sum_{\substack{u,v \in \V{t}\\u \neq v}}
        \Bigl|\EE{\rp}\!\bigl(f_u\,f_v\bigr)\Bigr|
        \;\le\;
        \tC_{\ref{lem: layerwise}} \,\dW{t}^{2}\,R
        \sum_{u \in \V{t}} \EE{\rp}\!\bigl(f_u^2\bigr).
        \tag*{$(*)$}
    \end{align*}
    Hence, if 
    \[
        \tC_{\ref{lem: layerwise}}\;\dW{t}^{2}\,R \;\;\le\;\; \tfrac{1}{2}\,\kappa,
    \]
    then
    \[
        1 + \tC_{\ref{lem: layerwise}}\;\dW{t}^{2}\,R \;\;\le\;\; 1 + \kappa,
        \quad\text{and}\quad
        1 - \tC_{\ref{lem: layerwise}}\;\dW{t}^{2}\,R \;\;\ge\;\; 1 - \tfrac{1}{2}\,\kappa 
        \;\;\ge\;\; \tfrac{1}{\,1+\kappa\,}.
    \]
    From these inequalities, the lemma follows.

    Recall from \eqref{def  dW} that
    \[
        \dW{t} \;=\;
        \begin{cases}
            \displaystyle \frac{1}{\tCB\,R}, & \text{if } t=0,\\[6pt]
            \displaystyle \Dt\,\exp\bigl(-\eps\,t\bigr), & \text{if } t>0.
        \end{cases}
    \]
    For \(t \ge 1\), we have \(\dW{t} \le \Dt\). Since 
    \(\Dt = \tC'\,\exp\bigl(-\eps\,\hB\bigr)\) for some prefixed $\tC' \in \CC$ and \(\hB \ge \tCR\,\bigl(\log(R)+1\bigr)\),
    it follows that 
    \[
        \tC_{\ref{lem: layerwise}}\;\dW{t}^{2}\,R \;\;\le\;\; \tfrac{1}{2}\,\kappa
        \quad
        \text{when \(\tCR\) is sufficiently large.}
    \]
    When \(t = 0\), we have
    \[
        \tC_{\ref{lem: layerwise}}\;\dW{0}^{2}\,R
        \;=\;
        \frac{\tC_{\ref{lem: layerwise}}}{\tCB}
        \;\le\;\;
        \tfrac{1}{2}\,\kappa
        \quad
        \text{provided}
        \quad
        \tCB \;\ge\; \sqrt{\frac{2}{\kappa}\,\tC_{\ref{lem: layerwise}}}.
    \]
    Therefore, by choosing 
    \[
        \tCB \;\ge\; \sqrt{\frac{2}{\kappa}\,\tC_{\ref{lem: layerwise}}},
    \]
    we ensure that the above condition holds for all \(t \in [0, \rkr]\), completing 
    the proof.
\end{proof}

\begin{proof}[ Proof of Lemma \ref{lem: layerwise}]
    \noindent
    Without loss of generality, assume \(t \ge r\). For simplicity, let 
    \[
      V_t = \V{t}
      \quad\text{and}\quad
      V_r = \V{r}.
    \]
    For \(u,v \in \V{}\), let \(\rho(u,v)\) denote the nearest common ancestor of \(u\) and \(v\).  
    We can express the sum as
    \begin{align*}
      \sum_{\substack{u \in V_t,\,v \in V_r\\u \neq v}}
      \maxnorm{\D{\rp}\bigl(f_u\,g_v\bigr)}
      &=
      \sum_{s = t}^{\rkr}
      \sum_{w \in \V{s}}
      \sum_{\substack{u \in V_t,\,v \in V_r\\\rho(u,v) = w\\u \neq v}}
      \maxnorm{\D{\rp}\bigl(f_u\,g_v\bigr)}.
    \end{align*}
    
    \step{Invoke Proposition \ref{prop RuRv}}
    By invoking \eqref{eq RuRv D} from Proposition~\ref{prop RuRv}, we obtain
    \begin{align*}
      (*) 
      &= 
      \sum_{s = t}^{\rkr}
      \sum_{w \in \V{s}}
      \sum_{\substack{u \in V_t,\,v \in V_r\\\rho(u,v) = w\\u \neq v}}
      \tC_{\ref{prop RuRv}}\,
      \dW{t}\,\lame^{\,s-t}\,\dW{r}\,\lame^{\,s-r}\,
      \exp\bigl(-\eps\,(\rkr - s)\bigr)\,
      \Unorm{f_u}{\rp}\,\Unorm{g_v}{\rp}.
    \end{align*}
    Since all summands are non-negative, we can relax the conditions \(\rho(u,v) = w\) and \(u \neq v\) to \(u,v \pe w\), obtaining
    \begin{align}
      \nonumber
      (*)
      &\le
      \sum_{s = t}^{\rkr}
      \sum_{w \in \V{s}}
      \sum_{\substack{u \in V_t,\,v \in V_r\\u,v \pe w}}
      \tC_{\ref{prop RuRv}}\,
      \dW{t}\,\lame^{\,s-t}\,\dW{r}\,\lame^{\,s-r}\,
      \exp\bigl(-\eps\,(\rkr - s)\bigr)\,
      \Unorm{f_u}{\rp}\,\Unorm{g_v}{\rp}
      \\
      \label{eq layerwise 01}
      &\le
      \tC_{\ref{prop RuRv}}\,
      \dW{t}\,\dW{r}
      \sum_{s = t}^{\rkr}
      \lame^{\,s-t}\,\lame^{\,s-r}\,
      \exp\bigl(-\eps\,(\rkr - s)\bigr)\,
      \sum_{\substack{w \in \V{s}\\u,v \pe w}}
      \Unorm{f_u}{\rp}\,\Unorm{g_v}{\rp}.
    \end{align}
    
    \step{Bounding the sum over \(u \in V_t, v \in V_r\) such that \(u,v \pe w\)}
    For each \(w \in V_s\) in the above summation, 
    \(\{u \in V_t : u \pe w\}\) is the \((s-t)\)-th descendant set of \(w\), and has size at most \(R\,d^{s-t}\) by our tree assumption. 
    Similarly, 
    \(\{v \in V_r : v \pe w\}\) is the \((s-r)\)-th descendant set of \(w\), also with size at most \(R\,d^{s-r}\).
    
    Recall a special case of the Cauchy--Schwarz inequality for a sum of real numbers \(\{a_i\}_{i \in I}\):
    \[
      \biggl(\sum_{i \in I} a_i\biggr)^{2}
      \;=\;
      \biggl(\sum_{i \in I} 1 \cdot a_i\biggr)^2
      \;\le\;
      |I| \sum_{i \in I} a_i^{2}.
    \]
    Applying this we get
    \begin{align*}
      \sum_{\substack{u \in V_t,\,v \in V_r\\u,v \pe w}}
      \Unorm{f_u}{\rp}\,\Unorm{g_v}{\rp}
      &= 
      \Biggl(\,\sum_{\substack{u \in V_t\\u \pe w}} \Unorm{f_u}{\rp}\Biggr)
      \Biggl(\,\sum_{\substack{v \in V_r\\v \pe w}} \Unorm{g_v}{\rp}\Biggr)
      \\
      &\le
      \sqrt{\,R\,d^{s-t}}\,
      \sqrt{\sum_{\substack{u \in V_t\\u \pe w}} \Unorm{f_u}{\rp}^{2}}
      \;\cdot \;
      \sqrt{\,R\,d^{s-r}}\,
      \sqrt{\sum_{\substack{v \in V_r\\v \pe w}} \Unorm{g_v}{\rp}^{2}}.
    \end{align*}
    Thus, the sum over \(w\) in \eqref{eq layerwise 01} is bounded by
    \begin{align*}
      \sum_{w \in \V{s}}
      \sum_{\substack{u \in V_t,\,v \in V_r\\u,v \pe w}}
      \Unorm{f_u}{\rp}\,\Unorm{g_v}{\rp}
      &\le
      R\,d^{\,s-t}\,(\sqrt{d}\,)^{\,t-r}
      \sum_{w \in \V{s}}
      \sqrt{\sum_{\substack{u \in V_t\\u \pe w}} \Unorm{f_u}{\rp}^{2}}
      \,
      \sqrt{\sum_{\substack{v \in V_r\\v \pe w}} \Unorm{g_v}{\rp}^{2}}
      \\
      &\le
      R\,d^{\,s-t}\,(\sqrt{d}\,)^{\,t-r}
      \sqrt{\sum_{u \in V_t} \Unorm{f_u}{\rp}^{2}}
      \,
      \sqrt{\sum_{v \in V_r} \Unorm{g_v}{\rp}^{2}},
    \end{align*}
    where the last inequality again follows from Cauchy--Schwarz. Substituting this back into \eqref{eq layerwise 01} yields
    \begin{align}
      \label{eq layerwise 02}
      \sum_{\substack{u \in V_t,\,v \in V_r\\u \neq v}}
      \maxnorm{\D{\rp}\bigl(f_u\,g_v\bigr)}
      \;\le\;&
      \tC_{\ref{prop RuRv}}\,
      \dW{t}\,\dW{r}\,R\,
      \bigl(d\,\lame^{2}\bigr)^{\tfrac{t-r}{2}}\,
      \left(
        \sum_{s = t}^{\rkr}
        \bigl(d\,\lame^{2}\bigr)^{s-t}\,
        \exp\!\bigl(-\eps\,(\rkr - s)\bigr)
      \right) \cdot  \\
      \nonumber
      &\; \cdot 
      \sqrt{\sum_{u \in V_t} \Unorm{f_u}{\rp}^{2}}
      \sqrt{\sum_{v \in V_r} \Unorm{g_v}{\rp}^{2}}.
    \end{align}
    
    We pause to note that the same argument applies to bound
    \(\sum_{\substack{u \in V_t,\,v \in V_r\\u \neq v}} \bigl|\EE{\rp}\bigl(f_u\,g_v\bigr)\bigr|\),
    the only difference being that we invoke \eqref{eq RuRv E} from Proposition~\ref{prop RuRv}
    instead of \eqref{eq RuRv D}. The result matches the above expression but omits the factor 
    \(\exp\bigl(-\eps\,(\rkr - s)\bigr)\):
    \begin{align}
      \label{eq layerwise 03}
      \sum_{\substack{u \in V_t,\,v \in V_r\\u \neq v}}
      \bigl|\EE{\rp}\bigl(f_u\,g_v\bigr)\bigr|
      &\le
      \tC_{\ref{prop RuRv}}\,
      \dW{t}\,\dW{r}\,R\,
      \bigl(d\,\lame^{2}\bigr)^{\tfrac{t-r}{2}}
      \left(
        \sum_{s = t}^{\rkr}
        \bigl(d\,\lame^{2}\bigr)^{s-t}
      \right)
      \sqrt{\sum_{u \in V_t} \Unorm{f_u}{\rp}^{2}}
      \sqrt{\sum_{v \in V_r} \Unorm{g_v}{\rp}^{2}}.
    \end{align}

    \step{Bounding the sum of \((d\,\lame^{2})^{\,s-t}\,\exp\bigl(-\eps\,(\rkr - s)\bigr)\)}
From the definitions of \(\eps\) and \(\lame\) in Definition~\ref{def varepsilon}, we have
\[
  d\,\lame^{2}\,\exp(\eps)
  \;\le\;
  \exp(- 2 \cdot 1.1  + \eps) 
  \;=\;
  \exp\bigl(-1.2\,\eps\bigr).
\]
Thus, 
\[
  \sum_{s = t}^{\rkr}
  \bigl(d\,\lame^{2}\bigr)^{\,s-t}
  \exp\bigl(-\eps\,(\rkr - s)\bigr)
  \;\le\;
    \sum_{ n = 0}^\infty  \exp\bigl(-1.2\,\eps\,n\bigr) \cdot \exp\bigl(-\eps\,(\rkr - t)\bigr)
  \;\le\;
  \tC\,\exp\bigl(-\eps\,(\rkr - t)\bigr)
\]
for some \(\tC \in \CC\). Substituting this into \eqref{eq layerwise 02} yields
\begin{align*}
  \sum_{\substack{u \in V_t,\,v \in V_r\\u \neq v}}
  \maxnorm{\D{\rp}\bigl(f_u\,g_v\bigr)}
  &\le
  \tC\,
  \dW{t}\,\dW{r}\,R\,
  \bigl(d\,\lame^{2}\bigr)^{\tfrac{t-r}{2}}\,
  \exp\bigl(-\eps\,(\rkr - t)\bigr)\,
  \sqrt{\sum_{u \in V_t} \Unorm{f_u}{\rp}^{2}}
  \,\sqrt{\sum_{v \in V_r} \Unorm{g_v}{\rp}^{2}},
\end{align*}
where \(\tC\) has been increased if necessary to absorb any constant factors.

\step{Bounding the sum of \((d\,\lame^{2})^{\,s-t}\)}
Since \(d\,\lame^{2} < \exp\bigl(-2\,\eps\bigr)\in \CC\) from Definition \ref{def varepsilon}, it follows that
\[
  \sum_{s = t}^{\rkr}
  \bigl(d\,\lame^{2}\bigr)^{\,s-t}
  \;\le\;
  \tC
\]
for some \(\tC \in \CC\). Substituting this into \eqref{eq layerwise 02} gives
\begin{align*}
  \sum_{\substack{u \in V_t,\,v \in V_r\\u \neq v}}
  \bigl|\EE{\rp}\bigl(f_u\,g_v\bigr)\bigr|
  &\le
  \tC\,
  \dW{t}\,\dW{r}\,R\,
  \bigl(d\,\lame^{2}\bigr)^{\tfrac{t-r}{2}}\,
  \sqrt{\sum_{u \in V_t} \Unorm{f_u}{\rp}^{2}}
  \,\sqrt{\sum_{v \in V_r} \Unorm{g_v}{\rp}^{2}}.
\end{align*}

\end{proof}

\subsection{Proof of Proposition \ref{prop layerwise main}}
\begin{proof}
    \step{Case 1: \(t \neq r\) and neither is zero}

    Assume \(t \neq r\) and \(t, r \neq 0\). From Definition \ref{def:Dt-and-dWu}, we have
    \[
      \dW{t} \;=\; \Dt\,\exp\bigl(-\eps\,t\bigr)
      \quad\text{and}\quad
      \dW{r} \;=\; \Dt\,\exp\bigl(-\eps\,r\bigr).
    \]
    First, observe that
    \begin{align*}
      \maxnorm{\D{\rp}\bigl(f_t\,g_r\bigr)}
      &\;\le\;
      \sum_{u \in \V{t}}\,\sum_{v \in \V{r}}
      \maxnorm{\D{\rp}\bigl(f_u\,g_v\bigr)}
      \;=\;
      \sum_{\substack{u \in \V{t},\,v \in \V{r}\\u \neq v}}
      \maxnorm{\D{\rp}\bigl(f_u\,g_v\bigr)},
    \end{align*}
    where the equality follows from the fact that \(u\) and \(v\) lie in distinct layers of the subtree \(\T{\rp}\).

    \noindent
    Applying Lemma~\ref{lem: layerwise} and Corollary~\ref{cor layerwise norm} gives
    \begin{align*}
      (*)
      &\;\le\;
      (1 + \kappa)\,\tC_{\ref{lem: layerwise}}\,R\,
      \Dt^{2}\,
      \exp\bigl(-\eps\,t\bigr)\,
      \exp\bigl(-\eps\,r\bigr)\,
      \bigl(d\,\lame^{2}\bigr)^{\tfrac{|t-r|}{2}}\,
      \exp\!\Bigl(-\eps\bigl(\rkr - \max\{t,r\}\bigr)\Bigr)\,
      \Unorm{f_t}{\rp}\,\Unorm{g_r}{\rp}.
    \end{align*}
    Recalling from Definition~\ref{def varepsilon} that \(\bigl(d\,\lame^{2}\bigr)^{1/2} \le \exp\bigl(-1.1\,\eps\bigr)\), 
    together with $$\exp(- \eps t) \exp( - \eps r) =  \exp(-\eps\,\max\{t,r\})\exp(-\eps\,\min\{t,r\})\,$$ 
   we simplify:
    \begin{align}
    \nonumber
      (*)
      &\;\le\;
      (1 + \kappa)\,\tC_{\ref{lem: layerwise}}\,R\,
      \Dt^{2}\,
      \exp\bigl(-\eps\,\max\{t,r\}\bigr)\,
      \exp\bigl(-\eps\,\min\{t,r\}\bigr)\,
      \exp\bigl(-1.1\,\eps\,|t-r|\bigr)
      \\ \nonumber
      &\qquad{}\times
      \exp\!\Bigl(-\eps\,\bigl(\rkr - \max\{t,r\}\bigr)\Bigr)\,
      \Unorm{f_t}{\rp}\,\Unorm{g_r}{\rp}
      \\ 
      \label{eq layerwise 04}
      &=\;
      (1 + \kappa)\,\tC_{\ref{lem: layerwise}}\,R\,
      \Dt^{2}\,
      \exp\bigl(-\eps\,\rkr\bigr)\,
      \exp\!\Bigl(-1.1\,\eps\,|t-r|\;-\;\eps\,\min\{t,r\}\Bigr)\,
      \Unorm{f_t}{\rp}\,\Unorm{g_r}{\rp}.
    \end{align}
    
    Next, since we can choose \(\tCR\) large enough after fixing \(\tCB\), recall
    \[
      \Dt \;\le\; \tC'\,R\,
      \exp\bigl(-\tCR\,\bigl(\log(R)+1\bigr)\bigr)
    \]
    for some prefixed $\tC' \in \CC$, 
    so for sufficiently large \(\tCR\),
    \[
      R\,\Delta
      \;\le\;
      \frac{1}{\tCB}.
    \]
    Under this assumption, the above bound is even stronger than the claim in the proposition. This condition will be used repeatedly throughout the proof.

    \noindent
    A similar argument bounds \(\EE{\rp}\bigl(f_t\,g_r\bigr)\), except that \(\exp\!\bigl(-\eps\,(\rkr - \max\{t,r\})\bigr)\) does not appear:
    \begin{align*}
      \bigl|\EE{\rp}\bigl(f_t\,g_r\bigr)\bigr|
      &\;\le\;
      (1 + \kappa)\,\tC_{\ref{lem: layerwise}}\,R\,
      \Dt^{2}\,
      \exp\Bigl(-\eps\,(t + r)\Bigr)\,
      \exp\bigl(-1.1\,\eps\,|t-r|\bigr)\,
      \Unorm{f_t}{\rp}\,\Unorm{g_r}{\rp}.
    \end{align*}

    \step{Case 2: \(t \neq r\) and \(\min\{t,r\} = 0\)}
Assume without loss of generality that \(0 = t \le r\). Then
\[
  \dW{0} \;=\; \frac{1}{\tCB\,R}.
\]
The difference of this case from Case~1 is precisely the change of the value $\dW{t}$.  
Using Lemma~\ref{lem: layerwise} and Corollary~\ref{cor layerwise norm}, we obtain
\begin{align*}
  \maxnorm{\D{\rp}\bigl(f_t\,g_r\bigr)}
  &\;\le\;
  (1 + \kappa)\,\tC_{\ref{lem: layerwise}}\,
  \frac{\Dt}{\tCB}\,
  \exp\bigl(-\eps\,r\bigr)\,
  \bigl(d\,\lame^{2}\bigr)^{\tfrac{|t-r|}{2}}\,
  \exp\!\Bigl(-\eps\,\bigl(\rkr - \max\{t,r\}\bigr)\Bigr)\,
  \Unorm{f_t}{\rp}\,\Unorm{g_r}{\rp}
  \\
  &\;\le\;
  (1 + \kappa)\,\tC_{\ref{lem: layerwise}}\,
  \frac{\Dt}{\tCB}\,
  \exp\bigl(-\eps\,\rkr\bigr)\,
  \exp\!\Bigl(-1.1\,\eps\,|t-r|\;-\;\eps\,\min\{t,r\}\Bigr)\,
  \Unorm{f_t}{\rp}\,\Unorm{g_r}{\rp}.
\end{align*}
For the expectation term, the absence of the factor 
\(\exp\bigl(-\eps\,(\rkr - \max\{t,r\})\bigr)\) yields
\begin{align*}
  \bigl|\EE{\rp}\bigl(f_t\,g_r\bigr)\bigr|
  &\;\le\;
  (1 + \kappa)\,\tC_{\ref{lem: layerwise}}\,
  \frac{\Dt}{\tCB}\,
  \exp\Bigl(-\eps\,(t + r)\Bigr)\,
  \exp\bigl(-1.1\,\eps\,|t-r|\bigr)\,
  \Unorm{f_t}{\rp}\,\Unorm{g_r}{\rp}.
\end{align*}

\step{Case 3: \(t = r\) and \(t \neq 0\)}
Since \(t = r\neq 0\), the estimate from Case~1 still applies to the sum over distinct indices:
\begin{align*}
  &\sum_{\substack{u \in \V{t},\,v \in \V{r}\\u \neq v}}
  \maxnorm{\D{\rp}\bigl(f_u\,g_v\bigr)} \\
  \;\le\;&
  (1 + \kappa)\,\tC_{\ref{lem: layerwise}}\,R\,
  \Dt^{2}\,
  \exp\bigl(-\eps\,\rkr\bigr)\,
  \exp\!\Bigl(-1.1\,\eps\,|t-r|\;-\;\eps\,\min\{t,r\}\Bigr)\,
  \Unorm{f_t}{\rp}\,\Unorm{g_r}{\rp}.
\end{align*}
The only additional piece is the diagonal terms:
\[
  \maxnorm{\D{\rp}\bigl(f_t\,g_r\bigr)}
  \;\le\;
  \sum_{\substack{u \in \V{t},\,v \in \V{r}\\u \neq v}}
  \maxnorm{\D{\rp}\bigl(f_u\,g_v\bigr)}
  \;+\;
  \sum_{u \in \V{t}}
  \maxnorm{\D{\rp}\bigl(f_u\,g_u\bigr)}.
\]
By Proposition~\ref{prop RuRv} (last statement), we have
\begin{align*}
  \sum_{u \in \V{t}}
  \maxnorm{\D{\rp}\bigl(f_u\,g_u\bigr)}
  &\;\le\;
  \sum_{u \in \V{t}}
  \tC_{\ref{prop RuRv}}\,
  \Dt\,
  \exp\bigl(-\eps\,\rkr\bigr)\,
  \Unorm{f_u}{\rp}\,\Unorm{g_u}{\rp}
  \\
  &\;\le\;
  \tC_{\ref{prop RuRv}}\,
  \Dt\,
  \exp\bigl(-\eps\,\rkr\bigr)\,
  \sqrt{\sum_{u \in \V{t}} \Unorm{f_u}{\rp}^{2}}
  \,\sqrt{\sum_{u \in \V{t}} \Unorm{g_u}{\rp}^{2}}
  \\
  &\;\le\;
  (1 + \kappa)\,\tC_{\ref{prop RuRv}}\,
  \Dt\,
  \exp\bigl(-\eps\,\rkr\bigr)\,
  \Unorm{f_t}{\rp}\,\Unorm{g_t}{\rp},
\end{align*}
where the final step follows from the Cauchy--Schwarz inequality. Combining both estimates completes the proof for this case.

\step{Case 4: \(t = r = 0\)}
In this final case, we have \(\dW{0} = \tfrac{1}{\tCB\,R}\). Then, 
applying Lemma \ref{lem: layerwise} and Corollary \ref{cor layerwise norm} gives
\begin{align*}
  \sum_{\substack{u \in \V{0},\,v \in \V{0}\\u \neq v}}
  \maxnorm{\D{\rp}\bigl(f_u\,g_v\bigr)}
  &\;\le\;
  (1 + \kappa)\,\tC_{\ref{lem: layerwise}}\,
  \frac{1}{(\tCB)^{2}\,R}\,
  \exp\bigl(-\eps\,\rkr\bigr)\,
  \Unorm{f_0}{\rp}\,\Unorm{g_0}{\rp}.
\end{align*}

\noindent
For the diagonal term, we again use the last statement of 
Proposition~\ref{prop RuRv}, along with the Cauchy--Schwarz inequality, and Corollary \ref{cor layerwise norm} to get:
\begin{align*}
  \sum_{u \in \V{0}}
  \maxnorm{\D{\rp}\bigl(f_u\,g_u\bigr)}
  &\;\le\;
  (1 + \kappa)\,\tC_{\ref{prop RuRv}}
  \Bigl(
    \tlam^{\rkr}
    \;+\;
    \Dt\,\exp\bigl(-\eps\,\rkr\bigr)
  \Bigr)\,
  \Unorm{f_0}{\rp}\,\Unorm{g_0}{\rp}.
\end{align*}
\end{proof}

\section{Proof of Theorem \ref{theor main}}
\label{sec: proof main}

As a first step, we establish the following norm-comparison lemma.

\begin{lemma}\label{lem f norm}
    Suppose 
    \[
      \tCB \;\ge\; \sqrt{\frac{2}{\kappa}\,\tC_{\ref{lem: layerwise}}},
    \]
    and let \(\tCR\) be sufficiently large. Then, for any \(\rp\) satisfying Assumption~\ref{Assume rp} and any \(f \in \mathcal{T}_{K+1}(\rp)\), the following holds.

    By Lemma~\ref{prop PT 2K+1}, \(f\) can be expressed as
    \[
      f 
      \;=\; 
      \sum_{u \in \V{}} f_u,
      \quad
      \text{where } f_u \in \dR{u}.
    \]
    Next, group the terms of \(f\) according to the layers of the subtree \(\T{\rp}\). For \(t \in [0, \rkr]\), define
    \[
      f_t
      \;=\;
      \sum_{u \in \V{t}} f_u.
    \]
    Then,
    \[
      \frac{1}{\,1+\kappa\,}
      \sum_{t \in [0, \rkr]} 
        \EE{\rp}\!\bigl(f_t^2\bigr)
      \;\;\le\;\;
      \EE{\rp}\!\bigl(f^2\bigr)
      \;\;\le\;\;
      (1+\kappa)
      \sum_{t \in [0, \rkr]}
        \EE{\rp}\!\bigl(f_t^2\bigr).
    \]
\end{lemma}

As a corollary of Lemma~\ref{lem f norm} and Corollary~\ref{cor layerwise norm}, we obtain the following.
\begin{cor}\label{cor f norm}
    With the same setup as in Lemma~\ref{lem f norm}, the function $f = \sum_{u \in \V{}} f_u$ satisfies 
    \[
      \frac{1}{\,(1+\kappa)^{2}\,}
      \sum_{u \in \V{}} 
        \EE{\rp}\!\bigl(f_u^2\bigr)
      \;\;\le\;\;
      \EE{\rp}\!\bigl(f^2\bigr)
      \;\;\le\;\;
      (1+\kappa)^{2}
      \sum_{u \in \V{}} 
        \EE{\rp}\!\bigl(f_u^2\bigr).
    \]
\end{cor}

  \begin{proof}[Proof of Lemma~\ref{lem f norm}]
    Consider the difference
    \begin{align*}
      \Bigl|\EE{\rp}\bigl(f^{2}\bigr)
        \;-\;
        \sum_{t \in [0, \rkr]} 
          \EE{\rp}\bigl(f_{t}^{2}\bigr)
      \Bigr|
      \;=\;
      \Bigl|\,
        \sum_{\substack{t,r \in [0, \rkr]\\t \neq r}} 
          \EE{\rp}\bigl(f_{t}\,f_{r}\bigr)
      \Bigr|
      \;\le\;
      \sum_{\substack{t,r \in [0, \rkr]\\t \neq r}}
        \Bigl|\EE{\rp}\bigl(f_{t}\,f_{r}\bigr)\Bigr|.
    \end{align*}
    By invoking Proposition~\ref{prop layerwise main}, we obtain
    \begin{align}
      \label{eq f norm 01}
      (*)
      \;\le\;
      \tC_{\ref{prop layerwise main}}\,
      \frac{\Dt}{\tCB}
      \sum_{\substack{t,r \in [0, \rkr]\\t \neq r}}
        \exp\!\bigl(-\eps\,(t + r)\bigr)\,
        \exp\!\bigl(-1.1\,\eps\,|t - r|\bigr)\,
        \Unorm{f_{t}}{\rp}\,\Unorm{f_{r}}{\rp}.
    \end{align}
    Next, applying the inequality 
    \(|ab| \le \frac{1}{2}\bigl(a^{2} + b^{2}\bigr)\), 
    we get
    \begin{align*}
      (*)
      &\;\le\;
      \tC_{\ref{prop layerwise main}}\,
      \frac{\Dt}{\tCB}
      \sum_{\substack{t,r \in [0, \rkr]\\t \neq r}}
        \exp\!\bigl(-\eps\,(t + r)\bigr)\,
        \exp\!\bigl(-1.1\,\eps\,|t - r|\bigr)\,
        \frac{\Unorm{f_{t}}{\rp}^{2} + \Unorm{f_{r}}{\rp}^{2}}{2}
      \\
      &=\;
      \tC_{\ref{prop layerwise main}}\,
      \frac{\Dt}{\tCB}
      \sum_{t \in [0, \rkr]}
      \Unorm{f_{t}}{\rp}^{2}
      \sum_{\substack{r \in [0, \rkr]\\r \neq t}}
        \exp\!\bigl(-\eps\,(t + r)\bigr)\,
        \exp\!\bigl(-1.1\,\eps\,|t - r|\bigr).
    \end{align*}

    \noindent
    We now bound the inner sum by a rough estimate:
    \begin{align*}
      \sum_{\substack{r \in [0, \rkr]\\r \neq t}}
        \exp\!\bigl(-\eps\,(t + r)\bigr)\,
        \exp\!\bigl(-1.1\,\eps\,|t - r|\bigr)
      &\;\le\;
      \sum_{\substack{r \in [0, \rkr]\\r \neq t}}
        \exp\!\bigl(-\eps\,(t + r)\bigr)\,
        \exp\!\bigl(-\eps\,|t - r|\bigr)
      \\
      &\;=\;
      \sum_{\substack{r \in [0, \rkr]\\r \neq t}}
        \exp\!\bigl(-2\,\eps\,\max\{t,r\}\bigr)
      \\
      &\;\le\;
      (t+1)\,\exp(-2\,\eps\,t)
      \;+\;
      \sum_{r = t+1}^{\infty}
        \exp\!\bigl(-2\,\eps\,r\bigr)
      \\
      &\;\le\;
      (\tC + t)\,\exp(-2\,\eps\,t)
      \;\;\le\;\;
      \tC'
    \end{align*}
    for some \(\tC, \tC' \in \CC\). Substituting this back into \eqref{eq f norm 01}, we get
    \begin{align*}
      \Bigl|\EE{\rp}\bigl(f^{2}\bigr)
        \;-\;
        \sum_{t \in [0, \rkr]} 
          \EE{\rp}\bigl(f_{t}^{2}\bigr)
      \Bigr|
      &\;\le\;
      \tC_{\ref{prop layerwise main}}\,
      \frac{\Dt}{\tCB}\,\tC'
      \sum_{t \in [0, \rkr]}
        \Unorm{f_{t}}{\rp}^{2}.
    \end{align*}

    \noindent
    Finally, using
    \[
      \Dt
      \;\le\;
      \tC''\,R\,
      \exp\!\Bigl(-\tCR\,\bigl(\log(R)+1\bigr)\Bigr)
      \quad
      \text{for some }
      \tC'' \in \CC,
    \]
    we can choose \(\tCR\) sufficiently large to ensure the lemma’s claim holds. 
    This completes the proof.
    \end{proof}

\begin{proof}[Proof of Theorem \ref{theor main}]

  First, observe that
\[
  \D{\rp}\!\bigl(fg\bigr)
  \;=\;
  \sum_{t,r \in [0, \rkr]}
    \D{\rp}\!\bigl(f_t\,g_r\bigr).
\]
Applying Proposition~\ref{prop layerwise main} to each term in this sum yields
\begin{align*}
  \maxnorm{\D{\rp}\!\bigl(fg\bigr)}
  &\;\le\;
  \sum_{t,r \in [0, \rkr]}
    \maxnorm{\D{\rp}\!\bigl(f_t\,g_r\bigr)}
  \\
  &\;\le\;
  \sum_{t,r \in [0, \rkr]} \Bigg(
    \frac{\tC_{\ref{prop layerwise main}}}{(\tCB)^{2}}\,
    \exp\bigl(-\eps\,\rkr\bigr)\,
    \exp\!\Bigl(
      -1.1\,\eps\,|t-r|
      \;-\;\eps\,\min\{t,r\}
    \Bigr)\,
    \Unorm{f_t}{\rp}\,\Unorm{g_r}{\rp} \\
  &\phantom{\;\le\;} \;+\;
  {\bf 1}(r = t)\,
  \tC_{\ref{prop layerwise main}}\,
  \Bigl(
    \tlam^{\rkr}
    + 
    \Dt\,\exp\bigl(-\eps\,\rkr\bigr)
  \Bigr)\,
  \Unorm{f_t}{\rp}\,\Unorm{g_r}{\rp} \Bigg)\,.
\end{align*}

We estimate these two sums separately.

\step{First summand}
We bound the first summand similarly to the proof of Lemma~\ref{lem f norm}:
\begin{align*}
  \sum_{t,r \in [0, \rkr]}
    \frac{\tC_{\ref{prop layerwise main}}}{(\tCB)^{2}}
    \exp\bigl(-\eps\,\rkr\bigr)\,
    \exp\!\Bigl(-1.1\,\eps\,|t-r|\;-\;\eps\,\min\{t,r\}\Bigr)\,
    \Unorm{f_t}{\rp}\,\Unorm{g_r}{\rp}.
\end{align*}
Define the matrix \(B\) of size \(\bigl(\rkr +1\bigr)\times\bigl(\rkr +1\bigr)\) by
\[
  B_{t,r}
  \;=\;
  \exp\!\Bigl(-1.1\,\eps\,|t-r|\;-\;\eps\,\min\{t,r\}\Bigr),
\]
and let
\(\vec{f} = \bigl(\Unorm{f_t}{\rp}\bigr)_{t \in [0,\rkr]}\)
and
\(\vec{g} = \bigl(\Unorm{g_r}{\rp}\bigr)_{r \in [0,\rkr]}\).
Then the sum above can be written as
\[
  (*)
  \;=\;
  \frac{\tC_{\ref{prop layerwise main}}}{(\tCB)^{2}}
  \exp\bigl(-\eps\,\rkr\bigr)\,
  (\vec{f})^{\top}\,B\,\vec{g}
  \;\;\le\;\;
  \frac{\tC_{\ref{prop layerwise main}}}{(\tCB)^{2}}
  \exp\bigl(-\eps\,\rkr\bigr)\,
  \|\vec{f}\|\;\|B\|\;\|\vec{g}\|,
\]
where \(\|B\|\) is the operator norm of the matrix \(B\), and
\[
  \|\vec{f}\|
  \;=\;
  \sqrt{\sum_{t \in [0,\rkr]} \Unorm{f_t}{\rp}^{2}},
  \quad
  \|\vec{g}\|
  \;=\;
  \sqrt{\sum_{r \in [0,\rkr]} \Unorm{g_r}{\rp}^{2}}.
\]
By Lemma~\ref{lem f norm}, we also have
\[
  \|\vec{f}\|
  \;\le\;
  \sqrt{(1+\kappa)\,\EE{\rp}\bigl(f^{2}\bigr)}
  \;=\;
  \sqrt{1+\kappa}\,\Unorm{f}{\rp},
  \quad
  \|\vec{g}\|
  \;\le\;
  \sqrt{1+\kappa}\,\Unorm{g}{\rp}.
\]
Hence,
\[
  (*)
  \;\le\;
  (1+\kappa)
  \,\frac{\tC_{\ref{prop layerwise main}}}{(\tCB)^{2}}
  \exp\bigl(-\eps\,\rkr\bigr)\,
  \Unorm{f}{\rp}\,\Unorm{g}{\rp}\,\|B\|.
\]
    \step{Bounding the operator norm of \(B\)}
    Since \(B\) is a symmetric matrix, there exists a unit vector \(\vec{v}\) (i.e., \(\|\vec{v}\| = 1\)) such that
    \[
      \|B\|
      \;=\;
      \vec{v}^{\top}\,B\,\vec{v}.
    \]
    Furthermore,
    \begin{align*}
      \|B\|
      &\;=\;
      \vec{v}^{\top}\,B\,\vec{v}
      \;=\;
      \sum_{t,r} 
        v_{t}\,B_{t,r}\,v_{r}
      \;\le\;
      \sum_{t,r} 
        |v_{t}|\,
        B_{t,r}\,
        |v_{r}|,
    \end{align*}
    where the last inequality follows from the non-negativity of \(B_{t,r}\).

    \noindent
    Next, we reduce to a scenario similar to that in Lemma~\ref{lem f norm}. Notice:
    \begin{align*}
      \sum_{t,r} 
        |v_{t}|\,
        B_{t,r}\,
        |v_{r}|
      &\;\le\;
      \sum_{t,r}
        \frac{|v_{t}|^{2} 
              \;+\;
              |v_{r}|^{2}}{2}
        \,B_{t,r}
      \;=\;
      \sum_{t}
        |v_{t}|^{2}
        \Bigl(
          \sum_{r} B_{t,r}
        \Bigr).
    \end{align*}
    We now analyze the term \(\sum_{r} B_{t,r}\) with respect to \(t\):
    \begin{align*}
      \sum_{r} B_{t,r}
      &\;=\;
      \sum_{r}
        \exp\Bigl(
          -1.1\,\eps\, |t-r|
          \;-\;\eps\,\min\{t,r\}
        \Bigr)
      \\
      &\;\le\;
      \sum_{r}
        \exp\Bigl(
          -\eps\,\max\{t,r\}
        \Bigr)
      \;\le\;
      \tC
    \end{align*}
    for some constant \(\tC \in \CC\). Consequently,
    \[
      \sum_{t,r} 
        |v_t|\,
        B_{t,r}\,
        |v_r|
      \;\le\;
      \sum_{t}
        |v_{t}|^{2}\,\tC
      \;=\;
      \tC.
    \]
    Thus,
    \[
      \|B\|
      \;=\;
      \vec{v}^{\top}\,B\,\vec{v}
      \;\le\;
      \tC.
    \]
    Returning to our previous bound, we conclude that the first summand can be further bounded by
    \[
      (1+\kappa)\,
      \frac{\tC_{\ref{prop layerwise main}}}{(\tCB)^{2}}\,\tC\,
      \exp\bigl(-\eps\,\rkr\bigr)\,
      \Unorm{f}{\rp}\,\Unorm{g}{\rp}.
    \]

    \step{Second summand}
\begin{align*}
    &\;\;\sum_{t,r \in [0, \rkr]}
    \mathbf{1}\{r=t\}\;\tC_{\ref{prop layerwise main}}
    \Bigl(\tlam^{\rkr} + \Dt\,\exp\bigl(-\eps\,\rkr\bigr)\Bigr)\,
    \Unorm{f_t}{\rp}\,\Unorm{g_r}{\rp}
    \\
    &= 
    \tC_{\ref{prop layerwise main}}
    \Bigl(\tlam^{\rkr} + \Dt\,\exp\bigl(-\eps\,\rkr\bigr)\Bigr)\,
    \sum_{t}
      \Unorm{f_t}{\rp}\,\Unorm{g_t}{\rp}
    \\
    &\;\le\;
    \tC_{\ref{prop layerwise main}}
    \Bigl(\tlam^{\rkr} + \Dt\,\exp\bigl(-\eps\,\rkr\bigr)\Bigr)\,
    \sqrt{\sum_{t} \Unorm{f_t}{\rp}^{2}}
    \;\sqrt{\sum_{t} \Unorm{g_t}{\rp}^{2}}
    \\
    &\;\le\;
    (1+\kappa)\,\tC_{\ref{prop layerwise main}}
    \Bigl(\tlam^{\rkr} + \Dt\,\exp\bigl(-\eps\,\rkr\bigr)\Bigr)\,
    \Unorm{f}{\rp}\,\Unorm{g}{\rp},
\end{align*}
where we use Lemma~\ref{lem f norm} and the Cauchy--Schwarz inequality in the last two steps.

\noindent
Recall from Definition~\ref{def varepsilon} that \(\tlam^{\rkr} \le \exp\bigl(-\eps\,\rkr\bigr)\), 
and \(\Dt\) becomes arbitrarily small when \(\tCR\) is large. Hence,
\[
  (*)
  \;\le\;
  2\,(1+\kappa)\,\tC_{\ref{prop layerwise main}}
  \,\exp\bigl(-\eps\,\rkr\bigr).
\]

\step{Final bound}
Combining this with our earlier estimates for the first summand gives
\[
  \maxnorm{\D{\rp}\bigl(fg\bigr)}
  \;\le\;
  \tC'\,\exp\bigl(-\eps\,\rkr\bigr)\,
  \Unorm{f}{\rp}\,\Unorm{g}{\rp},
\]
where \(\tC'\) does not depend on \(\tCB\) (since \(\tCB \ge 1\)). We can then absorb constant terms into the exponent:
\[
  \maxnorm{\D{\rp}\bigl(fg\bigr)}
  \;\le\;
  \exp\bigl(-\eps\,(\rkr - \tC'')\bigr)\,
  \Unorm{f}{\rp}\,\Unorm{g}{\rp},
\]
for some \(\tC'' \in \CC\).

\noindent
Recalling the definition of \(\rkr\), we get
\[
  \maxnorm{\D{\rp}\bigl(fg\bigr)}
  \;\le\;
  \exp\!\Bigl(
    -\eps\,
    \Bigl(\h(\rp)\;-\;\bigl(\h_K + \hB + \tC''\bigr)
    \Bigr)
  \Bigr).
\]
We may also assume \(\tCR \ge \tC''\). Thus,
\[
  \h_K + \hB + \tC''
  \;\le\;
  \h_K + 2\,\tCR\,\bigl(\log(R) + 1\bigr).
\]
Therefore, for sufficiently large \(\tCR\),
\[
  \maxnorm{\D{\rp}\bigl(fg\bigr)}
  \;\le\;
  \exp\!\Bigl(
    -\eps\,
    \Bigl(\h(\rp)\;-\;\bigl(\h_K + 2\,\tCR\,(\log(R) + 1)\bigr)
    \Bigr)
  \Bigr),
  \quad
  \text{provided } \h(\rp) \ge \h_K + \hB.
\]
Since
\[
  \h_K + \hB
  \;\le\;
  \h_K + 2\,\tCR\,(\log(R) + 1),
\]
we conclude
\[
  \h_{K+1}
  \;\le\;
  \h_K + 2\,\tCR\,(\log(R) + 1).
\]

\end{proof}


\section{\bf Linear Algebra and Tensor Algebra}
Let us recall a basic fact about representing a bilinear form as a matrix.
\begin{fact}
    \label{fact positiveSemiDefinite}
    If a finite-dimensional vector space $H$ is equipped with a symmetric, semi-positive definite bilinear form $E$, then there exists a basis $\{e_i\}_{i \in A \sqcup B}$ such that the index set can be partitioned into two disjoint sets $A$ and $B$ so that $E$ can be represented as a diagonal matrix with diagonal entries $1$ corresponding to $A$ and $0$ corresponding to $B$:
    $$
        E(e_i,e_j) = \delta_{ij} {\bf 1}_A(i){\bf 1}_A(j) \mbox{ for every } i,j \in A \sqcup B\,,
    $$
    where ${\bf 1}_A(i) = 1$ if $i \in A$ and $0$ otherwise.
\end{fact}
\subsection{Projections like operators}
\begin{lemma}
    \label{lem LA_Pi}
    Let $W \subseteq V$ be two finite-dimensional vector spaces, and let $L: V \times V \to \R$ is a symmetric semi-positive definite bilinear form.
    Then, there exists a basis of $V$ that can be partitioned into four sets $B_{W,0}, B_{W,1}, B_{V,0}, B_{V,1}$ such that:
    \begin{enumerate}
        \item  $B_{W,0} \cup B_{W,1}$ forms a basis for $W$.
        \item  Representing $L$ as a matrix with respect to this basis, then $L$ is a diagonal matrix with diagonal entries $1$ corresponding to $B_{W,1}\cup B_{V,1}$, and $0$ for entries corresponding to $B_{W,0}\cup B_{V,0}$.
    \end{enumerate}

    Furthermore, define the  projection $\Pi_W: V \to W$ as
    $$
        \Pi_W \left(\sum_{v \in B_{W,0}\cup B_{W_1} \cup B_{V,0} \cup B_{V,1}} a_v v\right)
        =
        \sum_{v \in B_{W,0}\cup B_{W_1}} a_v v\,.
    $$
    This projection $\Pi_W$ satisfies the property
    $$
        L(f,g) = L(\Pi_W f,g)
    $$
    for every $f \in V$ and $g \in W$.
\end{lemma}
\begin{proof}

    \step{Radical of $L$}
    Recall the radical of $L$ is defined as
    $$
        {\rm rad}(L) = \{v \in V\,:\, L(v,w) = 0 \mbox{ for every } w \in V\}
        = \{ v \in V \,:\, L(w,v) = 0 \mbox{ for every } w \in V\}\,,
    $$
    where the equality follows from the symmetry of $L$. Here we claim that
    $$
        {\rm rad}(L) = \{v \in V\,:\, L(v,v) = 0 \}\,.
    $$
    The inclusion $\subseteq$ is trivial. For the other direction, we consider $v \in V$ such that $L(v,v) = 0$. Let $w \in V$ be arbitrary. Consider the quadratic function
    \begin{align*}
        t \mapsto L(v + tw, v + tw) \,.
    \end{align*}
    The positive semi-definiteness of $L$ implies this function is non-negative for every $t \in \R$. Next, we expand the quadratic form and get
    \begin{align*}
        L(v + tw, v + tw) = L(v,v) + 2t L(v,w) + t^2 L(w,w)  = 2t L(v,w) + t^2 L(w,w) \ge 0\,.
    \end{align*}
    If $L(w,w) =0$, this implies $L(v,w) = 0$ by taking $t =1 $ and $t=-1$. If $L(w,w) > 0$, then we can choose $t = - \frac{L(v,w)}{L(w,w)}$, which leads to
    \begin{align*}
        -\frac{L(v,w)^2}{L(w,w)} \ge 0 \Rightarrow L(v,w) = 0\,.
    \end{align*}
    Therefore, the claim follows.

    \step{Construction of the basis}

    Let $B_{W,0}$ be a basis for $W \cap {\rm rad}(L)$.
    Next, extend $B_{W,0}$ to a basis $B_{W,0} \cup B_{W,1}$ for $W$. If ${\rm span}(B_{W,0}) \neq W$, repeatedly find a new vector $w \in W$ such that $L(w,w) = 1$ and $L(w,w') = 0$ for every previously found basis element $w'$.

    \step{Proof of Claim}

    We claim that this process will terminate, resulting a basis $B_{W,0} \cup B_{W,1}$ for $W$ such that, when $L$ is represented as a matrix with respect to this basis (restricted to $W$), then it is  diagonal with entries $1$ corresponding to $B_{W,1}$, and $0$ for entries corresponding to $B_{W,0}$.

    Let $k = {\rm dim}(W) - {\rm dim}(W \setminus {\rm rad}(L))$, and assume this process has been repeated $t$ times, where $w_1,w_2,\dots, w_t$ elements found in addition to those in $B_{W,0}$.
    For any linear combination $w = \sum_{i=1}^t a_i w_i + v$, where $v$ in the span of $B_{W,0}$, we have
    \begin{align*}
        L(w,w) =  \sum_{i=1}^t a_i^2 + L(v,v)\,,
    \end{align*}
    where the equality follows from the fact that $L(w_i,w_j) = 0$ for every $i \neq j$ in the construction and $v \in {\rm rad}(L)$.

    Thus, $L(w,w) = 0$ only if all $a_i = 0$, which implies $w$ can be zero only when $w=v$. And thus, $w=0$ if and only if $a_1,\dots, a_t = 0$ and $v = 0$. Therefore,
    we conclude that $w_1,\dots, w_t$ and $B_{W,0}$ are linearly independent. In particular, this implies $t \le k$.

    Now, suppose $t < k$. Pick any $v \in
        W \setminus {\rm span}(B_{W,0} \cup \{w_1,\dots, w_t\})$, and set
    \begin{align*}
        v' = v - \sum_{i=1}^t L(v,w_i) w_i\,.
    \end{align*}
    Clearly, we have $L(v',w_i) = 0$ for every $i = 1,\dots, t$. If $L(v',v') = 0$, then $v' \in {\rm rad}(L)$, which contradicts the assumption that $v \in W \setminus {\rm rad}(L)$. Therefore, $L(v',v') > 0$.
    This allows us to set
    \begin{align*}
        w_{t+1} = \frac{v'}{\sqrt{L(v',v')}}\,.
    \end{align*}
    Thus, we conclude that the process will terminate after $k$ steps, resulting in a basis $B_{W,0} \cup B_{W,1}$ with $B_{W,1} = \{w_1,\dots, w_k\}$. The fact that $L$ is diagonal with respect to this basis, taking values $1$ for $B_{W,1}$ and $0$ for $B_{W,0}$, follows directly from the construction.

    \step{Extend the basis}

    At the same time, we can extend $B_{W,0}$ to a basis $B_{W,0} \cup B_{V,0}$ for ${\rm rad}(L)$, we can simply just keep adding vector $v \in {\rm rad}(L)$ which are linear independent with the previously found basis elements.

    Finally, extend $B_{W,0} \cup B_{W,1}$ to a basis $B_{W,0} \cup B_{W,1} \cup B_{V,0} \cup B_{V,1}$ for $V$. Use the same procedure by finding $v \in V$ such that $L(v,v) = 1$ and $L(v,v') = 0$ for every previously found basis element. This completes the construction of the basis.

    The properties of this basis and $\Pi_W$ can be directly verified by the construction.
\end{proof}
\begin{lemma}
    \label{lem LA_Pi2}
    Let $W \subseteq V$ be two finite-dimensional vector spaces, and let $L: V \times V \to \R$ is a symmetric semi-positive definite bilinear form. Suppose we have an additional vector space $H \subseteq V$ such that
    \begin{align*}
        L(f,g) = 0 \mbox{ for every } f \in H \mbox{ and } g \in W\,.
    \end{align*}
    Then, there exists a basis of $V$ that can be partitioned into three sets $B_{W,1}, B_{V,0} \supseteq H$ , and $B_{V,1}$ such that:
    \begin{enumerate}
        \item ${\rm Span}(B_{W}) \subseteq W$.
        \item Representing $L$ as a matrix with respect to this basis, then it is a diagonal matrix with diagonal entries $1$ corresponding to $B_{W,1}\cup B_{V,1}$, and $0$ for entries corresponding to $B_{V,0}$.
    \end{enumerate}
    Furthermore, if we define the projection $\Pi: V \to {\rm span}(B_W)$ as
    \begin{align*}
        \Pi \left(\sum_{v \in B_{W,1} \cup B_{V,0} \cup B_{V,1}} a_v v\right) = \sum_{v \in B_{W,1}} a_v v\,,
    \end{align*}
    This projection $\Pi$ satisfies
    \begin{enumerate}
        \item $L(f,g) = L(\Pi f,g)$ for every $f \in V$ and $g \in W$.
        \item $\Pi f = 0$ for every $f \in H$.
    \end{enumerate}
\end{lemma}
\begin{proof}
    The proof is similar to that of Lemma \ref{lem LA_Pi}, we just outline the procedure.

    The space $H$ is contained in the radical of $L$:
    $$
        H \subseteq {\rm rad}(L) = \{v \in V\,:\, L(v,v) = 0\}\,,
    $$
    where the last equality was established in the proof of Lemma \ref{lem LA_Pi}.

    We start by taking $B_{V,0}$ to be a basis of ${\rm rad}(L)$. For every vector in $v \in V \setminus {\rm rad}(L)$,  we have $L(v,v)>0$.
    Next, we extend this basis to $B_{V,0} \cup B_{W,1}$ as a basis for $W$ by the same procedure as in the proof of Lemma \ref{lem LA_Pi}. Then, we further extend to $B_{V,0} \cup B_{W,1} \cup B_{V,1}$ as a basis for $V$, by selecting $v \in V$ such that $L(v,v) = 1$ and $L(v,v') = 0$ for every previously found basis element.
    As shown in the proof of Lemma \ref{lem LA_Pi}, the properties of the basis and the first property of $\Pi$ can be directly verified by the construction.

    The second property follows from the fact that $H \subseteq {\rm rad}(L)$ and $\Pi {\rm rad}(L) = 0$.

\end{proof}

\subsection{Submultiplicativity of the tensor product operator norms}
\begin{lemma}
    \label{lem: mainTensorCore}
    Consider four vector spaces $\{H_i^\pm\}_{i = 1,2}$, each equipped with a symmetric semi-positive definite bilinear form $E_i^\pm$. Suppose there are two bilinear map $L_i: H_i^+ \times H_i^- \to \R$ for $i = 1,2$ such that
    \begin{align*}
        |L_i(f,g)| \le \delta_i \sqrt{E_i^+(f,f) E_i^-(g,g)} \mbox{ for } f \in H_i^+, g \in H_i^-\,.
    \end{align*}
    then the their tensor product: $L_1 \otimes L_2: H_1^+ \otimes H_2^+ \times H_1^- \otimes H_2^- \to \R$ defined by
    $$
        (L_1 \otimes L_2)(f_1\otimes f_2,\,g_1 \otimes g_2) = L_1(f_1,g_1)L_2(f_2,g_2),
    $$
    satisfies
    \begin{align}
        \label{eq: mainTensorCore}
        |L_1\otimes L_2 (f,g)| \le \delta_1 \delta_2 \sqrt{E_1^+\otimes E_2^+(f,f) E_1^-\otimes E_2^-(g,g) }
        \mbox{ for } f \in H_1^+\otimes H_2^+, g \in H_1^-\otimes H_2^-.
    \end{align}
\end{lemma}

Here we restate the Lemma \ref{lem: mainTensorProduct}, which will be a direct consequence of the above lemma.
\begin{lemma}
    Consider $2k$ vectors spaces each equipped with a symmetric semi-positive definite bilinear form $(H_i^\pm  , E_i^\pm)_{i \in k}$
    and consider $k$ finite sets $U_1,\dots, U_k$.
    Suppose we have two bilinear maps $L_i: H_i^+ \times H_i^- \to \R^{U_i}$ such that
    \begin{align*}
        \|L_i(f,g)\|_{\rm max} \le \delta_i \sqrt{E_i^+(f,f) E_i^-(g,g)} \mbox{ for } f \in H_i^+, g \in H_i^-,
    \end{align*}
    where
    \begin{align*}
        \maxnorm{x} = \max_{u \in U} |x(u)| \mbox{ for } x \in \R^{U_i}\,.
    \end{align*}
    Then,
    \begin{align*}
        \|\bigotimes_{i \in [k]} L_i (f,g)\|_{\rm max} \le \prod_{i \in k}\delta_i \sqrt{E_1^+\otimes E_2^+(f,f) E_1^-\otimes E_2^-(g,g) }
        \mbox{ for } f \in \bigotimes_{i \in [k]} H_k^+, g \in \bigotimes_{i \in [k]}H_i^-,
    \end{align*}
    where
    \begin{align*}
        \maxnorm{x} = \max_{u = (u_1,u_2,\dots,u_k) \in U_1\times U_2 \cdots \times U_k} |x(u)| \mbox{ for } x \in \bigotimes_{i \in [k]}\R^{U_i} =
        \R^{U_1\times U_2 \cdots \times U_k}\,.
    \end{align*}
\end{lemma}
\begin{proof}
    \underline{Reduction to the case when $k=2$} It suffices to prove the case when $k=2$. The general case follows by induction:
    Simply view $\bigotimes_{i \in [k+1]} H_i^+ =  (\bigotimes_{i \in [k]} H_i^+) \otimes H_{k+1}^+$, $\bigotimes_{i \in [k+1]} H_i^- =  (\bigotimes_{i \in [k]} H_i^-) \otimes H_{k+1}^-$. Then, it is a direct check that we apply the corollary in the two tensor products case leads to the induction step.

    \underline{Case when $k=2$.}
    Here for each $u_i \in U_i$ we define
    $$
        L_{i,u_i}(f,g) = L_i(f,g)(u_i).
    $$
    With $|L_{i,u_i}(f,g)| \le \|L_i(f,g)\|_{\rm max} \le \delta_i$, we could apply the lemma to each $L_{i,u_i}$ and get
    $$
        |L_{1,u_1}\otimes L_{2,u_2}(f,g)| \le \delta_1 \delta_2 \sqrt{E_1^+\otimes E_2^+(f,f) E_1^-\otimes E_2^-(g,g) }.
    $$
    Since it holds for every $(u_1,u_2) \in U_1 \times U_2$, it follows that
    $$
        \|L_1 \otimes L_2 (f,g)\|_{\rm max} \le \delta_1 \delta_2 \sqrt{E_1^+\otimes E_2^+(f,f) E_1^-\otimes E_2^-(g,g) }.
    $$
\end{proof}

\begin{proof} [Proof of Lemma \ref{lem: mainTensorCore}]
    Let us first show that it is suffice to prove the lemma when $H_i^\pm$ are Hilbert spaces.

    \step{Reduction to Hilbert Space}:
    By Fact \ref{fact positiveSemiDefinite}, for each space $H_i^\pm$ we can find a basis $B_{i,0}^\pm \sqcup B_{i,1}^\pm$ such that $E_i^\pm$ is represented as a diagonal matrix with diagonal entries $1$ corresponding to $B_{i,1}^\pm$ and $0$ for $B_{i,0}^\pm$. For each $f \in H_i^+$ and $g \in H_i^-$, let $f = f_0 + f_1$ and $g = g_0 + g_1$ be the decomposition according to the basis. Then, the assumption on $L_i$ implies that
    $$
        L_i(f,g)
        =
        L_i(f_0, g_0) + L_i(f_0, g_1) + L_i(f_1, g_0) +
        L_i(f_1, g_1)
        =
        L_i(f_1, g_1)\,.
    $$

    \step{Tensor Product Basis}
    Next, we consider the tensor product $H_1^+ \otimes H_2^+$, which has a basis $v \otimes w$ for $v \in B_{1,1}^+ \sqcup B_{1,0}^+$ and $w \in B_{2,1}^+ \sqcup B_{2,0}^+$. For any $f \in H_1^+ \otimes H_2^+$, we express $f = f_0 + f_1$, where $f_1$ is the component corresponding to the basis $B_{1,1}^+ \otimes B_{2,1}^+$, and $f_0$ is the rest. Similarly, for $g \in H_1^- \otimes H_2^-$, we express $g = g_0 + g_1$ in the same way. We claim that
    \begin{align*}
        L_1 \otimes L_2 (f,g)
        =
        L_1 \otimes L_2 (f_1, g_1)\,.
    \end{align*}
    To see this, consider $L_1\otimes L_2 (f_0,g_1)$. Due to the bilinear property,
    $L_1\otimes L_2 (f_0,g_1)$ is a linear combination of
    \begin{align*}
        L_1 \otimes L_2 (v^+ \otimes w^+, v^- \otimes w^-) = L_1(v^+, v^-) L_2(w^+, w^-)\,,
    \end{align*}
    with either $v^+ \in B_{1,0}^+$ or $w^+ \in B_{2,0}^+$ or both.
    Thus, either $L_1(v^+, v^-)$ or $L_2(w^+, w^-)$ must be $0$ due to the property of $L_i$. The same argument applies to $L_1\otimes L_2 (f_1,g_0)$, and $L_1\otimes L_2 (f_0,g_0)$.

    Furthermore, it is also straightforward to verify that  $E_1^+ \otimes E_2^+(f,f) = E_1^+ \otimes E_2^+(f_1,f_1)$ and $E_1^- \otimes E_2^-(g,g) = E_1^- \otimes E_2^-(g_1,g_1)$.

    Therefore, it suffices to prove the statement for $f_1$ and $g_1$. This restricts $H_i^\pm$ to the span of $B_{i,1}^\pm$, where $E_i^\pm$ acts as the identity operator with the basis. Therefore, we reduce the lemma to the case when $H_i^\pm$ are Hilbert spaces.

    Now, we consider the case when $H_i^\pm$ are Hilbert spaces.

    \step{$L_i$ as linear operator from $H_i^+$ to $H_i^-$}
    When $H_i^\pm$ are Hilbert spaces where the inner product is introduced by $E_i^\pm$, we can view $L_i$ as a linear operator from $H_i^+$ to $H_i^-$. Specifically, for each $f \in H_i^+$, the map $g \mapsto L_i(f, g)$ for $g \in H_i^-$ is a linear functional on $H_i^-$, which can be identified with a unique element $f^* \in H_i^-$ such that $L_i(f,g) = \langle f^*, g\rangle$. We write $L_if = f^*$ and $L_i(f,g) = \langle L_if ,g\rangle$. The assumption of the lemma implies that the operator norm $\|L_i\| \le \delta_i$. Further, $H_1^+ \otimes H_2^-$ is also a Hilbert space with the inner product defined by $E_1^+ \otimes E_2^+$, and the same for $H_1^- \otimes H_2^-$.

    Now, take an orthonormal basis $\{v_i^+\}$ for $H_1^+$ and $\{w_i^+\}$ for $H_2^+$, and let $\{v_k^-\}$ and $\{w_\ell^-\}$ be the corresponding orthonormal basis for $H_1^-$ and $H_2^-$, respectively.
    For any $f \in H_1^+ \otimes H_2^+$ and $g \in H_1^- \otimes H_2^-$,
    we express  $ f= \sum_{i,\alpha} a_{i\alpha} v_i^+ \otimes w_\alpha^+$ and $g = \sum_{j,\beta} b_{j\beta} v_j^- \otimes w_\beta^-$. We start with expressing $L_1 \otimes L_2(f,g)$ in terms of the basis:
    \begin{align*}
        L_1 \otimes L_2(f,g)
        = &
        \sum_{i,j,\alpha,\beta} a_{i\alpha} b_{j\beta} L_1(v_i^+, v_j^-) L_2(w_\alpha^+, w_\beta^-) \,.
    \end{align*}
    Next, we invoke the bilinear property of $L_2$ and then $L_1$ to get
    \begin{align*}
        (*)
        = &
        \sum_{i,j,\alpha}a_{i\alpha}  L_1(v_i^+, v_j^-) L_2 \Big(w_\alpha^+, b_{j\beta} \sum_{\beta} w_\beta^- \Big)
        =
        \sum_j L_2 \Big( \sum_{i,\alpha} a_{i\alpha}  L_1(v_i^+, v_j^-) w_\alpha^+,\, \sum_\beta b_{j\beta} w_\beta^- \Big)\,.
    \end{align*}
    Now, to estimate the absolute value, we first have
    \begin{align*}
        |L_1 \otimes L_2(f,g)|
        \le &
        \sum_{j}
        \bigg| L_2 \Big( \sum_{i,\alpha} a_{i\alpha}  L_1(v_i^+, v_j^-) w_\alpha^+,\, \sum_\beta b_{j\beta} w_\beta^- \bigg) \Big|\,.
    \end{align*}
    For each $j$, we apply the assumption of $L_2$ to get
    \begin{align*}
        \bigg| L_2 \Big( \sum_{i,\alpha} a_{i\alpha}  L_1(v_i^+, v_j^-) w_\alpha^+,\, \sum_\beta b_{j\beta} w_\beta^- \bigg) \Big|
        \le &
        \delta_2 \sqrt{\sum_{\alpha} \Big(  L_1\big(\sum_i a_{i\alpha}v_i^+, v_j^-\big)\Big)^2} \sqrt{ \sum_\beta b_{j\beta}^2} \,,
    \end{align*}
    where we also relied on $\{w_\alpha^+\}$ and $\{w_\beta^-\}$ being orthonormal basis of $H_2^+$ and $H_2^-$, respectively. Substituting this back to the above estimate, we get
    \begin{align*}
        |L_1 \otimes L_2(f,g)|
        \le &
        \delta_2
        \sum_j  \sqrt{\sum_{\alpha} \Big(  L_1\big(\sum_i a_{i\alpha}v_i^+, v_j^-\big)\Big)^2} \sqrt{ \sum_\beta b_{j\beta}^2}
        \le
        \delta_2
        \sqrt{\sum_{j,\alpha} \Big(  L_1\big(\sum_i a_{i\alpha}v_i^+, v_j^-\big)\Big)^2} \sqrt{ \sum_{j,\beta} b_{j\beta}^2} \,,
    \end{align*}
    where the last inequality follows from the Cauchy-Schwarz inequality.

    Now, we consider $\sum_{j,\alpha} \Big(  L_1\big(\sum_i a_{i\alpha}v_i^+, v_j^-\big)\Big)^2$. Let us fix $\alpha$ and sum over $j$:
    \begin{align*}
        \sum_j \Big(  L_1\big(\sum_i a_{i\alpha}v_i^+, v_j^-\big)\Big)^2
        =
        \sum_j \Big  \langle \underbrace{L_1(\sum_i a_{i\alpha}v_i^+)}_{\in H_1^-},  v_j^- \Big\rangle ^2
        =
        \Big\| L_1(\sum_i a_{i\alpha}v_i^+) \Big\|^2
        \le \delta_1^2 \sum_i a_{i\alpha}^2\,.
    \end{align*}
    where we used the fact that $\{v_j^-\}$ is an orthonormal basis of $H_1^-$, the assumption that the operator norm of $L_1$ is bounded by $\delta_1$, and then $\{v_i^+\}$ is an orthonormal basis of $H_1^+$. Substituting this back to the above estimate, we get
    \begin{align*}
        |L_1 \otimes L_2(f,g)|
        \le
        \delta_1\delta_2 \sqrt{\sum_{i,\alpha}a_{i\alpha}^2} \sqrt{\sum_{j,\beta} b_{j\beta}^2}
        =
        \delta_1\delta_2 \|f\| \|g\|\,,
    \end{align*}
    and the lemma follows.

\end{proof}

Here we restate the Lemma \ref{lem tensorNorm}.
\begin{lemma}
    Consider $k$ finite-dimensional vector spaces $H_1, H_2, \dots, H_k$, each equipped with a symmetric, semi-positive definite bilinear form
    $E_i: H_i \times H_i \to \R$ for $i = 1,2,\dots, k$.
    Suppose there exists linear operators $L_i: H_i \rightarrow  H_i$ for $i \in [k]$ such that
    \begin{align*}
        E_i(L_i(f), L_i(f)) \le \delta_i E_i(f,f) \mbox{ for every } f \in H_i\,,
    \end{align*}
    for some $\delta_i \ge 0$.
    Then it follows that
    \begin{align*}
        \bigotimes_{i=1}^k E_i \left(
        \bigotimes_{i=1}^k L_i f, \bigotimes_{i=1}^k L_i f \right)
        \le \prod_{i=1}^k \delta_i
        \bigotimes_{i=1}^k E_i(f,f) \mbox{ for every } f \in \bigotimes_{i=1}^kH_i \,.
    \end{align*}
\end{lemma}
\begin{proof}

    Indeed, it is sufficient to prove the lemma for $k=2$, as the general case can always be reduced to this case by induction.
    For example, suppose the lemma holds for $k-1$. In the case $k$, we can treat
    $H_2 \otimes H_3 \otimes \dots \otimes H_k$ as a single vector space $H_2'$ equipped with the bilinear form $E_2 \otimes E_3 \otimes \dots \otimes E_k$. Further, define $L_2' = L_2 \otimes L_3 \otimes \dots \otimes L_k$, which satisfies the inequality condition with $\delta_2' = \delta_2 \delta_3 \dots \delta_k$ by the induction hypothesis. This reduces the case $k$ to the case $2$. It is sufficient to prove the lemma for $k=2$.

    Recalling Fact \ref{fact positiveSemiDefinite},
    we can choose bases for $H_1$ and $H_2$, say $e_1,\dots, e_{d_1}$ for $H_1$ and $u_1,\dots, u_{d_2}$ for $H_2$, such that the matrices representing $E_1$ with respect to $\{e_1,\dots, e_{d_1}\}$ and $E_2$ with respect to $\{u_1,\dots, u_{d_2}\}$ are diagonal matrices with diagonal entries either $1$ or $0$.

    Now, consider $f \in H_1 \otimes H_2$, and express it with respect to the basis $e_i \otimes u_j$ as
    $$f = \sum_{i,j} s_{ij} e_i \otimes u_j. $$
    Then,
    $$
        E_1 \otimes E_2(f,f) = \sum_{i,j} a_i b_j s_{ij}^2\,,
    $$
    where $a_i$ and $b_j$ are the diagonal entries of $E_1$ and $E_2$, respectively, each taking values of either $1$ or $0$.

    Let $I_2$ be the identity map on $H_2$. Then, $ L_1 \otimes I_2 f = \sum_{i,j} s_{ij} L_1(e_i) \otimes u_j$ satisfies
    \begin{align*}
        E_1 \otimes E_2 \left(  \sum_{i,j} s_{ij} L_1(e_i) \otimes u_j,\, \sum_{i,j} s_{ij} L_1(e_i) \otimes u_j \right)
        = &
        \sum_{j,j'} E_1 \otimes E_2 \left(\sum_{i} s_{ij} L_1(e_i) \otimes u_j, \sum_{i} s_{ij'} L_1(e_i) \otimes u_{j'} \right) \\
        = &
        \sum_{j,j'} E_1 (\sum_{i} s_{ij} L_1(e_i), \sum_{i} s_{ij'} L_1(e_i))   E_2(u_j,u_{j'})                                  \\
        = &
        \sum_{j} E_1 \left(\sum_{i} s_{ij} L_1(e_i), \sum_{i} s_{ij} L_1(e_i) \right)  b_j\,.
    \end{align*}
    Using the inequality condition on $L_1$, we have
    \begin{align*}
        (*)
        \le &
        \sum_{j} \delta_1 a_i s_{i,j}^2  b_j
        \le
        \delta_1 E_1\otimes E_2(f,f)\,.
    \end{align*}

    Next, express $ L_1 \otimes I_2 f$ in terms of the basis $e_i \otimes u_j$ as
    \begin{align*}
        L_1 \otimes I_2 f = \sum_{i,j} t_{ij} e_i \otimes u_j.
    \end{align*}
    Applying a similar argument to $I_1 \otimes L_2(L_1 \otimes I_2 f) $, we find that
    \begin{align*}
        E_1 \otimes E_2 \left(  \sum_{i,j} t_{ij} e_i \otimes L_2(u_j),\, \sum_{i,j} t_{ij} e_i \otimes L_2(u_j) \right)
        \le &
        \delta_2 E_1\otimes E_2\Big( L_1 \otimes I_2 f, L_1 \otimes I_2 f \Big) \\
        \le &
        \delta_1 \delta_2 E_1\otimes E_2(f,f)\,.
    \end{align*}

    Finally, noting that $(I_1 \otimes L_2) \circ (L_1 \otimes I_2) = L_1 \otimes L_2$, the lemma follows.
\end{proof}

\end{document}